\documentclass[10pt,a4paper]{article}
\usepackage[utf8]{inputenc}
\usepackage[T1]{fontenc}
\usepackage{geometry}
\usepackage{amsthm}
\usepackage[french,english]{babel}
\usepackage{amssymb}
\usepackage{lmodern}
\usepackage{amsmath,amsfonts}
\usepackage{mathrsfs} 
\usepackage{dsfont}
\usepackage{stmaryrd}
\DeclareMathOperator{\dive}{div} 
\usepackage{graphicx}
\usepackage{fullpage}
\usepackage[nottoc,notlot,notlof]{tocbibind}
\usepackage[colorlinks=true,urlcolor=black,linkcolor=black,citecolor=black]{hyperref}
\numberwithin{equation}{section}
\usepackage{mathtools,array}
\usepackage{verbatimbox}
\usepackage{color}
\usepackage{bm}
\usepackage{tikz-cd}
\usepackage{pst-node}
\usetikzlibrary{matrix}
\usepackage{mathtools}
\usepackage{verbatimbox}
\usepackage{diagbox}
\usepackage{array}
\newcolumntype{C}{>{$\displaystyle} c <{$}}
\usepackage{makecell}
\usepackage{float}
\usepackage{tabu}
\usepackage{cancel}
\usepackage{lscape}
\usepackage[babel=true]{csquotes}
\usepackage{bm}
\usepackage{xcolor}
\usepackage[new]{old-arrows}
\usepackage{stackengine}
\makeatletter
\def\env@dmatrix{\hskip -\arraycolsep
	\let\@ifnextchar\new@ifnextchar
	\def\arraystretch{2}%
	\array{*{\c@MaxMatrixCols}{>{\displaystyle}c}}%
}
\usepackage{soul}

\makeatother

\DeclareFontShape{OMX}{cmex}{m}{n}{
	<-7.5> cmex7
	<7.5-8.5> cmex8
	<8.5-9.5> cmex9
	<9.5-> cmex10
}{}
\SetSymbolFont{largesymbols}{normal}{OMX}{cmex}{m}{n}
\SetSymbolFont{largesymbols}{bold}  {OMX}{cmex}{m}{n}

\usepackage{enumitem}

\begin{document}

	\title{Morse Index Stability of Branched Willmore Immersions}

	\author{Alexis Michelat\footnote{Institute of Mathematics, EPFL B, Station 8, CH-1015 Lausanne, Switzerland.\hspace{.5em} \href{alexis.michelat@epfl.ch}{alexis.michelat@epfl.ch}}}%\and \;Tristan \selectlanguage{french}Rivière\selectlanguage{english}\footnote{Department of Mathematics, ETH Zentrum, CH-8093 Zürich, Switzerland.\hspace{2.85em}\href{mailto:tristan.riviere@math.ethz.ch}{tristan.riviere@math.ethz.ch}}}
	\date{\today}
	
	\maketitle
	
	\vspace{-0.5em}
	
	\begin{abstract} 
    We show that the sum of the Morse index and the nullity of Willmore immersions of bounded energy is lower semi-continuous without assuming that the limiting immersion and the bubbles are free of branch points. Our proof is based on a refined analysis of the properties of two families of fourth-order differential operators with regular singularities that depend on a parameter equal to the order of the branch points. The most technical results that justify the length of the article are Gagliardo–Nirenberg–Rellich inequalities in degenerating annuli that are necessary to show that the eigenvalues of the index operator with respect to a suitable weight are bounded from below.  
    %We show the Morse index stability for sequences of Willmore immersions of bounded energy without assuming that the limiting immersion and the bubbles are free of branch points. The analysis is based on our previous work in collaboration with Tristan Rivière (Memoirs of the European Mathematical Society, 2025)
	\end{abstract}

	\tableofcontents
	\vspace{0cm}
	\begin{center}
		{Mathematical subject classification : 49Q10 (primary),
        35J35, 35J48, 53A05, 53A10, 53C42.

        %• 49Q10: "Optimization of shapes other than minimal surfaces" (MSC2020)

%Secondary(optional)	

%• 35J35: "Variational methods for higher-order elliptic equations" (MSC2020)

%• 35J48: "Higher-order elliptic systems" (MSC2020)

%• 53A05: "Surfaces in Euclidean and related spaces" (MSC2020)

%• 53A10: "Minimal surfaces in differential geometry, surfaces with prescribed mean curvature" (MSC2020)

%• 53C42: "Differential geometry of immersions (minimal, prescribed curvature, tight, etc.)" (MSC2020)
		 }
	\end{center}

	\theoremstyle{plain}
	\newtheorem*{theorem*}{Theorem}
	\newtheorem{theorem}{Theorem}[section]
	\newenvironment{theorembis}[1]
	{\renewcommand{\thetheorem}{\ref{#1}$'$}%
		\addtocounter{theorem}{-1}%
		\begin{theorem}}
		{\end{theorem}}
	\renewcommand*{\thetheorem}{\Alph{theorem}}
	\newtheorem{lemme}[theorem]{Lemma}
	\newtheorem*{lemme*}{Lemma}
	\newtheorem{propdef}[theorem]{Definition-Proposition}
	\newtheorem*{propdef*}{Definition-Proposition}
	\newtheorem{prop}[theorem]{Proposition}
	\newtheorem{cor}[theorem]{Corollary}
	\theoremstyle{definition}
	\newtheorem*{definition}{Definition}
	\newtheorem{defi}[theorem]{Definition}
	\newtheorem{rem}[theorem]{Remark}
	\newtheorem*{rem*}{Remark}
	\newtheorem{rems}[theorem]{Remarks}
	\newtheorem{remimp}[theorem]{Important Remark}
	\newtheorem{exemple}[theorem]{Example}
	\newtheorem{defi2}{Definition}
	\newtheorem{propdef2}[defi2]{Proposition-Definition}
	\newtheorem{remintro}[defi2]{Remark}
	\newtheorem{remsintro}[defi2]{Remarks}
	\newtheorem{conj}{Conjecture}
	\newtheorem{question}{Open Question}
	\renewcommand\hat[1]{%
		\savestack{\tmpbox}{\stretchto{%
				\scaleto{%
					\scalerel*[\widthof{\ensuremath{#1}}]{\kern-.6pt\bigwedge\kern-.6pt}%
					{\rule[-\textheight/2]{1ex}{\textheight}}%WIDTH-LIMITED BIG WEDGE
				}{\textheight}% 
			}{0.5ex}}%
		\stackon[1pt]{#1}{\tmpbox}
	}
	\parskip 1ex
	\newcommand{\totimes}{\ensuremath{\,\dot{\otimes}\,}}
	\newcommand{\vc}[3]{\overset{#2}{\underset{#3}{#1}}}
	\newcommand{\conv}[1]{\ensuremath{\underset{#1}{\longrightarrow}}}
	\newcommand{\A}{\ensuremath{\vec{A}}}
	\newcommand{\B}{\ensuremath{\vec{B}}}
	\newcommand{\C}{\ensuremath{\mathbb{C}}}
	\newcommand{\D}{\ensuremath{\nabla}}
	\newcommand{\Disk}{\ensuremath{\mathbb{D}}}
	\newcommand{\E}{\ensuremath{\vec{E}}}
	\newcommand{\I}{\ensuremath{\mathbb{I}}}
	\newcommand{\Q}{\ensuremath{\vec{Q}}}
	\newcommand{\loc}{\ensuremath{\mathrm{loc}}}
	\newcommand{\z}{\ensuremath{\bar{z}}}
	\newcommand{\hh}{\ensuremath{\mathscr{H}}}
	\newcommand{\h}{\ensuremath{\vec{h}}}
	\newcommand{\vol}{\ensuremath{\mathrm{vol}}}
	\newcommand{\hs}[3]{\ensuremath{\left\Vert #1\right\Vert_{\mathrm{H}^{#2}(#3)}}}
	\newcommand{\R}{\ensuremath{\mathbb{R}}}
	\renewcommand{\P}{\ensuremath{\mathbb{P}}}
	\newcommand{\N}{\ensuremath{\mathbb{N}}}
	\newcommand{\Z}{\ensuremath{\mathbb{Z}}}
    \newcommand{\wconv}[1]{\ensuremath{\underset{#1}{\rightharpoonup}}}
	\newcommand{\p}[1]{\ensuremath{\partial_{#1}}}
	\newcommand{\Res}{\ensuremath{\mathrm{Res}}}
	\newcommand{\lp}[2]{\ensuremath{\mathrm{L}^{#1}(#2)}}
	\renewcommand{\wp}[3]{\ensuremath{\left\Vert #1\right\Vert_{\mathrm{W}^{#2}(#3)}}}
	\newcommand{\wpn}[3]{\ensuremath{\Vert #1\Vert_{\mathrm{W}^{#2}(#3)}}}
	\newcommand{\np}[3]{\ensuremath{\left\Vert #1\right\Vert_{\mathrm{L}^{#2}(#3)}}}
	\newcommand{\hp}[3]{\ensuremath{\left\Vert #1\right\Vert_{\mathrm{H}^{#2}(#3)}}}
	\newcommand{\ck}[3]{\ensuremath{\left\Vert #1\right\Vert_{\mathrm{C}^{#2}(#3)}}}
	\newcommand{\hardy}[2]{\ensuremath{\left\Vert #1\right\Vert_{\mathscr{H}^{1}(#2)}}}
	\newcommand{\lnp}[3]{\ensuremath{\left| #1\right|_{\mathrm{L}^{#2}(#3)}}}
	\newcommand{\npn}[3]{\ensuremath{\Vert #1\Vert_{\mathrm{L}^{#2}(#3)}}}
	\newcommand{\nc}[3]{\ensuremath{\left\Vert #1\right\Vert_{C^{#2}(#3)}}}
	\renewcommand{\Re}{\ensuremath{\mathrm{Re}\,}}
	\renewcommand{\Im}{\ensuremath{\mathrm{Im}\,}}
	\newcommand{\dist}{\ensuremath{\mathrm{dist}}}
	\newcommand{\diam}{\ensuremath{\mathrm{diam}\,}}
	\newcommand{\leb}{\ensuremath{\mathscr{L}}}
	\newcommand{\supp}{\ensuremath{\mathrm{supp}\,}}
	\renewcommand{\phi}{\ensuremath{\vec{\Phi}}}
	\renewcommand{\H}{\ensuremath{\vec{H}}}
	\renewcommand{\L}{\ensuremath{\vec{L}}}
	\renewcommand{\lg}{\ensuremath{\mathscr{L}_g}}
	\renewcommand{\ker}{\ensuremath{\mathrm{Ker}}}
	\renewcommand{\epsilon}{\ensuremath{\varepsilon}}
	\renewcommand{\bar}{\ensuremath{\overline}}
	\newcommand{\s}[2]{\ensuremath{\langle #1,#2\rangle}}
	\newcommand{\pwedge}[2]{\ensuremath{\,#1\wedge#2\,}}
	\newcommand{\bs}[2]{\ensuremath{\left\langle #1,#2\right\rangle}}
	\newcommand{\scal}[2]{\ensuremath{\langle #1,#2\rangle}}
	\newcommand{\sg}[2]{\ensuremath{\left\langle #1,#2\right\rangle_{\mkern-3mu g}}}
	\newcommand{\n}{\ensuremath{\vec{n}}}
	\newcommand{\ens}[1]{\ensuremath{\left\{ #1\right\}}}
	\newcommand{\lie}[2]{\ensuremath{\left[#1,#2\right]}}
	\newcommand{\g}{\ensuremath{g}}
	\newcommand{\dzeta}{\ensuremath{\det\hphantom{}_{\kern-0.5mm\zeta}}}
	\newcommand{\e}{\ensuremath{\vec{e}}}
	\newcommand{\f}{\ensuremath{\vec{f}}}
	\newcommand{\ig}{\ensuremath{|\vec{\mathbb{I}}_{\phi}|}}
	\newcommand{\ik}{\ensuremath{\left|\mathbb{I}_{\phi_k}\right|}}
	\newcommand{\w}{\ensuremath{\vec{w}}}
	\newcommand{\hooklongrightarrow}{\lhook\joinrel\longrightarrow}
	\renewcommand{\tilde}{\ensuremath{\widetilde}}
	\newcommand{\vg}{\ensuremath{\mathrm{vol}_g}}
	\newcommand{\im}{\ensuremath{\mathrm{W}^{2,2}_{\iota}(\Sigma,N^n)}}
	\newcommand{\imm}{\ensuremath{\mathrm{W}^{2,2}_{\iota}(\Sigma,\R^3)}}
	\newcommand{\timm}[1]{\ensuremath{\mathrm{W}^{2,2}_{#1}(\Sigma,T\R^3)}}
	\newcommand{\tim}[1]{\ensuremath{\mathrm{W}^{2,2}_{#1}(\Sigma,TN^n)}}
	\renewcommand{\d}[1]{\ensuremath{\partial_{x_{#1}}}}
	\newcommand{\dg}{\ensuremath{\mathrm{div}_{g}}}
	\renewcommand{\Res}{\ensuremath{\mathrm{Res}}}
	\newcommand{\un}[2]{\ensuremath{\bigcup\limits_{#1}^{#2}}}
	\newcommand{\res}{\mathbin{\vrule height 1.6ex depth 0pt width
			0.13ex\vrule height 0.13ex depth 0pt width 1.3ex}}
    \newcommand{\antires}{\mathbin{\vrule height 0.13ex depth 0pt width 1.3ex\vrule height 1.6ex depth 0pt width 0.13ex}}
	\newcommand{\ala}[5]{\ensuremath{e^{-6\lambda}\left(e^{2\lambda_{#1}}\alpha_{#2}^{#3}-\mu\alpha_{#2}^{#1}\right)\left\langle \nabla_{\e_{#4}}\vec{w},\vec{\mathbb{I}}_{#5}\right\rangle}}
	\setlength\boxtopsep{1pt}
	\setlength\boxbottomsep{1pt}
	\newcommand\norm[1]{%
		\setbox1\hbox{$#1$}%
		\setbox2\hbox{\addvbuffer{\usebox1}}%
		\stretchrel{\lvert}{\usebox2}\stretchrel*{\lvert}{\usebox2}%
	}
	\allowdisplaybreaks
	\newcommand*\mcup{\mathbin{\mathpalette\mcapinn\relax}}
	\newcommand*\mcapinn[2]{\vcenter{\hbox{$\mathsurround=0pt
				\ifx\displaystyle#1\textstyle\else#1\fi\bigcup$}}}
	\def\Xint#1{\mathchoice
		{\XXint\displaystyle\textstyle{#1}}%
		{\XXint\textstyle\scriptstyle{#1}}%
		{\XXint\scriptstyle\scriptscriptstyle{#1}}%
		{\XXint\scriptscriptstyle\scriptscriptstyle{#1}}%
		\!\int}
	\def\XXint#1#2#3{{\setbox0=\hbox{$#1{#2#3}{\int}$ }
			\vcenter{\hbox{$#2#3$ }}\kern-.58\wd0}} 
	\def\ddashint{\Xint=}
	\newcommand{\dashint}[1]{\ensuremath{{\Xint-}_{\mkern-10mu #1}}}
	\newcommand\ccancel[1]{\renewcommand\CancelColor{\color{red}}\cancel{#1}}
	\newcommand\colorcancel[2]{\renewcommand\CancelColor{\color{#2}}\cancel{#1}}
	\newcommand{\abs}[1]{\left\lvert #1 \right \rvert}
	
	\renewcommand{\thetheorem}{\thesection.\arabic{theorem}}

 \section{Introduction}
 
 \subsection{Overture}
 
 Geometers, as Chern predicted in a famous interview, can no longer restrict themselves to the smooth category. In this article, we address the second analytic step of the long-standing research program initiated by Tristan Rivière a decade ago (\cite{eversion}) to solve Kusner's conjecture—that consists in finding the optimal sphere eversion. After Smale proved in $1957$ that the path of immersions from the round sphere $S^2$ into $\R^3$ is path-connected (\cite{smale}), geometers struggled to find an explicit path of immersions that connects the standard immersion of the round sphere $\iota:S^2\rightarrow \R^3$ to the antipodal embedding $-\iota$. Examples were eventually constructed by several mathematicians, but were quite involved and it was unclear if one could make them simpler. In $1983$, Robert Bryant classified the Willmore spheres in $\R^3$ and showed that they are all inversions of minimal surfaces with embedded planar ends (\cite{bryant})—this result was partially extended to various settings (\cite{lamm,classification}) before Dorian Martino finally established it to all branched immersions a couple of years ago (\cite{martino_classification}). The relevance of this result lies in two facts: the Willmore energy can be seen as a distance function from the immersion of the round sphere and it should be possible to compute the Morse index (for the Willmore energy) of inversions of minimal surfaces. The second point was actually not considered until T. Rivière shared his vision with me at the end of $2015$ (\cite{eversion}). It should not come as a surprise to anyone aware of the important recent developments of geometric analysis that once more, this vision was correct (\cite{indexS3,index2}; see also \cite{hirsch,elena}). The first observation led Robert Kusner (\cite{kusner}) to imagine that one may flow the Willmore energy from a given Willmore surface to the round sphere. The idea was to find a surface that would have the suitable geometric properties to be also flowed to the antipodal embedding of the sphere. Amongst the surfaces classified by R. Bryant, the first non-trivial family has energy $16\pi$. It was therefore natural to assume that the \emph{cost} of the sphere eversion, or in other words, the width of the min-max associated with Kusner's conjecture, is equal to $16\pi$.

 Although the approach using gradient flows did not bear fruits in this direction so far (see however \cite{palmurella_flow1,palmurella_flow2}), T. Rivière designed an alternative approach using min-max methods. Before describing it, let us introduce the Willmore energy by a rather unusual approach that draws ideas from particle physics. It makes the Willmore energy proceed from an alternative to the problem of Plateau, that consists in finding a surface of prescribed boundary that minimises the area. This problem caught the attention of countless mathematicians and constitutes the third important development in the calculus of variations after a proof of the existence of a solution to the Dirichlet's problem by Hilbert in 1900 and a proof of the existence of a non-trivial closed geodesic on any simply connected surface by Birkhoff in 1917. Assuming the reader familiar with the rudiments of Sobolev spaces, we recall that for any domain $\Omega\subset \R^d$ (where $d\geq 1$), the Dirichlet energy $E:W^{1,2}(\Omega)\rightarrow \R$ is defined by:
\begin{align*}
   E(u)=\frac{1}{2}\int_{\Omega}|\D u|^2dx.
\end{align*}
Independently solved in 1930 by Radó and in 1931 by Douglas for curves in Euclidean spaces, the problem of Plateau can also be formulated in Riemannian manifolds and treated by similar methods. The most obvious generalisation is to look at model spaces, namely the round sphere equipped with its induced metric from its embedding into the Euclidean space and the hyperbolic space. In the latter example, we place ourselves in the half-space model $\mathbb{\R}^3_+=\R^2\times\R_+^*$ equipped with the metric
\begin{align*}
g_{\mathrm{hyp}}=\frac{dx^2+dy^2+dz^2}{z^2}
\end{align*}
of constant sectional curvature $-1$. The problem of Plateau in this setting can be formulated as follows: for any Jordan curve $\Gamma\subset\partial_{\infty}\mathbb{H}^3=\R^2\times\ens{0}$ (here, this notation of asymptotic boundary means that we take the boundary of the space for a conformal change of metric that makes it smooth up to the boundary), does there exist a surface $\Sigma\subset\mathbb{H}^3$ that locally minimises the area and such that $\partial_{\infty}\Sigma=\Gamma$? Anderson (\cite{anderson}) proved the existence of a solution to this problem for any rectifiable curve, and a solution happens to be a current of finite mass as a subset of this Euclidean space. However, due to the singularity of the metric at infinity, its area is infinite. This is where an important idea of physics—renormalisation—comes into play. The idea, that in this specific context can be attributed to Graham--Witten (\cite{graham}), consists in computing the area of
\begin{align*}
\Sigma_{\epsilon}=\Sigma\cap\ens{(x,y,z): z>\epsilon}
\end{align*}
and to remove the singularity terms in $\epsilon$ to get a finite quantity. Although the final result depends on the choice of the defining function and the limit has no reason to be finite, it was identified by Alexakis--Mazzeo (\cite{alexakismazzeo1}; see also \cite{alexakismazzeo2}) who proved the following result.
\begin{theorem}[Alexakis--Mazzeo]
Let $0<\alpha<1$ and $\Gamma\subset \partial_{\infty}\mathbb{H}^3$ be a $C^{3,\alpha}$ curve and $\Sigma\subset\mathbb{H}^3$ be a minimal surface such that $\partial_{\infty}\Sigma=\Gamma$. Then, the renormalised area defined by
\begin{align*}
\mathscr{RA}(\Sigma)=\lim_{\epsilon\rightarrow 0}\left(\mathrm{Area}(\Sigma_{\epsilon})-\frac{\mathscr{L}(\partial \Sigma_{\epsilon})}{\epsilon}\right)
\end{align*}
is finite and the following identity holds:
\begin{align*}
\mathscr{RA}(\Sigma)=-2\pi\chi(\Sigma)-\frac{1}{2}\int_{\Sigma}|\mathring{A}|^2\,d\mathrm{vol}_{\Sigma},
\end{align*}
where $\chi(\Sigma)$ is the Euler characteristic and $\mathring{A}$ is the trace-free second fundamental form of $\Sigma$.
\end{theorem}
Now, we will adopt a parametric approach and fix a Riemann surface $\Sigma$ and an immersion $\phi:\Sigma\rightarrow\mathbb{H}^3$. The quantity
\begin{align*}
\mathscr{W}(\phi)=\frac{1}{2}\int_{\Sigma}|\mathring{A}|^2\,d\mathrm{vol}_g,
\end{align*}
where $g=\phi^*g_{\mathrm{hyp}}$ is the induced metric on $\Sigma$, is known as the Willmore energy. For immersions into the Euclidean space, the Willmore energy is usually defined as follows:
\begin{align*}
W(\phi)=\int_{\Sigma}H^2\,d\mathrm{vol}_g
\end{align*}
where $H$ is the mean curvature. 
Thanks to Gauss--Bonnet theorem, both functionals only differ by a topological constant and are therefore analytically equivalent. A property first established by Blaschke and Thomsen in 1923 (to be precise, one can find it already in the Doktorarbeit of Thomsen in 1922, and a certain W. Schadow—whose entire existence is shrouded with mystery—is also credited for deriving the Euler--Lagrange of the Willmore energy) is the conformal invariance of the Willmore energy (\cite{blaschke,thomsen}). In other words, it is invariant by ambient translations, rotations, dilations, and inversions. More precisely, the quantity $|\mathring{A}|^2d\vg$ is a pointwise conformally invariant quantity, which implies that inversions of complete minimal surfaces—for which the mean curvature $H$ vanishes identically—with finite total curvature are compact Willmore surfaces. However, in general, those surfaces may have branch points, where the map $\phi$ can be expanded in a conformally chart as a function which is harmonic at first order and satisfies (according to work of Kuwert--Schätzle \cite{kuwert} and Bernard--Rivière \cite{beriviere})
\begin{align*}
    \phi(z)=\Re\left(\vec{A}_0z^m\right)+O(|z|^{m+1}),
\end{align*}
where $m\geq 2$ is an integer. Those are the immersions that we consider in this article, or more precisely, we place ourselves in the setting of \cite{willmore_scs} but allow the limiting immersions and the bubbles to develop branch points.

\subsection{Technical Aspects}

As in \cite{willmore_scs}, an important first step in the proof is to show that negative variations cannot be located in neck regions. This is accomplished by using the Rellich and Hardy--Rellich inequalities already established in \cite[Chapter 2]{willmore_scs}, and more refined estimates using the properties of the family of operators:
\begin{align*}
\mathscr{L}_m=|x|^{1-m}\Delta(|x|^{m-1}(\,\cdot\,))=\Delta+2(m-1)\frac{x}{|x|^2}\cdot \D+\frac{(m-1)^2}{|x|^2},
\end{align*}
where $m\geq 1$ is the \emph{integer} multiplicity of the branch point. 
Using the decomposition of the forth-order elliptic differential operator with regular singularities $\leb_m^{\ast}\leb_m$ as the non-trivial sum of two positive operators and integrating by parts the weighted $L^2$ norm of $\leb_mu$, it is rather straightforward to show that for all $m>1$ and $\alpha>0$, there exists a universal constant $C_{\alpha}<\infty$ such that for all $0<a<b<\infty$, defining $\Omega=B_b\setminus \bar{B}_a(0)$, for all $u\in W^{2,2}_0(\Omega)$, we have
\begin{align*}
\int_{\Omega}\left(\frac{u^2}{|x|^4}+\frac{|\D u|^2}{|x|^2}\right)\left(\left(\frac{|x|}{b}\right)^{2\alpha}+\left(\frac{a}{|x|}\right)^{2\alpha}\right)dx\leq C_{\alpha}\int_{\Omega}\left(\leb_mu\right)^2dx.
\end{align*}
From this inequality and the refined estimates established for the second fundamental form of Willmore immersions in neck regions, one can show that negative variations cannot be located in the neck regions. However, we face another difficulty that is far more challenging than in the unbranched case $m=1$ when we try to establish that eigenvalues for the weighted Willmore operator are bounded from below. As in \cite{willmore_scs}, the proof is based on a rather involved weighted Gagliardo--Nirenberg inequality. In the case $m=1$, by a standard decomposition of any function $u\in W^{2,2}(\Omega)$ between a function that vanishes on the boundary and a function that solves a Dirichlet problem, the estimate followed from establishing the inequality for biharmonic functions. The obvious generalisation is to consider solutions of $\leb_m^{\ast}\leb_mu=0$, but if the required estimate does indeed hold, it does not suffice for our purpose. Indeed, in the next argument using the Bochner identity, due to the non-flatness of the conical metric in neck regions, comparing the Bochner identity for the flat metric on which a weight is applied \emph{a posteriori} (this is the expression furnished by our Gagliardo--Nirenberg inequality), in order to estimate one of the extra terms, it is necessary to have a weight on the right-hand side of the Gagliardo--Nirenberg inequality. However, such an estimate is \emph{neither} verified for functions in $W^{2,2}_0(\Omega)$ \emph{nor} for solutions of $\leb_m^{\ast}\leb_mu=0$ (due to the loss of the logarithmic singularity once integrated with respect to a weight). The situation may seem hopeless but the remedy is to notice that the correct generalisation of biharmonic functions in our context is not solutions of the partial differential equation $\leb_m^{\ast}\leb_mu=0$ but (real) solutions of
\begin{align*}
\mathfrak{D}_mu=16\,\Re\left(\bar{\mathfrak{L}_m}^{\ast}\mathfrak{L}_mu\right)=0,
\end{align*}
where
\begin{align*}
\mathfrak{L}_m=|z|^{1-m}\p{z}^2\left(|z|^{m-1}(\,\cdot\,)\right)=\p{z}^2+\frac{(m-1)}{z}\p{z}+\frac{(m-1)(m-3)}{4z^2}.
\end{align*}
Notice that both operators both have leading term given by the bi-Laplacian and have regular singularities (\cite{PR}). Otherwise, their properties radically differ, although it may not be obvious to see on their expanded expressions:
\begin{align*}
            \mathfrak{D}_m&=\Delta^2-\frac{6(m^2-1)}{|x|^2}\Delta +4(m^2-1)\left(\frac{x}{|x|^2}\right)^t\cdot \D^2(\,\cdot\,)\cdot \left(\frac{x}{|x|^2}\right)+\frac{8(m-1)^2}{|x|^2}\frac{x}{|x|^2}\cdot \D \\
            &+\frac{(m+1)(m-1)^2(m-3)}{|x|^4},
        \end{align*}
and
        \begin{align*}
            \leb_m^{\ast}\leb_m=\Delta^2+\frac{2(m^2-1)}{|x|^2}\Delta -4(m^2-1)\left(\frac{x}{|x|^2}\right)^t\cdot \D^2(\,\cdot\,)\cdot \left(\frac{x}{|x|^2}\right)+\frac{(m-1)^2}{|x|^4}.
        \end{align*}
In other words, instead of seeing the biharmonic energy as 
\begin{align*}
\int_{\Omega}\left(\Delta u\right)^2dx=16\int_{\Omega}\left(\p{z\z}^2u\right)^2dx,
\end{align*}
we see it as
\begin{align*}
16\int_{\Omega}|\p{z}^2u|^2|dz|^2.
\end{align*}
and by an immediate integration by parts, both functionals coincide on $W^{2,2}_0(\Omega)$. From the analytical viewpoint, the advantage of working with $\mathfrak{D}_m$ (seen as an operator acting on \emph{real-valued} functions) has a four-dimensional Kernel (while the Kernel of $\leb_m^{\ast}\leb_m$ is infinite-dimensional—is it equal to the space of harmonic functions weighted by $|x|^{1-m}$). This explains why the doubly weighted Rellich and Hardy--Rellich inequalities for $\mathfrak{L}_m$ do hold true on $W^{2,2}_0(\Omega)$, as long as we assume that $m\geq 3$. This restriction on $m$ may seem problematic, but in codimension $1$, Laurain--Rivière showed that the branch point must be odd due to a connexity argument on the first homotopy group of $\mathrm{SO}(3)$ (\cite{quantamoduli}). Although it seems reductive, for the previous Gagliardo--Nirenberg inequality, we need to assume 
\begin{align*}
    m&>\frac{2}{3}+\frac{1}{3}\sqrt{-1-\frac{2^{\frac{2}{3}}}{\sqrt[3]{32+3\sqrt{114}}}+\sqrt[3]{2(32+3\sqrt{114})}}\\
        &+\frac{1}{2}\sqrt{-\frac{8}{9}+\frac{2^{\frac{8}{3}}}{9\sqrt[3]{32+3\sqrt{114}}}-\frac{4}{9}\sqrt[3]{2(32+3\sqrt{114})}+\frac{88}{9\sqrt{-1-\frac{2^{\frac{2}{3}}}{\sqrt[3]{32+3\sqrt{114}}}+\sqrt[3]{2(32+3\sqrt{114})}}}}\\
        &=2.039423\cdots
\end{align*}
and this condition seems unavoidable. Furthermore, the radial terms can be estimated without having to use the delicate argument involving a refinement of the Cauchy--Schwarz inequality that was necessary for the Gagliardo--Nirenberg inequality for solutions of $\leb_m^{\ast}\leb_mu=0$ due to the logarithm components. However, the Gagliardo--Nirenberg inequality for solutions of $\mathfrak{D}_mu=0$ is excessively more technical for the Taylor expansion mostly (that is, for all except the first two Fourier frequencies) involves transcendent functions that have no primitive expressible with respect to standard functions (and the roots are quite involved too), which makes it especially difficult to estimate the \enquote{crossed terms} of the Fourier coefficients in the various $L^2$ norms. This forces us to use the refined Cauchy--Schwarz inequality for most frequencies, but the exact technical details would be tedious to explain here so let us simply state the useful lemma that involves the Cauchy--Schwarz identity that is instrumental in the proof. 
\begin{lemme}
Let $(X,\mu)$ be a $\sigma$-finite measured space and $f_1,f_2\in L^2(X,\mu,\C)$. Then, for all $\lambda_1,\lambda_2\in \C$, we have
\begin{align*}
\frac{|\lambda_1||\lambda_2|}{2}\int_{X\times X}\left|\det\begin{pmatrix}f_1(x) & f_2(x)\\
f_1(y) & f_2(y)
\end{pmatrix}\right|^2d\mu(x)\,d\mu(y)\leq \np{f_1}{2}{X}\np{f_2}{2}{X}\int_{X}\left|\lambda_1\,f_1+\lambda_2\,f_2\right|^2d\mu.
\end{align*}
\end{lemme}

\subsection{Main Theorem}

\begin{theorem}\label{main_theorem}
    Let $n\geq 3$. Then, there exists a universal constant $0<\Lambda_n<\infty$ with the following property. Let $\Sigma$ be a closed Riemann surface and $\{\phi_k\}_{k\in\N}\subset \mathrm{Imm}(\Sigma,\R^n)$ be a sequence of Willmore immersions of bounded energy:
    \begin{align*}
        \limsup_{k\rightarrow \infty}W(\phi_k)<\infty.
    \end{align*}
    Assume that $\{\phi_k\}_{k\in \N}$ bubble-converges to $(\phi_{\infty}^1,\cdots,\phi_{\infty}^m,\vec{\Psi}_1,\cdots,\vec{\Psi}_p,\vec{\chi}_1,\cdots,\vec{\chi}_q)$, and assume that \textbf{either} $n=3$, \textbf{or} that no limiting Willmore surface has branched points of order $m=2$ and let $\ens{\gamma_k^1,\cdots,\gamma_k^N}\subset \Sigma$ be the set of shrinking geodesics of $(\Sigma,g_k)$. There exists a universal constant $\Lambda_n>0$ such that the bound
    \begin{align}\label{smallness_residue_2}
        \limsup_{k\rightarrow \infty}\max_{1\leq l\leq N}\frac{|\vec{\gamma}_1(\phi_k,\gamma_k^l)|}{\leb(\gamma_k^l)}\leq \Lambda_n
    \end{align}
    implies that 
    \begin{align}\label{scs_branched}
        \limsup_{k\rightarrow \infty}\left(\mathrm{Ind}_W(\phi_k)+\mathrm{Null}_W(\phi_k)\right)\leq \sum_{l=1}^m\mathrm{Ind}_W^0(\phi_{\infty}^l)+\sum_{i=1}^p\mathrm{Ind}_W^0(\vec{\Psi}_{i})+\sum_{j=1}^q\mathrm{Ind}_W^0(\vec{\chi}_{j})<\infty
    \end{align}
    where $\mathrm{Ind}^0_W=\mathrm{Ind}_W+\mathrm{Null}_W$.
\end{theorem}
For the definition of bubble-convergence, see \cite{quanta} and \cite[Definition 1.1.6]{willmore_scs}. Notice that the quantization and bubble-convergence always hold thanks to the combined work of Bernard--Rivière \cite{quanta} and Laurain--Rivière \cite{quantamoduli} (in the latter case, this is because our condition \eqref{smallness_residue_2} implies the condition of Laurain--Rivière for the quantization of energy).  
 
 \subsection{Finale}
  
 What would be the next step in Rivière's program? It consists in generalising the present article to the case of branched Willmore spheres obtained by min-max (\cite{eversion}). The first difficulty is to control the non-local terms arising in the second derivative. Otherwise, the main obstacle is to understand how Rellich-type estimates can show that the contribution coming from the viscosity term is negligible. Here, the analogy with harmonic maps from Da Lio--Gianocca--Rivière's theory (\cite{riviere_morse_scs}) breaks down for min-max need not work for Sacks–Uhlenbeck-type estimates. However, the Willmore energy is more flexible and there are no topological obstructions for the upper semi-continuity of the Morse index plus nullity.

\section{Inequalities for the Classical Family of Weighted Inequalities}

In this technical section, we establish two fundamental estimates that play a crucial role in the proof. They are by no means immediate, and when it comes to the second inequality, it does not follow from a mere application of the methods first developed in \cite{willmore_scs} and we have to come up with several refinements of the proof.

\subsection{Weighted Inequalities Associated to a Family of Fourth-Order Elliptic Operators}

Compared to \cite{willmore_scs}, for $m>1$, the proof is significantly more technical, but also holds for a larger range of exponents.% (and more precise if $m>1$ is small enough, but we did not precise that in the statement of the theorem).
\begin{theorem}\label{poincare_weight_m}
    Let $m>1$ and define
    \begin{align*}
        \leb_m=\Delta +2(m-1)\frac{x}{|x|^2}\cdot \D +\frac{(m-1)^2}{|x|^2}.
    \end{align*}
    For all $\alpha>0$, there exist a universal constant $C_{m,\alpha}<\infty$ with the following property.
    Fix $0<a<b<\infty$ and let $\Omega=B_b\setminus\bar{B}_a(0)$. Then, for all $u\in W^{2,2}_0(\Omega)$, we have
    \begin{align}
    \left\{\begin{alignedat}{2}
        &\int_{\Omega}\frac{u^2}{|x|^4}\left(\left(\frac{|x|}{b}\right)^{\alpha}+\left(\frac{a}{|x|}\right)^{\alpha}\right)dx&&\leq C_{m,\alpha}\int_{\Omega}\left(\leb_mu\right)^2dx\\
        &\int_{\Omega}\frac{|\D u|^2}{|x|^2}\left(\left(\frac{|x|}{b}\right)^{\alpha}+\left(\frac{a}{|x|}\right)^{\alpha}\right)dx&&\leq C_{m,\alpha}\int_{\Omega}\left(\leb_mu\right)^2dx.
        \end{alignedat}\right.
    \end{align}
\end{theorem}
\begin{proof}
Recall that (see \cite[(2.2.5) p.97]{willmore_scs})
\begin{align}\label{fourth_order_m}
    \leb_m^{\ast}\leb_m=\Delta^2+2(m^2-1)\frac{1}{|x|^2}\Delta-4(m^2-1)\left(\frac{x}{|x|^2}\right)^t\cdot  \D^2(\,\cdot\,)\cdot\left(\frac{x}{|x|^2}\right)+\frac{(m^2-1)^2}{|x|^4}.
\end{align}
Furthermore, for all $u\in W^{2,2}_0(\Omega)$, the following identity holds:
\begin{align}\label{ipp_m}
    \int_{\Omega}(\mathscr{L}_mu)^2dx&=\int_{\Omega}\left(\Delta u+(m^2-1)\frac{u}{|x|^2}\right)^2dx    +4(m^2-1)\int_{\Omega}\left(\frac{x}{|x|^2}\cdot \D u-\frac{u}{|x|^2}\right)^2dx.
\end{align}
This identity is based on the following decomposition
\begin{align}\label{decomp_adjoint}
    \leb_{m,1}^{\ast}\leb_{m,1}+4(m^2-1)\mathscr{D}_2^{\ast}\mathscr{D}_2=\leb_m^{\ast}\leb_m,
\end{align}
where
\begin{align*}
    \leb_{m,1}=\Delta+\frac{(m^2-1)}{|x|^2}\qquad\text{and}\qquad \mathscr{D}_2=\frac{x}{|x|^2}\cdot \D-\frac{1}{|x|^2}=\frac{1}{r}\p{r}-\frac{1}{r^2}
\end{align*}
Let us check the identity \eqref{decomp_adjoint}. Since $\leb_{m,1}$ is a self-adjoint operator, we have
\begin{align}\label{adjoint_part1}
    \Delta \left(\leb_{m,1}\right)&=\Delta^2+(m^2-1)\frac{1}{|x|^2}\Delta -4(m^2-1)\frac{x}{|x|^4}\cdot \D (\,\cdot\,)+\frac{4(m^2-1)}{|x|^4}\nonumber\\
    \leb_{m,1}^{\ast}\leb_{m,1}&=\Delta^2+2(m^2-1)\frac{1}{|x|^2}\Delta -4(m^2-1)\frac{x}{|x|^4}\cdot \D (\,\cdot\,)+\frac{(m^2-1)^2+4(m^2-1)}{|x|^4}.
\end{align}
On the other hand, we have
\begin{align*}
    \p{r}\mathscr{D}_2=\frac{1}{r}\p{r}^2-\frac{2}{r^2}\p{r}+\frac{2}{r^3},
\end{align*}
which shows that
\begin{align}\label{adjoint_part2}
    \mathscr{D}_2^{\ast}\mathscr{D}_2&=-\frac{1}{r^2}\p{r}^2+\frac{2}{r^3}\p{r}-\frac{2}{r^4}-\frac{1}{r^3}\p{r}+\frac{1}{r^4}\nonumber\\
    &=-\frac{1}{r^2}\p{r}^2+\frac{1}{r^3}\p{r}-\frac{1}{r^4}\nonumber\\
    &=-\left(\frac{x}{|x|^2}\right)^t\cdot \D^2(\,\cdot\,)\cdot \left(\frac{x}{|x|^2}\right)+\frac{x}{|x|^4}\cdot \D (\,\cdot\,)-\frac{1}{|x|^4}
\end{align}
and we deduce that 
\begin{align*}
    &\leb_{m,1}^{\ast}\leb_{m,1}+4(m^2-1)\mathscr{D}_2^{\ast}\mathscr{D}_2=\Delta^2+2(m^2-1)\frac{1}{|x|^2}\Delta -4(m^2-1)\frac{x}{|x|^4}\cdot \D (\,\cdot\,)+\frac{(m^2-1)^2+4(m^2-1)}{|x|^4}\\
    &-4(m^2-1)\left(\frac{x}{|x|^2}\right)^t\cdot \D^2(\,\cdot\,)\cdot \left(\frac{x}{|x|^2}\right)+4(m^2-1)\frac{x}{|x|^4}\cdot \D (\,\cdot\,)-\frac{4(m^2-1)}{|x|^4}\\
    &=\Delta^2+2(m^2-1)\frac{1}{|x|^2}\Delta-4(m^2-1)\left(\frac{x}{|x|^2}\right)^t\cdot  \D^2(\,\cdot\,)\cdot\left(\frac{x}{|x|^2}\right)+\frac{(m^2-1)^2}{|x|^4},
\end{align*}
which allows us to recover \eqref{fourth_order_m}, and provides a new proof of \eqref{ipp_m}. 

\textbf{Part 1: Radial estimates.} 

Let $\alpha>0$. Thanks to \eqref{ipp_m}, we have
\begin{align*}
    \int_{\Omega}\left(\frac{x}{|x|^2}\cdot \D u-\frac{u}{|x|^2}\right)^2dx\leq \frac{1}{4(m^2-1)}\int_{\Omega}\left(\leb_mu\right)^2dx.
\end{align*}
Furthermore, for all $\alpha\in \R$, notice that
\begin{align}\label{simple_lm_1}
    \int_{\Omega}\frac{u}{|x|^2}\left(\frac{x}{|x|^2}\cdot \D u\right)|x|^{\alpha}dx&=\frac{1}{2}\int_{\Omega}\dive\left(\frac{x}{|x|^2}u^2\right)|x|^{\alpha-2}dx=-\frac{1}{2}\int_{\Omega}u^2\left(\frac{x}{|x|^2}\cdot \D \left(|x|^{\alpha}\right)\right)dx\nonumber\\
    &=-\frac{\alpha-2}{2}\int_{\Omega}\frac{u^2}{|x|^4}|x|^{\alpha}dx.
\end{align}
where we used that $\D\log|x|=\dfrac{x}{|x|^2}$ and the harmonicity of $\log$ in dimension $2$. Therefore, we have by \eqref{simple_lm_1}
\begin{align}\label{simple_lm_2}
    \int_{\Omega}\left(\frac{x}{|x|^2}\cdot \D u-\frac{u}{|x|^2}\right)^2|x|^{\alpha}dx&=\int_{\Omega}\left(\frac{x}{|x|^2}\cdot \D u\right)^2|x|^{\alpha}dx-2\int_{\Omega}\frac{u}{|x|^2}\left(\frac{x}{|x|^2}\cdot \D u\right)|x|^{\alpha}dx+\int_{\Omega}\frac{u^2}{|x|^4}|x|^{\alpha}dx\nonumber\\
    &=\int_{\Omega}\left(\frac{x}{|x|^2}\cdot \D u\right)^2|x|^{\alpha}dx+(\alpha-1)\int_{\Omega}\frac{u^2}{|x|^4}|x|^{\alpha}dx.
\end{align}
Therefore, for all $\alpha\geq 1$, we deduce that
\begin{align}\label{simple_lm_3}
    \int_{\Omega}\left(\frac{x}{|x|^2}\cdot \D u\right)^2\left(\frac{|x|}{b}\right)^{\alpha}dx&\leq \int_{\Omega}\left(\frac{x}{|x|^2}\cdot \D u-\frac{u}{|x|^2}\right)^2\left(\frac{|x|}{b}\right)^{\alpha}dx\leq \int_{\Omega}\left(\frac{x}{|x|^2}\cdot \D u-\frac{u}{|x|^2}\right)^2dx\nonumber\\
    &\leq \frac{1}{4(m^2-1)}\int_{\Omega}\left(\leb_mu\right)^2dx.
\end{align}
Likewise, for all $\alpha>1$, we get
\begin{align}\label{simple_lm_4}
    \int_{\Omega}\frac{u^2}{|x|^4}\left(\frac{|x|}{b}\right)^{\alpha}dx\leq \frac{1}{4(\alpha-1)(m^2-1)}\int_{\Omega}\left(\leb_mu\right)^2dx.
\end{align}
Now, notice that \eqref{simple_lm_1} and Cauchy--Schwarz inequality show that for all $\alpha\neq 2$, we have
\begin{align*}
    \int_{\Omega}\frac{u^2}{|x|^4}|x|^{\alpha}dx&=\frac{-2}{\alpha-2}\int_{\Omega}\frac{u}{|x|^2}\left(\frac{x}{|x|^2}\cdot \D u\right)|x|^{\alpha}dx\\
    &\leq \frac{2}{|\alpha-2|}\left(\int_{\Omega}\frac{u^2}{|x|^4}|x|^{\alpha}dx\right)^{\frac{1}{2}}\left(\int_{\Omega}\left(\frac{x}{|x|^2}\cdot \D u\right)^2|x|^{\alpha}dx\right)^{\frac{1}{2}}.
\end{align*}
Therefore, we deduce that for all $\alpha\neq 2$
\begin{align}\label{simple_lm_5}
    \int_{\Omega}\frac{u^2}{|x|^4}|x|^{\alpha}dx\leq \frac{4}{(\alpha-2)^2}\int_{\Omega}\left(\frac{x}{|x|^2}\cdot \D u\right)^2|x|^{\alpha}dx.
\end{align}
Therefore, if $\alpha<2$, we get by \eqref{simple_lm_2} and \eqref{simple_lm_5}
\begin{align}\label{simple_lm_6}
    \int_{\Omega}\left(\frac{x}{|x|^2}\cdot \D u\right)^2|x|^{\alpha}dx+(\alpha-1)\int_{\Omega}\frac{u^2}{|x|^4}|x|^{\alpha}dx&\geq \left(1-\frac{4(1-\alpha)}{(\alpha-2)^2}\right)\int_{\Omega}\left(\frac{x}{|x|^2}\cdot \D u\right)^2|x|^{\alpha}dx\nonumber\\
    &=\frac{\alpha^2}{(\alpha-2)^2}\int_{\Omega}\left(\frac{x}{|x|^2}\cdot \D u\right)^2|x|^{\alpha}dx.
\end{align}
Therefore, for all $0<\alpha<2$, we deduce that
\begin{align}\label{simple_lm_7}
    \int_{\Omega}\left(\frac{x}{|x|^2}\cdot \D u\right)^2\left(\frac{|x|}{b}\right)^{\alpha}dx\leq \frac{(2-\alpha)^2}{4\alpha^2(m^2-1)}\int_{\Omega}\left(\leb_mu\right)^2dx,
\end{align}
which shows that for all $0<\alpha<2$, we have by \eqref{simple_lm_5} and \eqref{simple_lm_7}
\begin{align}\label{simple_lm_8}
    \int_{\Omega}\frac{u^2}{|x|^4}\left(\frac{|x|}{b}\right)^{\alpha}dx\leq \frac{1}{\alpha^2(m^2-1)}\int_{\Omega}\left(\leb_mu\right)^2dx.
\end{align}
Furthermore, since the estimate also holds for negative $\alpha$, we also get for all $\alpha>0$
\begin{align}\label{simple_lm_9}
    \int_{\Omega}\left(\frac{x}{|x|^2}\cdot \D u\right)^2\left(\frac{a}{|x|}\right)^{\alpha}dx\leq \frac{(2+\alpha)^2}{4\alpha^2(m^2-1)}\int_{\Omega}\left(\leb_mu\right)^2dx
\end{align}
and
\begin{align}\label{simple_lm_10}
    \int_{\Omega}\frac{u^2}{|x|^4}\left(\frac{a}{|x|}\right)^{\alpha}dx\leq \frac{1}{\alpha^2(m^2-1)}\int_{\Omega}\left(\leb_mu\right)^2dx.
\end{align}
Therefore, we need only establish the tangential estimate.

\textbf{Part 2: Tangential estimates.}

Interestingly enough, in the most technical lemma of the article, this is the radial estimate that is the most delicate to estimate, while here, the tangential estimate is slightly more difficult to establish. For all $u\in C^{\infty}_c(\Omega)$ and $v\in C^{\infty}(\Omega)$, we have
\begin{align*}
    \leb_m(uv)&=\left(\Delta +2(m-1)\frac{x}{|x|^2}\cdot \D u+\frac{(m-1)^2}{|x|^2}\right)(uv)=v\left(\Delta u+2(m-1)\frac{x}{|x|^2}\cdot \D u+\frac{(m-1)^2}{|x|^2}u\right)\\
    &+2\,\D u\cdot \D v+u\left(\Delta v+2(m-1)\frac{x}{|x|^2}\cdot \D v\right)\\
    &=v\,\leb_mu+2\,\D u\cdot \D v+u\left(\Delta v+2(m-1)\frac{x}{|x|^2}\cdot \D v\right).
\end{align*}
Therefore, we have
\begin{align}\label{gen_weight1}
    \int_{\Omega}u\left(\leb_m^{\ast}\leb_mu\right)v\,dx&=\int_{\Omega}\left(\leb_mu\right)^2v\,dx+2\int_{\Omega}\left(\D u\cdot \D v\right)\leb_mu\,dx\\
    &+\int_{\Omega}u\,\leb_mu\left(\Delta v+2(m-1)\frac{x}{|x|^2}\cdot \D v\right)dx.
\end{align}
On the other hand, we have
\begin{align*}
    \leb_{m,1}\left(uv\right)=\left(\Delta +\frac{m^2-1}{|x|^2}\right)(uv)=v\leb_{m,1}u+2\,\D u\cdot \D v+u\,\Delta v,
\end{align*}
and
\begin{align*}
    \mathscr{D}_2(uv)=\left(\frac{x}{|x|^2}\cdot \D-\frac{1}{|x|^2}\right)\left(uv\right)=v\,\mathscr{D}_2u+u\left(\frac{x}{|x|^2}\cdot \D v\right).
\end{align*}
Using the identity \eqref{decomp_adjoint}, we deduce that 
\begin{align}\label{gen_weight2}
    &\int_{\Omega}u\left(\leb_{m}^{\ast}\leb_mu\right)\,v\,dx=\int_{\Omega}u\left(\leb_{m,1}^{\ast}\leb_{m,1}u\right)v\,dx+4(m^2-1)\int_{\Omega}u\left(\mathscr{D}_2^{\ast}\mathscr{D}_2u\right)v\,dx\nonumber\\
    &=\int_{\Omega}\left(\leb_{m,1}u\right)^2v\,dx+4(m^2-1)\int_{\Omega}\left(\mathscr{D}_2u\right)^2v\,dx+2\int_{\Omega}\left(\D u\cdot \D v\right)\leb_{m,1}u\,dx+\int_{\Omega}\left(u\,\leb_{m,1}u\right)\Delta v\,dx\nonumber\\
    &+4(m^2-1)\int_{\Omega}u\,\mathscr{D}_2u\left(\frac{x}{|x|^2}\cdot \D v\right)dx.
\end{align}
Then, we have
\begin{align}\label{gen_weight3}
    \leb_m-\leb_{m,1}&=\Delta +2(m-1)\frac{x}{|x|^2}\cdot \D u+\frac{(m-1)^2}{|x|^2}-\left(\Delta +\frac{m^2-1}{|x|^2}\right)=2(m-1)\frac{x}{|x|^2}\cdot \D-\frac{2(m-1)}{|x|^2}\nonumber\\
    &=2(m-1)\mathscr{D}_2.
\end{align}
Comparing \eqref{gen_weight1} and \eqref{gen_weight2}, we deduce by \eqref{gen_weight3} that
\begin{align}\label{ipp_weight}
    &\int_{\Omega}\left(\leb_mu\right)^2v\,dx=\int_{\Omega}\left(\leb_{m,1}u\right)^2v\,dx+4(m^2-1)\int_{\Omega}\left(\mathscr{D}_2u\right)^2v\,dx-4(m-1)\int_{\Omega}\left(\D u\cdot \D v\right)\mathscr{D}_2u\,dx\nonumber\\
    &-2(m-1)\int_{\Omega}\left(u\,\mathscr{D}_2u\right)\Delta v\,dx-2(m-1)\int_{\Omega}u\,\leb_mu\left(\frac{x}{|x|^2}\cdot \D v\right)dx+4(m^2-1)\int_{\Omega}u\,\mathscr{D}_2u\left(\frac{x}{|x|^2}\cdot \D v\right)dx.
\end{align}
Now, if we take $v(x)=|x|^{\alpha}$ (where $\alpha\in \R$), we get
\begin{align*}
    &\D v=\alpha\frac{x}{|x|^2}|x|^{\alpha}\\
    &\Delta v=\left(\p{r}^2+\frac{1}{r}\p{r}\right)r^{\alpha}=\alpha^2|x|^{\alpha-2}.
\end{align*}
Therefore, the identity \eqref{ipp_weight} becomes
\begin{align}\label{gen_weight_4}
    &\int_{\Omega}\left(\leb_mu\right)^2|x|^{\alpha}dx=\int_{\Omega}\left(\leb_{m,1}u\right)^2|x|^{\alpha}\,dx+4(m^2-1)\int_{\Omega}\left(\mathscr{D}_2u\right)^2|x|^{\alpha}dx\nonumber\\
    &-4\alpha(m-1)\int_{\Omega}\left(\frac{x}{|x|^2}\cdot \D u\right)\left(\frac{x}{|x|^2}\cdot \D u-\frac{u}{|x|^2}\right)|x|^{\alpha}dx-2\alpha(m-1)\int_{\Omega}u\,\leb_mu\,|x|^{\alpha-2}dx\nonumber\\
    &+\left(4\alpha(m^2-1)-2\alpha^2(m-1)\right)\int_{\Omega}\frac{u}{|x|^2}\left(\frac{x}{|x|^2}\cdot \D u-\frac{u}{|x|^2}\right)|x|^{\alpha}dx.
\end{align}
Using \eqref{simple_lm_1}, we deduce that 
\begin{align}\label{gen_weight_5}
    &-4\alpha(m-1)\int_{\Omega}\left(\frac{x}{|x|^2}\cdot \D u\right)\left(\frac{x}{|x|^2}\cdot \D u-\frac{u}{|x|^2}\right)|x|^{\alpha}dx=-4\alpha(m-1)\int_{\Omega}\left(\frac{x}{|x|^2}\cdot \D u\right)^2|x|^{\alpha}dx\nonumber\\
    &-2\alpha(\alpha-2)(m-1)\int_{\Omega}\frac{u^2}{|x|^4}|x|^{\alpha}dx.
\end{align}
Then, we have
\begin{align}\label{gen_weight_6}
    &-2\alpha(m-1)\int_{\Omega}u\,\Delta u\,|x|^{\alpha-2}dx=2\alpha(m-1)\int_{\Omega}\frac{|\D u|^2}{|x|^2}|x|^{\alpha}dx\nonumber\\
    &+2\alpha(\alpha-2)(m-1)\int_{\Omega}\frac{u}{|x|^2}\left(\frac{x}{|x|^2}\cdot \D u\right)|x|^{\alpha}dx=2\alpha(m-1)\int_{\Omega}\frac{|\D u|^2}{|x|^2}|x|^{\alpha}dx\nonumber\\
    &-\alpha(\alpha-2)^2(m-1)\int_{\Omega}\frac{u^2}{|x|^4}|x|^{\alpha}dx.
\end{align}
Therefore, we get
\begin{align}\label{gen_weight_7}
    &-2\alpha(m-1)\int_{\Omega}u\,\leb_mu\,|x|^{\alpha-2}dx=-2\alpha(m-1)\int_{\Omega}u\,\Delta u\,|x|^{\alpha-2}dx\\
    &-4\alpha(m-1)^2\int_{\Omega}\frac{u}{|x|^2}\left(\frac{x}{|x|^2}\cdot \D u\right)|x|^{\alpha}dx
    -2\alpha(m-1)^3\int_{\Omega}\frac{u^2}{|x|^4}|x|^{\alpha}dx\nonumber\\
    &=2\alpha(m-1)\int_{\Omega}\frac{|\D u|^2}{|x|^2}|x|^{\alpha}dx+\left(-\alpha(\alpha-2)^2(m-1)+2\alpha(\alpha-2)(m-1)^2-2\alpha(m-1)^3\right)\int_{\Omega}\frac{u^2}{|x|^4}|x|^{\alpha}dx.
\end{align}
Finally, we have 
\begin{align}\label{gen_weight_8}
    \int_{\Omega}\frac{u}{|x|^2}\left(\frac{x}{|x|^2}\cdot \D u-\frac{u}{|x|^2}\right)|x|^{\alpha}dx=-\frac{(\alpha-2)}{2}\int_{\Omega}\frac{u^2}{|x|^4}|x|^{\alpha}dx-\int_{\Omega}\frac{u^2}{|x|^4}|x|^{\alpha}dx=-\frac{\alpha}{2}\int_{\Omega}\frac{u^2}{|x|^4}|x|^{\alpha}dx
\end{align}
and since $4\alpha(m^2-1)-2\alpha^2(m-1)=2\alpha(m-1)(2(m+1)-\alpha)$, we get
\begin{align}\label{gen_weight_9}
    &\left(4\alpha(m^2-1)-2\alpha^2(m-1)\right)\int_{\Omega}\frac{u}{|x|^2}\left(\frac{x}{|x|^2}\cdot \D u-\frac{u}{|x|^2}\right)|x|^{\alpha}dx\nonumber\\
    &=-\alpha^2(m-1)(2(m+1)-\alpha)\int_{\Omega}\frac{u^2}{|x|^4}|x|^{\alpha}dx.
\end{align}
Gathering \eqref{gen_weight_4}, \eqref{gen_weight_5}, \eqref{gen_weight_6}, \eqref{gen_weight_7},  and \eqref{gen_weight_9}, we deduce that 
\begin{align*}
    &\int_{\Omega}\left(\leb_mu\right)^2|x|^{\alpha}dx=\int_{\Omega}\left(\Delta u+\frac{m^2-1}{|x|^2}u\right)^2+4(m^2-1)\int_{\Omega}\left(\frac{x}{|x|^2}\cdot \D u-\frac{u}{|x|^2}\right)|x|^{\alpha}dx\\
    &-2\alpha(m-1)\int_{\Omega}\left(\frac{x}{|x|^2}\cdot \D u\right)^2|x|^{\alpha}dx+2\alpha(m-1)\int_{\Omega}\frac{\left|\p{\theta}u\right|^2}{|x|^4}|x|^{\alpha}dx\\
    &+\Big(-2\alpha(\alpha-2)(m-1)-\alpha(\alpha-2)^2(m-1)+2\alpha(\alpha-2)(m-1)^2-2\alpha(m-1)^3\\
    &-\alpha^2(m-1)(2(m+1)-\alpha)\Big)\int_{\Omega}\frac{u^2}{|x|^4}|x|^{\alpha}dx\\
    &=\int_{\Omega}\left(\Delta u+\frac{m^2-1}{|x|^2}u\right)^2|x|^{\alpha}dx+4(m^2-1)\int_{\Omega}\left(\frac{x}{|x|^2}\cdot \D u-\frac{u}{|x|^2}\right)^2|x|^{\alpha}dx\\
    &-2\alpha(m-1)\int_{\Omega}\left(\frac{x}{|x|^2}\cdot \D u\right)^2|x|^{\alpha}dx+2\alpha(m-1)\int_{\Omega}\frac{\left|\p{\theta}u\right|^2}{|x|^4}|x|^{\alpha}dx\\
    &-2\alpha(m-1)\Big(m^2-1+\alpha\Big)\int_{\Omega}\frac{u^2}{|x|^4}|x|^{\alpha}dx.
\end{align*}
Therefore, we can compactly rewrite the identity as
\begin{align}\label{gen_final1}
    &\int_{\Omega}\left(\leb_mu\right)^2|x|^{\alpha}dx=\int_{\Omega}\left(\Delta u+\frac{m^2-1}{|x|^2}u\right)^2|x|^{\alpha}dx+4(m^2-1)\int_{\Omega}\left(\frac{x}{|x|^2}\cdot \D u-\frac{u}{|x|^2}\right)^2|x|^{\alpha}dx\nonumber\\
    &-2\alpha(m-1)\int_{\Omega}\left(\frac{x}{|x|^2}\cdot \D u\right)^2|x|^{\alpha}dx+2\alpha(m-1)\int_{\Omega}\frac{\left|\p{\theta}u\right|^2}{|x|^4}|x|^{\alpha}dx\nonumber\\
    &-2\alpha(m-1)\Big(m^2-1+\alpha\Big)\int_{\Omega}\frac{u^2}{|x|^4}|x|^{\alpha}dx.
\end{align}
Therefore, if $\alpha>0$, we get by \eqref{simple_lm_3}, \eqref{simple_lm_4}, \eqref{simple_lm_7}, and \eqref{simple_lm_8}
\begin{align}\label{lm_weight_final1}
    &\int_{\Omega}\frac{|\p{\theta}u|^2}{|x|^4}\left(\frac{|x|}{b}\right)^{\alpha}dx\leq \frac{1}{2\alpha(m-1)}\int_{\Omega}\left(\leb_mu\right)^{2}\left(\frac{|x|}{b}\right)^{\alpha}dx+\int_{\Omega}\left(\frac{x}{|x|^2}\cdot \D u\right)^2\left(\frac{|x|}{b}\right)^{\alpha}dx\nonumber\\
    &+(m^2-1+\alpha)\int_{\Omega}\frac{u^2}{|x|^4}\left(\frac{|x|}{b}\right)^{\alpha}dx\nonumber\\
    &\leq \left\{\begin{alignedat}{2}
        &\left(\frac{1}{2\alpha(m-1)}+\frac{1}{4(m^2-1)}+\frac{m^2-1+\alpha}{4(\alpha-1)(m^2-1)}\right)\int_{\Omega}\left(\leb_mu\right)^2dx\qquad &&\text{if}\;\,\alpha>1\\
        &\left(\frac{1}{2\alpha(m-1)}+\frac{(2-\alpha)^2}{4\alpha^2(m^2-1)}+\frac{m^2-1+\alpha}{\alpha^2(m^2-1)}\right)\int_{\Omega}\left(\leb_mu\right)^2dx\qquad &&\text{if}\;\, 0<\alpha<2,
    \end{alignedat}\right.
\end{align}
while we get by \eqref{simple_lm_9} and \eqref{simple_lm_10} (recall \eqref{ipp_m} too) for that for $\alpha>0$ 
\begin{align}\label{lm_weight_final2}
    &\int_{\Omega}\frac{|\p{\theta}u|^2}{|x|^4}\left(\frac{a}{|x|}\right)^{\alpha}dx\leq \frac{1}{2\alpha(m-1)}\left(\int_{\Omega}\left(\Delta u+\frac{m^2-1}{|x|^2}u\right)^2\left(\frac{a}{|x|}\right)^{\alpha}dx\right.\nonumber\\
    &\left.+4(m^2-1)\int_{\Omega}\left(\frac{x}{|x|^2}\cdot \D u-\frac{u}{|x|^2}\right)^2\left(\frac{a}{|x|}\right)^{\alpha}dx\right)+\int_{\Omega}\left(\frac{x}{|x|^2}\cdot \D u\right)^2\left(\frac{a}{|x|}\right)^{\alpha}dx\nonumber\\
    &+(m^2-1-\alpha)_+\int_{\Omega}\frac{u^2}{|x|^4}\left(\frac{a}{|x|}\right)^{\alpha}dx\nonumber\\
    &\leq \left(\frac{1}{2\alpha(m-1)}+\frac{(2+\alpha)^2}{4\alpha^2(m^2-1)}+\frac{1}{\alpha^2(m^2-1)}\right)\int_{\Omega}\left(\leb_mu\right)^2dx
\end{align}
and this concludes the proof of the theorem. 
\end{proof}

We need to strengthen this estimate by adding weights on the right-hand side, but the cost is to have to add a term corresponding to the $L^2$ norm of $\p{z}^2u$ once we make a change of variable.

\subsection{Refined Weighted Inequalities}

If we want to add a weight in the right-hand side of the inequality, we need to replace $\leb_m$ by the operator $\mathfrak{L}_m$ that correspond to $|z|^{1-m}\p{z}^2$ after a change of variation $v=|z|^{m-1}u$.
 
\begin{theorem}\label{poincare_weight_m_2}
    For all $3\leq m<\infty$, define
    \begin{align*}
        \mathfrak{L}_m=|z|^{1-m}\p{z}^2(|z|^{m-1}(\,\cdot\,))=\p{z}^2+\frac{m-1}{z}\p{z}+\frac{(m-1)(m-3)}{4z^2}.
    \end{align*}
    Fix $0<a<b<\infty$ and let $\Omega=B_b\setminus\bar{B}_a(0)$. Then, for all $u\in W^{2,2}_0(\Omega)$, we have for all $-\infty<\beta<\dfrac{m-1}{4}$ \emph{(}such that $\beta\neq 1/2$\emph{)} and for all $-\infty<\gamma<\dfrac{m-1}{2}$ \emph{(}such that $\gamma\neq 1$\emph{)}
    \begin{align}\label{ineq_frak_m}
    \left\{\begin{alignedat}{1}
        &\int_{\Omega}\frac{u^2}{|x|^4}|x|^{4\beta}dx\leq \frac{8}{(m-1)(m-1-4\beta)(1-2\beta)^2}\int_{\Omega}\left|\mathfrak{L}_mu\right|^2|x|^{4\beta}dx\\
        &\int_{\Omega}\frac{|\D u|^2}{|x|^2}|x|^{2\gamma}\leq \frac{8}{(m-1)(m-1-2\gamma)}\int_{\Omega}\left|\mathfrak{L}_mu\right|^2|x|^{2\gamma}dx.
        \end{alignedat}\right.
    \end{align}
\end{theorem}
\begin{proof}
    First, 
the identity \eqref{simple_lm_1}
\begin{align}\label{full_hessian_1}
    \int_{\Omega}\frac{u}{|x|^2}\left(\frac{x}{|x|^2}\cdot \D u\right)|x|^{\alpha}dx=-\frac{(\alpha-2)}{2}\int_{\Omega}\frac{u^2}{|x|^4}|x|^{\alpha}dx.
\end{align}
It implies that for all $\alpha\neq 2$, we have
\begin{align}\label{elementary_inequality}
    \int_{\Omega}\frac{u^2}{|x|^4}|x|^{\alpha}dx\leq \frac{4}{(\alpha-2)^2}\int_{\Omega}\frac{|\D u|^2}{|x|^2}|x|^{\alpha}dx.
\end{align}
    Now, we compute for all $\alpha\neq 0$
    \begin{align}\label{full_hessian_2}
    	&\int_{\Omega}|\mathfrak{L}_mu|^2|z|^{\alpha}|dz|^2=\int_{\Omega}\left|\p{z}^2 u+\frac{m-1}{z}\p{z}u+\frac{(m-1)(m-3)}{4z^2}u\right|^2|z|^{\alpha}|dz|^2\nonumber\\
        &=\int_{\Omega}|\p{z}^2u|^2|z|^{\alpha}|dz|^2+(m-1)^2\int_{\Omega}\frac{|\p{z}u|^2}{|z|^2}|z|^{\alpha}|dz|^2+\frac{(m-1)^2(m-3)^2}{16}\int_{\Omega}\frac{u^2}{|z|^4}|z|^{\alpha}|dz|^2\nonumber\\
    	&+2(m-1)\Re\int_{\Omega}\p{z}^2u\,\p{\z}u\frac{|z|^{\alpha}}{\z}|dz|^2+\frac{(m-1)(m-3)}{2}\Re\int_{\Omega}u\,\p{z}^2u\,\frac{|z|^{\alpha}}{\z^2}|dz|^2\nonumber\\
    	&+\frac{(m-1)^2(m-3)}{2}\Re\int_{\Omega}\frac{\p{z}u}{z}u\frac{|z|^{\alpha}}{\z^2}|dz|^2.
    \end{align}
    We have the elementary computation $\p{z}|z|^{\alpha}=\p{z}(z\z)^{\frac{\alpha}{2}}=\dfrac{\alpha}{2}\z|z|^{\alpha-2}$, which yields
    \begin{align}\label{full_hessian_3}
    	\int_{\Omega}\p{z}^2u\,\p{\z}u\frac{|z|^{\alpha}}{\z}|dz|^2&=-\int_{\Omega}\p{z}u\,\p{z\z}^2u\frac{|z|^{\alpha}}{\z}|dz|^2-\frac{\alpha}{2}\int_{\Omega}\p{z}u\cdot\p{\z}u|z|^{\alpha-2}|dz|^2\nonumber\\
        &=-\frac{1}{2}\int_{\Omega}\p{\z}\left(\p{z}u\right)^2\frac{|z|^{\alpha}}{\z}|dz|^2-\frac{\alpha}{2}\int_{\Omega}\frac{|\p{z}u|^2}{|z|^{2}}|z|^{\alpha}|dz|^2\nonumber\\
        &=\frac{(\alpha-2)}{4}\int_{\Omega}\left(\frac{z}{|z|^2}\,\p{z}u\right)^2|z|^{\alpha}|dz|^2-\frac{\alpha}{2}\int_{\Omega}\frac{|\p{z}u|^2}{|z|^2}|z|^{\alpha}|dz|^2.
    \end{align}
%    where we used that
 %   \begin{align*}
  %      \p{\z}\left(\frac{|z|^{\alpha}}{\z}\right)=-\frac{|z|^{\alpha}}{\z^2}+\frac{\alpha}{2}\frac{z}{\z}|z|^{\alpha-2}=\frac{(\alpha-2)}{2}z^2|z|^{\alpha-4},
   % \end{align*}
    %which can also be found by writing
    %\begin{align*}
     %   \p{\z}\left(\frac{|z|^{\alpha}}{\z}\right)=\p{\z}\left(z|z|^{\alpha-2}\right)=\frac{(\alpha-2)}{2}z^2|z|^{\alpha-4}.
    %\end{align*}
    Next, we have
    \begin{align}\label{full_hessian_4}
        \int_{\Omega}u\,\p{z}^2u\frac{|z|^{\alpha}}{\z^2}|dz|^2&=-\int_{\Omega}\left(\frac{z}{|z|^2}\p{z}u\right)^2|z|^{\alpha}|dz|^2-\frac{\alpha}{2}\int_{\Omega}u\,\p{z}u\frac{|z|^{\alpha-2}}{\z}|dz|^2\nonumber\\
        &=-\int_{\Omega}\left(\frac{z}{|z|^2}\p{z}u\right)^2|z|^{\alpha}|dz|^2-\frac{\alpha}{4}\int_{\Omega}\p{z}\left(u^2\right)\frac{|z|^{\alpha-2}}{\z}|dz|^2\nonumber\\
        &=-\int_{\Omega}\left(\frac{z}{|z|^2}\p{z}u\right)^2|z|^{\alpha}|dz|^2+\frac{\alpha(\alpha-2)}{8}\int_{\Omega}\frac{u^2}{|z|^4}|z|^{\alpha}|dz|^2.
    \end{align}
    Finally, we also have
    \begin{align}\label{full_hessian_5}
        \int_{\Omega}\frac{\p{z}u}{z}u\frac{|z|^{\alpha}}{\z^2}|dz|^2=\int_{\Omega}u\,\p{z}u\frac{|z|^{\alpha-2}}{\z}|dz|^2=-\frac{(\alpha-2)}{4}\int_{\Omega}\frac{u^2}{|z|^4}|z|^{\alpha}|dz|^2.
    \end{align}
    Gathering \eqref{full_hessian_2}, \eqref{full_hessian_3}, \eqref{full_hessian_4}, and \eqref{full_hessian_5}, we find that
    \begin{align}\label{full_hessian_6}
        &\int_{\Omega}\left|\mathfrak{L}_mu\right|^2|z|^{\alpha}|dz|^2=\int_{\Omega}|\p{z}^2u|^2|z|^{\alpha}|dz|^2+(m-1)^2\int_{\Omega}\frac{|\p{z}u|^2}{|z|^2}|z|^{\alpha}|dz|^2\nonumber\\
        &+\frac{(m-1)^2(m-3)^2}{16}\int_{\Omega}\frac{u^2}{|z|^4}|z|^{\alpha}|dz|^2
        +\frac{(m-1)(\alpha-2)}{2}\int_{\Omega}\Re\left\{\left(\frac{z}{|z|^2}\p{z}u\right)^2\right\}|z|^{\alpha}|dz|^2\nonumber\\
        &-(m-1)\alpha\int_{\Omega}\frac{|\p{z}u|^2}{|z|^2}|dz|^2
        -\frac{(m-1)(m-3)}{2}\int_{\Omega}\Re\left\{\left(\frac{z}{|z|^2}\p{z}u\right)^2\right\}|z|^{\alpha}|dz|^2\nonumber\\
        &+\frac{(m-1)(m-3)\alpha(\alpha-2)}{8}\int_{\Omega}\frac{u^2}{|z|^4}|z|^{\alpha}|dz|^2
        -\frac{(m-1)^2(m-3)(\alpha-2)}{8}\int_{\Omega}\frac{u^2}{|z|^4}|z|^{\alpha}|dz|^2.
    \end{align}
    Using the elementary inequality
    \begin{align*}
        \left|\Re\int_{\Omega}\left(\frac{z}{|z|^2}\p{z}u\right)|z|^{\alpha}|dz|^2\right|\leq \int_{\Omega}\frac{|\p{z}u|^2}{|z|^2}|z|^{\alpha}|dz|^2,
    \end{align*}
    we deduce that for all $\alpha\in \R$
    \begin{align}\label{full_hessian_7}
        &\int_{\Omega}\left|\mathfrak{L}_mu\right|^2|z|^{\alpha}|dz|^2\geq \int_{\Omega}|\p{z}^2u|^2|z|^{\alpha}|dz|^2\nonumber\\
        &+\left((m-1)^2+\frac{(m-1)(\alpha-2)}{2}-\frac{(m-1)(m-3)}{2}-(m-1)\alpha\right)\int_{\Omega}\frac{|\p{z}u|^2}{|z|^2}|z|^{\alpha}|dz|^2\nonumber\\
        &+\left(\frac{(m-1)^2(m-3)^2}{16}+\frac{(m-1)(m-3)\alpha(\alpha-2)}{8}-\frac{(m-1)^2(m-3)(\alpha-2)}{8}\right)\int_{\Omega}\frac{u^2}{|z|^4}|z|^{\alpha}|dz|^2\nonumber\\
        &=\int_{\Omega}|\p{z}^2u|^2|z|^{\alpha}|dz|^2+\frac{1}{2}(m-1)\left(m-1-\alpha\right)\int_{\Omega}\frac{|\p{z}u|^2}{|z|^2}|z|^{\alpha}|dz|^2\nonumber\\
        &+\frac{(m-1)(m-3)}{16}\left(2\alpha^2-2(m+1)\alpha+(m^2-1)\right)\int_{\Omega}\frac{u^2}{|z|^4}|z|^{\alpha}|dz|^2.
    \end{align}
    The discriminant of the polynomial $2\,X^2-2(m+1)X+(m^2-1)$ is equal to 
    \begin{align*}
        D=-4(m^2-2m-3)=-4(m+1)(m-3)\leq 0,
    \end{align*}
    which shows that for all $m\geq 3$ and for all $\alpha<m-1$, we have the estimate
    \begin{align*}
        \int_{\Omega}\frac{|\D u|^2}{|x|^2}|x|^{\alpha}dx\leq \frac{8}{(m-1)(m-1-\alpha)}\int_{\Omega}\left|\mathfrak{L}_mu\right|^2|z|^{\alpha}|dz|^2,
    \end{align*}
    and for all $\alpha\neq 2$ and $\alpha<m-1$, the inequality \eqref{elementary_inequality} shows that
    \begin{align*}
        \int_{\Omega}\frac{u^2}{|x|^4}|x|^{\alpha}dx\leq \frac{32}{(m-1)(m-1-\alpha)(\alpha-2)^2}\int_{\Omega}\left|\mathfrak{L}_mu\right|^2|z|^{\alpha}|dz|^2.
    \end{align*}
    Notice that if $m>3$, for all $\alpha\leq m-1$, we get the estimate
    \begin{align*}
        \int_{\Omega}\frac{u^2}{|x|^4}|x|^{\alpha}dx\leq \frac{16}{(m-1)(m-3)(2\alpha^2-2(m+1)\alpha+(m^2-1))}\int_{\Omega}\left|\mathfrak{L}_mu\right|^2|z|^{\alpha}|dz|^2.
    \end{align*}
    Now, if $1<m<3$, we have
    \begin{align*}
        2\alpha^2-2(m+1)\alpha+(m^2-1)&=2\left(\alpha-\frac{(m+1)-\sqrt{(m+1)(3-m)}}{2}\right)\\
        &\times\left(\alpha-\frac{(m+1)+\sqrt{(m+1)(3-m)}}{2}\right),
    \end{align*}
    which shows that 
    \begin{align*}
        \frac{(m-1)(m-3)}{16}\left(2\alpha^2-2(m+1)\alpha+(m^2-1)\right)\left(2\alpha^2-2(m+1)\alpha+(m^2-1)\right)\geq 0
    \end{align*}
    if and only if
    \begin{align*}
        \frac{(m+1)-\sqrt{(m+1)(3-m)}}{2}\leq \alpha\leq \frac{(m+1)+\sqrt{(m+1)(3-m)}}{2}.
    \end{align*}
    Since
    \begin{align*}
        \frac{(m+1)-\sqrt{(m+1)(3-m)}}{2}=\frac{m^2-1}{m+1+\sqrt{(m+1)(3-m)}}=\alpha_m\in ]0,2[,
    \end{align*}
    we have to find a new argument for $\alpha<\alpha_m$. Therefore, assuming that $\alpha<\alpha_m$, we get thanks to \eqref{elementary_inequality} and \eqref{full_hessian_7} (notice also that $|\D u|^2=4|\p{z}u|^2$)
    \begin{align*}
        &\int_{\Omega}\left|\mathfrak{L}_mu\right|^2|z|^{\alpha}|dz|^2\geq \int_{\Omega}|\p{z}^2u|^2|z|^{\alpha}|dz|^2+\frac{1}{2}(m-1)\left(m-1-\alpha\right)\int_{\Omega}\frac{|\p{z}u|^2}{|z|^2}|z|^{\alpha}|dz|^2\\
        &+\frac{(m-1)(m-3)}{(\alpha-2)^2}\left(2\alpha^2-2(m+1)\alpha+(m^2-1)\right)\int_{\Omega}\frac{|\p{z}u|^2}{|z|^2}|z|^{\alpha}|dz|^2\\
        &=\int_{\Omega}|\p{z}^2u|^2|z|^{\alpha}|dz|^2\\
        &+\frac{(m-1)}{2(\alpha-2)^2}\left(-\alpha^3+(5m-9)\alpha^2+4(m^2-m-3)\alpha+2(m^3-3m^2+m+1)\right)\int_{\Omega}\frac{|\p{z}u|^2}{|z|^2}|z|^{\alpha}|dz|^2.
    \end{align*}
    We see that in general, we will only get an estimate for a restricted range of $\alpha$, but we are mostly interested to prove an estimate for $|\alpha|$ small. Furthermore, since we are mostly interested in the case $m=2$, we now assume that $m=2$, which yields
    \begin{align*}
        -\alpha^3+(5m-9)\alpha^2+4(m^2-m-3)\alpha+2(m^3-3m^2+m+1)=-\alpha^3+\alpha^2+4\alpha-2. 
    \end{align*}
    However, the roots of $Q=-X^3+X^2+4\,X-2$ are all real and given by
    \begin{align*}
        r_1&=\frac{1}{3}-\frac{13(1+i\sqrt{3})}{6\sqrt[3]{-8+3i\sqrt{237}}}-\frac{1}{6}(1-i\sqrt{3})\sqrt[3]{-8+3i\sqrt{237}}=-1.81360\cdots\\
        r_2&=\frac{1}{3}-\frac{13(1-i\sqrt{3})}{6\sqrt[3]{-8+3i\sqrt{237}}}-\frac{1}{6}(1+i\sqrt{3})\sqrt[3]{-8+3i\sqrt{237}}=0.47068\cdots\\
        r_3&=\frac{1}{3}+\frac{13}{3\sqrt[3]{-8+3i\sqrt{237}}}+\frac{1}{3}\sqrt[3]{-8+3i\sqrt{237}}=2.34429\cdots
    \end{align*}
    In particular, the quantity is negative for $r_1<\alpha<r_2$, and we cannot obtain the expected estimate for negative $\alpha$. However, the estimate holds true for $r_2<\alpha<r_3$, and it holds in particular if $\alpha=2\gamma$ with $\frac{1}{4}\leq \gamma<1$. However, for $\alpha<0$, even if we use \eqref{gen_final1}, we will not be able to get the needed estimate since the tangential part of the gradient has the wrong sign. The breakdown of estimates for $m=2$ is an issue that will reappear in the next estimates. 
\end{proof}

\subsection{Weighted Gagliardo--Nirenberg-type Inequality}

\begin{theorem}\label{lm_last_lemma}
    Let $P=3\,X^4-8\,X^3+10\,X^2-16\,X+7\in \R[X]$ and let 
    \begin{align*}
        m_0&=\frac{2}{3}+\frac{1}{3}\sqrt{-1-\frac{2^{\frac{2}{3}}}{\sqrt[3]{32+3\sqrt{114}}}+\sqrt[3]{2(32+3\sqrt{114})}}\\
        &+\frac{1}{2}\sqrt{-\frac{8}{9}+\frac{2^{\frac{8}{3}}}{9\sqrt[3]{32+3\sqrt{114}}}-\frac{4}{9}\sqrt[3]{2(32+3\sqrt{114})}+\frac{88}{9\sqrt{-1-\frac{2^{\frac{2}{3}}}{\sqrt[3]{32+3\sqrt{114}}}+\sqrt[3]{2(32+3\sqrt{114})}}}}\\
        &=2.039423\cdots
    \end{align*}
    be its largest real root. Fix $m_0<m<\infty$, and let 
    \begin{align*}
        \leb_m=\Delta+2(m-1)\frac{x}{|x|^2}\cdot \D +\frac{(m-1)^2}{|x|^2}
    \end{align*}
    and
    \begin{align*}
        \mathfrak{L}_m=\p{z}^2+\frac{m-1}{z}\p{z}+\frac{(m-1)(m-3)}{4z^2}.
    \end{align*}
    For all $0<\beta<1/2$ and $0<\gamma<1$, there exists $\Lambda_{\gamma}<\infty$ such that for all $0<a<b<\infty$, assuming that
    \begin{align*}
        \log\left(\frac{b}{a}\right)\geq \Lambda_{\gamma},
    \end{align*}
    the following property is verified.
    Define $\Omega=B_b\setminus\bar{B}_a(0)\subset \R^2$. For all function $u\in W^{2,2}(\Omega)$ such that $\leb_m^{\ast}\leb_mu=0$, there exists a universal constant $C_{m,\beta,\gamma}<\infty$ such that
     \begin{align}\label{lm_last_lemma_ineq}
        &\int_{\Omega}\frac{1}{|x|^2}\left|\D u+(m-1)\frac{x}{|x|^2}u\right|^2\left(\left(\frac{|x|}{b}\right)^{2\gamma}+\left(\frac{a}{|x|}\right)^{2\gamma}\right)dx\leq C_{m,\beta,\gamma}\left(\int_{\Omega}\frac{u^2}{|x|^4}\left(\left(\frac{|x|}{b}\right)^{4\beta}+\left(\frac{a}{|x|}\right)^{4\beta}\right)dx\right.\nonumber\\
        &\left.+\int_{\Omega}\left(\Delta u+2(m-1)\frac{x}{|x|^2}\cdot \D u+\frac{(m-1)^2}{|x|^2}u\right)^2dx+\int_{\Omega}\left|\p{z}^2u+\frac{m-1}{z}\p{z}u+\frac{(m-1)(m-3)}{4z^2}u\right|^2|dz|^2\right).
    \end{align}
\end{theorem}

    The proof, as the one of the corresponding inequality for $m=1$ proved in \cite{willmore_scs}, simply consists in showing that the $L^2$ norm of a $m$-biharmonic function (a function such that $\leb_m^{\ast}\leb_mu=0$) controls the $l^2$ norm of its Fourier coefficients. In fact, not all coefficients of the radial part of $u$ will be directly estimable, so we will have to use another argument. Since it is the most technical part of the proof, we give the fundamental technical lemmas separately. Since the lack of direct estimate comes from the insufficiency of the Cauchy--Schwarz inequality, we first compute the missing term in the inequality.

    \subsubsection{Refinements of the Cauchy--Schwarz Inequality}
    
    \begin{lemme}\label{Cauchy--Schwarz_missing-term_theorem}
        Let $(X,\mu)$ be a $\sigma$-finite measured space. Then, for all $f_1,f_2\in L^2(X,\mu,\C)$, we have
        \begin{align}\label{Cauchy--Schwarz_missing-term}
            \left(\int_{X}|f_1|^2d\mu\right)\left(\int_{X}|f_2|^2d\mu\right)-\left|\int_{X}f_1\bar{f_2}\,d\mu\right|^2=\frac{1}{2}\int_{X\times X}\left|f_1(x)f_2(y)-f_1(y)f_2(y)\right|^2d\mu(x)d\mu(y).
        \end{align}
    \end{lemme}
    \begin{proof}

        Indeed, Fubini's theorem implies that
        \begin{align}\label{symmetry_int}
            \left(\int_{X}|f_1|^2d\mu\right)\left(\int_{X}|f_2|^2d\mu\right)=\int_{X\times X}|f_1(x)|^2|f_2(y)|^2d\mu(x)d\mu(y).
        \end{align}
        On the other hand, we have
        \begin{align}\label{eq2}
            \left|\int_{X}f_1\bar{f_2}d\mu\right|^2&=\left(\int_{X}f_1(x)\bar{f_2(x)}d\mu(x)\right)\left(\int_{X}\bar{f_1(y)}f_2(y)d\mu(y)\right)\nonumber\\
            &=\int_{X\times X}f_1(x)f_2(y)\bar{f_1(y)f_2(x)}d\mu(x)d\mu(y)\nonumber\\
            &=\Re\left(\int_{X\times X}f_1(x)f_2(y)\bar{f_1(y)f_2(x)}d\mu(x)d\mu(y)\right)\nonumber\\
            &=\int_{X\times X}\Re\left(f_1(x)f_2(y)\bar{f_1(y)f_2(x)}\right)d\mu(x)d\mu(y)
        \end{align}
        Furthermore, we have by \eqref{symmetry_int} and symmetry
        \begin{align}\label{eq3}
            \left(\int_{X}|f_1|^2d\mu\right)\left(\int_{X}|f_2|^2d\mu\right)=\frac{1}{2}\int_{X\times X}\left(|f_1(x)|^2|f_2(y)|^2+|f_1(y)|^2|f_2(x)|^2\right)d\mu(x)d\mu(y),
        \end{align}
        and finally, we have by \eqref{eq2} and \eqref{eq3}
        \begin{align}
            &\left(\int_{X}|f_1|^2d\mu\right)\left(\int_{X}|f_2|^2d\mu\right)-\left|\int_{X}f_1\bar{f_2}\,d\mu\right|^2\nonumber\\
            &=\frac{1}{2}\int_{X\times X}\left(|f_1(x)|^2|f_2(y)|^2+|f_1(y)|^2|f_2(x)|^2-2\,\Re\left(f_1(x)f_2(y)\bar{f_1(y)f_2(x)}\right)\right)d\mu(x)d\mu(y)\nonumber\\
            &=\frac{1}{2}\int_{X\times X}\left|f_1(x)f_2(y)-f_1(y)f_2(x)\right|^2d\mu(x)d\mu(y),
        \end{align}
        which concludes the proof.
    \end{proof}
    \begin{lemme}\label{estimée_somme_Cauchy--Schwarz_theorem}
    Let $(X,\mu)$ be a $\sigma$-finite measured space. For all $f_1,f_2\in L^2(X,\mu,\C)$, for all $\lambda_1,\lambda_2\in \R$ we have
    \begin{align}\label{estimée_somme_Cauchy--Schwarz}
        \frac{|\lambda_1\lambda_2|}{2}\int_{X\times X}\left|\det\begin{pmatrix}
            f_1(x) & f_2(x)\\
            f_1(y) & f_2(y)
        \end{pmatrix}\right|^2d\mu(x)d\mu(y)\leq \np{f_1}{2}{X}\np{f_2}{2}{X}\int_{X}\left|\lambda_1\,f_1+\lambda_2\,f_2\right|^2d\mu.
    \end{align}    
    \end{lemme}
    \begin{proof}
        For all $\lambda\in \R$, we estimate thanks to the identity \eqref{Cauchy--Schwarz_missing-term} of Lemma \ref{Cauchy--Schwarz_missing-term_theorem}
        \small
    \begin{align*}
        &\int_{X}\left|f_1+\lambda\,f_2\right|^2d\mu=\int_{X}|f_1|^2d\mu+\lambda^2\int_{X}|f_1|^2d\mu+2\lambda\,\Re\left(\int_{X}f_1\bar{f_2}d\mu\right)\\
        &\geq \int_{X}|f_1|^2d\mu+\lambda^2\int_{X}|f_2|^2d\mu-2|\lambda|\left|\int_{X}f_1\bar{f_2}d\mu\right|\\
        &=\int_{X}|f_1|^2d\mu+\lambda^2\int_{X}|f_2|^2d\mu\\
        &-2|\lambda|\sqrt{\left(\int_{X}|f_1|^2d\mu\right)\left(\int_{X}|f_2|^2d\mu\right)-\frac{1}{2}\int_{X\times X}\left|f_1(x)f_2(y)-f_1(y)f_2(x)\right|^2d\mu(x)d\mu(y)}\\
        &=\int_{X}|f_1|^2d\mu-2|\lambda|\sqrt{\left(\int_{X}|f_1|^2d\mu\right)\left(\int_{X}|f_2|^2d\mu\right)}+\lambda^2\int_{X}|f_2|^2d\mu\\
        &+2|\lambda|\sqrt{\left(\int_{X}|f_1|^2d\mu\right)\left(\int_{X}|f_2|^2d\mu\right)}\\
        &-2|\lambda|\sqrt{\left(\int_{X}|f_1|^2d\mu\right)\left(\int_{X}|f_2|^2d\mu\right)-\frac{1}{2}\int_{X\times X}\left|f_1(x)f_2(y)-f_1(y)f_2(x)\right|^2d\mu(x)d\mu(y)}\\
        &=\left(\sqrt{\int_{X}|f_1|^2d\mu}-|\lambda|\sqrt{\int_{X}|f_2|^2d\mu}\right)^2\\
        &+\frac{|\lambda| \int_{X\times X}\left|f_1(x)f_2(y)-f_1(y)f_2(x)\right|^2d\mu(x)d\mu(y)}{\sqrt{\left(\int_{X}|f_1|^2d\mu\right)\left(\int_{X}|f_2|^2d\mu\right)}+\sqrt{\left(\int_{X}|f_1|^2d\mu\right)\left(\int_{X}|f_2|^2d\mu\right)-\frac{1}{2}\int_{X\times X}\left|f_1(x)f_2(y)-f_1(y)f_2(x)\right|^2d\mu(x)d\mu(y)}}.
    \end{align*}
    \normalsize
    Now, replacing $f_1$ by $\lambda_1f_1$ and $\lambda$ by $\lambda_2$, we have (provided that $\lambda_1\neq 0$, but the estimate is trivial if $\lambda_1$ or $\lambda_2$ vanishes)
    \small
    \begin{align}\label{improved_cauchy_schwarz}
        &\int_{X}|\lambda_1\,f_1+\lambda_2\,f_2|^2d\mu=\lambda_1^2\int_{X}\left|f_1+\frac{\lambda_2}{\lambda_1}f_2\right|^2d\mu\nonumber\\
        &\geq \left(|\lambda_1|\sqrt{\int_{X}|f_1|^2d\mu}-|\lambda_2|\sqrt{\int_{X}|f_2|^2d\mu}\right)^2\nonumber\\
        &+\frac{|\lambda_1\lambda_2|\int_{X\times X}\left|f_1(x)f_2(y)-f_1(y)f_2(x)\right|^2d\mu(x)d\mu(y)}{\sqrt{\left(\int_{X}|f_1|^2d\mu\right)\left(\int_{X}|f_2|^2d\mu\right)}+\sqrt{\left(\int_{X}|f_1|^2d\mu\right)\left(\int_{X}|f_2|^2d\mu\right)-\frac{1}{2}\int_{X\times X}\left|f_1(x)f_2(y)-f_1(y)f_2(x)\right|^2d\mu(x)d\mu(y)}}.
    \end{align}
    \normalsize
    In particular, we get the inequality
    \begin{align}\label{improved_cauchy_schwarz2}
        \frac{|\lambda_1\lambda_2|}{2}\frac{\int_{X\times X}\left|f_1(x)f_2(y)-f_1(y)f_2(x)\right|^2d\mu(x)d\mu(y)}{\sqrt{\left(\int_{X}|f_1|^2d\mu\right)\left(\int_{X}|f_2|^2d\mu\right)}}\leq \int_{X}\left|\lambda_1\,f_1+\lambda_2\,f_2\right|^2d\mu,
    \end{align}
    which concludes the proof of the lemma.
    \end{proof}

    \begin{lemme}\label{log_estimate_sum_theorem}
        Let $0<a<b\leq e^{-1}$ and $\gamma\in \R\setminus\ens{0}$. Define $f_1(r)=r^{\gamma-\frac{1}{2}}\log(r)$ and $f_2(r)=r^{\gamma-\frac{1}{2}}$. For all $\lambda>0$, we have 
        \begin{align}\label{log_estimate_sum}
            \lambda\, \frac{b^{2\gamma}}{\log\left(\frac{1}{b}\right)}\leq \frac{8\gamma^2(1+\gamma)}{\left(1-\left(\frac{a}{b}\right)^{2\gamma}\right)^2-4\gamma^2\left(\frac{a}{b}\right)^{2\gamma}\log^2\left(\frac{b}{a}\right)}\int_{a}^b\left(f_1+\lambda\,f_2\right)^2dr.
        \end{align}
    \end{lemme}
    \begin{proof}
    We have
    \begin{align*}
        &\int_{a}^bf_1^2(r)dr=\int_{a}^br^{2\gamma-1}\log^2(r)dr=\left[\frac{1}{2\gamma}r^{2\gamma}\log^2(r)\right]_a^b-\frac{1}{\gamma}\int_{a}^br^{2\gamma-1}\log(r)dr\\
        &=\frac{1}{2\gamma}\left(b^{2\gamma}\log^2\left(\frac{1}{b}\right)-a^{2\gamma}\log^2\left(\frac{1}{a}\right)\right)+\left[\frac{1}{2\gamma^2}r^{2\gamma}\log\left(\frac{1}{b}\right)\right]_a^b+\frac{1}{2\gamma^2}\int_{a}^br^{2\gamma-1}dr\\
        &=\frac{1}{2\gamma}\left(b^{2\gamma}\log^2\left(\frac{1}{b}\right)-a^{2\gamma}\log^2\left(\frac{1}{a}\right)\right)+\frac{1}{2\gamma^2}\left(b^{2\gamma}\log\left(\frac{1}{b}\right)-a^{2\gamma}\log\left(\frac{1}{a}\right)\right)+\frac{1}{4\gamma^3}b^{2\gamma}\left(1-\left(\frac{a}{b}\right)^{2\gamma}\right).
    \end{align*}
    Then, we have
    \begin{align*}
        &\int_{a}^bf_1(r)f_2(r)dr=\int_{a}^br^{2\gamma-1}\log(r)dr=-\frac{1}{2\gamma}\left(b^{2\gamma}\log\left(\frac{1}{b}\right)-a^{2\gamma}\log\left(\frac{1}{a}\right)\right)-\frac{1}{4\gamma^2}b^{2\gamma}\left(1-\left(\frac{a}{b}\right)^{2\gamma}\right),
    \end{align*}
    and finally
    \begin{align*}
        \int_{a}^bf_2^2(r)=\int_{a}^br^{2\gamma-1}dr=\frac{1}{2\gamma}b^{2\gamma}\left(1-\left(\frac{a}{b}\right)^{2\gamma}\right).
    \end{align*}
    We readily see that the triangle does not give us any non-trivial estimate. Therefore, we need the refinenment of the Cauchy--Schwarz inequality proven above.    
    Now, we get
    \begin{align*}
        \int_{a}^b\int_{a}^b\left(f_1(r)f_2(s)-f_1(s)f_2(r)\right)^2drds=\int_{a}^b\int_{a}^br^{2\gamma-1}s^{2\gamma-1}\log^2\left(\frac{r}{s}\right)dr\,ds.
    \end{align*}
    We have
    \begin{align*}
        &\int_{a}^br^{2\gamma-1}\log^2\left(\frac{r}{s}\right)ds=s^{2\gamma}\int_{\frac{a}{s}}^{\frac{b}{s}}t^{2\gamma-1}\log^2(t)dt=s^{2\gamma}\left[\frac{1}{2\gamma}t^{2\gamma}\log^2\left(t\right)-\frac{1}{2\gamma^2}t^{2\gamma}\log\left(t\right)+\frac{1}{4\gamma^3}t^{2\gamma}\right]_{\frac{a}{s}}^{\frac{b}{s}}\\
        &=\frac{1}{2\gamma}\left(b^{2\gamma}\log^2\left(\frac{s}{b}\right)-a^{2\gamma}\log^2\left(\frac{s}{a}\right)\right)+\frac{1}{2\gamma^2}\left(b^{2\gamma}\log\left(\frac{s}{b}\right)-a^{2\gamma}\log\left(\frac{s}{a}\right)\right)+\frac{1}{4\gamma^3}b^{2\gamma}\left(1-\left(\frac{a}{b}\right)^{2\gamma}\right)
    \end{align*}
    Therefore, we have
    \begin{align}\label{I_V}
        \int_{a}^b\int_{a}^br^{2\gamma-1}s^{2\gamma-1}\log^2\left(\frac{r}{s}\right)dr\,ds&=\mathrm{(I)}+\mathrm{(II)}+\mathrm{(III)}+\mathrm{(IV)}+\mathrm{(V)}.
    \end{align}
    We first have
    \begin{align}\label{estimate_I}
        \mathrm{(I)}&=\frac{1}{2\gamma}b^{2\gamma}\int_{a}^{b}s^{2\gamma-1}\log^2\left(\frac{s}{b}\right)ds=\frac{1}{2\gamma}b^{4\gamma}\int_{\frac{a}{b}}^1t^{2\gamma-1}\log^2(t)dt\nonumber\\
        &=\frac{1}{2\gamma}b^{4\gamma}\left[\frac{1}{2\gamma}t^{2\gamma}\log^2(t)-\frac{1}{2\gamma^2}\log(t)+\frac{1}{4\gamma^3}t^{2\gamma}\right]_{\frac{a}{b}}^{1}\nonumber\\
        &=\frac{1}{2\gamma}b^{4\gamma}\left\{\frac{1}{4\gamma^3}\left(1-\left(\frac{a}{b}\right)^{2\gamma}\right)-\frac{1}{2\gamma}\left(\frac{a}{b}\right)^{2\gamma}\log^2\left(\frac{b}{a}\right)-\frac{1}{2\gamma^2}\left(\frac{a}{b}\right)^{2\gamma}\log\left(\frac{b}{a}\right)\right\}\nonumber\\
        &=b^{4\gamma}\left\{\frac{1}{8\gamma^4}\left(1-\left(\frac{a}{b}\right)^{2\gamma}\right)-\frac{1}{4\gamma^2}\left(\frac{a}{b}\right)^{2\gamma}\log^2\left(\frac{b}{a}\right)-\frac{1}{4\gamma^3}\left(\frac{a}{b}\right)^{2\gamma}\log\left(\frac{b}{a}\right)\right\}.
    \end{align}
    Then, we compute
    \begin{align}\label{estimate_II}
        \mathrm{(II)}&=-\frac{1}{2\gamma}a^{2\gamma}\int_{a}^bs^{2\gamma-1}\log^2\left(\frac{s}{a}\right)ds=-\frac{1}{2\gamma}a^{4\gamma}\int_{1}^{\frac{b}{a}}t^{2\gamma-1}\log^2(t)dt\nonumber\\
        &=-\frac{1}{2\gamma}a^{4\gamma}\left[\frac{1}{2\gamma}t^{2\gamma}\log^2(t)-\frac{1}{2\gamma^2}t^{2\gamma}\log(t)+\frac{1}{4\gamma^3}r^{2\gamma}\right]_1^{\frac{b}{a}}\nonumber\\
        &=-\frac{1}{2\gamma}a^{4\gamma}\left\{\frac{1}{2\gamma}\left(\frac{b}{a}\right)^{2\gamma}\log^2\left(\frac{b}{a}\right)-\frac{1}{2\gamma^2}\left(\frac{b}{a}\right)^{2\gamma}\log\left(\frac{b}{a}\right)+\frac{1}{4\gamma^3}\left(\left(\frac{b}{a}\right)^{2\gamma}-1\right)\right\}\nonumber\\
        &=b^{4\gamma}\left\{-\frac{1}{4\gamma^2}\left(\frac{a}{b}\right)^{2\gamma}+\frac{1}{4\gamma^3}\left(\frac{a}{b}\right)^{2\gamma}\log\left(\frac{b}{a}\right)-\frac{1}{8\gamma^4}\left(1-\left(\frac{a}{b}\right)^{2\gamma}\right)\left(\frac{a}{b}\right)^{2\gamma}\right\}.
    \end{align}
    Likewise, we have
    \begin{align}\label{estimate_III}
        \mathrm{(III)}&=\frac{1}{2\gamma^2}b^{2\gamma}\int_{a}^{b}s^{2\gamma-1}\log\left(\frac{s}{b}\right)ds=\frac{1}{2\gamma^2}b^{4\gamma}\int_{\frac{a}{b}}^1t^{2\gamma-1}\log(t)dt=\frac{1}{2\gamma^2}b^{4\gamma}\left[\frac{1}{2\gamma}t^{2\gamma}\log(t)-\frac{1}{4\gamma^2}t^{2\gamma}\right]_{\frac{a}{b}}^1\nonumber\\
        &=\frac{1}{2\gamma^2}b^{4\gamma}\left\{\frac{1}{2\gamma}\left(\frac{a}{b}\right)^{2\gamma}\log\left(\frac{b}{a}\right)-\frac{1}{4\gamma^2}\left(1-\left(\frac{a}{b}\right)^{2\gamma}\right)\right\}\nonumber\\
        &=b^{4\gamma}\left\{-\frac{1}{8\gamma^4}\left(1-\left(\frac{a}{b}\right)^{2\gamma}\right)^{2\gamma}+\frac{1}{4\gamma^3}\left(\frac{a}{b}\right)^{2\gamma}\log\left(\frac{b}{a}\right)\right\}.
    \end{align}
    Likewise, we have
    \begin{align}\label{estimate_IV}
        \mathrm{(IV)}&=-\frac{1}{2\gamma^2}a^{2\gamma}\int_{a}^bs^{2\gamma-1}\log\left(\frac{s}{a}\right)ds=-\frac{1}{2\gamma^2}a^{4\gamma}\int_{1}^{\frac{b}{a}}t^{2\gamma-1}\log(t)dt\nonumber\\
        &=-\frac{1}{2\gamma^2}a^{4\gamma}\left[\frac{1}{2\gamma}t^{2\gamma}\log(t)-\frac{1}{4\gamma^2}t^{2\gamma}\right]_{1}^{\frac{b}{a}}\nonumber\\
        &=-\frac{1}{2\gamma^2}a^{4\gamma}\left\{\frac{1}{2\gamma}\left(\frac{b}{a}\right)^{2\gamma}\log\left(\frac{b}{a}\right)-\frac{1}{4\gamma^2}\left(\left(\frac{b}{a}\right)^{2\gamma}-1\right)\right\}\nonumber\\
        &=b^{4\gamma}\left\{-\frac{1}{4\gamma^3}\left(\frac{a}{b}\right)^{2\gamma}\log\left(\frac{b}{a}\right)+\frac{1}{8\gamma^4}\left(1-\left(\frac{a}{b}\right)^{2\gamma}\right)\left(\frac{a}{b}\right)^{2\gamma}\right\}.
    \end{align}
    Finally, we have
    \begin{align}\label{estimate_V}
        \mathrm{(V)}=\frac{1}{4\gamma^3}b^{2\gamma}\left(1-\left(\frac{a}{b}\right)^{2\gamma}\right)\int_{a}^bs^{2\gamma-1}ds=\frac{1}{8\gamma^4}b^{4\gamma}\left(1-\left(\frac{a}{b}\right)^{2\gamma}\right)^2.
    \end{align}
    Finally, gathering the estimates \eqref{I_V}, \eqref{estimate_I}, \eqref{estimate_II}, \eqref{estimate_III}, \eqref{estimate_IV}, and \eqref{estimate_V}, we deduce that
    \begin{align}\label{rest_cs}
        \int_{a}^b\int_{a}^br^{2\gamma-1}s^{2\gamma-1}\log^2\left(\frac{r}{s}\right)dr\,ds&=b^{4\gamma}\left\{\frac{1}{8\gamma^4}b^{4\gamma}\left(1-\left(\frac{a}{b}\right)^{2\gamma}\right)^{2}-\frac{1}{2\gamma^2}\left(\frac{a}{b}\right)^{2\gamma}\log^2\left(\frac{b}{a}\right)\right\}\nonumber\\
        &=\frac{1}{8\gamma^4}b^{4\gamma}\left\{\left(1-\left(\frac{a}{b}\right)^{2\gamma}\right)^2-4\gamma^2\left(\frac{a}{b}\right)^{2\gamma}\log^2\left(\frac{b}{a}\right)\right\},
    \end{align}
    which notably furnishes the elementary inequality
    \begin{align}\label{elementary}
        2\gamma\,x^{\gamma}\log\left(\frac{1}{x}\right)< 1-x^{2\gamma}\quad \text{for all}\;\gamma>0,\;\, \text{for all}\; 0<x<1
    \end{align}
    which can also be checked directly. Now, we have the elementary estimate
    \begin{align*}
        \int_{a}^bf_1^2(r)dr\leq \frac{1}{2\gamma}\left(1+\frac{1}{\gamma}+\frac{1}{2\gamma^2}\right)b^{2\gamma}\log^2\left(\frac{1}{b}\right),
    \end{align*}
    which implies that 
    \begin{align*}
        \sqrt{\int_{a}^bf_1^2(r)dr\int_{a}^bf_2^2(r)dr}&\leq \sqrt{\frac{1}{2\gamma}\left(1+\frac{1}{\gamma}+\frac{1}{2\gamma^2}\right)b^{2\gamma}\log^2\left(\frac{1}{b}\right)\times \frac{1}{2\gamma}b^{2\gamma}}\\
        &\leq \frac{1}{2\gamma}\sqrt{1+\frac{1}{\gamma}+\frac{1}{2\gamma^2}}b^{2\gamma}\log\left(\frac{1}{b}\right).
    \end{align*}
    Therefore, we have
    \begin{align*}
        &\frac{\int_{a}^b\int_{a}^b\big(f_1(r)f_2(s)-f_1(s)f_2(r)\big)^2}{\sqrt{\int_{a}^bf_1^2(r)dr\int_{a}^bf_2^2(r)dr}}\geq \frac{1}{8\gamma^4}b^{4\gamma}\left\{\left(1-\left(\frac{a}{b}\right)^{2\gamma}\right)^2-4\gamma^2\left(\frac{a}{b}\right)^{2\gamma}\log^2\left(\frac{b}{a}\right)\right\}\\
        &\times \frac{2\sqrt{2}\gamma^2}{\sqrt{2\gamma^2+2\gamma+1}}\frac{1}{b^{2\gamma}\log\left(\frac{1}{b}\right)}\\
        &=\frac{1}{2\sqrt{2}\gamma^2\sqrt{2\gamma^2+2\gamma+1}}\frac{b^{2\gamma}}{\log\left(\frac{1}{b}\right)}\left\{\left(1-\left(\frac{a}{b}\right)^{2\gamma}\right)^2-4\gamma^2\left(\frac{a}{b}\right)^{2\gamma}\log^2\left(\frac{b}{a}\right)\right\}.
    \end{align*}
    Finally, using inequality \eqref{Cauchy--Schwarz_missing-term} from Lemma \ref{Cauchy--Schwarz_missing-term_theorem}, we deduce that for all $\lambda>0$
    \begin{align*}
        \frac{\lambda}{4\sqrt{2}\gamma^2\sqrt{2\gamma^2+2\gamma+1}}\frac{b^{2\gamma}}{\log\left(\frac{1}{b}\right)}\left\{\left(1-\left(\frac{a}{b}\right)^{2\gamma}\right)^2-4\gamma^2\left(\frac{a}{b}\right)^{2\gamma}\log^2\left(\frac{b}{a}\right)\right\}\leq \int_{a}^b\left(f_1+\lambda\,f_2\right)^2dr,
    \end{align*}
    which concludes the proof of the lemma. 
    \end{proof}

    \subsubsection{Proof of the Main Inequality}

\begin{proof}
    \textbf{Step 1.} Decomposition in Fourier modes.
    Thanks to the appendix of \cite{willmore_scs}, the inequality is verified for $m=1$. We can therefore assume without loss of generality that $m>1$.
    Expand $u$ as 
    \begin{align*}
        u(r,\theta)=\sum_{n\in \Z}u_n(r)e^{in\theta}.
    \end{align*}
    Recall that 
    \begin{align*}
        \leb_m^{\ast}\leb_m=\Delta^2+2(m^2-1)\frac{1}{|x|^2}\Delta-4(m^2-1)\left(\frac{x}{|x|}\right)^t\cdot \D^2(\,\cdot\,)\cdot\left(\frac{x}{|x|^2}\right)+\frac{(m^2-1)^2}{|x|^4}.
    \end{align*}
    We also have
    \begin{align*}
        \Delta&=\p{r}^2+\frac{1}{r}\p{r}+\frac{1}{r^2}\p{\theta}^2\\
        \p{r}\Delta&=\p{r}^3+\frac{1}{r}\p{r}^2-\frac{1}{r^2}\p{r}+\frac{1}{r^2}\p{r}\p{\theta}^2-\frac{2}{r^3}\p{\theta}^2\\
        \p{r}^2\Delta&=\p{r}^4+\frac{1}{r}\p{r}^3-\frac{2}{r^2}\p{r}^2+\frac{2}{r^3}\p{r}+\frac{1}{r^2}\p{r}^2\p{\theta}^2-\frac{4}{r^3}\p{r}\p{\theta}^2+\frac{6}{r^4}\p{\theta}^2\\
        \Delta^2&=\p{r}^2\Delta+\frac{1}{r}\p{r}\Delta+\frac{1}{r^2}\p{\theta}^2\Delta\\
        &=\p{r}^4+\frac{1}{r}\p{r}^3-\frac{2}{r^2}\p{r}^2+\frac{2}{r^3}\p{r}+\frac{1}{r^2}\p{r}^2\p{\theta}^2-\frac{4}{r^3}\p{r}\p{\theta}^2+\frac{6}{r^4}\p{\theta}^2\\
        &+\frac{1}{r}\p{r}^3+\frac{1}{r^2}\p{r}^2-\frac{1}{r^3}\p{r}+\frac{1}{r^3}\p{r}\p{\theta}^2-\frac{2}{r^4}\p{\theta}^2\\
        &+\frac{1}{r^2}\p{r}^2\p{\theta}^2+\frac{1}{r^3}\p{r}\p{\theta}^2+\frac{1}{r^4}\p{\theta}^4\\
        &=\p{r}^4+\frac{2}{r}\p{r}^3-\frac{1}{r^2}\p{r}^2+\frac{1}{r^3}\p{r}+\frac{2}{r^2}\p{r}^2\p{\theta}^2-\frac{2}{r^3}\p{r}\p{\theta}^2+\frac{4}{r^4}\p{\theta}^2+\frac{1}{r^4}\p{\theta}^4.
    \end{align*}
    Therefore, we have
    \begin{align*}
        \leb_m^{\ast}\leb_m&=\p{r}^4+\frac{2}{r}\p{r}^3-\frac{1}{r^2}\p{r}^2+\frac{1}{r^3}\p{r}+\frac{2}{r^2}\p{r}^2\p{\theta}^2-\frac{2}{r^3}\p{r}\p{\theta}^2+\frac{4}{r^4}\p{\theta}^2+\frac{1}{r^4}\p{\theta}^4\\
        &+2(m^2-1)\left(\frac{1}{r^2}\p{r}^2+\frac{1}{r^3}\p{r}+\frac{1}{r^4}\p{\theta}^2\right)-4(m^2-1)\p{r}^2+\frac{(m^2-1)^2}{r^4}\\
        &=\p{r}^4+\frac{2}{r}\p{r}^3-\frac{2m^2-1}{r^2}\p{r}^2+\frac{2m^2-1}{r^3}\p{r}+\frac{(m^2-1)^2}{r^4}+\frac{2}{r^2}\p{r}^2\p{\theta}^2-\frac{2}{r}\p{r}\p{\theta}^2+\frac{2(m^2+1)}{r^4}\p{\theta}^2+\frac{1}{r^4}\p{\theta}^4.
    \end{align*}
    Notice that 
    \begin{align*}
        \leb_m=\Delta+2(m-1)\frac{x}{|x|^2}\cdot \D +\frac{(m-1)^2}{|x|^2}=\p{r}^2+(2m-1)\frac{1}{r}\p{r}+\frac{(m-1)^2}{r^2}+\frac{1}{r^2}\p{\theta}^2u.
    \end{align*}
    Therefore, we have
    \begin{align*}
        \Pi_{n^2}(\leb_m)=\p{r}^2+(2m-1)\frac{1}{r}\p{r}+\frac{(m-1)^2-n^2}{r^2}.
    \end{align*}
    If $u=Y(\log(r))$, then
    \begin{align*}
    \left\{\begin{alignedat}{1}
        u'&=\frac{1}{r}Y'\\
        u''&=\frac{1}{r^2}(Y''-Y'),
        \end{alignedat}\right.
    \end{align*}
    and we get
    \begin{align*}
        \Pi_{n^2}(\leb_m)u&=\frac{1}{r^2}\left(Y''-Y'+(2m-1)Y'+(m-1-n)(m-1+n)Y\right)\\
        &=\frac{1}{r^2}\left(Y''+2(m-1)Y'+(m-1-n)(m-1+n)Y\right).
    \end{align*}
    Its characteristic polynomial is given by 
    \begin{align*}
        P_m=X^2+2(m-1)X+(m-1)^2-n^2,
    \end{align*}
    whose discriminant is given by
    \begin{align*}
        \Delta_m=4(m-1)^2-4((m-1)^2-n^2)=4n^2,
    \end{align*}
    which shows that the roots of $P_m$ are given by $1-m\pm n$. Therefore, if $n\neq 0$, we have
    \begin{align*}
        \mathrm{Ker}\left(\Pi_{n^2}\left(\leb_m\right)\right)=\mathrm{Span}(r^{1-m+n},r^{1-m-n}).
    \end{align*}
    For $n=0$, we have
    \begin{align*}
        \mathrm{Ker}\left(\Pi_{0}\left(\leb_m\right)\right)=\mathrm{Span}\left(r^{1-m},r^{1-m}\log(r)\right).
    \end{align*}
    Then, we have
    \begin{align*}
        \leb_m^{\ast}=\Delta-2(m-1)\frac{x}{|x|^2}\cdot \D+\frac{(m-1)^2}{|x|^2}=\p{r}^2-(2m-3)\frac{1}{r}\p{r}+\frac{(m-1)^2}{r^2}+\frac{1}{r^2}\p{\theta}^2.
    \end{align*}
    We deduce that
    \begin{align*}
        \Pi_{n^2}(\leb_m^{\ast})=\p{r}^2-(2m-3)\frac{1}{r}\p{r}+\frac{(m-1)^2-n^2}{r^2}.
    \end{align*}
    Making the same change of variable $u=Y(\log(r))$, we get
    \begin{align*}
        \Pi_{n^2}(\leb_m^{\ast})u=\frac{1}{r^2}\left(Y''-2(m-1)Y'+(m-1-n)(m-1+n)Y\right),
    \end{align*}
    which shows that (this result can also be inferred from \cite[p. 118]{willmore_scs})
    \begin{align*}
        \mathrm{Ker}\left(\Pi_{n^2}\left(\leb_m^{\ast}\right)\right)=\mathrm{Span}(r^{1+m+n},r^{1+m-n}).
    \end{align*}
    For $n=0$, we have
    \begin{align*}
        \mathrm{Ker}\left(\Pi_{0}\left(\leb_m^{\ast}\right)\right)=\mathrm{Span}\left(r^{1+m},r^{1+m}\log(r)\right).
    \end{align*}
    Therefore, if $\leb_m^{\ast}\leb_mu=0$ in $\Omega=B_b\setminus\bar{B}_a(0)$, we deduce that there exists $\ens{a_n}_{n\in \Z},\ens{b_n}_{n\in\Z}\subset \C$ and $\alpha_1,\alpha_2\in \R$ such that
    \begin{align}\label{formula_u_m-biharmonic}
        u(z)&=|z|^{1-m}\left(\alpha_1\log|z|+2\,\Re\left(\sum_{n\in \Z}a_nz^n\right)\right)+|z|^{m+1}\left(\alpha_2\log|z|+2\,\Re\left(\sum_{n\in \Z}b_nz^n\right)\right)\\
        &=\alpha_1\,r^{1-m}\log(r)+2\,\Re\left(\sum_{n\in \Z}a_n\,r^{1-m+n}e^{in\theta}\right)+\alpha_2r^{m+1}\log(r)+2\,\Re\left(\sum_{n\in \Z}b_n\,r^{m+1+n}e^{in\theta}\right).\nonumber
    \end{align}

    \textbf{Step 2.} Estimations of the $L^2$ norm of $\leb_mu$.

    For all  $\alpha\in \R$, we have
    \begin{align*}
        &\p{r}\left(r^{\alpha}\log(r)\right)=\alpha\, r^{\alpha-1}\log(r)+r^{\alpha-1}\\
        &\p{r}^2\left(r^{\alpha}\log(r)\right)=\alpha(\alpha-1)r^{\alpha-2}\log(r)+(2\alpha-1)r^{\alpha-2}\\
        &\leb_m\left(r^{\alpha}\log(r)\right)=\left(\p{r}^2+(2m-1)\frac{1}{r}\p{r}+\frac{(m-1)^2}{r^2}\right)\left(r^{\alpha}\log(r)\right)\\
        &=\alpha(\alpha-1)r^{\alpha-2}\log(r)+(2\alpha-1)r^{\alpha-2}+(2m-1)\alpha\,r^{\alpha-2}\log(r)+(2m-1)r^{\alpha-2}\\
        &+(m-1)^2r^{\alpha-2}\log(r)\\
        &=\left((m-1)^2+\alpha\left(2(m-1)+\alpha\right)\right)r^{\alpha-2}\log(r)+2(m-1+\alpha)r^{\alpha-2}.
    \end{align*}
    Therefore, we have
    \begin{align*}
        \leb_m\left(r^{m+1}\log(r)\right)&=\left((m-1)^2+(m+1)\left(2(m-1)+m+1\right)\right)r^{m-1}\log(r)+2(2m-1)r^{m-1}\\
        &=4m^2\,r^{m-1}\log(r)+4mr^{m-1}.
    \end{align*}
    Then, we have
    \begin{align*}
        \leb_m\left(r^{m+1+n}e^{in\theta}\right)&=\left(\p{r}^2+(2m-1)\frac{1}{r}\p{r}+\frac{(m-1)^2}{r^2}+\frac{1}{r^2}\p{\theta}^2\right)\left(r^{m+1+n}e^{in\theta}\right)\\
        &=\left((m+n+1)(m+n)+(2m-1)(m+n+1)+(m-1)^2-n^2\right)r^{m+n-1}e^{in\theta}\\
        &=4m(m+n)r^{m+n-1}e^{in\theta}.
    \end{align*}
    Therefore, we deduce that 
    \begin{align*}
        \leb_mu&=4m\bigg(m\alpha_2r^{m-1}\log(r)+\left(\alpha_2+2mb_0\right)r^{m-1}\\
        &+\sum_{n\in \Z^{\ast}}\left((m+n)b_nr^{m+n-1}+(m-n)\bar{b_{-n}}r^{m-n-1}\right)e^{in\theta}\bigg).
    \end{align*}
    Therefore, if $\mathrm{Rad}$ stands for the radial projection, we have
    \begin{align*}
        &\mathrm{Rad}\left(\leb_mu\right)^2=16m^2\left(m^2\alpha_2^2r^{2m-2}\log^2(r)+\left(\alpha_2+2mb_0\right)^2r^{2m-2}+2m\alpha_2\left(\alpha_2+2mb_0\right)r^{2m-2}\log(r)\right.\\
        &\left.+\sum_{n\in \Z^{\ast}}\left|(m+n)b_nr^{m+n-1}+(m-n)\bar{b_{-n}}r^{m-n-1}\right|^2\right)\\
        &=16m^2\left(m^2\alpha_2^2r^{2m-2}\log^2(r)+\left(\alpha_2+2mb_0\right)^2r^{2m-2}+2m\alpha_2\left(\alpha_2+2mb_0\right)r^{2m-2}\log(r)\right.\\
        &\left.+2\sum_{n\in \Z^{\ast}}(m+n)^2|b_n|^2r^{2m+2n-2}+4\sum_{n=1}^{\infty}\left(m^2-n^2\right)\Re\left(b_nb_{-n}\right)r^{2m-2}\right).
    \end{align*}
    Then, we have for $\alpha\neq -1$
    \begin{align*}
        \int r^{\alpha}\log(r)dr&=%\frac{r^{\alpha+1}}{\alpha+1}\log(r)-\frac{1}{\alpha+1}\int r^{\alpha}dr=
        \frac{r^{\alpha+1}}{\alpha+1}\log(r)-\frac{1}{(\alpha+1)^2}r^{\alpha+1}\\
        \int r^{\alpha}\log^2(r)dr&=%\frac{r^{\alpha+1}}{\alpha+1}\log^2(r)-\frac{2}{\alpha+1}\int r^{\alpha}\log(r)dr\\
        %&=
        \frac{r^{\alpha+1}}{\alpha+1}\log^2(r)-\frac{2}{(\alpha+1)^2}r^{\alpha+1}\log(r)+\frac{2}{(\alpha+1)^3}r^{\alpha+1},
    \end{align*}
    while for all $\beta>0$
    \begin{align*}
        \int \frac{\log^{\beta}(r)}{r}dr=\frac{1}{\beta+1}\log^{\beta+1}(r).
    \end{align*}
    Therefore, we get
    \begin{align}\label{l2_norm_lmu}
        &\int_{\Omega}\left(\leb_mu\right)^2dx=32\pi\, m^2\int_{a}^b\left(m^2\alpha_2^2r^{2m-1}\log^2(r)+\left(\alpha_2+2mb_0\right)^2r^{2m-1}\right.\\
        &+2m\alpha_2\left(\alpha_2+2mb_0\right)r^{2m-1}\log(r)
        +2\sum_{n\in \Z^{\ast}}(m+n)^2|b_n|^2r^{2m+2n-1}\nonumber\\
        &\left.+4\sum_{n=1}^{\infty}\left(m^2-n^2\right)\Re\left(b_nb_{-n}\right)r^{2m-1}\right)dr\nonumber\\
        &=32\pi\,m^2\bigg\{\frac{m}{2}\alpha_2^2\left(b^{2m}\log^2\left(\frac{1}{b}\right)+\frac{1}{m}b^{2m}\log\left(\frac{1}{b}\right)+\frac{1}{2m^2}b^{2m}\right.\nonumber\\
        &\left.-\left(a^{2m}\log^2\left(\frac{1}{a}\right)+\frac{1}{m}a^{2m}\log\left(\frac{1}{a}\right)+\frac{1}{2m^2}a^{2m}\right)\right)\nonumber\\
        &+\frac{1}{2m}\left(\alpha_2+2mb_0\right)^2b^{2m}\left(1-\left(\frac{a}{b}\right)^{2m}\right)\nonumber\\
        &-\alpha_2\left(\alpha_2+2mb_0\right)\left(b^{2m}\log\left(\frac{1}{b}\right)+\frac{1}{2m}b^{2m}-\left(a^{2m}\log\left(\frac{1}{a}\right)+\frac{1}{2m}a^{2m}\right)\right)\nonumber\\
        &+\sum_{n=1-m}^{\infty}(m+n)|b_n|^2b^{2(m+n)}\left(1-\left(\frac{a}{b}\right)^{2(m+n)}\right)+\sum_{n=m+1}^{\infty}(n-m)|b_{-n}|^2\frac{1}{a^{2(n-m)}}\left(1-\left(\frac{a}{b}\right)^{2(n-m)}\right)\nonumber\\
        &+\frac{2}{m}\sum_{n=1}^{\infty}(m^2-n^2)\Re\left(b_nb_{-n}\right)b^{2m}\left(1-\left(\frac{a}{b}\right)^{2m}\right)\bigg\}.
    \end{align}

    \textbf{Step 3.} Estimation of the frequencies for $|n|\geq m+1$.
    
    First, for all $n\geq m+1$, we have by Cauchy's inequality $2ab\leq \epsilon\, a^2+\dfrac{1}{\epsilon}b^2$ (for $\epsilon=\dfrac{1}{2}$)
    \begin{align*}
        &m^2(m+n)|b_n|^2b^{2(m+n)}\left(1-\left(\frac{a}{b}\right)^{2(m+n)}\right)+m^2(n-m)|b_{-n}|^2\frac{1}{a^{2(n-m)}}\left(1-\left(\frac{a}{b}\right)^{2(n-m)}\right)\\
        &+2m(m+n)(m-n)\Re\left(b_n{b_{-n}}\right)b^{2m}\left(1-\left(\frac{a}{b}\right)^{2m}\right)\\
        &\geq \frac{1}{2}m^2(m+n)|b_n|^{2(n+m)}\left(1-\left(\frac{a}{b}\right)^{2(m+n)}\right)\\
        &+m^2(n-m)|b_{-n}|^2\left(\frac{1}{a^{2(n-m)}}\left(1-\left(\frac{a}{b}\right)^{2(n-m)}\right)-2b^{2(n-m)}\frac{\left(1-\left(\frac{a}{b}\right)^{2m}\right)^2}{1-\left(\frac{a}{b}\right)^{2(m+n)}}\right)\\
        &=\frac{1}{2}m^2(m+n)|b_n|^{2(n+m)}\left(1-\left(\frac{a}{b}\right)^{2(m+n)}\right)\\
        &+\frac{m^2(n-m)}{1-\left(\frac{a}{b}\right)^{2(n+m)}}|b_{-n}|^2\frac{1}{a^{2(n-m)}}\left(\left(1-\left(\frac{a}{b}\right)^{2(n-m)}\right)\left(1-\left(\frac{a}{b}\right)^{2(n+m)}\right)-2\left(\frac{a}{b}\right)^{2(n-m)}\left(1-\left(\frac{a}{b}\right)^{2m}\right)^2\right)\\
        &\geq \frac{1}{2}m^2(m+n)|b_n|^{2(n+m)}\left(1-\left(\frac{a}{b}\right)^{2(m+n)}\right)+\frac{m^2(n-m)}{1-\left(\frac{a}{b}\right)^{2(n+m)}}|b_{-n}|^2\frac{1}{a^{2(n-m)}}\left(1-4\left(\frac{a}{b}\right)^2\right).
    \end{align*}
    Therefore, assuming that
    \begin{align*}
        \log\left(\frac{b}{a}\right)\geq \log\left(\frac{4}{\sqrt{3}}\right),
    \end{align*}
    we deduce that
    \begin{align}\label{m_biharmonique_1}
        &m^2(m+n)|b_n|^2b^{2(m+n)}\left(1-\left(\frac{a}{b}\right)^{2(m+n)}\right)+m^2(n-m)|b_{-n}|^2\frac{1}{a^{2(n-m)}}\left(1-\left(\frac{a}{b}\right)^{2(n-m)}\right)\nonumber\\
        &+2m(m+n)(m-n)\Re\left(b_n{b_{-n}}\right)b^{2m}\left(1-\left(\frac{a}{b}\right)^{2m}\right)\nonumber\\
        &\geq \frac{1}{2}m^2(m+n)|b_n|^{2(n+m)}\left(1-\left(\frac{a}{b}\right)^{2(m+n)}\right)+\frac{1}{4}m^2(n-m)|b_{-n}|^2\frac{1}{a^{2(n-m)}}\nonumber\\
        &\geq \frac{1}{4}m^2\left((m+n)|b_n|^{2}b^{2(n+m)}+(n-m)|b_{-n}|^2\frac{1}{a^{2(n-m)}}\right).
    \end{align}
    In particular, we have
    \begin{align}\label{bound_lm1}
        8\pi\,m^2\sum_{n=m+1}^{\infty}(m+n)|b_n|^2b^{2(m+n)}+8\pi\,m^2\sum_{n=m+1}^{\infty}(n-m)|b_{-n}|^2\frac{1}{a^{2(n-m)}}\leq \int_{\Omega}\left(\leb_mu\right)^2dx.
    \end{align}

    \textbf{Step 4.} Estimation of the frequencies for $1\leq |n|\leq m$.
    
    Notice that the crossed terms vanish for $m=n$, so we get for free
    \begin{align}\label{bound_lm1_extra}
        &64\pi\,m^3|b_m|^2b^{4m}\left(1-\left(\frac{a}{b}\right)^{4m}\right)+8\pi\,m^2\sum_{n=m+1}^{\infty}(m+n)|b_n|^2b^{2(m+n)}+8\pi\,m^2\sum_{n=m+1}^{\infty}(n-m)|b_{-n}|^2\frac{1}{a^{2(n-m)}}\nonumber\\
        &\leq \int_{\Omega}\left(\leb_mu\right)^2dx.
    \end{align}
    Now, for $1\leq n\leq m-1$, we have
    \begin{align*}
        &m^2(m+n)|b_n|^2b^{2(m+n)}\left(1-\left(\frac{a}{b}\right)^{2(m+n)}\right)+m^2(m-n)|b_{-n}|^2b^{2(m-n)}\left(1-\left(\frac{a}{b}\right)^{2(m-n)}\right)\\
        &+2m(m+n)(m-n)\Re\left(b_n{b_{-n}}\right)b^{2m}\left(1-\left(\frac{a}{b}\right)^{2m}\right)\\
        &\geq m^2(m+n)|b_n|^2b^{2(m+n)}\left(1-\left(\frac{a}{b}\right)^{2(m+n)}\right)-(m+n)^2(m-n)|b_n|^2b^{2(m+n)}\frac{\left(1-\left(\frac{a}{b}\right)^{2m}\right)^{2}}{1-\left(\frac{a}{b}\right)^{2(m-n)}}\\
        &=\frac{(m+n)|b_n|^2}{1-\left(\frac{a}{b}\right)^{2(m-n)}}b^{2(m+n)}\left(m^2\left(1-\left(\frac{a}{b}\right)^{2(m+n)}\right)\left(1-\left(\frac{a}{b}\right)^{2(m-n)}\right)-\left(m^2-n^2\right)\left(1-\left(\frac{a}{b}\right)^{2m}\right)^2\right)\\
        &\geq \frac{(m+n)|b_n|^2}{1-\left(\frac{a}{b}\right)^{2(m-n)}}\left(n^2-2m^2\left(\frac{a}{b}\right)^{2(m-n)}\right).
    \end{align*}
    Therefore, assuming that 
    \begin{align*}
        2m^2\left(\frac{a}{b}\right)^{2(m-n)}\leq \frac{n^2}{2},
    \end{align*}
    or
    \begin{align}\label{additional_conformal_class1}
        \log\left(\frac{b}{a}\right)\geq \frac{1}{m-n}\log\left(\frac{2m}{n}\right),
    \end{align}
    we deduce that 
    \begin{align*}
        &m^2(m+n)|b_n|^2b^{2(m+n)}\left(1-\left(\frac{a}{b}\right)^{2(m+n)}\right)+m^2(m-n)|b_{-n}|^2b^{2(m-n)}\left(1-\left(\frac{a}{b}\right)^{2(m-n)}\right)\\
        &+2m(m+n)(m-n)\Re\left(b_n{b_{-n}}\right)b^{2m}\left(1-\left(\frac{a}{b}\right)^{2m}\right)\geq \frac{1}{2}(m+n)n^2|b_n|^2b^{2(m+n)}.
    \end{align*}
    Likewise, we have
    \begin{align*}
        &m^2(m+n)|b_n|^2b^{2(m+n)}\left(1-\left(\frac{a}{b}\right)^{2(m+n)}\right)+m^2(m-n)|b_{-n}|^2b^{2(m-n)}\left(1-\left(\frac{a}{b}\right)^{2(m-n)}\right)\\
        &+2m(m+n)(m-n)\Re\left(b_n{b_{-n}}\right)b^{2m}\left(1-\left(\frac{a}{b}\right)^{2m}\right)\\
        &\geq m^2(m-n)|b_{-n}|^2b^{2(m-n)}\left(1-\left(\frac{a}{b}\right)^{2(m-n)}\right)-(m+n)(m-n)^2|b_{-n}|^2b^{2(m-n)}\frac{\left(1-\left(\frac{a}{b}\right)^{2m}\right)^2}{1-\left(\frac{a}{b}\right)^{2(m+n)}}\\
        &=\frac{(m-n)|b_{-n}|^2}{1-\left(\frac{a}{b}\right)^{2(m+n)}}b^{2(m-n)}\left(m^2\left(1-\left(\frac{a}{b}\right)^{2(m-n)}\right)\left(1-\left(\frac{a}{b}\right)^{2(m+n)}\right)-\left(m^2-n^2\right)\left(1-\left(\frac{a}{b}\right)^{2m}\right)^{2}\right)\\
        &\geq \frac{(m-n)|b_{-n}|^2}{1-\left(\frac{a}{b}\right)^{2(m+n)}}b^{2(m-n)}\left(n^2-2m^2\left(\frac{a}{b}\right)^{2(m-n)}\right)\\
        &\geq \frac{1}{2}n^2(m-n)|b_{-n}|^2b^{2(m-n)}
    \end{align*}
    thanks to the previous condition \eqref{additional_conformal_class1} on the conformal class. Gathering both estimates, we finally deduce by \eqref{bound_lm1} that
    \begin{align}\label{bound_lm1_improved}
        &\sum_{n=1}^{m-1}n^2(m+n)|b_n|^2b^{2(m+n)}+\sum_{n=1}^{m-1}n^2(m-n)|b_{-n}|^2b^{2(m-n)}+8m^3|b_m|^2b^{4m}\left(1-\left(\frac{a}{b}\right)^{4m}\right)\nonumber\\
        &+m^2\sum_{n=m+1}^{\infty}(m+n)|b_n|^2b^{2(m+n)}+m^2\sum_{n=m+1}^{\infty}(n-m)|b_{-n}|^2\frac{1}{a^{2(n-m)}}\leq \frac{1}{8\pi}\int_{\Omega}\left(\leb_mu\right)^2dx.
    \end{align}
    Notice that we could control all $b_n$ frequencies except for $b_{-m}$ and $b_0$.

    \textbf{Step 5.} Estimation of the $L^2$ norm of $\mathfrak{L}_mu$.
    
    Then, consider the following quantity
    \begin{align*}
        \int_{\Omega}|z|^{2-2m}\left|\p{z}^2v\right|^2|dz|^2.  
    \end{align*}
    If $v=|z|^{m-1}u=(z\z)^{\frac{m-1}{2}}u$, we get
    \begin{align*}
        \p{z}v&=|z|^{m-1}\p{z}u+\frac{m-1}{2}\z|z|^{m-3}u\\
        \p{z}^2v&=|z|^{m-1}\p{z}^2u+(m-1)\z|z|^{m-3}u+\frac{(m-1)(m-3)}{4}\z^2|z|^{m-5}u\\
        &=|z|^{m-1}\left(\p{z}^2u+\frac{m-1}{z}\p{z}u+\frac{(m-1)(m-3)}{4z^2}u\right).
    \end{align*}
    Therefore, we have
    \begin{align*}
        \int_{\Omega}|z|^{2-2m}\left|\p{z}^2v\right|^2|dz|^2=\int_{\Omega}\left|\p{z}^2u+\frac{m-1}{z}\p{z}u+\frac{(m-1)(m-3)}{4z^2}u\right|^2|dz|^2,
    \end{align*}
    which justifies the above choice in the inequality. We therefore define the complex elliptic operator (acting on real functions) with regular singularities
    \begin{align}\label{frak_lm}
       \mathfrak{L}_m=\p{z}^2+\frac{m-1}{z}\p{z}+\frac{(m-1)(m-3)}{4z^2}.
    \end{align}
    Recall that
    \begin{align*}
        u(z)&=|z|^{1-m}\left(\alpha_1\log|z|+2\,\Re\left(\sum_{n\in \Z}a_nz^n\right)\right)+|z|^{m+1}\left(\alpha_2\log|z|+2\,\Re\left(\sum_{n\in \Z}b_nz^n\right)\right).
    \end{align*}
    We get
    \begin{align*}
        \p{z}u&=\frac{1-m}{2}\z|z|^{-m-1}\left(\alpha_1\log|z|+2\,\Re\left(\sum_{n\in \Z}a_nz^n\right)\right)+|z|^{1-m}\left(\frac{\alpha_1}{2z}+\sum_{n\in \Z}n\,a_nz^{n-1}\right)\\
        &+\frac{m+1}{2}\z|z|^{m-1}\left(\alpha_2\log|z|+2\,\Re\left(\sum_{n\in \Z}b_nz^n\right)\right)+|z|^{m+1}\left(\frac{\alpha_2}{2z}+\sum_{n\in \Z}n\,b_nz^{n-1}\right)\\
        &=\frac{1-m}{2}|z|^{1-m}\frac{1}{z}\left(\alpha_1\log|z|+2\,\Re\left(\sum_{n\in \Z}a_nz^n\right)\right)+|z|^{1-m}\left(\frac{\alpha_1}{2z}+\sum_{n\in \Z}n\,a_nz^{n-1}\right)\\
        &+\frac{m+1}{2}|z|^{m+1}\frac{1}{z}\left(\alpha_2\log|z|+2\,\Re\left(\sum_{n\in \Z}b_nz^n\right)\right)+|z|^{m+1}\left(\frac{\alpha_2}{2z}+\sum_{n\in \Z}n\,b_nz^{n-1}\right)
    \end{align*}
    and
    \begin{align*}
        \p{z}^2u&=\frac{m^2-1}{4}\z^2|z|^{-m-3}\left(\alpha_1\log|z|+2\,\Re\left(\sum_{n\in \Z}a_nz^n\right)\right)+(1-m)\z|z|^{-m-1}\left(\frac{\alpha_1}{2z}+\sum_{n\in\Z}n\,a_nz^{n-1}\right)\\
        &+|z|^{1-m}\left(-\frac{\alpha_1}{2z^2}+\sum_{n\in \Z}n(n-1)a_nz^{n-2}\right)+\frac{m^2-1}{4}\z^2|z|^{m-3}\left(\alpha_2\log|z|+2\,\Re\left(\sum_{n\in \Z}b_nz^n\right)\right)\\
        &+(m+1)\z|z|^{m-1}\left(\frac{\alpha_2}{2z}+\sum_{n\in \Z}n\,b_nz^{n-1}\right)+|z|^{m+1}\left(-\frac{\alpha_2}{2z^2}+\sum_{n\in \Z}n(n-1)b_n\,z^{n-2}\right)\\
        &=\frac{m^2-1}{4}|z|^{1-m}\frac{1}{z^2}\left(\alpha_1\log|z|+2\,\Re\left(\sum_{n\in \Z}a_nz^n\right)\right)+(1-m)|z|^{1-m}\left(\frac{\alpha_1}{2z^2}+\sum_{n\in \Z}n\,a_nz^{n-2}\right)\\
        &+|z|^{1-m}\left(-\frac{\alpha_1}{2z^2}+\sum_{n\in \Z}n(n-1)a_nz^{n-2}\right)+\frac{m^2-1}{4}|z|^{m+1}\frac{1}{z^2}\left(\alpha_2\log|z|+2\,\Re\left(\sum_{n\in \Z}b_nz^n\right)\right)\\
        &+(m+1)|z|^{m+1}\left(\frac{\alpha_2}{2z^2}+\sum_{n\in \Z}n\,b_nz^{n-2}\right)+|z|^{m+1}\left(-\frac{\alpha_2}{2z^2}+\sum_{n\in \Z}n(n-1)b_nz^{n-2}\right).
    \end{align*}
    Therefore, we have
    \begin{align*}
       \mathfrak{L}_mu&=\p{z}^2u+\frac{m-1}{z}\p{z}u+\frac{(m-1)(m-3)}{4z^2}u\\
         &=|z|^{1-m}\left(\alpha_1\frac{\log|z|}{z^2}-\frac{\alpha_1}{z^2}+\sum_{n\in \Z}\left(n(n-1)-(m-1)\right)a_nz^{n-2}-\sum_{n\in \Z}(m-1)\bar{a_n}\frac{\z^n}{z^2}\right)\\
        &+|z|^{m+1}\left(m(m-1)\alpha_2\frac{\log|z|}{z^2}+m\frac{\alpha_2}{z^2}+\sum_{n\in \Z}\left(n(2m+n-1)+m(m-1)\right)b_nz^{n-2}\right.\\
        &\left.+m(m-1)\sum_{n\in \Z}\bar{b_n}\frac{\z^n}{z^2}\right).
    \end{align*}
    Therefore, we get
    \begin{align*}
        &\mathrm{Rad}\left(|\mathfrak{L}_mu|^2\right)=|z|^{2-2m}\left(\alpha_1^2\frac{\log^2|z|}{|z|^4}+\frac{\alpha_1^2}{|z|^4}+\sum_{n\in \Z}\left(\left(n(n-1)-(m-1)\right)^2+(m-1)^2\right)|a_n|^2|z|^{2n-4}\right.\\
        &\left.-2\alpha_1^2\frac{\log|z|}{|z|^4}-4(m-1)\alpha_1a_0\frac{\log|z|}{|z|^4}+4(m-1)\frac{\alpha_2a_0}{|z|^4}\right.\\
        &\left.-2(m-1)\sum_{n\in \Z}\left(n(n-1)-(m-1)\right)\Re\left(a_n\bar{a_{-n}}\right)\frac{1}{|z|^4}\right)\\
        &+|z|^{2m+2}\left(m^2(m-1)^2\alpha_2^2\frac{\log^2|z|}{|z|^4}+m^2\frac{\alpha_2^2}{|z|^4}\right.\\
        &\left.+\sum_{n\in \Z}\left(\left(n(2m+n-1)+m(m-1)\right)^2+m^2(m-1)^2\right)|b_n|^2|z|^{2n-4}\right.\\
        &+2m^2(m-1)\alpha_2^2\frac{\log|z|}{|z|^4}+4m^2(m-1)^2\alpha_2b_0\frac{\log|z|}{|z|^4}+4m^2(m-1)\frac{\alpha_2b_0}{|z|^4}\\
        &\left.+2m(m-1)\sum_{n\in \Z}\left(n(2m+n-1)+m(m-1)\right)\Re\left(b_n\bar{b_{-n}}\right)\frac{1}{|z|^4}\right)\\
        &+2m(m-1)\alpha_1\alpha_2\frac{\log^2|z|}{|z|^2}-m(m-3)\alpha_1\alpha_2\frac{\log|z|}{|z|^2}-2m\frac{\alpha_1\alpha_2}{|z|^2}\\
        &+\left(4m(m-1)\alpha_1b_0-4m(m-1)^2\alpha_2a_0\right)\frac{\log|z|}{|z|^2}
        -\left(4m(m-1)\alpha_1b_0+4m(m-1)\alpha_2a_0\right)\frac{1}{|z|^2}\\
        &+2\sum_{n\in \Z}\Big\{\left(n(n-1)-(m-1)\right)\left(n(2m+n-1)+m(m-1)\right)-m(m-1)^2\Big\}\Re\left(a_n\bar{b_{n}}\right)|z|^{2n-2}\\
        &+2\sum_{n\in \Z}\left(m(m-1)\left(n(n-1)-(m-1)\right)+(m-1)\left(n(n-2m-1)-m(m-1)\right)\right)\Re\left(a_n\bar{b_{-n}}\right)\frac{1}{|z|^2}.
    \end{align*}
    Then, we have for $\alpha\neq -1$
    \begin{align*}
        \int r^{\alpha}\log(r)dr&=\frac{r^{\alpha+1}}{\alpha+1}\log(r)-\frac{1}{\alpha+1}\int r^{\alpha}dr=\frac{r^{\alpha+1}}{\alpha+1}\log(r)-\frac{1}{(\alpha+1)^2}r^{\alpha+1}\\
        \int r^{\alpha}\log^2(r)dr%&=\frac{r^{\alpha+1}}{\alpha+1}\log^2(r)-\frac{2}{\alpha+1}\int r^{\alpha}\log(r)dr\\
        &=\frac{r^{\alpha+1}}{\alpha+1}\log^2(r)-\frac{2}{(\alpha+1)^2}r^{\alpha+1}\log(r)+\frac{2}{(\alpha+1)^3}r^{\alpha+1},
    \end{align*}
    and for all $\beta>0$
    \begin{align*}
        \int \frac{\log^{\beta}(r)}{r}dr=\frac{1}{\beta+1}\log^{\beta+1}(r).
    \end{align*}
    Therefore, we get
    \small
    \begin{align}\label{frak_m_l2}
        &\int_{\Omega}\left|\mathfrak{L}_mu\right|^2|dz|^2=2\pi\int_{a}^b\bigg(\alpha_1^2r^{-2m-1}\log^2(r)+\alpha_1^2r^{-2m-1}\nonumber\\
        &+\sum_{n\in \Z}\left(\left(n(n-1)-(m-1)^2\right)^2+(m-1)^2\right)|a_n|^2r^{2n-2m-1}-2\alpha_1^2r^{-2m-1}\log(r)\nonumber\\
        &-4(m-1)\alpha_1a_0r^{-2m-1}
        -2(m-1)\sum_{n\in \Z}\left(n(n-1)-(m-1)\right)\Re\left(a_n\bar{a_{-n}}\right)r^{-2m-1}\nonumber\\
        &+m^2(m-1)^2\alpha_2^2r^{2m-1}\log^2(r)+m^2\alpha^2_2r^{2m-1}\nonumber\\
        &+
        \sum_{n\in \Z}\left(\left(n(2m+n-1)+m(m-1)\right)^2+m^2(m-1)^2\right)|b_n|^2r^{2n+2m-1}\nonumber\\
        &+2m^2(m-1)\alpha_2^2r^{2m-1}\log(r)+4m^2(m-1)^2\alpha_2b_0r^{2m-1}\log(r)+4m^2(m-1)r^{2m-1}\alpha_2b_0\nonumber\\
        &+2m(m-1)\sum_{n\in \Z}\left(n(2m+n-1)+m(m-1)\right)\Re\left(b_n\bar{b_{-n}}\right)r^{2m-1}\nonumber\\
        &+2m(m-1)\alpha_1\alpha_2r^{-1}\log^2(r)-m(m-3)\alpha_1\alpha_2r^{-1}\log(r)-2m\alpha_1\alpha_2r^{-1}\nonumber\\
        &+4m(m-1)\left(\alpha_1b_0-(m-1)\alpha_2a_0\right)r^{-1}\log(r)-4m(m-1)\left(\alpha_1b_0+\alpha_2a_0\right)r^{-1}\nonumber\\
        &+2\sum_{n\in \Z}\Big\{\left(n(n-1)-(m-1)\right)\left(n(2m+n-1)+m(m-1)\right)-m(m-1)^2\Big\}\Re\left(a_n\bar{b_{n}}\right)r^{2n-1}\nonumber\\
        &+2\sum_{n\in \Z}\left(m(m-1)\left(n(n-1)-(m-1)\right)+(m-1)\left(n(n-2m-1)-m(m-1)\right)\right)\Re\left(a_n\bar{b_{-n}}\right)r^{-1}\bigg)dr\nonumber\\
        &=2\pi\bigg(\frac{1}{2m}\alpha_1^2\left(\frac{1}{a^{2m}}\log^2(a)-\frac{1}{m}\frac{1}{a^{2m}}\log\left(\frac{1}{a}\right)-\frac{1}{2m^2}\frac{1}{a^{2m}}-\left(\frac{1}{b^{2m}}\log^2(b)+\frac{1}{m}\frac{1}{b^{2m}}\log\left(\frac{1}{b}\right)+\frac{1}{2m^2}\frac{1}{b^{2m}}\right)\right)\nonumber\\
        &+\frac{1}{2m}\alpha_1^2\frac{1}{a^{2m}}\left(1-\left(\frac{a}{b}\right)^{2m}\right)+\frac{1}{2}\sum_{n=m+1}^{\infty}\frac{\left(n(n-1)-(m-1)\right)^2+(m-1)^2}{n-m}|a_n|^2b^{2(n-m)}\left(1-\left(\frac{a}{b}\right)^{2(n-m)}\right)\nonumber\\
        &+(m-1)^2\left((m-1)^2+1\right)|a_m|^2\log\left(\frac{b}{a}\right)\nonumber\\
        &+\frac{1}{2}\sum_{n=1-m}^{\infty}\frac{\left(n(n+1)-(m-1)^2\right)^2+(m-1)^2}{n+m}|a_{-n}|^2\frac{1}{a^{2(n+m)}}\left(1-\left(\frac{a}{b}\right)^{2(n+m)}\right)\nonumber\\
        &+\frac{1}{m}\alpha_1^2\left(\frac{1}{a^{2m}}\log\left(\frac{1}{a}\right)+\frac{1}{2m}\frac{1}{a^{2m}}-\left(\frac{1}{b^{2m}}\log\left(\frac{1}{b}\right)+\frac{1}{2m}\frac{1}{b^{2m}}\right)\right)-\frac{2(m-1)}{m}\alpha_1a_0\frac{1}{a^{2m}}\left(1-\left(\frac{a}{b}\right)^{2m}\right)\nonumber\\
        &-\frac{m-1}{m}\sum_{n\in \Z}\left(n(n-1)-(m-1)\right)\Re\left(a_n\bar{a_{-n}}\right)\frac{1}{a^{2m}}\left(1-\left(\frac{a}{b}\right)^{2m}\right)\nonumber\\
        &+\frac{m(m-1)^2}{2}\alpha_2^2\left(b^{2m}\log^2(b)+\frac{1}{m}b^{2m}\log\left(\frac{1}{b}\right)+\frac{1}{2m^2}b^{2m}\log\left(\frac{1}{b}\right)\right.\nonumber\\
        &\left.-\left(a^{2m}\log^2(a)+\frac{1}{m}a^{2m}\log\left(\frac{1}{a}\right)+\frac{1}{2m^2}a^{2m}\right)\right)+\frac{m}{2}\alpha_2^2b^{2m}\left(1-\left(\frac{a}{b}\right)^{2m}\right)\nonumber\\
        &+\frac{1}{2}\sum_{n=1-m}^{\infty}\frac{\left(n(2m+n-1)+m(m-1)\right)^2+m^2(m-1)^2}{n+m}|b_n|^2b^{2(n+m)}\left(1-\left(\frac{a}{b}\right)^{2(n+m)}\right)\nonumber\\
        &+m^2(m-1)^2|b_{-m}|^2\log\left(\frac{b}{a}\right)\nonumber\\
        &+\frac{1}{2}\sum_{n=m+1}^{\infty}\frac{\left(n(n-2m-1)+m(m-1)\right)^2+m^2(m-1)^2}{n-m}|b_{-n}|^2\frac{1}{a^{2(n-m)}}\left(1-\left(\frac{a}{b}\right)^{2(n-m)}\right)\nonumber\\
        &-m(m-1)\alpha_2^2\left(b^{2m}\log\left(\frac{1}{b}\right)+\frac{1}{2m}b^{2m}-\left(a^{2m}\log\left(\frac{1}{a}\right)+\frac{1}{2m}a^{2m}\right)\right)\nonumber\\
        &-2m(m-1)^2\alpha_2b_0\left(b^{2m}\log\left(\frac{1}{b}\right)+\frac{1}{2m}b^{2m}-\left(a^{2m}\log\left(\frac{1}{a}\right)+\frac{1}{2m}a^{2m}\right)\right)\nonumber\\     
        &+2m(m-1)\alpha_0b_0b^{2m}\left(1-\left(\frac{a}{b}\right)^{2m}\right)\nonumber\\
        &+(m-1)\sum_{n\in \Z}\left(n(2m+n-1)+m(m-1)\right)\Re\left(b_n\bar{b_{-n}}\right)b^{2m}\left(1-\left(\frac{a}{b}\right)^{2m}\right)\nonumber\\
        &+\frac{2m(m-1)}{3}\alpha_1\alpha_2\left(\log^3(b)-\log^3(a)\right)+\frac{m(m-3)}{2}\alpha_1\alpha_2\left(\log^2(b)-\log^2(a)\right)-2m\,\alpha_1\alpha_2\log\left(\frac{b}{a}\right)\nonumber\\
        &-4m(m-1)\left(\alpha_1b_0-(m-1)\alpha_2a_0\right)\left(\log^2(b)-\log^2(a)\right)-4(m-1)\left(\alpha_1b_0+\alpha_2a_0\right)\log\left(\frac{b}{a}\right)\nonumber\\
        &+\sum_{n=1}^{\infty}\frac{\left(n(n-1)-(m-1)\right)\left(n(2m+n-1)+m(m-1)\right)-m(m-1)^2}{n}\Re\left(a_n\bar{b_n}\right)b^{2n}\left(1-\left(\frac{a}{b}\right)^{2n}\right)\nonumber\\
        &-4m(m-1)^2a_0b_0\log\left(\frac{b}{a}\right)\nonumber\\
        &+\sum_{n=1}^{\infty}\frac{\left(n(n+1)-(m-1)\right)\left(n(n-2m-1)+m(m-1)\right)-m(m-1)^2}{n}\Re\left(a_{-n}\bar{b_{-n}}\right)\frac{1}{a^{2n}}\left(1-\left(\frac{a}{b}\right)^{2n}\right)\nonumber\\
        &+2\sum_{n\in \Z}\Big\{m(m-1)\left(n(n-1)-(m-1)\right)+(m-1)\left(n(n-2m-1)-m(m-1)\right)\Big\}\Re\left(a_n\bar{b_{-n}}\right)\log\left(\frac{b}{a}\right)\bigg).
    \end{align}
    \normalsize
    Let us first estimate the coefficients involving $a_n$ since we have already estimated all coefficients involving $b_n$ in the $L^2$ norm of $\leb_mu$.
    
    \textbf{Step 6. Estimation of the coefficients for $|n|\geq m+1$.}
    
    For all $n\geq m+1$, consider the following quantity
    \small
    \begin{align*}
        &\frac{1}{2}\frac{\left(n(n-1)-(m-1)\right)^2+(m-1)^2}{n-m}|a_n|^2b^{2(n-m)}\left(1-\left(\frac{a}{b}\right)^{2(n-m)}\right)\\
        &+\frac{1}{2}\frac{\left(n(2m+n-1)+m(m-1)\right)^2+m^2(m-1)^2}{n+m}|b_n|^2b^{2(n+m)}\left(1-\left(\frac{a}{b}\right)^{2(n+m)}\right)\\
        &+\frac{\left(n(n-1)-(m-1)\right)\left(n(2m+n-1)+m(m-1)\right)-m(m-1)^2}{n}\Re\left(a_n\bar{b_n}\right)b^{2n}\left(1-\left(\frac{a}{b}\right)^{2n}\right)\\
        &\geq \frac{1}{2\left(1-\left(\frac{a}{b}\right)^{2(n+m)}\right)}|a_n|^2b^{2(n-m)}\left(\frac{\left(n(n-1)-(m-1)\right)^2+(m-1)^2}{n-m}\left(1-\left(\frac{a}{b}\right)^{2(n-m)}\right)\left(1-\left(\frac{a}{b}\right)^{2(n+m)}\right)\right.\\
        &-\left(\frac{\left(n(n-1)-(m-1)\right)\left(n(2m+n-1)+m(m-1)\right)-m(m-1)^2}{n}\right)^2\\
        &\left.\times \frac{n+m}{\left(n(2m+n-1)+m(m-1)\right)^2+m^2(m-1)^2}\left(1-\left(\frac{a}{b}\right)^{2n}\right)^2\right).
    \end{align*}
    \normalsize
    Define
    \begin{align*}
        \left\{\begin{alignedat}{1}
            X&=n(2m+n-1)+m(m-1)\\
            Y&=n(n-1)-(m-1).
        \end{alignedat}\right.
    \end{align*}
    We have 
    \begin{align*}
        &\frac{\left(n(n-1)-(m-1)\right)^2+(m-1)^2}{n-m}\\
        &-\left(\frac{\left(n(n-1)-(m-1)\right)\left(n(2m+n-1)+m(m-1)\right)-m(m-1)^2}{n}\right)^2\\
        &\times \frac{n+m}{\left(n(2m+n-1)+m(m-1)\right)^2+m^2(m-1)^2}\\
        &=\frac{Y^2+(m-1)^2}{n-m}-\frac{(XY-m(m-1)^2)^2}{n^2}\frac{n+m}{X^2+m^2(m-1)^2}\\
        &=\frac{1}{n^2(n-m)\left(X^2+m^2(m-1)^2\right)}\\
        &\times\bigg(n^2\left(X^2+m^2(m-1)^2\right)\left(Y^2+(m-1)^2\right)-\left(n^2-m^2\right)\left(XY-m(m-1)^2\right)^2\bigg),
    \end{align*}
    and
    \begin{align*}
        &n^2\left(X^2+m^2(m-1)^2\right)\left(Y^2+(m-1)^2\right)-\left(n^2-m^2\right)\left(XY-m(m-1)^2\right)^2\\
        &=n^2X^2Y^2+(m-1)^2n^2X^2+m^2(m-1)^2Y^2+m^2(m-1)^4n^2\\
        &-\left(n^2-m^2\right)\left(X^2Y^2-2m(m-1)XY+m^2(m-1)^4\right)\\
        &=m^2X^2Y^2+(m-1)^2n^2X^2+m^2(m-1)^2Y^2+m^2(m-1)^4n^2+2m(m-1)^2\left(n^2-m^2\right)XY\\
        &+m^4(m-1)^4\geq m^2X^2Y^2.
    \end{align*}
    Now, we have
    \begin{align*}
        X=n^2+(2m-1)n+m(m-1)\geq n^2.
    \end{align*}
    On the other hand, we have
    \begin{align*}
        Y=n^2-n-(m-1),
    \end{align*}
    which implies that $Y\geq \frac{1}{2}n^2$ if and only if
    $
        n^2-2n-2(m-1)\geq 0
    $,
    which is equivalent to
    \begin{align*}
        n\geq 1+\sqrt{2m-1}.
    \end{align*}
    Since $n\geq m+1$ and $m+1\geq 1+\sqrt{2m-1}$, the condition is always verified and we finally get 
    \begin{align*}
        n^2\left(X^2+m^2(m-1)^2\right)\left(Y^2+(m-1)^2\right)-\left(n^2-m^2\right)\left(XY-m(m-1)\right)^2\geq \frac{m^2}{4}n^8.
    \end{align*}
    Since $n\geq m+1$, we have
    \begin{align*}
        X=n(2m+n-1)+m(m-1)\leq n(3n-3)+(n-1)(n-2)\leq 4n^2, 
    \end{align*}
    which implies that
    \begin{align*}
        n^2(n-m)\left(X^2+m^2(m-1)^2\right)\leq n^3\left(16n^4+(n-1)^2(n-2)^2\right)\leq 17n^7,
    \end{align*}
    and finally
    \begin{align}\label{bound_coefficient}
        &\frac{\left(n(n-1)-(m-1)\right)^2+(m-1)^2}{n-m}\nonumber\\
        &-\left(\frac{\left(n(n-1)-(m-1)\right)\left(n(2m+n-1)+m(m-1)\right)-m(m-1)^2}{n}\right)^2\nonumber\\
        &\times \frac{n+m}{\left(n(2m+n-1)+m(m-1)\right)^2+m^2(m-1)^2}\geq \frac{m^2}{68}n.
    \end{align}
    Therefore, we have
    \begin{align*}
        &\frac{\left(n(n-1)-(m-1)\right)^2+(m-1)^2}{n-m}\left(1-\left(\frac{a}{b}\right)^{2(n-m)}\right)\left(1-\left(\frac{a}{b}\right)^{2(n+m)}\right)\\
        &-\left(\frac{\left(n(n-1)-(m-1)\right)\left(n(2m+n-1)+m(m-1)\right)-m(m-1)^2}{n}\right)^2\\
        &\times \frac{n+m}{\left(n(2m+n-1)+m(m-1)\right)^2+m^2(m-1)^2}\left(1-\left(\frac{a}{b}\right)^{2n}\right)^2\\
        &\geq \frac{m^2}{68}n-\frac{2\left(\left(n(n-1)-(m-1)\right)^2+(m-1)^2\right)}{n-m}\left(\frac{a}{b}\right)^{2(n-m)}.
    \end{align*}
    We now look for a condition on the conformal class so that
    \begin{align*}
        \frac{2\left(\left(n(n-1)-(m-1)\right)^2+(m-1)^2\right)}{n-m}\left(\frac{a}{b}\right)^{2(n-m)}\leq \frac{1}{68}-\frac{1}{100},
    \end{align*}
    or
    \begin{align*}
        \left(\frac{a}{b}\right)^{2(n-m)}\leq \frac{1}{425}\frac{n-m}{\left(n(n-1)-(m-1)\right)^2+m^2(m-1)^2},
    \end{align*}
    or
    \begin{align*}
        \log\left(\frac{b}{a}\right)\geq \frac{1}{2(n-m)}\log\left(\frac{425\left(\left(n(n-1)-(m-1)\right)^2+m^2(m-1)^2\right)}{n-m}\right).
    \end{align*}
    Since the right-hand side converges to $0$ as $n\rightarrow \infty$, this quantity is bounded in $n\geq m+1$.
    Let
    \begin{align*}
        f(x)=\frac{1}{x}\log\left(\frac{425\left(\left(x^2+(2m-1)x+(m-1)^2\right)^2+m^2(m-1)^2\right)^2}{x}\right)
    \end{align*}
    for $x\geq 1$. We have
    \begin{align*}
        \left(x^2+(2m-1)x+(m-1)^2\right)^2+m^2(m-1)^2\leq (m^2+1)^2x^4+m^2(m-1)^2x^4,
    \end{align*}
    which yields
    \begin{align*}
        f(x)\leq \frac{1}{x}\log\left(425\left((m^2+1)^2+m^2(m-1)^2\right)x^3\right)\leq \frac{1}{3e}+\log\left(425\left((m^2+1)^2+m^2(m-1)^2\right)^2\right),
    \end{align*}
    since
    \begin{align*}
        \max_{x\geq 1}\frac{\log(x)}{x}=\frac{1}{e}
    \end{align*}
    and this function is strictly decreasing. Finally, the condition on the conformal class holds provided that
    \begin{align}\label{lm_log1}
        \log\left(\frac{b}{a}\right)\geq \frac{1}{3e}+\log\left(425\left((m^2+1)^2+m^2(m-1)^2\right)^2\right).
    \end{align}
    If the condition is satisfied, we get
    \begin{align}\label{borne_an_1}
        &\frac{1}{2}\frac{\left(n(n-1)-(m-1)\right)^2+(m-1)^2}{n-m}|a_n|^2b^{2(n-m)}\left(1-\left(\frac{a}{b}\right)^{2(n-m)}\right)\nonumber\\
        &+\frac{1}{2}\frac{\left(n(2m+n-1)+m(m-1)\right)^2+m^2(m-1)^2}{n+m}|b_n|^2b^{2(n+m)}\left(1-\left(\frac{a}{b}\right)^{2(n+m)}\right)\nonumber\\
        &+\frac{\left(n(n-1)-(m-1)\right)\left(n(2m+n-1)+m(m-1)\right)-m(m-1)^2}{n}\Re\left(a_n\bar{b_n}\right)b^{2n}\left(1-\left(\frac{a}{b}\right)^{2n}\right)\nonumber\\
        &\geq \frac{m^2}{200}n|a_n|^2b^{2(n-m)}.
    \end{align}
    Then, for $n\geq m+1$, consider
    \begin{align*}
        &\frac{1}{2}\frac{\left(n(n+1)-(m-1)\right)^2}{n+m}|a_{-n}|^2\frac{1}{a^{2(n+m)}}\left(1-\left(\frac{a}{b}\right)^{2(n+m)}\right)\\
        &+\frac{1}{2}\frac{\left(n(n-2m-1)+m(m-1)\right)^2+m^2(m-1)^2}{n-m}|b_{-n}|^2\frac{1}{a^{2(n-m)}}\left(1-\left(\frac{a}{b}\right)^{2(n-m)}\right)\\
        &+\frac{\left(n(n+1)-(m-1)\right)\left(n(n-2m-1)+m(m-1)\right)-m(m-1)^2}{n}\Re\left(a_{-n}\bar{b_{-n}}\right)\frac{1}{a^{2n}}\left(1-\left(\frac{a}{b}\right)^{2n}\right)\\
        &\geq \frac{1}{2}\frac{\left(n(n+1)-(m-1)\right)^2}{n+m}|a_{-n}|^2\frac{1}{a^{2(n+m)}}\left(1-\left(\frac{a}{b}\right)^{2(n+m)}\right)\\
        &-\frac{1}{2}\left(\frac{\left(n(n+1)-(m-1)\right)\left(n(n-2m-1)+m(m-1)\right)-m(m-1)^2}{n}\right)^2\\
        &\times \frac{n-m}{\left(n(n-2m-1)+m(m-1)\right)^2+m^2(m-1)^2}|a_{-n}|^2\frac{1}{a^{2(n+m)}}\frac{\left(1-\left(\frac{a}{b}\right)^{2n}\right)^2}{1-\left(\frac{a}{b}\right)^{2(n-m)}}.
    \end{align*}
    Define as above
    \begin{align*}
        \left\{\begin{alignedat}{1}
            X&=n(n-2m-1)+m(m-1)\\
            Y&=n(n+1)-(m-1).            
        \end{alignedat}\right.
    \end{align*}
    Then we have
    \begin{align*}
        &\frac{\left(n(n+1)-(m-1)\right)^2}{n+m}\\
        &-\left(\frac{\left(n(n+1)-(m-1)\right)\left(n(n-2m-1)+m(m-1)\right)-m(m-1)^2}{n}\right)^2\\
        &\times \frac{n-m}{\left(n(n-2m-1)+m(m-1)\right)^2+m^2(m-1)^2}\\
        &=\frac{Y^2}{n+m}-\frac{\left(XY-m(m-1)^2\right)^2}{n^2}\frac{n-m}{X^2+m^2(m-1)^2}\\
        &=\frac{1}{n^2(n+m)\left(X^2+m^2(m-1)^2\right)}\left(m^2X^2Y^2+(m-1)^2n^2X^2+m^2(m-1)^2Y^2+m^2(m-1)^4n^2\right.\\\
        &\left.+2m(m-1)^2\left(n^2-m^2\right)XY+m^4(m-1)^2\right)\\
        &\geq \frac{m^2X^2Y^2+m^4(m-1)^2}{n^2(n+m)\left(X^2+m^2(m-1)^2\right)}.
    \end{align*}
    Notice that $X=0$ if
    \begin{align*}
        n=m+\frac{1+\sqrt{8m+1}}{2}.
    \end{align*}
    We see that for $n\geq 2m+1$, we have
    \begin{align*}
        X\geq m(m-1),
    \end{align*}
    and for $n\geq 3m$, we easily get 
    \begin{align*}
        X\geq \frac{m-1}{3m}n^2.
    \end{align*}
    On the other hand, since $n\geq m+1$, we have $Y\geq n^2$. Then, we trivially estimate for all $m+1\leq n\leq 3m-1$
    \begin{align*}
        m^4(m-1)^2\geq \frac{m^4(m-1)^2}{(3m-1)^7}n^7.
    \end{align*}
    Therefore, we finally get that for all $n\geq m+1$, we have
    \begin{align*}
        m^2X^2Y^2+m^4(m-1)^2\geq \frac{m^4(m-1)^2}{(3m-1)^7}n^7.
    \end{align*}
    On the other hand, we have
    \begin{align*}
        n^2(n+m)(X^2+m^2(m-1)^2)\leq n^2(2n+1)(4n^2)\leq 8n^6,
    \end{align*}
    and finally, we get
    \begin{align*}
        &\frac{\left(n(n+1)-(m-1)\right)^2}{n+m}\\
        &-\left(\frac{\left(n(n+1)-(m-1)\right)\left(n(n-2m-1)+m(m-1)\right)-m(m-1)^2}{n}\right)^2\\
        &\times \frac{n-m}{\left(n(n-2m-1)+m(m-1)\right)^2+m^2(m-1)^2}\geq \frac{m^4(m-1)^2}{8(3m-1)^7}n.
    \end{align*}
    This condition implies that 
    \begin{align*}
        &\frac{1}{2}\frac{\left(n(n+1)-(m-1)\right)^2}{n+m}|a_{-n}|^2\frac{1}{a^{2(n+m)}}\left(1-\left(\frac{a}{b}\right)^{2(n+m)}\right)\\
        &-\frac{1}{2}\left(\frac{\left(n(n+1)-(m-1)\right)\left(n(n-2m-1)+m(m-1)\right)-m(m-1)^2}{n}\right)^2\\
        &\times \frac{n-m}{\left(n(n-2m-1)+m(m-1)\right)^2+m^2(m-1)^2}|a_{-n}|^2\frac{1}{a^{2(n+m)}}\frac{\left(1-\left(\frac{a}{b}\right)^{2n}\right)^2}{1-\left(\frac{a}{b}\right)^{2(n-m)}}\\
        &\geq \frac{1}{2\left(1-\left(\frac{a}{b}\right)^{2(n-m)}\right)}|a_{-n}|^2\frac{1}{a^{2(n+m)}}\left(\frac{m^4(m-1)^2}{8(3m-1)^7}n-\frac{2\left((n(n+1)-(m-1)\right)^2}{n+m}\left(\frac{a}{b}\right)^{2(n-m)}\right).
    \end{align*}
    Therefore, we look for a condition on the conformal class such that 
    \begin{align*}
        \frac{2\left((n(n+1)-(m-1)\right)^2}{n+m}\left(\frac{a}{b}\right)^{2(n-m)}\leq \frac{m^4(m-1)^2}{16(3m-1)^7}n,
    \end{align*}
    or that for all $n\geq m+1$
    \begin{align*}
        \log\left(\frac{b}{a}\right)\geq \frac{1}{2(n-m)}\log\left(\frac{32(3m-1)^7}{m^4(m-1)^2}\frac{\left(n(n+1)-(m-1)\right)^2}{n(n+m)}\right).
    \end{align*}
    Since 
    \begin{align*}
        \frac{\left(n(n+1)-(m-1)\right)^2}{n(n+m)}\leq n^2,
    \end{align*}
    we deduce that the condition is satisfied for
    \begin{align}\label{lm_log_2}
        \log\left(\frac{b}{a}\right)\geq \frac{1}{e}+\frac{1}{2}\log\left(\frac{32(3m-1)^7}{m^4(m-1)^2}\right). 
    \end{align}
    If this condition is satisfied, we get for all $n\geq m+1$
    \begin{align}\label{borne_an_2}
        &\frac{1}{2}\frac{\left(n(n+1)-(m-1)\right)^2}{n+m}|a_{-n}|^2\frac{1}{a^{2(n+m)}}\left(1-\left(\frac{a}{b}\right)^{2(n+m)}\right)\nonumber\\
        &+\frac{1}{2}\frac{\left(n(n-2m-1)+m(m-1)\right)^2+m^2(m-1)^2}{n-m}|b_{-n}|^2\frac{1}{a^{2(n-m)}}\left(1-\left(\frac{a}{b}\right)^{2(n-m)}\right)\nonumber\\
        &+\frac{\left(n(n+1)-(m-1)\right)\left(n(n-2m-1)+m(m-1)\right)-m(m-1)^2}{n}\Re\left(a_{-n}\bar{b_{-n}}\right)\frac{1}{a^{2n}}\left(1-\left(\frac{a}{b}\right)^{2n}\right)\nonumber\\
        &\geq \frac{m^4(m-1)^2}{32(3m-1)^7}\,n|a_{-n}|^2\frac{1}{a^{2(n+m)}}.
    \end{align}
    Thanks to \eqref{frak_m_l2}, \eqref{borne_an_1}, and \eqref{borne_an_2}, we deduce that
    \small
    \begin{align*}
        &\frac{m^2}{200}\sum_{n=m+1}^{\infty}n|a_n|^2b^{2(n-m)}+\frac{m^4(m-1)^2}{32(3m-1)^7}\sum_{n=m+1}^{\infty}n|a_{-n}|^2\frac{1}{a^{2(n+m)}}\\
        &-\frac{(m-1)}{m}\sum_{n=m+1}^{\infty}\left(n(n-1)-(m-1)\right)\Re\left(a_n\bar{a_{-n}}\right)\frac{1}{a^{2m}}\left(1-\left(\frac{a}{b}\right)^{2m}\right)\\
        &-\frac{(m-1)}{m}\sum_{n=m+1}^{\infty}\left(n(n+1)-(m-1)\right)\Re\left(a_n\bar{a_{-n}}\right)\frac{1}{a^{2m}}\left(1-\left(\frac{a}{b}\right)^{2m}\right)\\
        &+2\sum_{n=m+1}^{\infty}\Big\{m(m-1)\left(n(n-1)-(m-1)\right)+(m-1)\left(n(n-2m-1)-m(m-1)\right)\Big\}\Re\left(a_n\bar{b_{-n}}\right)\log\left(\frac{b}{a}\right)\\
        &+2\sum_{n=m+1}^{\infty}\Big\{m(m-1)\left(n(n+1)-(m-1)\right)+(m-1)\left(n(n+2m+1)-m(m-1)\right)\Big\}\Re\left(a_{-n}\bar{b_{n}}\right)\log\left(\frac{b}{a}\right)\\
        &\leq \frac{1}{2\pi}\int_{\Omega}\left|\mathfrak{L}_mu\right|^2|dz|^2.
    \end{align*}
    \normalsize
    On the other hand, we have by \eqref{bound_lm1}
    \begin{align*}
        4\,m^2\sum_{n=m+1}^{\infty}(m+n)|b_n|^2b^{2(n+m)}+4\,m^2\sum_{n=m+1}^{\infty}(n-m)|b_{-n}|^2\frac{1}{a^{2(n-m)}}\leq \frac{1}{2\pi}\int_{\Omega}\left(\leb_mu\right)^2dx.
    \end{align*}
    We estimate
    \begin{align*}
        -\frac{2(m-1)}{m}\left(n^2-(m-1)\right)\Re\left(a_n\bar{a_{-n}}\right)\frac{1}{a^{2m}}&\geq -\frac{64(3m-1)^7}{m^6}n^3|a_n|^2\frac{1}{a^{2(m-n)}}\\
        &-\frac{m^4(m-1)^2}{64(3m-1)^7}\,n|a_{-n}|^2\frac{1}{a^{2(n+m)}}.
    \end{align*}
    Therefore, we get
    \begin{align*}
        &\frac{m^2}{200}\sum_{n=m+1}^{\infty}n|a_n|^2b^{2(n-m)}+\frac{m^4(m-1)^2}{32(3m-1)^7}\sum_{n=m+1}^{\infty}n|a_{-n}|^2\frac{1}{a^{2(n+m)}}\\
        &-\frac{(m-1)}{m}\sum_{n=m+1}^{\infty}\left(n^2-(m-1)\right)\Re\left(a_n\bar{a_{-n}}\right)\frac{1}{a^{2m}}\left(1-\left(\frac{a}{b}\right)^{2m}\right)\\
        &\geq \frac{m^2}{200}\sum_{n=m+1}^{\infty}n|a_n|^2b^{2(n-m)}\left(1-\frac{12800(3m-1)^7}{m^8}n^2\left(\frac{a}{b}\right)^{2(n-m)}\right)\nonumber\\
        &+\frac{m^4(m-1)^2}{64(3m-1)^7}\sum_{n=m+1}^{\infty}n|a_{-n}|^2\frac{1}{a^{2(n+m)}}.
    \end{align*}
    We choose a conformal class such that for all $n\geq m+1$,
    \begin{align*}
        \frac{12800(3m-1)^7}{m^7}n^2\left(\frac{a}{b}\right)^{2(n-m)}\leq \frac{1}{2},
    \end{align*}
    or
    \begin{align*}
        \log\left(\frac{b}{a}\right)\geq \frac{1}{2(n-m)}\log\left(\frac{25600(3m-1)^7}{m^8}n^2\right).
    \end{align*}
    Therefore, the condition is satisfied if 
    \begin{align}\label{lm_log_3}
        \log\left(\frac{b}{a}\right)\geq \frac{1}{e}+m+\frac{1}{2}\log\left(\frac{25600(3m-1)^7}{m^8}\right).
    \end{align}
    Indeed, we have
    \begin{align*}
        \frac{\log(n)}{n-m}=\frac{\log(n-m)}{n-m}+\frac{\log\left(1+\frac{m}{n-m}\right)}{n-m}\leq \frac{1}{e}+\frac{m}{(n-m)^2}\leq \frac{1}{e}+m,
    \end{align*}
    where we used that
    \begin{align*}
        \sup_{x\geq 1}\frac{\log(x)}{x}=\frac{1}{e}\qquad\text{and}\;\, \log(1+x)\leq x\qquad\text{for all}\;\,x\geq 0.
    \end{align*}
    If \eqref{lm_log_3} is satisfied, we get 
    \begin{align*}
        &\frac{m^2}{200}\sum_{n=m+1}^{\infty}n|a_n|^2b^{2(n-m)}+\frac{m^4(m-1)^2}{32(3m-1)^7}\sum_{n=m+1}^{\infty}n|a_{-n}|^2\frac{1}{a^{2(n+m)}}\\
        &-\frac{2(m-1)}{m}\sum_{n=m+1}^{\infty}\left(n^2-(m-1)\right)\Re\left(a_n\bar{a_{-n}}\right)\frac{1}{a^{2m}}\left(1-\left(\frac{a}{b}\right)^{2m}\right)\\
        &\geq \frac{m^2}{400}\sum_{n=m+1}^{\infty}n|a_n|^2b^{2(n-m)}+\frac{m^4(m-1)^2}{64(3m-1)^7}\sum_{n=m+1}^{\infty}n|a_{-n}|^2\frac{1}{a^{2(n+m)}}.
    \end{align*}
    We have
    \begin{align*}
        \Big|m(m-1)\left(n(n-1)-(m-1)\right)+(m-1)\left(n(n-2m-1)-m(m-1)\right)\Big|\leq 4m(m-1)n^2.
    \end{align*}
    Therefore, we get 
    \begin{align*}
        &2\Big\{m(m-1)\left(n(n-1)-(m-1)\right)+(m-1)\left(n(n-2m-1)-m(m-1)\right)\Big\}\Re\left(a_n\bar{b_{-n}}\right)\log\left(\frac{b}{a}\right)\\
        &\geq -\frac{8(m-1)^2n^4}{n-m}|a_n|^2a^{2(n-m)}\log^2\left(\frac{b}{a}\right)-2m^2(n-m)|b_{-n}|^2\frac{1}{a^{2(n-m)}}.
    \end{align*}
    Now, we have
    \begin{align*}
        &\frac{m^2}{400}\sum_{n=m+1}^{\infty}\left(n|a_n|^2b^{2(n-m)}-\frac{8(m-1)^2n^4}{n-m}|a_n|^2a^{2(n-m)}\log^2\left(\frac{b}{a}\right)\right)\\
        &=\frac{m^2}{400}\sum_{n=m+1}^{\infty}n|a_n|^2b^{2(n-m)}\left(1-\frac{3200(m-1)^2n^3}{n-m}\log^2\left(\frac{b}{a}\right)\left(\frac{a}{b}\right)^{2(n-m)}\right).
    \end{align*}
    Therefore, we look for a conformal class such that for all $n\geq m+1$, we have
    \begin{align}\label{conformal_class_condition}
        \frac{3200(m-1)^2n^3}{n-m}\log^2\left(\frac{b}{a}\right)\left(\frac{a}{b}\right)^{2(n-m)}\leq \frac{1}{2}.
    \end{align}
    Consider the function
    \begin{align*}
        f(x)=x\log^2(x)\qquad 0<x<1.
    \end{align*}
    We have
    \begin{align*}
        f(x)=\log^2(x)+2\log(x)=\log(x)\left(\log(x)+2\right)=\log\left(\frac{1}{x}\right)\left(-2+\log\left(\frac{1}{x}\right)\right).
    \end{align*}
    We deduce that $f$ is strictly increasing on $]0,e^{-2}[$ and strictly decreasing $]e^{-2},1[$. Therefore, we have
    \begin{align*}
        \sup_{0<x<1}f(x)=f(e^{-2})=4\,e^{-2}.
    \end{align*}
    Therefore, we get 
    \begin{align*}
        \frac{3200(m-1)^2n^3}{n-m}\log^2\left(\frac{b}{a}\right)\left(\frac{a}{b}\right)^{2(n-m)}\leq \frac{12800(m-1)^2n^3}{e^2(n-m)}\left(\frac{a}{b}\right)^{2(n-m-\frac{1}{2})},
    \end{align*}
    and \eqref{conformal_class_condition} is satisfied provided that 
    \begin{align*}
        \log\left(\frac{b}{a}\right)\geq \frac{1}{2n-(2m+1)}\log\left(\frac{25600(m-1)^2n^3}{e^2(n-m)}\right).
    \end{align*}
    Finally, we have for all $n\geq m+1$
    \begin{align*}
        \frac{1}{2n-(2m+1)}&\log\left(\frac{25600(m-1)^2n^3}{e^2(n-m)}\right)\leq \log\left(\frac{25600(m-1)^2}{e^2}\right)+\frac{3}{2n-(2m+1)}\log(n)\\
        &\leq \log\left(\frac{25600(m-1)^2}{e^2}\right)+\frac{3}{2e}+3(2m+1).
    \end{align*}
    Indeed, we have
    \begin{align*}
        \frac{\log(n)}{2n-(2m+1)}&=\frac{n-(m+\frac{1}{2})+(m+\frac{1}{2})}{2n-(2m+1)}=\frac{\log(n-(m+\frac{1}{2}))}{2n-(2m+1)}+\frac{\log\left(1+\frac{2m+1}{2n-(2m+1)}\right)}{2n-(2m+1)}\\
        &\leq \frac{1}{2e}+\frac{2m+1}{\left(2n-(2m+1)\right)^2}\leq \frac{1}{2e}+2m+1,
    \end{align*}
    Therefore, provided that 
    \begin{align}\label{lm_log_4}
         \log\left(\frac{b}{a}\right)\geq \log\left(\frac{25600(m-1)^2}{e^2}\right)+\frac{3}{2e}+3(2m+1).
    \end{align}
    Assuming that this estimate holds, we deduce that  
    \begin{align*}
        &\frac{m^2}{200}\sum_{n=m+1}^{\infty}n|a_n|^2b^{2(n-m)}+\frac{m^4(m-1)^2}{32(3m-1)^7}\sum_{n=m+1}^{\infty}n|a_{-n}|^2\frac{1}{a^{2(n+m)}}\\
        &+4\,m^2\sum_{n=m+1}^{\infty}(n-m)|b_{-n}|^2\frac{1}{a^{2(n-m)}}\\
        &-\frac{2(m-1)}{m}\sum_{n=m+1}^{\infty}\left(n^2-(m-1)\right)\Re\left(a_n\bar{a_{-n}}\right)\frac{1}{a^{2m}}\left(1-\left(\frac{a}{b}\right)^{2m}\right)\\
        &+2\sum_{n=m+1}^{\infty}\Big\{m(m-1)\left(n(n-1)-(m-1)\right)+(m-1)\left(n(n-2m-1)-m(m-1)\right)\Big\}\\
        &\times \Re\left(a_n\bar{b_{-n}}\right)\log\left(\frac{b}{a}\right)\\
        &\geq \frac{m^2}{800}\sum_{n=m+1}^{\infty}n|a_n|^2b^{2(n-m)}+\frac{m^4(m-1)^2}{64(3m-1)^7}\sum_{n=m+1}^{\infty}n|a_{-n}|^2\frac{1}{a^{2(n+m)}}+2\,m^2\sum_{n=m+1}^{\infty}(n-m)|b_{-n}|^2\frac{1}{a^{2(n-m)}}.
    \end{align*}
    Now, we estimate 
    \begin{align*}
        &m(m-1)\left(n(n+1)-(m-1)\right)+(m-1)\left(n(n+2m+1)-m(m-1)\right)\\
        &\leq 2m(m-1)n^2+4(m-1)n^2\leq 6m(m-1)n^2,
    \end{align*}
    which implies that
    \begin{align*}
        &2\Big\{m(m-1)\left(n(n+1)-(m-1)\right)+(m-1)\left(n(n+2m+1)-m(m-1)\right)\Big\}\Re\left(a_{-n}\bar{b_{n}}\right)\log\left(\frac{b}{a}\right)\\
        &\geq -\frac{m^4(m-1)^2}{128(3m-1)^7}n|a_{-n}|^2\frac{1}{a^{n+m)}}-\frac{4608(3m-1)^7n^3}{m^2}|b_n|^2a^{2(n+m)}\log^2\left(\frac{b}{a}\right).
    \end{align*}
    We have
    \begin{align*}
        &4m^2\sum_{n=m+1}^{\infty}(m+n)|b_n|^2b^{2(n+m)}-\sum_{n=m+1}^{\infty}\frac{4608(3m-1)^7n^3}{m^2}|b_n|^2a^{2(n+m)}\log^2\left(\frac{b}{a}\right)\\
        &=4m^2\sum_{n=m+1}^{\infty}(m+n)|b_n|^2b^{2(n+m)}\left(1-\frac{1152(3m-1)^7n^3}{m^4(n+m)}\left(\frac{a}{b}\right)^{2(n+m)}\log^2\left(\frac{b}{a}\right)\right).
    \end{align*}
    Therefore, we look for a condition on the conformal class such that 
    \begin{align*}
        \frac{1152(3m-1)^7n^3}{m^4(n+m)}\left(\frac{a}{b}\right)^{2(n+m)}\log^2\left(\frac{b}{a}\right)\leq \frac{1}{2}.
    \end{align*}
    Using once more the inequality $x\log^2(x)\leq 4\,e^{-2}$ for all $0<x<1$, the condition is verified if 
    \begin{align*}
        \frac{1152(3m-1)^7n^3}{m^4(n+m)}\left(\frac{a}{b}\right)^{2(n+m-\frac{1}{2})}\leq \frac{1}{2},
    \end{align*}
    or
    \begin{align*}
        \log\left(\frac{b}{a}\right)\geq \frac{1}{2m+2n-1}\log\left(\frac{2304(3m-1)^7n^3}{m^4(n+m)}\right).
    \end{align*}
    We have 
    \begin{align*}
        \frac{1}{2m+2n-1}\log\left(\frac{n^3}{n+m}\right)\leq \frac{1}{2m+2n-1}\log\left((2m+2n-1)^2\right)\leq \frac{2}{e},
    \end{align*}
    so the new condition on the conformal class is given by 
    \begin{align}\label{lm_log_5}
        \log\left(\frac{b}{a}\right)\geq \frac{2}{e}+\frac{1}{4m+1}\log\left(\frac{2304(3m-1)^7}{m^4}\right).
    \end{align}
    If it is satisfied, we finally get 
    \begin{align*}
    &\frac{m^2}{200}\sum_{n=m+1}^{\infty}n|a_n|^2b^{2(n-m)}+\frac{m^4(m-1)^2}{32(3m-1)^7}\sum_{n=m+1}^{\infty}n|a_{-n}|^2\frac{1}{a^{2(n+m)}}\\
        &+4\,m^2\sum_{n=m+1}(n+m)|b_n|^2b^{2(n+m)}+4\,m^2\sum_{n=m+1}^{\infty}(n-m)|b_{-n}|^2\frac{1}{a^{2(n-m)}}\\
        &-\frac{2(m-1)}{m}\sum_{n=m+1}^{\infty}\left(n^2-(m-1)\right)\Re\left(a_n\bar{a_{-n}}\right)\frac{1}{a^{2m}}\left(1-\left(\frac{a}{b}\right)^{2m}\right)\\
        &+2\sum_{n=m+1}^{\infty}\Big\{m(m-1)\left(n(n-1)-(m-1)\right)+(m-1)\left(n(n-2m-1)-m(m-1)\right)\Big\}\\
        &\times \Re\left(a_n\bar{b_{-n}}\right)\log\left(\frac{b}{a}\right)\\
        &+2\sum_{n=m+1}^{\infty}\Big\{m(m-1)\left(n(n+1)-(m-1)\right)+(m-1)\left(n(n+2m+1)-m(m-1)\right)\Big\}\\
        &\times\Re\left(a_{-n}\bar{b_{n}}\right)\log\left(\frac{b}{a}\right)\\
        &\geq 2\,m^2\sum_{n=m+1}(n+m)|b_n|^2b^{2(n+m)}+2\,m^2\sum_{n=m+1}^{\infty}(n-m)|b_{-n}|^2\frac{1}{a^{2(n-m)}}\\
        &+\frac{m^2}{800}\sum_{n=m+1}^{\infty}n|a_n|^2b^{2(n-m)}+\frac{m^4(m-1)^2}{128(3m-1)^7}\sum_{n=m+1}^{\infty}n|a_{-n}|^2\frac{1}{a^{2(n+m)}}.
    \end{align*} 
    Using \eqref{bound_lm1} and \eqref{frak_m_l2}, we finally get 
    \begin{align}\label{bound_lm2}
        &2\,m^2\sum_{n=m+1}^{\infty}(n+m)|b_n|^2b^{2(n+m)}+2\,m^2\sum_{n=m+1}^{\infty}(n-m)|b_{-n}|^2\frac{1}{a^{2(n-m)}}\nonumber\\
        &+\frac{m^2}{800}\sum_{n=m+1}^{\infty}n|a_n|^2b^{2(n-m)}+\frac{m^4(m-1)^2}{128(3m-1)^7}\sum_{n=m+1}^{\infty}n|a_{-n}|^2\frac{1}{a^{2(n+m)}}\nonumber\\
        &\leq \frac{1}{2\pi}\int_{\Omega}\left(\leb_mu\right)^2dx+\frac{1}{2\pi}\int_{\Omega}\left|\mathfrak{L}_mu\right|^2|dz|^2.
    \end{align}

    \textbf{Step 7.}    Now, recall that
    \begin{align*}
        u(r,\theta)=\alpha_1\,r^{1-m}\log(r)+2\,\Re\left(\sum_{n\in \Z}a_nr^{1-m+n}e^{in\theta}\right)+\alpha_2r^{m+1}\log(r)+2\,\Re\left(\sum_{n\in \Z}b_nr^{m+1+n}e^{in\theta}\right).
    \end{align*}
    If $v=|x|^{m-1}u$
    \begin{align*}
        |\D v|^2|x|^{2-2m}=\left|\D u+(m-1)\frac{x}{|x|^2}u\right|^2,
    \end{align*}
    while
    \begin{align*}
        |\p{r}v|^2|x|^{2-2m}=\left|\p{r}u+\frac{m-1}{r}u\right|^2\qquad \text{and}\qquad \left|\frac{1}{r}\p{\theta}v\right|^2|x|^{2-2m}=\frac{1}{r^2}|\p{\theta}u|^2,
    \end{align*}
    which readily shows that 
    \begin{align*}
        \left|\D u+(m-1)\frac{x}{|x|^2}u\right|^2=\left|\p{r}u+\frac{m-1}{r}u\right|^2+\frac{1}{r^2}|\p{\theta}u|^2.
    \end{align*}
    Now, we compute 
    \begin{align*}
        \p{r}u&=(1-m)\alpha_1\,r^{-m}\log(r)+\alpha_1r^{-m}+2\,\Re\left(\sum_{n\in \Z}(1-m+n)a_nr^{-m+n}e^{in\theta}\right)\\
        &+(m+1)\alpha_2\,r^m\log(r)+\alpha_2r^{m}+2\,\Re\left(\sum_{n\in \Z}(m+1+n)b_nr^{m+n}e^{in\theta}\right),
    \end{align*}
    which implies that
    \begin{align*}
        \p{r}u+\frac{m-1}{r}u&=\colorcancel{(1-m)\alpha_1\,r^{-m}\log(r)}{red}+\alpha_1r^{-m}+2\,\Re\left(\sum_{n\in \Z}(\colorcancel{(1-m)}{blue}+n)a_nr^{-m+n}e^{in\theta}\right)\\
        &+(m+1)\alpha_2\,r^m\log(r)+\alpha_2r^{m}+2\,\Re\left(\sum_{n\in \Z}(m+1+n)b_nr^{m+n}e^{in\theta}\right)\\
        &+\colorcancel{(m-1)\alpha_1\,r^{-m}\log(r)}{red}+\colorcancel{2\,\Re\left(\sum_{n\in \Z}(m-1)a_n\,r^{-m+n}e^{in\theta}\right)}{blue}\\
        &+(m-1)\alpha_2\,r^m\log(r)+2\,\Re\left(\sum_{n\in \Z}(m-1)b_n\,r^{m+n}e^{in\theta}\right)\\
        &=\alpha_1\,r^{-m}+2\,\Re\left(\sum_{n\in \Z}n\,a_n\,r^{-m+n}e^{in\theta}\right)\\
        &+2m\,\alpha_2\,r^m\log(r)+\alpha_2\,r^m+2\,\Re\left(\sum_{n\in \Z}(2m+n)b_n\,r^{m+n}e^{in\theta}\right).
    \end{align*}
    On the other hand, we trivially have
    \begin{align*}
        \frac{1}{r}\p{\theta}u=2\,\Re\left(\sum_{n\in \Z^{\ast}}n\,a_n\,r^{-m+n}e^{in\theta}\right)+2\,\Re\left(\sum_{n\in \Z}n\,b_n\,r^{m+n}e^{in\theta}\right).
    \end{align*}
    Using the elementary inequality
    \begin{align*}
        \left|\sum_{i=1}^dz_i\right|^2\leq d\sum_{i=1}^d|z_n|^2\qquad \forall (z_1,\cdots,z_d)\in \C^d,
    \end{align*}
    we deduce that
    \begin{align*}
        &\mathrm{Rad}\left(\left|\D u+(m-1)\frac{x}{|x|^2}u\right|^2\right)=\mathrm{Rad}\left(\left|\p{r}u+\frac{m-1}{r}u\right|^2+\frac{1}{r^2}|\p{\theta}u|^2\right)\\
        &\leq 5\alpha_1^2r^{-2m}+20\sum_{n\in \Z^{\ast}}|n|^2|a_n|^2r^{-2m+2n}+20m^2\alpha_2^2r^{2m}\log^2(r)+5\alpha_2^2r^{2m}+20\sum_{n\in \Z}(2m+n)^2|b_n|^2r^{2m+2n}\\
        &+8\sum_{n\in \Z^{\ast}}n^2|a_n|^2r^{-2m+2n}+8\sum_{n\in \Z^{\ast}}n^2|b_n|^2r^{2m+2n}\\
        &=5\alpha_1^2r^{-2m}+28\sum_{n\in \Z^{\ast}}|n|^2|a_n|^2r^{-2m+2n}+20m^2\alpha_2^2r^{2m}\log^2(r)+5\alpha_2^2r^{2m}\\
        &+4\sum_{n\in \Z}\left(5(2m+n)^2+2n^2\right)|b_n|^2r^{2m+2n}.
    \end{align*}
    Therefore, we have for all $0<|\gamma|<1$
    \begin{align*}
        &\int_{\Omega}\frac{1}{|x|^2}\left|\D u+(m-1)\frac{x}{|x|^2}u\right|^2|x|^{2\gamma}dx\leq 2\pi\int_{a}^b\bigg(5\alpha_1^2r^{-2m-1+2\gamma}+28\sum_{n\in \Z^{\ast}}|n|^2|a_n|^2r^{2n-2m-1+2\gamma}\\
        &+20m^2\alpha_2^2r^{2m-1+2\gamma}\log^2(r)+5\alpha_2^2r^{2m-1+2\gamma}+4\sum_{n\in \Z}\left(5(2m+n)^2+2n^2\right)|b_n|^2r^{2m+2n-1+2\gamma}\bigg)dr\\
        &=2\pi\bigg(\frac{5\alpha_1^2}{2(m-\gamma)}\frac{1}{a^{2(m-\gamma)}}\left(1-\left(\frac{a}{b}\right)^{2(m-\gamma)}\right)+14\sum_{n\in\Z^{\ast}}\frac{|n|^2}{|n-m+\gamma|}|a_n|^2\left|b^{2(n-m+\gamma)}-a^{2(n-m+\gamma)}\right|\\
        &+\frac{10m^2}{m+\gamma}\alpha_2^2\left(b^{2m+2\gamma}\log^2(b)+\frac{1}{m+\gamma}b^{2m+2\gamma}\log\left(\frac{1}{b}\right)+\frac{1}{2(m+\gamma)^2}b^{2m+2\gamma}\right.\\
        &\left.-\left(a^{2m+2\gamma}\log^2(a)+\frac{1}{m+\gamma}a^{2m+2\gamma}\log\left(\frac{1}{a}\right)+\frac{1}{2(m+\gamma)^2}a^{2m+2\gamma}\right)\right)\\
        &+\frac{5\alpha_2^2}{2(m+\gamma)}b^{2(m+\gamma)}\left(1-\left(\frac{a}{b}\right)^{2(m+\gamma)}\right)+4\sum_{n\in \Z}\frac{5(2m+n)^2+2n^2}{|m+n+\gamma|}|b_n|^2\left|b^{2(n+m+\gamma)}-a^{2(n+m+\gamma)}\right|\bigg).
    \end{align*}
    Assuming that $0<\gamma<1$, we deduce that 
    \begin{align*}
        &\int_{\Omega}\frac{1}{|x|^2}\left|\D u+(m-1)\frac{x}{|x|^2}u\right|^2|x|^{2\gamma}dx\leq 2\pi\bigg(\frac{5\alpha_1^2}{2(m-\gamma)}\frac{1}{a^{2(m-\gamma)}}\left(1-\left(\frac{a}{b}\right)^{2(m-\gamma)}\right)\\
        &+14\sum_{n=m}^{\infty}\frac{n^2}{n-m+\gamma}|a_n|^2b^{2(n-m+\gamma)}\left(1-\left(\frac{a}{b}\right)^{2(n-m+\gamma)}\right)\\
        &+14\sum_{n=1-m}^{\infty}\frac{n^2}{n+m-\gamma}|a_{-n}|^2\frac{1}{a^{2(n+m-\gamma)}}\left(1-\left(\frac{a}{b}\right)^{2(n+m-\gamma)}\right)\\
        &+\frac{10m^2}{m+\gamma}\alpha_2^2\left(b^{2m+2\gamma}\log^2(b)+\frac{1}{m+\gamma}b^{2m+2\gamma}\log\left(\frac{1}{b}\right)+\frac{1}{2(m+\gamma)^2}b^{2m+2\gamma}\right.\\
        &\left.-\left(a^{2m+2\gamma}\log^2(a)+\frac{1}{m+\gamma}a^{2m+2\gamma}\log\left(\frac{1}{a}\right)+\frac{1}{2(m+\gamma)^2}a^{2m+2\gamma}\right)\right)\\
        &+\frac{5\alpha_2^2}{2(m+\gamma)}b^{2(m+\gamma)}\left(1-\left(\frac{a}{b}\right)^{2(m+\gamma)}\right)+4\sum_{n=-m}^{\infty}\frac{5(2m+n)^2+2n^2}{n+m+\gamma}|b_n|^2b^{2(n+m+\gamma)}\left(1-\left(\frac{a}{b}\right)^{2(n+m+\gamma)}\right)\\
        &+4\sum_{n=1-m}^{\infty}\frac{5(n-2m)^2+2n^2}{n-m-\gamma}|b_{-n}|^2\frac{1}{a^{2(n-m-\gamma)}}\left(1-\left(\frac{a}{b}\right)^{2(n-m-\gamma)}\right)\bigg).
    \end{align*}
    Now, recall \eqref{bound_lm2}. 
    \begin{align}\label{bound_lm2*}
        &\frac{m^2}{800}\sum_{n=m+1}^{\infty}n|a_n|^2b^{2(n-m)}+\frac{m^4(m-1)^2}{128(3m-1)^7}\sum_{n=m+1}^{\infty}n|a_{-n}|^2\frac{1}{a^{2(n+m)}}\nonumber\\
        &+2\,m^2\sum_{n=m+1}^{\infty}(n+m)|b_n|^2b^{2(n+m)}+2\,m^2\sum_{n=m+1}^{\infty}(n-m)|b_{-n}|^2\frac{1}{a^{2(n-m)}}\nonumber\\
        &\leq \frac{1}{2\pi}\int_{\Omega}\left(\leb_mu\right)^2dx+\frac{1}{2\pi}\int_{\Omega}\left|\mathfrak{L}_mu\right|^2|dz|^2.
    \end{align}
    We trivially estimate
    \begin{align*}
        &28\pi\sum_{n=m+1}^{\infty}\frac{n^2}{n-m+\gamma}b^{2(n-m+\gamma)}\left(1-\left(\frac{a}{b}\right)^{2(n-m+\gamma)}\right)\leq 28\pi\,b^{2\gamma}\sum_{n=m+1}^{\infty}n|a_n|^2b^{2(n-m)}\\
        &\leq \frac{11200}{m^2}b^{2\gamma}\left(\int_{\Omega}\left(\leb_mu\right)^2dx+\int_{\Omega}|\mathfrak{L}_mu|^2|dz|^2\right).
    \end{align*}
    Likewise, we have
    \begin{align*}
        &28\pi\sum_{n=m+1}^{\infty}\frac{n^2}{n+m-\gamma}|a_{-n}|^2\frac{1}{a^{2(n+m-\gamma)}}\left(1-\left(\frac{a}{b}\right)^{2(n+m-\gamma)}\right)\\
        &\leq \frac{1792(3m-1)^7}{m^4(m-1)^2}a^{2\gamma}\left(\int_{\Omega}\left(\leb_mu\right)^2dx+\int_{\Omega}|\mathfrak{L}_mu|^2|dz|^2\right).
    \end{align*}
    Then, we readily obtain for all $n\geq m+1$
    \begin{align*}
        5(2m+n)^2+2n^2\leq 47n^2,
    \end{align*}
    which shows that 
    \begin{align*}
        &8\pi\sum_{n=m+1}^{\infty}\frac{5(2m+n)^2+2n^2}{n+m+\gamma}|b_n|^2b^{2(n+m+\gamma)}\left(1-\left(\frac{a}{b}\right)^{2(n+m+\gamma)}\right)\leq 376\pi\,b^{2\gamma}\sum_{n=m+1}^{\infty}(n+m)|b_n|^2b^{2(n+m)}\\
        &\leq \frac{188}{m^2}b^{2\gamma}\left(\int_{\Omega}\left(\leb_mu\right)^2dx+\int_{\Omega}|\mathfrak{L}_mu|^2|dz|^2\right),
    \end{align*}
    and likewise
    \begin{align*}
        &4\sum_{n=m+1}^{\infty}\frac{5(n-2m)^2+2n^2}{n-m-\gamma}|b_{-n}|^2\frac{1}{a^{2(n-m-\gamma)}}\left(1-\left(\frac{a}{b}\right)^{2(n-m-\gamma)}\right)\\
        &\leq \frac{188}{m^2}\left(\int_{\Omega}\left(\leb_mu\right)^2dx+\int_{\Omega}|\mathfrak{L}_mu|^2|dz|^2\right).
    \end{align*}
    Therefore, we only have finitely many terms left to estimate in the series, and also the logarithm terms. We first estimate $a_n$ for $1-m\leq n\leq m-1$. 
    Since $1<2\beta<2$, we have
    \begin{align*}
        &\sum_{n\in \Z^{\ast}}\frac{|a_n|^2}{|-m+n+2\beta|}\left|b^{-2m+2n+4\beta}-a^{-2m+2n+4\beta}\right|\\
        &=\sum_{n=m-1}^{\infty}\frac{|a_n|^2}{n-m+2\beta}b^{2(n-m+2\beta)}\left(1-\left(\frac{a}{b}\right)^{2(n-m+2\beta)}\right)\\
        &+\sum_{n=2-m,n\neq 0}^{\infty}\frac{|a_{-n}|^2}{n+m-2\beta}\frac{1}{a^{2(n+m-2\beta)}}\left(1-\left(\frac{a}{b}\right)^{2(n+m-2\beta)}\right).
    \end{align*}
    Likewise, we have
    \begin{align*}
        &\sum_{n\in \Z^{\ast}}\frac{|b_n|^2}{|m+n+2\beta|}\left|b^{2m+2n+4\beta}-a^{2m+2n+4\beta}\right|\\
        &=\sum_{n=-m-1,n\neq 0}^{\infty}\frac{|b_n|^2}{m+n+2\beta}b^{2(m+n+2\beta)}\left(1-\left(\frac{a}{b}\right)^{2(m+n+2\beta)}\right)\\
        &+\sum_{n=m+2}^{\infty}\frac{|b_{-n}|^2}{n-m-2\beta}\frac{1}{a^{2(n-m-2\beta)}}\left(1-\left(\frac{a}{b}\right)^{2(n+m-2\beta)}\right).
    \end{align*}
    Then, we have
    \begin{align*}
        &2\sum_{n\in \Z^{\ast}}\frac{\Re\left(a_n\bar{b_n}\right)}{n+2\beta}\left(b^{2n+4\beta}-a^{2n-2+4\beta}\right)=2\sum_{n=1}^{\infty}\frac{\Re\left(a_n\bar{b_n}\right)}{n+2\beta}b^{2(n+2\beta)}\left(1-\left(\frac{a}{b}\right)^{2(n+2\beta)}\right)\\
        &+2\frac{\Re\left(a_{-1}b_{-1}\right)}{2\beta-1}b^{2(2\beta-1)}\left(1-\left(\frac{a}{b}\right)^{2(2\beta-1)}\right)+2\sum_{n=2}^{\infty}\frac{\Re\left(a_{-n}\bar{b_{-n}}\right)}{n-2\beta}\frac{1}{a^{2(n-2\beta)}}\left(1-\left(\frac{a}{b}\right)^{2(n-2\beta)}\right).
    \end{align*}
    First, consider the quantity
    \small
    \begin{align*}
        &\sum_{n=2}^{m-2}\frac{|a_n|^2}{m-n-2\beta}\frac{1}{a^{2(m-n-2\beta)}}\left(1-\left(\frac{a}{b}\right)^{2(m-n-2\beta)}\right)+\sum_{n=2}^{m-2}\frac{|b_n|^2}{m+n+2\beta}b^{2(m+n+2\beta)}\left(1-\left(\frac{a}{b}\right)^{2(m+n+2\beta)}\right)\\
        &+\sum_{n=2}^{m-2}\frac{|a_{-n}|^2}{n+m-2\beta}\frac{1}{a^{2(m+n-2\beta)}}\left(1-\left(\frac{a}{b}\right)^{2(m+n-2\beta)}\right)+\sum_{n=2}^{m-2}\frac{|b_{-n}|^2}{m-n+2\beta}b^{2(m-n+2\beta)}\left(1-\left(\frac{a}{b}\right)^{2(m-n+2\beta)}\right)\\
        &-\frac{2}{m-2\beta}\sum_{n=2}^{m-2}\Re\left(a_na_{-n}\right)\frac{1}{a^{2m-4\beta}}\left(1-\left(\frac{a}{b}\right)^{2m-4\beta}\right)+\frac{2}{m+2\beta}\sum_{n=2}^{m-2}\Re\left(b_nb_{-n}\right)b^{2m+4\beta}\left(1-\left(\frac{a}{b}\right)^{2m+4\beta}\right).%\\
        %&+4\sum_{n=2}^{m-2}\Re\left(a_nb_{-n}+a_{-n}b_n\right)\log\left(\frac{b}{a}\right).
    \end{align*}
    \normalsize
    We first have for all $2\leq n\leq m-2$
    \begin{align*}
        &\frac{|a_n|^2}{m-n-2\beta}\left(1-\left(\frac{a}{b}\right)^{2(m-n-2\beta)}\right)+\frac{|a_{-n}|^2}{m+n-2\beta}\frac{1}{a^{2(n+m-2\beta)}}\left(1-\left(\frac{a}{b}\right)^{2(m+n-2\beta)}\right)\\
        &-\frac{2}{m-2\beta}\Re\left(a_na_{-n}\right)\frac{1}{a^{2m-4\beta}}\left(1-\left(\frac{a}{b}\right)^{2m-4\beta}\right)\\
        &\geq \frac{|a_n|^2}{m-n-2\beta}\frac{1}{a^{2(m-n-2\beta)}}\left(1-\left(\frac{a}{b}\right)^{2(m-n-2\beta)}\right)-\frac{m+n-2\beta}{(m-2\beta)^2}|a_{-n}|^2\frac{1}{a^{2(m-n-2\beta)}}\frac{\left(1-\left(\frac{a}{b}\right)^{2m-4\beta}\right)^2}{1-\left(\frac{a}{b}\right)^{2(m+n-2\beta)}}\\
        &=\frac{|a_n|^2}{1-\left(\frac{a}{b}\right)^{2(m+n-2\beta)}}\frac{1}{a^{2(m-n-2\beta)}}\left(\frac{1}{m-n-2\beta}\left(1-\left(\frac{a}{b}\right)^{2(m-n-2\beta)}\right)\left(1-\left(\frac{a}{b}\right)^{2(m+n-2\beta)}\right)\right.\\
        &\left.-\frac{m+n-2\beta}{(m-2\beta)^2}\left(1-\left(\frac{a}{b}\right)^{2m-4\beta}\right)^2\right)\\
        &\geq \frac{|a_n|^2}{1-\left(\frac{a}{b}\right)^{2(m+n-2\beta)}}\frac{1}{a^{2(m-n-2\beta)}}\left(\frac{n^2}{(m-n-2\beta)(m-2\beta)^2}-\frac{2}{m-n-2\beta}\left(\frac{a}{b}\right)^{2(m-n-2\beta)}\right),
    \end{align*}
    where we used that
    \begin{align*}
        \frac{1}{m-n-2\beta}-\frac{m+n-2\beta}{(m-2\beta)^2}=\frac{(m-\beta)^2-\left((m-2\beta)^2-n^2\right)}{(m-n-2\beta)(m-2\beta)^2}=\frac{n^2}{(m-n-2\beta)(m-2\beta)^2}.
    \end{align*}
    Therefore, we assume that 
    \begin{align*}
        \frac{2}{m-n-2\beta}\left(\frac{a}{b}\right)^{2(m-n-2\beta)}\leq \frac{n^2}{2(m-n-2\beta)(m-2\beta)^2},
    \end{align*}
    which is equivalent to
    \begin{align*}
        \log\left(\frac{b}{a}\right)\geq \frac{1}{2(m-n-2\beta)}\log\left(\frac{4(m-2\beta)^2}{n^2}\right).
    \end{align*}
    Since $2\leq n\leq m-2$, the condition is satisfied provided that
    \begin{align}\label{lm_log_6}
        \log\left(\frac{b}{a}\right)\geq \frac{1}{2(1-\beta)}\log\left(m-2\beta\right).
    \end{align}
    If this estimate is satisfied, we obtain
    \begin{align}\label{extra_an1}
        &\frac{|a_n|^2}{m-n-2\beta}\left(1-\left(\frac{a}{b}\right)^{2(m-n-2\beta)}\right)+\frac{|a_{-n}|^2}{m+n-2\beta}\frac{1}{a^{2(n+m-2\beta)}}\left(1-\left(\frac{a}{b}\right)^{2(m+n-2\beta)}\right)\nonumber\\
        &-\frac{2}{m-2\beta}\Re\left(a_na_{-n}\right)\frac{1}{a^{2m-4\beta}}\left(1-\left(\frac{a}{b}\right)^{2m-4\beta}\right)\nonumber\\
        &\geq \frac{n^2}{2(m-n-2\beta)(m-2\beta)^2}|a_n|^2\frac{1}{a^{2(m-n-2\beta)}}.
    \end{align}
    Likewise, we estimate
    \begin{align*}
        &\frac{|a_n|^2}{m-n-2\beta}\left(1-\left(\frac{a}{b}\right)^{2(m-n-2\beta)}\right)+\frac{|a_{-n}|^2}{m+n-2\beta}\frac{1}{a^{2(n+m-2\beta)}}\left(1-\left(\frac{a}{b}\right)^{2(m+n-2\beta)}\right)\\
        &-\frac{2}{m-2\beta}\Re\left(a_na_{-n}\right)\frac{1}{a^{2m-4\beta}}\left(1-\left(\frac{a}{b}\right)^{2m-4\beta}\right)\\
        &\geq \frac{|a_{-n}|^2}{1-\left(\frac{a}{b}\right)^{2(m-n-2\beta)}}\frac{1}{a^{2(m+n-2\beta)}}\left(\frac{n^2}{(m+n-2\beta)(m-2\beta)^2}-\frac{2}{m+n-2\beta}\left(\frac{a}{b}\right)^{2(m-n-2\beta)}\right).
    \end{align*}
    Therefore, we impose the following condition on the conformal class
    \begin{align*}
        \log\left(\frac{b}{a}\right)\geq \frac{1}{2(m-n-2\beta)}\log\left(\frac{4(m-2\beta)^2}{n^2}\right),
    \end{align*}
    that is satisfied thanks to \eqref{lm_log_6}. Therefore, we deduce that 
    \begin{align}\label{extra_an2}
        &\frac{|a_n|^2}{m-n-2\beta}\left(1-\left(\frac{a}{b}\right)^{2(m-n-2\beta)}\right)+\frac{|a_{-n}|^2}{m+n-2\beta}\frac{1}{a^{2(n+m-2\beta)}}\left(1-\left(\frac{a}{b}\right)^{2(m+n-2\beta)}\right)\nonumber\\
        &-\frac{2}{m-2\beta}\Re\left(a_na_{-n}\right)\frac{1}{a^{2m-4\beta}}\left(1-\left(\frac{a}{b}\right)^{2m-4\beta}\right)\geq \frac{n^2}{2(m+n-2\beta)(m-2\beta)^2}|a_{-n}|^2\frac{1}{a^{2(m+n-2\beta)}}.
    \end{align}
    Therefore, combining \eqref{extra_an1} with \eqref{extra_an2}, we deduce that 
    \begin{align}\label{extra_an}
        &\frac{|a_n|^2}{m-n-2\beta}\left(1-\left(\frac{a}{b}\right)^{2(m-n-2\beta)}\right)+\frac{|a_{-n}|^2}{m+n-2\beta}\frac{1}{a^{2(n+m-2\beta)}}\left(1-\left(\frac{a}{b}\right)^{2(m+n-2\beta)}\right)\nonumber\\
        &-\frac{2}{m-2\beta}\Re\left(a_na_{-n}\right)\frac{1}{a^{2m-4\beta}}\left(1-\left(\frac{a}{b}\right)^{2m-4\beta}\right)\nonumber\\
        &\geq \frac{n^2}{4(m-n-2\beta)(m-2\beta)^2}|a_n|^2\frac{1}{a^{2(m-n-2\beta)}}+\frac{n^2}{4(m+n-2\beta)(m-2\beta)^2}|a_{-n}|^2\frac{1}{a^{2(m+n-2\beta)}}.
    \end{align}
    Now, we estimate the terms involving $b_n$. 
    We have for $1\leq n\leq m$
    \begin{align*}
        &\frac{|b_n|^2}{m+n+2\beta}b^{2(m+n+2\beta)}\left(1-\left(\frac{a}{b}\right)^{2(m+n+2\beta)}\right)+\frac{|b_{-n}|^2}{m-n+2\beta}b^{2(m-n+2\beta)}\left(1-\left(\frac{a}{b}\right)^{2(m-n+2\beta)}\right)\\
        &+\frac{2\,\Re\left(b_nb_{-n}\right)}{m+2\beta}b^{2m+4\beta}\left(1-\left(\frac{a}{b}\right)^{2(m+4\beta)}\right)\\
        &\geq \frac{|b_n|^2}{1-\left(\frac{a}{b}\right)^{2(m-n+2\beta)}}b^{2(m+n+2\beta)}\left(\frac{1}{m+n+2\beta}\left(1-\left(\frac{a}{b}\right)^{2(m+n+2\beta)}\right)\left(1-\left(\frac{a}{b}\right)^{2(m-n+2\beta)}\right)\right.\\
        &\left.-\frac{m-n+2\beta}{(m+2\beta)^2}\left(1-\left(\frac{a}{b}\right)^{2m+4\beta}\right)^2\right)\\
        &\geq \frac{|b_n|^2}{1-\left(\frac{a}{b}\right)^{2(m-n+2\beta)}}b^{2(m+n+2\beta)}\left(\frac{n^2}{(m+n+2\beta)(m+2\beta)^2}-\frac{2}{m+n+2\beta}\left(\frac{a}{b}\right)^{2(m-n+2\beta)}\right).
    \end{align*}
    Therefore, we impose the condition
    \begin{align*}
        \frac{2}{m+n+2\beta}\left(\frac{a}{b}\right)^{2(m-n+2\beta)}\leq \frac{n^2}{(m+n+2\beta)(m+2\beta)^2},
    \end{align*}
    which yields the condition
    \begin{align*}
        \log\left(\frac{b}{a}\right)\geq \frac{1}{2(m-n+2\beta)}\log\left(\frac{4n^2}{(m+2\beta)^2}\right),
    \end{align*}
    which is satisfied provided that 
    \begin{align}\label{lm_log_7}
        \log\left(\frac{b}{a}\right)\geq \frac{1}{2\beta}\log\left(\frac{2m}{m+2\beta}\right).
    \end{align}
    Therefore, we get 
    \begin{align}\label{extra_bn1}
        &\frac{|b_n|^2}{m+n+2\beta}b^{2(m+n+2\beta)}\left(1-\left(\frac{a}{b}\right)^{2(m+n+2\beta)}\right)+\frac{|b_{-n}|^2}{m-n+2\beta}b^{2(m-n+2\beta)}\left(1-\left(\frac{a}{b}\right)^{2(m-n+2\beta)}\right)\nonumber\\
        &+\frac{2\,\Re\left(b_nb_{-n}\right)}{m+2\beta}b^{2m+4\beta}\left(1-\left(\frac{a}{b}\right)^{2(m+4\beta)}\right)\nonumber\\
        &\geq \frac{n^2}{2(m+n+2\beta)(m+2\beta)^2}|b_n|^2b^{2(m+n+2\beta)}.
    \end{align}
    The same argument (that holds when the same condition on the conformal class is satisfied) also shows that
    \begin{align}\label{extra_bn2}
        &\frac{|b_n|^2}{m+n+2\beta}b^{2(m+n+2\beta)}\left(1-\left(\frac{a}{b}\right)^{2(m+n+2\beta)}\right)+\frac{|b_{-n}|^2}{m-n+2\beta}b^{2(m-n+2\beta)}\left(1-\left(\frac{a}{b}\right)^{2(m-n+2\beta)}\right)\nonumber\\
        &+\frac{2\,\Re\left(b_nb_{-n}\right)}{m+2\beta}b^{2m+4\beta}\left(1-\left(\frac{a}{b}\right)^{2(m+4\beta)}\right)\nonumber\\
        &\geq \frac{n^2}{2(m-n+2\beta)(m+2\beta)^2}|b_{-n}|^2b^{2(m-n+2\beta)}.
    \end{align}
    Combining \eqref{extra_bn1} and \eqref{extra_bn2} yields
    \begin{align}\label{extra_bn}
        &\frac{|b_n|^2}{m+n+2\beta}b^{2(m+n+2\beta)}\left(1-\left(\frac{a}{b}\right)^{2(m+n+2\beta)}\right)+\frac{|b_{-n}|^2}{m-n+2\beta}b^{2(m-n+2\beta)}\left(1-\left(\frac{a}{b}\right)^{2(m-n+2\beta)}\right)\nonumber\\
        &+\frac{2\,\Re\left(b_nb_{-n}\right)}{m+2\beta}b^{2m+4\beta}\left(1-\left(\frac{a}{b}\right)^{2(m+4\beta)}\right)\nonumber\\
        &\geq \frac{n^2}{4(m+n+2\beta)(m+2\beta)^2}|b_n|^2b^{2(m+n+2\beta)}+\frac{n^2}{4(m-n+2\beta)(m+2\beta)^2}|b_{-n}|^2b^{2(m-n+2\beta)}.
    \end{align}
    Therefore, we get by \eqref{extra_an} and \eqref{extra_bn}
    \small
    \begin{align*}
        &\sum_{n=2}^{m-2}\frac{|a_n|^2}{m-n-2\beta}\frac{1}{a^{2(m-n-2\beta)}}\left(1-\left(\frac{a}{b}\right)^{2(m-n-2\beta)}\right)+\sum_{n=2}^{m-2}\frac{|b_n|^2}{m+n+2\beta}b^{2(m+n+2\beta)}\left(1-\left(\frac{a}{b}\right)^{2(m+n+2\beta)}\right)\\
        &+\sum_{n=2}^{m-2}\frac{|a_{-n}|^2}{n+m-2\beta}\frac{1}{a^{2(m+n-2\beta)}}\left(1-\left(\frac{a}{b}\right)^{2(m+n-2\beta)}\right)+\sum_{n=2}^{m-2}\frac{|b_{-n}|^2}{m-n+2\beta}b^{2(m-n+2\beta)}\left(1-\left(\frac{a}{b}\right)^{2(m-n+2\beta)}\right)\\
        &-\frac{2}{m-2\beta}\sum_{n=2}^{m-2}\Re\left(a_na_{-n}\right)\frac{1}{a^{2m-4\beta}}\left(1-\left(\frac{a}{b}\right)^{2m-4\beta}\right)+\frac{2}{m+2\beta}\sum_{n=2}^{m-2}\Re\left(b_nb_{-n}\right)b^{2m+4\beta}\left(1-\left(\frac{a}{b}\right)^{2m+4\beta}\right)\\
        &\geq \sum_{n=2}^{m-2}\frac{n^2}{4(m-n-2\beta)(m-2\beta)^2}|a_n|^2\frac{1}{a^{2(m-n-2\beta)}}+\sum_{n=2}^{m-2}\frac{n^2}{4(m+n-2\beta)(m-2\beta)^2}|a_{-n}|^2\frac{1}{a^{2(m+n-2\beta)}}\\
        &+\sum_{n=2}^{m-2}\frac{n^2}{4(m+n+2\beta)(m+2\beta)^2}|b_n|^2b^{2(m+n+2\beta)}+\sum_{n=2}^{m-2}\frac{n^2}{4(m-n+2\beta)(m+2\beta)^2}|b_{-n}|^2b^{2(m-n+2\beta)}.
    \end{align*}
    \normalsize
    Then, we estimate
    \small
    \begin{align*}
        &4\sum_{n=2}^{m-2}\Re\left(a_nb_{-n}+a_{-n}b_n\right)\log\left(\frac{b}{a}\right)\\
        &\geq -32\sum_{n=2}^{m-2}\frac{(m-n+2\beta)(m+2\beta)^2}{n^2}|a_n|^2b^{2(n-m+2\beta)}\log^2\left(\frac{b}{a}\right)-\sum_{n=2}^{m-2}\frac{n^2}{8(m-n+2\beta)(m+2\beta)^2}|b_{-n}|^2b^{2(m-n+2\beta)}\\
        &-32\sum_{n=2}^{m-2}\frac{(m+n+2\beta)(m+2\beta)^2}{n^2}|a_{-n}|^2b^{2(-m-n+2\beta)}\log^2\left(\frac{b}{a}\right)-\sum_{n=2}^{m-2}\frac{n^2}{8(m+n+2\beta)(m+2\beta)^2}|b_n|^2b^{2(m+n+2\beta)}.
    \end{align*}
    \normalsize
    Therefore, we deduce that 
    \small
    \begin{align*}
        &\sum_{n=2}^{m-2}\frac{|a_n|^2}{m-n-2\beta}\frac{1}{a^{2(m-n-2\beta)}}\left(1-\left(\frac{a}{b}\right)^{2(m-n-2\beta)}\right)+\sum_{n=2}^{m-2}\frac{|b_n|^2}{m+n+2\beta}b^{2(m+n+2\beta)}\left(1-\left(\frac{a}{b}\right)^{2(m+n+2\beta)}\right)\\
        &+\sum_{n=2}^{m-2}\frac{|a_{-n}|^2}{n+m-2\beta}\frac{1}{a^{2(m+n-2\beta)}}\left(1-\left(\frac{a}{b}\right)^{2(m+n-2\beta)}\right)+\sum_{n=2}^{m-2}\frac{|b_{-n}|^2}{m-n+2\beta}b^{2(m-n+2\beta)}\left(1-\left(\frac{a}{b}\right)^{2(m-n+2\beta)}\right)\\
        &-\frac{2}{m-2\beta}\sum_{n=2}^{m-2}\Re\left(a_na_{-n}\right)\frac{1}{a^{2m-4\beta}}\left(1-\left(\frac{a}{b}\right)^{2m-4\beta}\right)+\frac{2}{m+2\beta}\sum_{n=2}^{m-2}\Re\left(b_nb_{-n}\right)b^{2m+4\beta}\left(1-\left(\frac{a}{b}\right)^{2m+4\beta}\right)\\
        &+4\sum_{n=2}^{m-2}\Re\left(a_nb_{-n}+a_{-n}b_n\right)\log\left(\frac{b}{a}\right)\\
        &\geq \sum_{n=2}^{m-2}|a_n|^2\frac{1}{a^{2(m-n-2\beta)}}\left(\frac{n^2}{4(m-n-2\beta)(m-2\beta)^2}-\frac{32(m-n+2\beta)(m+2\beta)^2}{n^2}\left(\frac{a}{b}\right)^{2(m-n-2\beta)}\log^2\left(\frac{b}{a}\right)\right)\\
        &+\sum_{n=2}^{\infty}|a_{-n}|^2\frac{1}{a^{2(m+n-2\beta)}}\left(\frac{n^2}{4(m+n-2\beta)(m-2\beta)^2}-\frac{32(m+n+2\beta)(m+2\beta)^2}{n^2}\left(\frac{a}{b}\right)^{2(m+n-2\beta)}\log^2\left(\frac{b}{a}\right)\right)\\
        &+\sum_{n=2}^{m-2}\frac{n^2}{8(m+n+2\beta)(m+2\beta)^2}|b_n|^2b^{2(m+n+2\beta)}+\sum_{n=2}^{m-2}\frac{n^2}{8(m-n+2\beta)(m+2\beta)^2}|b_{-n}|^2b^{2(m-n+2\beta)}.
    \end{align*}
    \normalsize
    Therefore, we impose
    \begin{align*}
        &\frac{32(m-n+2\beta)(m+2\beta)^2}{n^2}\left(\frac{a}{b}\right)^{2(m-n-2\beta)}\log^2\left(\frac{b}{a}\right)\leq \frac{n^2}{8(m-n-2\beta)(m-2\beta)^2}\\
        &\frac{32(m+n+2\beta)(m+2\beta)^2}{n^2}\left(\frac{a}{b}\right)^{2(m+n-2\beta)}\log^2\left(\frac{b}{a}\right)\leq \frac{n^2}{8(m+n-2\beta)(m-2\beta)^2}.
    \end{align*}
    For all $\epsilon>0$, consider the function
    \begin{align*}
        f(x)=x^{\epsilon}\log^2(x)\qquad 0<x<1.
    \end{align*}
    We have
    \begin{align*}
        f'(x)=\epsilon\,x^{\epsilon-1}\log^2(x)+2\,x^{\epsilon-1}\log(x)=x^{\epsilon-1}\log\left(\frac{1}{x}\right)\left(\epsilon\log\left(\frac{1}{x}\right)-2\right).
    \end{align*}
    Therefore, $f$ is strictly increasing on $]0,e^{-\frac{2}{\epsilon}}[$ and strictly decreasing on $]e^{-\frac{2}{\epsilon}},1[$, which implies that
    \begin{align*}
        \sup_{0<x<1}f(x)=f(e^{-\frac{2}{\epsilon}})=\frac{4\,e^{-2}}{\epsilon}.
    \end{align*}
    Therefore, if $0<\epsilon<2(1-\beta)$, the first estimate is implied provided that 
    \begin{align*}
        \frac{32(m-n+2\beta)(m+2\beta)^2}{n^2}\left(\frac{a}{b}\right)^{2(m-n-2\beta-\epsilon)}\leq \frac{e^2\epsilon\, n^2}{16(m-n+2\beta)(m-2\beta)^2},
    \end{align*}
    or
    \begin{align*}
        \log\left(\frac{b}{a}\right)\geq \frac{1}{2(m-n-2\beta-\epsilon)}\log\left(\frac{1024((m-n)^2-4\beta^2)(m^2-4\beta^2)^2}{e^2\,\epsilon\,n^4}\right).
    \end{align*}
    Choosing $\epsilon=1-\beta$, the estimate is satisfied provided that 
    \begin{align}\label{lm_log_8}
        \log\left(\frac{b}{a}\right)\geq \frac{1}{2(1-\beta)}\log\left(\frac{1024(m^2-4\beta^2)^3}{e^2(1-\beta)}\right).
    \end{align}
    Likewise, the second condition is satisfied provided that 
    \begin{align}\label{lm_log_9}
         \log\left(\frac{b}{a}\right)\geq \frac{1}{2(1-\beta)}\log\left(\frac{4048m^2(m^2-4\beta^2)^2}{e^2(1-\beta)}\right).
    \end{align}
    Therefore, we finally get
    \small
    \begin{align*}
        &\sum_{n=2}^{m-2}\frac{|a_n|^2}{m-n-2\beta}\frac{1}{a^{2(m-n-2\beta)}}\left(1-\left(\frac{a}{b}\right)^{2(m-n-2\beta)}\right)+\sum_{n=2}^{m-2}\frac{|b_n|^2}{m+n+2\beta}b^{2(m+n+2\beta)}\left(1-\left(\frac{a}{b}\right)^{2(m+n+2\beta)}\right)\\
        &+\sum_{n=2}^{m-2}\frac{|a_{-n}|^2}{n+m-2\beta}\frac{1}{a^{2(m+n-2\beta)}}\left(1-\left(\frac{a}{b}\right)^{2(m+n-2\beta)}\right)+\sum_{n=2}^{m-2}\frac{|b_{-n}|^2}{m-n+2\beta}b^{2(m-n+2\beta)}\left(1-\left(\frac{a}{b}\right)^{2(m-n+2\beta)}\right)\\
        &-\frac{2}{m-2\beta}\sum_{n=2}^{m-2}\Re\left(a_na_{-n}\right)\frac{1}{a^{2m-4\beta}}\left(1-\left(\frac{a}{b}\right)^{2m-4\beta}\right)+\frac{2}{m+2\beta}\sum_{n=2}^{m-2}\Re\left(b_nb_{-n}\right)b^{2m+4\beta}\left(1-\left(\frac{a}{b}\right)^{2m+4\beta}\right)\\
        &+4\sum_{n=2}^{m-2}\Re\left(a_nb_{-n}+a_{-n}b_n\right)\log\left(\frac{b}{a}\right)\\
        &\geq \sum_{n=2}^{m-2}\frac{n^2}{8(m-n-2\beta)(m-2\beta)^2}|a_n|^2\frac{1}{a^{2(m-n-2\beta)}}+\sum_{n=2}^{m-2}\frac{n^2}{8(m+n-2\beta)(m-2\beta)^2}|a_{-n}|^2\frac{1}{a^{2(m+n-2\beta)}}\\
        &+\sum_{n=2}^{m-2}\frac{n^2}{8(m+n+2\beta)(m+2\beta)^2}|b_n|^2b^{2(m+n+2\beta)}+\sum_{n=2}^{m-2}\frac{n^2}{8(m-n+2\beta)(m+2\beta)^2}|b_{-n}|^2b^{2(m-n+2\beta)}.
    \end{align*}
    \normalsize
    The remaining cross term is given by 
    \begin{align*}
        \sum_{n=2}^{m-2}\frac{2\,\Re\left(a_n\bar{b_n}\right)}{n+2\beta}b^{2(n+2\beta)}\left(1-\left(\frac{a}{b}\right)^{2(n+2\beta)}\right)+\sum_{n=2}^{m-2}\frac{2\,\Re\left(a_{-n}\bar{b_{-n}}\right)}{n-2\beta}\frac{1}{a^{2(n-2\beta)}}\left(1-\left(\frac{a}{b}\right)^{2(n-2\beta)}\right).
    \end{align*}
    We first have
    \begin{align*}
        \frac{2\,\Re\left(a_n\bar{b_n}\right)}{n+2\beta}b^{2(n+2\beta)}\left(1-\left(\frac{a}{b}\right)^{2(n+2\beta)}\right)&\geq -\frac{n^2}{16(m+n+2\beta)(m+2\beta)^2}|b_n|^2b^{2(m+n+2\beta)}\\
        &-\frac{16(m+n+2\beta)(m+2\beta)^2}{n^2(n+2\beta)^2}|a_n|^2b^{2(-m+n+2\beta)}.
    \end{align*}
    We have 
    \begin{align*}
        &\frac{n^2}{8(m-n-2\beta)(m-2\beta)^2}|a_n|^2\frac{1}{a^{2(m-n-2\beta)}}-\frac{16(m+n+2\beta)(m+2\beta)^2}{n^2(n+2\beta)^2}|a_n|^2b^{2(-m+n+2\beta)}\\
        &=|a_n|^2\frac{1}{a^{2(m-n-2\beta)}}\left(\frac{n^2}{8(m-n-2\beta)(m-2\beta)^2}-\frac{16(m+n+2\beta)(m+2\beta)^2}{n^2(n+2\beta)^2}\left(\frac{a}{b}\right)^{2(m-n-2\beta)}\right).
    \end{align*}
    Therefore, we impose the condition 
    \begin{align*}
        \frac{16(m+n+2\beta)(m+2\beta)^2}{n^2(n+2\beta)^2}\left(\frac{a}{b}\right)^{2(m-n-2\beta)}\leq \frac{n^2}{16(m-n-2\beta)(m-2\beta)^2},
    \end{align*}
    which is equivalent to
    \begin{align*}
        \log\left(\frac{b}{a}\right)\geq \frac{1}{2(m-n-2\beta)}\log\left(\frac{256\left(m^2-(n+2\beta)^2\right)\left(m^2-4\beta^2\right)^2}{n^4(n+2\beta)^2}\right).
    \end{align*}
    This condition is satisfied if 
    \begin{align}\label{lm_log_10}
        \log\left(\frac{b}{a}\right)\geq \frac{1}{4(1-\beta)}\log\left(\frac{4\left(m^2-4\beta^2\right)^3}{(1+\beta)^2}\right).
    \end{align}
    Likewise, we get
    \begin{align*}
        \frac{2\,\Re\left(a_{-n}\bar{b_{-n}}\right)}{n-2\beta}\frac{1}{a^{2(n-2\beta)}}\left(1-\left(\frac{a}{b}\right)^{2(n-2\beta)}\right)&\geq -\frac{n^2}{16(m-n+2\beta)(m+2\beta)^2}|b_{-n}|^2b^{2(m-n+2\beta)}\\
        &-\frac{16(m-n+2\beta)(m+2\beta)^2}{n^2(n-2\beta)^2}|a_{-n}|^2b^{2(-m-n+2\beta)}.
    \end{align*}
    Then, we get
    \begin{align*}
        &\frac{n^2}{8(m+n-2\beta)(m-2\beta)^2}|a_{-n}|^2\frac{1}{a^{2(m+n-2\beta)}}-\frac{16(m-n+2\beta)(m+2\beta)^2}{n^2(m-2\beta)^2}|a_{-n}|^2b^{2(-m-n+2\beta)}\\
        &=|a_{-n}|^2\frac{1}{a^{2(m+n-2\beta)}}\left(\frac{n^2}{8(m+n-2\beta)(m-2\beta)^2}-\frac{16(m-n+2\beta)(m+2\beta)^2}{n^2(m-2\beta)^2}\left(\frac{a}{b}\right)^{2(m+n-2\beta)}\right).
    \end{align*}
    Therefore, we impose
    \begin{align*}
        \frac{16(m-n+2\beta)(m+2\beta)^2}{n^2(n-2\beta)^2}\left(\frac{a}{b}\right)^{2(m+n-2\beta)}\leq \frac{n^2}{16(m+n-2\beta)(m-2\beta)^2},
    \end{align*}
    or
    \begin{align*}
        \log\left(\frac{b}{a}\right)\geq \frac{1}{2(m+n-2\beta)}\log\left(\frac{256\left(m^2-(n-2\beta)^2\right)\left(m^2-4\beta^2\right)^2}{n^4(n-2\beta)^2}\right),
    \end{align*}
    which holds for $2\leq n\leq m-4$ provided that
    \begin{align}\label{lm_log_11}
        \log\left(\frac{b}{a}\right)\geq \frac{1}{2(m+2(1-\beta)}\log\left(\frac{4\left(m^2-4(1-\beta)^2\right)\left(m^2-4\beta^2\right)^2}{(1-\beta)^2}\right).
    \end{align}
    Finally, we deduce that 
    \small
    \begin{align}
        &\sum_{n=2}^{m-2}\frac{|a_n|^2}{m-n-2\beta}\frac{1}{a^{2(m-n-2\beta)}}\left(1-\left(\frac{a}{b}\right)^{2(m-n-2\beta)}\right)+\sum_{n=2}^{m-2}\frac{|b_n|^2}{m+n+2\beta}b^{2(m+n+2\beta)}\left(1-\left(\frac{a}{b}\right)^{2(m+n+2\beta)}\right)\nonumber\\
        &+\sum_{n=2}^{m-2}\frac{|a_{-n}|^2}{n+m-2\beta}\frac{1}{a^{2(m+n-2\beta)}}\left(1-\left(\frac{a}{b}\right)^{2(m+n-2\beta)}\right)+\sum_{n=2}^{m-2}\frac{|b_{-n}|^2}{m-n+2\beta}b^{2(m-n+2\beta)}\left(1-\left(\frac{a}{b}\right)^{2(m-n+2\beta)}\right)\nonumber\\
        &-\frac{2}{m-2\beta}\sum_{n=2}^{m-2}\Re\left(a_na_{-n}\right)\frac{1}{a^{2m-4\beta}}\left(1-\left(\frac{a}{b}\right)^{2m-4\beta}\right)+\frac{2}{m+2\beta}\sum_{n=2}^{m-2}\Re\left(b_nb_{-n}\right)b^{2m+4\beta}\left(1-\left(\frac{a}{b}\right)^{2m+4\beta}\right)\nonumber\\
        &+4\sum_{n=2}^{m-2}\Re\left(a_nb_{-n}+a_{-n}b_n\right)\log\left(\frac{b}{a}\right)\nonumber\\
        &+\sum_{n=2}^{m-2}\frac{2\,\Re\left(a_n\bar{b_n}\right)}{n+2\beta}b^{2(n+2\beta)}\left(1-\left(\frac{a}{b}\right)^{2(n+2\beta)}\right)+\sum_{n=2}^{m-2}\frac{2\,\Re\left(a_{-n}\bar{b_{-n}}\right)}{n-2\beta}\frac{1}{a^{2(n-2\beta)}}\left(1-\left(\frac{a}{b}\right)^{2(n-2\beta)}\right)\nonumber\\
        &\geq \sum_{n=2}^{m-2}\frac{n^2}{16(m-n-2\beta)(m-2\beta)^2}|a_n|^2\frac{1}{a^{2(m-n-2\beta)}}+\sum_{n=2}^{m-2}\frac{n^2}{16(m+n-2\beta)(m-2\beta)^2}|a_{-n}|^2\frac{1}{a^{2(m+n-2\beta)}}\nonumber\\
        &+\sum_{n=2}^{m-2}\frac{n^2}{16(m+n+2\beta)(m+2\beta)^2}|b_n|^2b^{2(m+n+2\beta)}+\sum_{n=2}^{m-2}\frac{n^2}{16(m-n+2\beta)(m+2\beta)^2}|b_{-n}|^2b^{2(m-n+2\beta)}.
    \end{align}
    \normalsize
    Besides the logarithm components that will require a special (also difficult and lengthy) argument, we will estimate the finitely many additional terms by the weighted $L^2$ norm of $u$. Recall that $u$ is given by 
    \begin{align*}
        u(r,\theta)&=\alpha_1\,r^{1-m}\log(r)+2\,\Re\left(\sum_{n\in \Z}a_n\,r^{1-m+n}e^{in\theta}\right)+\alpha_2\,r^{m+1}\log(r)+2\,\Re\left(\sum_{n\in \Z}b_n\,r^{m+1+n}e^{in\theta}\right)\\
        &=\alpha_1\,r^{1-m}\log(r)+\alpha_2\,r^{m+1}\log(r)+2\,a_0\,r^{1-m}+2\,b_0\,r^{m+1}\\
        &+\sum_{n\in \Z^{\ast}}\left(a_n\,r^{1-m+n}+\bar{a_{-n}}r^{1-m-n}+b_n\,r^{m+1+n}+\bar{b_{-n}}r^{m+1-n}\right)e^{in\theta}.
    \end{align*}
    Therefore, we have
    \begin{align}\label{sum_squares}
        &\mathrm{Rad}(u^2)=\left(\alpha_1\,r^{1-m}\log(r)+\alpha_2\,r^{m+1}\log(r)+2\,a_0\,r^{1-m}+2\,b_0\,r^{m+1}\right)^2\nonumber\\
        &+\sum_{n\in \Z^{\ast}}\left|a_n\,r^{1-m+n}+\bar{a_{-n}}r^{1-m-n}+b_n\,r^{m+1+n}+\bar{b_{-n}}r^{m+1-n}\right|^2\nonumber\\
        &=\alpha_1^2r^{2-2m}\log^2(r)+\alpha_2^2r^{2m+2}\log^2(r)+4a_0^2r^{2-2m}+4b_0^2r^{2m+2}+2\alpha_1\alpha_2\log^2(r)+4\,\alpha_1a_0r^{2-2m}\log(r)\nonumber\\
        &+4\left(\alpha_1b_0+\alpha_2a_0\right)\log(r)+4\,\alpha_2b_0r^{2m+2}\log(r)+8\,a_0b_0\nonumber\\
        &+2\sum_{n\in \Z^{\ast}}|a_n|^2r^{2-2m+2n}+2\sum_{n\in \Z^{\ast}}|b_n|^2r^{2m+2+2n}+2\,\sum_{n\in \Z^{\ast}}\Re\left(a_na_{-n}\right)r^{2-2m}+2\,\sum_{n\in \Z^{\ast}}\Re\left(b_nb_{-n}\right)r^{2m+2}\nonumber\\
        &+4\sum_{n\in \Z^{\ast}}\Re\left(a_n\bar{b_n}\right)r^{2n+2}+4\sum_{n\in \Z^{\ast}}\Re\left(a_nb_{-n}\right)r^2.
    \end{align}
    Then, we have for $\alpha\neq -1$
    \begin{align*}
        \int r^{\alpha}\log(r)dr&=%\frac{r^{\alpha+1}}{\alpha+1}\log(r)-\frac{1}{\alpha+1}\int r^{\alpha}dr=
        \frac{r^{\alpha+1}}{\alpha+1}\log(r)-\frac{1}{(\alpha+1)^2}r^{\alpha+1}\\
        \int r^{\alpha}\log^2(r)dr%&=\frac{r^{\alpha+1}}{\alpha+1}\log^2(r)-\frac{2}{\alpha+1}\int r^{\alpha}\log(r)dr\\
        &=\frac{r^{\alpha+1}}{\alpha+1}\log^2(r)-\frac{2}{(\alpha+1)^2}r^{\alpha+1}\log(r)+\frac{2}{(\alpha+1)^3}r^{\alpha+1},
    \end{align*}
    and for all $\beta>0$
    \begin{align*}
        \int \frac{\log^{\beta}(r)}{r}dr=\frac{1}{\beta+1}\log^{\beta+1}(r).
    \end{align*}
    This implies that for all $1/2<\beta<\mathrm{min}\ens{1,\frac{m}{2}}$
    \begin{align}\label{int_u_square}
        &\int_{\Omega}\frac{u^2}{|x|^4}|x|^{4\beta}dx\nonumber\\
        &=2\pi\int_{\Omega}\bigg(\alpha_1^2r^{-2m-1}\log^2(r)+\alpha_2^2r^{2m-1}\log^2(r)+4a_0^2r^{-2m-1}+4b_0^2r^{2m-1}+2\alpha_1\alpha_2r^{-3}\log^2(r)\nonumber\\
        &+4\,\alpha_1a_0r^{-2m-1}\log(r)+4\left(\alpha_1b_0+\alpha_2a_0\right)r^{-3}\log(r)+4\,\alpha_2b_0r^{2m-1}\log(r)+8\,a_0b_0\,r^{-3}\nonumber\\
        &+2\sum_{n\in \Z^{\ast}}|a_n|^2r^{-2m+2n-1}+2\sum_{n\in \Z^{\ast}}|b_n|^2r^{2m+2n-1}+2\,\sum_{n\in \Z^{\ast}}\Re\left(a_na_{-n}\right)r^{-2m-1}+2\,\sum_{n\in \Z^{\ast}}\Re\left(b_nb_{-n}\right)r^{2m-1}\nonumber\\
        &+4\sum_{n\in \Z^{\ast}}\Re\left(a_n\bar{b_n}\right)r^{2n-1}+4\sum_{n\in \Z^{\ast}}\Re\left(a_nb_{-n}\right)r^{-1}\bigg)r^{4\beta}dr\nonumber\\
        &=2\pi\left(\frac{\alpha_1^2}{2m-4\beta}\left(\frac{1}{a^{2m-4\beta}}\log^2(a)+\frac{1}{m-2\beta}\frac{1}{a^{2m-4\beta}}\log\left(\frac{1}{a}\right)+\frac{1}{2(m-2\beta)^2}\frac{1}{a^{2m-4\beta}}\right.\right.\nonumber\\
        &\left.-\left(\frac{1}{b^{2m-4\beta}}\log^2(b)+\frac{1}{m-2\beta}\frac{1}{b^{2m}}\log\left(\frac{1}{b}\right)+\frac{1}{2(m-2\beta)^2}\frac{1}{b^{2m-4\beta}}\right)\right)\nonumber\\
        &+\frac{\alpha_2^2}{2m+4\beta}\left(b^{2m+4\beta}\log^2(b)+\frac{1}{m+2\beta}b^{2m+4\beta}\log\left(\frac{1}{b}\right)+\frac{1}{2(m+2\beta)^2}b^{2m+4\beta}\right.\nonumber\\
        &\left.\left(a^{2m+4\beta}\log^2(a)+\frac{1}{m+2\beta}a^{2m+4\beta}\log\left(\frac{1}{a}\right)+\frac{1}{2(m+2\beta)^2}a^{2m+4\beta}\right)\right)\nonumber\\
        &+\frac{2}{m-2\beta}a_0^2\frac{1}{a^{2m-4\beta}}\left(1-\left(\frac{a}{b}\right)^{2m-4\beta}\right)+\frac{2}{m+2\beta}b_0^2b^{2m+4\beta}\left(1-\left(\frac{a}{b}\right)^{2m+4\beta}\right)\nonumber\\
        &+\frac{\alpha_1\alpha_2}{2\beta-1}\left(b^{4\beta-2}\log^2(b)+\frac{1}{2\beta-1}b^{4\beta-2}\log\left(\frac{1}{b}\right)+\frac{1}{2(2\beta-1)^2}b^{4\beta-2}\right.\nonumber\\
        &\left.-\left(a^{4\beta-2}\log^2(a)+\frac{1}{2\beta-1}a^{4\beta-2}\log\left(\frac{1}{a}\right)+\frac{1}{2(2\beta-1)^2}a^{4\beta-2}\right)\right)\nonumber\\
        &-\frac{2\alpha_1a_0}{m-2\beta}\left(\frac{1}{a^{2m-4\beta}}\log\left(\frac{1}{a}\right)+\frac{1}{2m-4\beta}\frac{1}{a^{2m-4\beta}}-\left(\frac{1}{b^{2m-4\beta}}\log\left(\frac{1}{b}\right)+\frac{1}{2m-4\beta}\frac{1}{b^{2m-4\beta}}\right)\right)\nonumber\\
        &-\frac{2(\alpha_1b_0+\alpha_2a_0)}{2\beta-1}\left(b^{2\beta-1}\log\left(\frac{1}{b}\right)+\frac{1}{4\beta-2}b^{2\beta-1}-\left(a^{2\beta-1}\log\left(\frac{1}{a}\right)+\frac{1}{4\beta-2}a^{2\beta-1}\right)\right)\nonumber\\
        &-\frac{2\alpha_2b_0}{m+2\beta}\left(b^{2m+4\beta}\log\left(\frac{1}{b}\right)+\frac{1}{2m+4\beta}b^{2m+4\beta}-\left(a^{2m+4\beta}\log\left(\frac{1}{a}\right)+\frac{1}{2m+4\beta}a^{2m+4\beta}\right)\right)\nonumber\\
        &+\frac{4a_0b_0}{2\beta-1}b^{4\beta-2}\left(1-\left(\frac{a}{b}\right)^{4\beta-2}\right)+\sum_{n\in \Z^{\ast}}\frac{|a_n|^2}{|-m+n+2\beta|}\left|b^{-2m+2n+4\beta}-a^{-2m+2n+4\beta}\right|\nonumber\\
        &+\sum_{n\in \Z^{\ast}}\frac{|b_n|^2}{|m+n+2\beta|}\left|b^{2m+2n+4\beta}-a^{2m+2n+4\beta}\right|-\frac{1}{m-2\beta}\sum_{n\in \Z^{\ast}}\Re\left(a_na_{-n}\right)\frac{1}{a^{2m-4\beta}}\left(1-\left(\frac{a}{b}\right)^{2m-4\beta}\right)\nonumber\\
        &+\frac{1}{m+2\beta}\sum_{n\in \Z^{\ast}}\Re\left(b_nb_{-n}\right)b^{2m+4\beta}\left(1-\left(\frac{a}{b}\right)^{2m+4\beta}\right)+2\sum_{n\in \Z^{\ast}}\frac{\Re\left(a_n\bar{b_n}\right)}{n+2\beta}\left(b^{2n+4\beta}-a^{2n+4\beta}\right)\nonumber\\
        &\left.+4\sum_{n\in \Z^{\ast}}\Re\left(a_nb_{-n}\right)\log\left(\frac{b}{a}\right)\right).
    \end{align}
    Provided that $m>m_0=2.03\cdots$, we also get the missing estimate on $a_m$, $a_{-m}$, and $b_m$ (that does not follow from the estimation of the $L^2$ norm of $u$). Now, we also need to estimate for $0<\beta<\dfrac{1}{2}$
    \begin{align*}
        &\frac{1}{2\pi}\int_{\Omega}\frac{u_0^2}{|x|^4}\frac{dx}{|x|^{4\beta}}=2\sum_{n\in \Z^{\ast}}\int_{a}^b|a_n|^2r^{-2m+2n-4\beta-1}dr+2\sum_{n\in \Z^{\ast}}|b_n|^2r^{2m+2n-4\beta-1}dr\\
        &+2\sum_{n\in \Z^{\ast}}\int_{a}^{b}\Re\left(a_na_{-n}\right)r^{-2m-4\beta-1}dr+2\sum_{n\in \Z^{\ast}}\int_{a}^b\Re\left(b_nb_{-n}\right)r^{2m-1-4\beta}dr\\
        &+4\sum_{n\in \Z^{\ast}}\Re\left(a_n\bar{b_n}\right)r^{2n-4\beta-1}dr+4\sum_{n\in\Z^{\ast}}\Re\left(a_nb_{-n}\right)r^{-4\beta-1}dr\\
        &=\sum_{n=m+1}^{\infty}\frac{|a_n|^2}{n-m-2\beta}b^{2(n-m-2\beta)}\left(1-\left(\frac{a}{b}\right)^{2(n-m-2\beta)}\right)\\
        &+\sum_{n=-m,n\neq 0}^{\infty}\frac{|a_{-n}|^2}{m+n+2\beta}\frac{1}{a^{2(m+n+2\beta)}}\left(1-\left(\frac{a}{b}\right)^{2(m+n+2\beta)}\right)\\
        &+\sum_{n=1-m,n\neq 0}^{\infty}\frac{|b_n|^2}{m+n-2\beta}b^{2(m+n-2\beta)}\left(1-\left(\frac{a}{b}\right)^{2(m+n-2\beta)}\right)\\
        &+\sum_{n=m}^{\infty}\frac{|b_{-n}|^2}{n-m+2\beta}\frac{1}{a^{2(n-m+2\beta)}}\left(1-\left(\frac{a}{b}\right)^{2(n-m+2\beta)}\right)\\
        &+2\sum_{n=1}^{\infty}\frac{\Re\left(a_na_{-n}\right)}{m+2\beta}\frac{1}{a^{2(m+2\beta)}}\left(1-\left(\frac{a}{b}\right)^{2(m+2\beta)}\right)+2\sum_{n=1}^{\infty}\frac{\Re\left(b_nb_{-n}\right)}{m-2\beta}b^{2(m-2\beta)}\left(1-\left(\frac{a}{b}\right)^{2(m-2\beta)}\right)\\
        &+2\sum_{n=1}^{\infty}\frac{\Re\left(a_n\bar{b_n}\right)}{n-2\beta}b^{2(n-2\beta)}\left(1-\left(\frac{a}{b}\right)^{2(n-2\beta)}\right)+2\sum_{n=1}^{\infty}\frac{\Re\left(a_{-n}\bar{b_{-n}}\right)}{n+2\beta}\frac{1}{a^{2(n+2\beta)}}\left(1-\left(\frac{a}{b}\right)^{2(n+2\beta)}\right)\\
        &+\frac{1}{\beta}\sum_{n=1}^{\infty}\left(\Re\left(a_nb_{-n}\right)+\Re\left(a_{-n}b_n\right)\right)\frac{1}{a^{4\beta}}\left(1-\left(\frac{a}{b}\right)^{4\beta}\right).
    \end{align*}
    Consider for $1\leq n\leq m-1$ the quantity
    \begin{align*}
        &\frac{|a_n|^2}{m-n+2\beta}\frac{1}{a^{2(m-n+2\beta)}}\left(1-\left(\frac{a}{b}\right)^{2(m-n+2\beta)}\right)\\
        &+\frac{|a_{-n}|^2}{m+n+2\beta}\frac{1}{a^{2(m+n+2\beta)}}\frac{1}{a^{2(m+n+2\beta)}}\left(1-\left(\frac{a}{b}\right)^{2(m+n+2\beta)}\right)\\
        &+\frac{|b_{-n}|^2}{m-n-2\beta}b^{2(m-n-2\beta)}\left(1-\left(\frac{a}{b}\right)^{2(m-n-2\beta)}\right)+\frac{|b_n|^2}{m+n-2\beta}b^{2(m+n-2\beta)}\left(1-\left(\frac{a}{b}\right)^{2(m+n-2\beta)}\right)\\
        &+\frac{2\,\Re\left(a_na_{-n}\right)}{m+2\beta}\frac{1}{a^{2(m+2\beta)}}\left(1-\left(\frac{a}{b}\right)^{2(m+2\beta)}\right)+\frac{2\,\Re\left(b_nb_{-n}\right)}{m-2\beta}b^{2(m-2\beta)}\left(1-\left(\frac{a}{b}\right)^{2(m-2\beta)}\right)\\
        &+\frac{2\,\Re\left(a_n\bar{b_n}\right)}{n-2\beta}b^{2(n-2\beta)}\left(1-\left(\frac{a}{b}\right)^{2(n-2\beta)}\right)+\frac{2\,\Re\left(a_{-n}\bar{b_{-n}}\right)}{n+2\beta}\frac{1}{a^{2(n+2\beta)}}\left(1-\left(\frac{a}{b}\right)^{2(n+2\beta)}\right)\\
        &+\frac{1}{\beta}\left(\Re\left(a_nb_{-n}\right)+\Re\left(a_{-n}b_n\right)\right)\frac{1}{a^{4\beta}}\left(1-\left(\frac{a}{b}\right)^{4\beta}\right).
    \end{align*}
    We first estimate
    \begin{align*}
        &\frac{|a_n|^2}{m-n+2\beta}\frac{1}{a^{2(m-n+2\beta)}}\left(1-\left(\frac{a}{b}\right)^{2(m-n+2\beta)}\right)\\
        &+\frac{|a_{-n}|^2}{m+n+2\beta}\frac{1}{a^{2(m+n+2\beta)}}\frac{1}{a^{2(m+n+2\beta)}}\left(1-\left(\frac{a}{b}\right)^{2(m+n+2\beta)}\right)\\
        &+\frac{2\,\Re\left(a_na_{-n}\right)}{m+2\beta}\frac{1}{a^{2(m+2\beta)}}\left(1-\left(\frac{a}{b}\right)^{2(m+2\beta)}\right)\\
        &\geq \frac{|a_n|^2}{m-n+2\beta}\frac{1}{a^{2(m-n+2\beta)}}\left(1-\left(\frac{a}{b}\right)^{2(m-n+2\beta)}\right)-\frac{(m+n+2\beta)|a_{n}|^2}{(m+2\beta)^2}\frac{1}{a^{2(m-n+2\beta)}}\frac{\left(1-\left(\frac{a}{b}\right)^{2(m+2\beta)}\right)^2}{1-\left(\frac{a}{b}\right)^{2(m+n+2\beta)}}\\
        &\geq \frac{|a_n|^2}{1-\left(\frac{a}{b}\right)^{2(m+n+2\beta)}}\frac{1}{a^{2(m-n+2\beta)}}\left(\frac{n^2}{(m+2\beta)^2(m-n+2\beta)}-\frac{2}{m-n+2\beta}\left(\frac{a}{b}\right)^{2(m-n+2\beta)}\right).
    \end{align*}
    Therefore, we impose the condition
    \begin{align*}
        \log\left(\frac{b}{a}\right)\geq \frac{1}{m-n+2\beta}\log\left(\frac{2(m+2\beta)}{n}\right),
    \end{align*}
    to get
    \begin{align}\label{neg_weight_an1}
        &\frac{|a_n|^2}{m-n+2\beta}\frac{1}{a^{2(m-n+2\beta)}}\left(1-\left(\frac{a}{b}\right)^{2(m-n+2\beta)}\right)\nonumber\\
        &+\frac{|a_{-n}|^2}{m+n+2\beta}\frac{1}{a^{2(m+n+2\beta)}}\frac{1}{a^{2(m+n+2\beta)}}\left(1-\left(\frac{a}{b}\right)^{2(m+n+2\beta)}\right)\nonumber\\
        &+\frac{2\,\Re\left(a_na_{-n}\right)}{m+2\beta}\frac{1}{a^{2(m+2\beta)}}\left(1-\left(\frac{a}{b}\right)^{2(m+2\beta)}\right)\geq \frac{n^2|a_n|^2}{(m+2\beta)^2(m-n+2\beta)}\frac{1}{a^{2(m-n+2\beta)}}.
    \end{align}
    Likewise, if
    \begin{align*}
        \log\left(\frac{b}{a}\right)\geq \frac{1}{m+n+2\beta}\log\left(\frac{2(m+2\beta)}{n}\right),
    \end{align*}
    we get
    \begin{align}\label{neg_weight_an2}
        &\frac{|a_n|^2}{m-n+2\beta}\frac{1}{a^{2(m-n+2\beta)}}\left(1-\left(\frac{a}{b}\right)^{2(m-n+2\beta)}\right)\nonumber\\
        &+\frac{|a_{-n}|^2}{m+n+2\beta}\frac{1}{a^{2(m+n+2\beta)}}\frac{1}{a^{2(m+n+2\beta)}}\left(1-\left(\frac{a}{b}\right)^{2(m+n+2\beta)}\right)\nonumber\\
        &+\frac{2\,\Re\left(a_na_{-n}\right)}{m+2\beta}\frac{1}{a^{2(m+2\beta)}}\left(1-\left(\frac{a}{b}\right)^{2(m+2\beta)}\right)\geq \frac{n^2|a_{-n}|^2}{2(m+2\beta)^2(m+n+2\beta)^2}\frac{1}{a^{2(m+n+2\beta)}}.
    \end{align}
    Finally, assuming that both conditions on the conformal class hold, we get
    \begin{align*}
        &\frac{|a_n|^2}{m-n+2\beta}\frac{1}{a^{2(m-n+2\beta)}}\left(1-\left(\frac{a}{b}\right)^{2(m-n+2\beta)}\right)\nonumber\\
        &+\frac{|a_{-n}|^2}{m+n+2\beta}\frac{1}{a^{2(m+n+2\beta)}}\frac{1}{a^{2(m+n+2\beta)}}\left(1-\left(\frac{a}{b}\right)^{2(m+n+2\beta)}\right)\nonumber\\
        &+\frac{2\,\Re\left(a_na_{-n}\right)}{m+2\beta}\frac{1}{a^{2(m+2\beta)}}\left(1-\left(\frac{a}{b}\right)^{2(m+2\beta)}\right)\nonumber\\
        &\geq \frac{n^2|a_n|^2}{4(m+2\beta)^2(m-n+2\beta)}\frac{1}{a^{2(m-n+2\beta)}}+\frac{n^2|a_{-n}|^2}{4(m+2\beta)^2(m+n+2\beta)^2}\frac{1}{a^{2(m+n+2\beta)}}.
    \end{align*}
    Likewise, we have 
    \begin{align}\label{neg_weight_bn}
        &\frac{|b_{-n}|^2}{m-n-2\beta}b^{2(m-n-2\beta)}\left(1-\left(\frac{a}{b}\right)^{2(m-n-2\beta)}\right)+\frac{|b_n|^2}{m+n-2\beta}b^{2(m+n-2\beta)}\left(1-\left(\frac{a}{b}\right)^{2(m+n-2\beta)}\right)\nonumber\\
        &+\frac{2\,\Re\left(b_nb_{-n}\right)}{m-2\beta}b^{2(m-2\beta)}\nonumber\\
        &\geq \frac{n^2|b_n|^2}{4(m-2\beta)^2(m+n-2\beta)}b^{2(m+n-2\beta)}+\frac{n^2|b_{-n}|^2}{4(m-2\beta)^2(m-n-2\beta)}b^{2(m-n-2\beta)}.
    \end{align}
    provided that 
    \begin{align*}
        \log\left(\frac{b}{a}\right)\geq \frac{1}{m-n-2\beta}\log\left(\frac{2(m-2\beta)}{n}\right).
    \end{align*}
    Now, we have
    \begin{align*}
        &\frac{2\,\Re\left(a_n\bar{b_n}\right)}{n-2\beta}b^{2(n-2\beta)}\left(1-\left(\frac{a}{b}\right)^{2(n-2\beta)}\right)\geq -\frac{16(m-2\beta)^2(m+n-2\beta)|a_n|^2}{n^2(n-2\beta)^2}\frac{1}{b^{2(m+n-2\beta)}}\\
        &-\frac{n^2|b_n|^2}{16(m-2\beta)^2(m-n-2\beta)}b^{2(m-n-2\beta)}.
    \end{align*}
    Then, we estimate
    \begin{align*}
        &\frac{2\,\Re\left(a_{-n}\bar{b_{-n}}\right)}{n+2\beta}\frac{1}{a^{2(n+2\beta)}}\left(1-\left(\frac{a}{b}\right)^{2(n+2\beta)}\right)\geq -\frac{n^2|a_{-n}|^2}{8(m+2\beta)^2(m+n+2\beta)}\frac{1}{a^{2(m+n+2\beta)}}\\
        &-\frac{8(m+2\beta)^2(m+n+2\beta)|b_{-n}|^2}{n^2(n+2\beta)^2}a^{2(m-n-2\beta)}.
    \end{align*}
    Finally, we have
    \begin{align*}
        &\frac{1}{\beta}\Re\left(a_nb_{-n}\right)\frac{1}{a^{4\beta}}\left(1-\left(\frac{a}{b}\right)^{4\beta}\right)\geq -\frac{n^2|a_n|^2}{8(m+2\beta)^2(m-n+2\beta)}\frac{1}{a^{2(m-n+2\beta)}}\\
        &-\frac{8(m+2\beta)^2(m-n+2\beta)|b_{-n}|^2}{\beta^2n^2}a^{2(m-n-2\beta)}
    \end{align*}
    and
    \begin{align*}
        &\frac{1}{\beta}\Re\left(a_{-n}b_n\right)\frac{1}{a^{4\beta}}\left(1-\left(\frac{a}{b}\right)^{4\beta}\right)\geq -\frac{16(m-2\beta)^2(m-n-2\beta)|a_{-n}|^2}{\beta^2n^2}\frac{1}{b^{2(m+n+2\beta)}}\\
        &-\frac{n^2|b_{n}|^2}{16(m-2\beta)^2(m-n-2\beta)}b^{2(m-n-2\beta)}.
    \end{align*}
    Therefore, we get
    \begin{align*}
        &\frac{|a_n|^2}{m-n+2\beta}\frac{1}{a^{2(m-n+2\beta)}}\left(1-\left(\frac{a}{b}\right)^{2(m-n+2\beta)}\right)\\
        &+\frac{|a_{-n}|^2}{m+n+2\beta}\frac{1}{a^{2(m+n+2\beta)}}\frac{1}{a^{2(m+n+2\beta)}}\left(1-\left(\frac{a}{b}\right)^{2(m+n+2\beta)}\right)\\
        &+\frac{|b_{-n}|^2}{m-n-2\beta}b^{2(m-n-2\beta)}\left(1-\left(\frac{a}{b}\right)^{2(m-n-2\beta)}\right)+\frac{|b_n|^2}{m+n-2\beta}b^{2(m+n-2\beta)}\left(1-\left(\frac{a}{b}\right)^{2(m+n-2\beta)}\right)\\
        &+\frac{2\,\Re\left(a_na_{-n}\right)}{m+2\beta}\frac{1}{a^{2(m+2\beta)}}\left(1-\left(\frac{a}{b}\right)^{2(m+2\beta)}\right)+\frac{2\,\Re\left(b_nb_{-n}\right)}{m-2\beta}b^{2(m-2\beta)}\left(1-\left(\frac{a}{b}\right)^{2(m-2\beta)}\right)\\
        &+\frac{2\,\Re\left(a_n\bar{b_n}\right)}{n-2\beta}b^{2(n-2\beta)}\left(1-\left(\frac{a}{b}\right)^{2(n-2\beta)}\right)+\frac{2\,\Re\left(a_{-n}\bar{b_{-n}}\right)}{n+2\beta}\frac{1}{a^{2(n+2\beta)}}\left(1-\left(\frac{a}{b}\right)^{2(n+2\beta)}\right)\\
        &+\frac{1}{\beta}\left(\Re\left(a_nb_{-n}\right)+\Re\left(a_{-n}b_n\right)\right)\frac{1}{a^{4\beta}}\left(1-\left(\frac{a}{b}\right)^{4\beta}\right)\\
        &\geq \frac{n^2|a_n|^2}{4(m+2\beta)^2(m-n+2\beta)}\frac{1}{a^{2(m-n+2\beta)}}+\frac{n^2|a_{-n}|^2}{4(m+2\beta)^2(m+n+2\beta)^2}\frac{1}{a^{2(m+n+2\beta)}}\\
        &+\frac{n^2|b_n|^2}{4(m-2\beta)^2(m+n-2\beta)}b^{2(m+n-2\beta)}+\frac{n^2|b_{-n}|^2}{4(m-2\beta)^2(m-n-2\beta)}b^{2(m-n-2\beta)}\\
        &-\frac{16(m-2\beta)^2(m+n-2\beta)|a_n|^2}{n^2(n-2\beta)^2}\frac{1}{b^{2(m+n-2\beta)}}
        -\frac{n^2|b_n|^2}{16(m-2\beta)^2(m-n-2\beta)}b^{2(m-n-2\beta)}\\
        &-\frac{n^2|a_{-n}|^2}{8(m+2\beta)^2(m+n+2\beta)}\frac{1}{a^{2(m+n+2\beta)}}-\frac{8(m+2\beta)^2(m+n+2\beta)|b_{-n}|^2}{n^2(n+2\beta)^2}a^{2(m-n-2\beta)}\\
        &-\frac{n^2|a_n|^2}{8(m+2\beta)^2(m-n+2\beta)}\frac{1}{a^{2(m-n+2\beta)}}-\frac{8(m+2\beta)^2(m-n+2\beta)|b_{-n}|^2}{\beta^2n^2}a^{2(m-n-2\beta)}\\
        &-\frac{16(m-2\beta)^2(m-n-2\beta)|a_{-n}|^2}{\beta^2n^2}\frac{1}{b^{2(m+n+2\beta)}}-\frac{n^2|b_{n}|^2}{16(m-2\beta)^2(m-n-2\beta)}b^{2(m-n-2\beta)}\\
        &=\frac{n^2|a_n|^2}{8(m+2\beta)^2(m-n+2\beta)}\frac{1}{a^{2(m-n+2\beta)}}\left(1-\frac{128\left(m^2-4\beta^2\right)^2\left(m^2-(n-2\beta)^2\right)}{n^2(n-2\beta)^2}\left(\frac{a}{b}\right)^{2(m-n+2\beta)}\right)\\
        &+\frac{n^2|a_{-n}|^2}{8(m+2\beta)^2(m+n+2\beta)^2}\frac{1}{a^{2(m+n+2\beta)}}\left(1-\frac{128\left(m^2-4\beta^2\right)^2\left(m^2-(n+2\beta)^2\right)}{\beta^2n^2}\left(\frac{a}{b}\right)^{2(m+n+2\beta)}\right)\\
        &+\frac{n^2|b_n|^2}{8(m-2\beta)^2(m-n-2\beta)}b^{2(m-n-2\beta)}\\
        &+\frac{n^2|b_{-n}|^2}{4(m-2\beta)^2(m-n-2\beta)}b^{2(m-n-2\beta)}\left(1-\frac{64\left(m^2-4\beta^2\right)^2\left(m^2-(n+2\beta)^2\right)}{n^2(n+2\beta)^2}\left(\frac{a}{b}\right)^{2(m-n-2\beta)}\right.\\
        &\left.-\frac{64\left(m^2-4\beta^2\right)^2\left(m^2-(n-2\beta)^2\right)}{\beta^2n^2}\left(\frac{a}{b}\right)^{2(m-n-2\beta)}\right).
    \end{align*}
    Therefore, we impose the conditions
    \begin{align*}
        &\log\left(\frac{b}{a}\right)\geq \frac{1}{2(m-n+2\beta)}\log\left(\frac{256\left(m^2-4\beta^2\right)^2\left(m^2-(n-2\beta)^2\right)}{n^2(n-2\beta)^2}\right)\\
        &\log\left(\frac{b}{a}\right)\geq \frac{1}{2(m+n+2\beta)}\log\left(\frac{256\left(m^2-4\beta^2\right)^2\left(m^2-(n+2\beta)^2\right)}{\beta^2n^2}\right)\\
        &\log\left(\frac{b}{a}\right)\geq \frac{1}{2(m-n-2\beta)}\log\left(\frac{256\left(m^2-4\beta^2\right)^2\left(m^2-(n+2\beta)^2\right)}{n^2(n+2\beta)^2}\right)\\
        &\log\left(\frac{b}{a}\right)\geq \frac{1}{2(m-n-2\beta)}\log\left(\frac{256\left(m^2-4\beta^2\right)^2\left(m^2-(n-2\beta)^2\right)}{\beta^2n^2}\right),
    \end{align*}
    and this finally yields
    \begin{align*}
        &\sum_{n=1}^{m-1}\left(\frac{n^2|a_n|^2}{16(m+2\beta)^2(m-n+2\beta)^2}\frac{1}{a^{2(m-n)}}+\frac{n^2|a_{-n}|^2}{16(m+2\beta)^2(m+n+2\beta)^2}\frac{1}{a^{2(m+n)}}\right)\\
        &+\left(\frac{a}{b}\right)^{4\beta}\sum_{n=1}^{m-1}\left(\frac{n^2|b_n|^2}{8(m-2\beta)^2(m+n-2\beta)}b^{2(m+n)}+\frac{n^2|b_{-n}|^2}{8(m-2\beta)^2(m-n-2\beta)}b^{2(m-n)}\right)\\
        &\leq \frac{1}{2\pi}\int_{\Omega}\frac{u_0^2}{|x|^4}\left(\frac{a}{|x|}\right)^{4\beta}dx
    \end{align*}
    
    Thanks to \eqref{int_u_square} (and \eqref{sum_squares} that shows that the integrated expression is a sum of squares and that bounding a subset of them gives a control by its left-hand side), we deduce that 
    \begin{align*}
        &\sum_{n=2}^{m-2}\frac{n^2}{16(m-n-2\beta)(m-2\beta)^2}|a_n|^2\frac{1}{a^{2(m-n-2\beta)}}+\sum_{n=2}^{m-2}\frac{n^2}{16(m+n-2\beta)(m-2\beta)^2}|a_{-n}|^2\frac{1}{a^{2(m+n-2\beta)}}\nonumber\\
        &+\sum_{n=2}^{m-2}\frac{n^2}{16(m+n+2\beta)(m+2\beta)^2}|b_n|^2b^{2(m+n+2\beta)}+\sum_{n=2}^{m-2}\frac{n^2}{16(m-n+2\beta)(m+2\beta)^2}|b_{-n}|^2b^{2(m-n+2\beta)}\\
        &\leq \frac{1}{2\pi}\int_{\Omega}\frac{u^2}{|x|^4}|x|^{4\beta}dx.
    \end{align*}
    Now, besides the logarithm terms, we need only estimate $a_{n}$ and $b_n$ for $m-1\leq |n|\leq m$ and $0\leq |n|\leq 1$. For $n=m$, we find the quantity
    \begin{align*}
        &\frac{|a_m|^2}{2\beta}b^{4\beta}\left(1-\left(\frac{a}{b}\right)^{4\beta}\right)
        +\frac{|a_{-m}|^2}{2(m-\beta)}\frac{1}{a^{4(m-\beta)}}\left(1-\left(\frac{a}{b}\right)^{4(m-\beta)}\right)\\
        &+\frac{|b_m|^2}{2(m+\beta)}b^{4(m+\beta)}\left(1-\left(\frac{a}{b}\right)^{4(m+\beta)}\right)
        +\frac{|b_{-m}|^2}{2\beta}b^{4\beta}\left(1-\left(\frac{a}{b}\right)^{4\beta}\right)\\
        &-\frac{2}{m-2\beta}\Re(a_ma_{-m})\frac{1}{a^{2m-4\beta}}\left(1-\left(\frac{a}{b}\right)^{2m-4\beta)}\right)\\
        &+\frac{2}{m+2\beta}\Re(b_mb_{-m})b^{2m+4\beta}\left(1-\left(\frac{a}{b}\right)^{2m+4\beta}\right)+\frac{2\,\Re\left(a_m\bar{b_m}\right)}{m+2\beta}b^{2(m+2\beta)}\left(1-\left(\frac{a}{b}\right)^{2(m+2\beta)}\right)\\
        &+\frac{2\,\Re\left(a_{-m}\bar{a_{-m}}\right)}{m-2\beta}\frac{1}{a^{2(m-2\beta)}}\left(1-\left(\frac{a}{b}\right)^{2(m-2\beta)}\right)\\
        &+4\,\Re\left(a_mb_{-m}+a_{-m}b_m\right)\log\left(\frac{b}{a}\right).
    \end{align*}
    As previously, we estimate 
    \begin{align*}
        &\frac{|a_m|^2}{2\beta}b^{4\beta}\left(1-\left(\frac{a}{b}\right)^{4\beta}\right)+\frac{|a_{-m}|^2}{2(m-\beta)}\frac{1}{a^{4(m-\beta)}}\left(1-\left(\frac{a}{b}\right)^{4(m-\beta)}\right)\\
        &-\frac{2}{m-2\beta}\Re(a_ma_{-m})\frac{1}{a^{2m-4\beta}}\left(1-\left(\frac{a}{b}\right)^{2m-4\beta)}\right)\\
        &\geq \frac{|a_m|^2}{2\beta}b^{4\beta}\left(1-\left(\frac{a}{b}\right)^{4\beta}\right)-\frac{4(m-\beta)}{(m-2\beta)^2}|a_m|^2a^{4\beta}\frac{\left(1-\left(\frac{a}{b}\right)^{2m-4\beta}\right)^2}{1-\left(\frac{a}{b}\right)^{4(m-\beta)}}\\
        &+\frac{|a_{-m}|^2}{4(m-\beta)}\frac{1}{a^{4(m-\beta)}}\left(1-\left(\frac{a}{b}\right)^{4(m-\beta)}\right)\\
        &=\frac{|a_m|^2}{1-\left(\frac{a}{b}\right)^{4(m-\beta)}}b^{4\beta}\left(\frac{1}{2\beta}\left(1-\left(\frac{a}{b}\right)^{4\beta}\right)\left(1-\left(\frac{a}{b}\right)^{4(m-\beta)}\right)-\frac{2(m-\beta)}{(m-2\beta)^2}\left(\frac{a}{b}\right)^{4\beta}\left(1-\left(\frac{a}{b}\right)^{2m-4\beta}\right)^2\right)\\
        &+\frac{|a_{-m}|^2}{4(m-\beta)}\frac{1}{a^{4(m-\beta)}}\left(1-\left(\frac{a}{b}\right)^{4(m-\beta)}\right)\\
        &\geq \frac{|a_m|^2}{1-\left(\frac{a}{b}\right)^{4(m-\beta)}}b^{4\beta}\left(\frac{1}{2\beta}-\frac{1}{2\beta}\left(\frac{a}{b}\right)^{4\beta}-\frac{1}{2\beta}\left(\frac{a}{b}\right)^{4(m-\beta)}-\frac{4(m-\beta)}{(m-2\beta)^2}\left(\frac{a}{b}\right)^{4\beta}\right)\\
        &+\frac{|a_{-m}|^2}{4(m-\beta)}\frac{1}{a^{4(m-\beta)}}\left(1-\left(\frac{a}{b}\right)^{4(m-\beta)}\right)
    \end{align*}
    We impose the following conditions on the conformal class
    \begin{align*}
    \left\{\begin{alignedat}{1}
        &\frac{1}{2\beta}\left(\frac{a}{b}\right)^{4\beta}\leq \frac{1}{16\beta}\\
        &\frac{1}{2\beta}\left(\frac{a}{b}\right)^{4(m-\beta)}\leq \frac{1}{16\beta}\\
        &\frac{4(m-\beta)}{(m-2\beta)^2}\left(\frac{a}{b}\right)^{4\beta}\leq \frac{1}{16\beta}.
        \end{alignedat}\right.
    \end{align*}
    Therefore, the condition is equivalent to
    \begin{align}\label{lm_log_12}
        \log\left(\frac{b}{a}\right)\geq \max\ens{\frac{1}{4\beta}\log(8),\frac{1}{4(m-\beta)}\log(8),\frac{1}{4\beta}\log\left(\frac{64\beta(m-\beta)}{(m-2\beta)^2}\right)}.
    \end{align}
    If this condition is satisfied, we deduce that
    \begin{align}\label{am_est1}
        &\frac{|a_m|^2}{2\beta}b^{4\beta}\left(1-\left(\frac{a}{b}\right)^{4\beta}\right)+\frac{|a_{-m}|^2}{2(m-\beta)}\frac{1}{a^{4(m-\beta)}}\left(1-\left(\frac{a}{b}\right)^{4(m-\beta)}\right)\nonumber\\
        &-\frac{2}{m-2\beta}\Re(a_ma_{-m})\frac{1}{a^{2m-4\beta}}\left(1-\left(\frac{a}{b}\right)^{2m-4\beta)}\right)\geq \frac{1}{4\beta}|a_m|^2b^{4\beta}+\frac{|a_{-m}|^2}{4(m-\beta)}\frac{1}{a^{4(m-\beta)}}\left(1-\left(\frac{a}{b}\right)^{4(m-\beta)}\right)\nonumber\\
        &\geq \frac{1}{4\beta}|a_m|^2b^{4\beta}+\frac{|a_{-m}|^2}{8(m-\beta)}\frac{1}{a^{4(m-\beta)}}.
    \end{align}
    Now, for $n=m-1$, the relevant quantity is given by
    \begin{align*}
        &\frac{|a_{m-1}|^2}{2\beta-1}b^{2(2\beta-1)}\left(1-\left(\frac{a}{b}\right)^{2(2\beta-1)}\right)+\frac{|a_{1-m}|^2}{2m-1-2\beta}\frac{1}{a^{2(2m-1-2\beta)}}\left(1-\left(\frac{a}{b}\right)^{2(2m-1-2\beta)}\right)
        \\
        &-\frac{2\,\Re\left(a_{m-1}a_{1-m}\right)}{m-2\beta}\frac{1}{a^{2m-4\beta}}\left(1-\left(\frac{a}{b}\right)^{2m-4\beta}\right)\\
        &\geq \frac{|a_{m-1}|^2}{2\beta-1}b^{2(2\beta-1)}\left(1-\left(\frac{a}{b}\right)^{2(2\beta-1)}\right)-\frac{2(2m-1-2\beta)}{(m-2\beta)^2}|a_{m-1}|^2\frac{1}{a^{2(1-2\beta)}}\frac{\left(1-\left(\frac{a}{b}\right)^{2m-4\beta}\right)^2}{1-\left(\frac{a}{b}\right)^{2(2m-1-2\beta)}}\\
        &+\frac{|a_{1-m}|^2}{2(2m-1-2\beta)}\frac{1}{a^{2(2m-1-2\beta)}}\left(1-\left(\frac{a}{b}\right)^{2(2m-1-2\beta)}\right)\\
        &\geq \frac{|a_{m-1}|^2}{1-\left(\frac{a}{b}\right)^{2(2m-1-2\beta)}}b^{2(2\beta-1)}\left(\frac{1}{2\beta-1}\left(1-\left(\frac{a}{b}\right)^{2(2\beta-1)}\right)\left(1-\left(\frac{a}{b}\right)^{2(2m-1-2\beta)}\right)\right.\\
        &\left.-\frac{2(2m-1-2\beta)}{(m-2\beta)^2}\left(\frac{a}{b}\right)^{2(2\beta-1)}\left(1-\left(\frac{a}{b}\right)^{2m-4\beta}\right)^2\right).
    \end{align*}
    Therefore, we impose the following conditions
    \begin{align*}
        \left\{\begin{alignedat}{1}
            &\frac{1}{2\beta-1}\left(\frac{a}{b}\right)^{2(2\beta-1)}\leq \frac{1}{8(2\beta-1)}\\
            &\frac{1}{2\beta-1}\left(\frac{a}{b}\right)^{2(2m-1-2\beta)}\leq \frac{1}{8(2\beta-1)}\\
            &\frac{2(2m-1-2\beta)}{(m-2\beta)^2}\left(\frac{a}{b}\right)^{2(2\beta-1)}\leq \frac{1}{8(2\beta-1)}
        \end{alignedat}\right.,
    \end{align*}
    which yields the condition
    \begin{align}\label{lm_log_13}
        \log\left(\frac{b}{a}\right)\geq \max\ens{\frac{1}{2(2\beta-1)}\log(8),\frac{1}{2(2m-1-2\beta)}\log(8),\frac{1}{2(2\beta-1)}\log\left(\frac{16(2\beta-1)(2m-1-2\beta)}{(m-2\beta)^2}\right)},
    \end{align}
    and we get
    \begin{align*}
        &\frac{|a_{m-1}|^2}{2\beta-1}b^{2(2\beta-1)}\left(1-\left(\frac{a}{b}\right)^{2(2\beta-1)}\right)+\frac{|a_{1-m}|^2}{2m-1-2\beta}\frac{1}{a^{2(2m-1-2\beta)}}\left(1-\left(\frac{a}{b}\right)^{2(2m-1-2\beta)}\right)
        \\
        &-\frac{2\,\Re\left(a_{m-1}a_{1-m}\right)}{m-2\beta}\frac{1}{a^{2m-4\beta}}\left(1-\left(\frac{a}{b}\right)^{2m-4\beta}\right)\\
        &\geq \frac{|a_{m-1}|^2}{2(2\beta-1)}b^{2(2\beta-1)}+\frac{|a_{1-m}|^2}{2(2m-1-2\beta)}\frac{1}{a^{2(2m-1-2\beta)}}\left(1-\left(\frac{a}{b}\right)^{2(2m-1-2\beta)}\right)\nonumber\\
        &\geq \frac{|a_{m-1}|^2}{2(2\beta-1)}b^{2(2\beta-1)}+\frac{|a_{1-m}|^2}{4(2m-1-2\beta)}\frac{1}{a^{2(2m-1-2\beta)}}.
    \end{align*}
    Now, we consider 
    \begin{align*}
        &\frac{|a_1|^2}{m-1-2\beta}\frac{1}{a^{2(m-1-2\beta)}}\left(1-\left(\frac{a}{b}\right)^{2(m-1-2\beta)}\right)+\frac{|a_{-1}|^2}{m+1-2\beta}\frac{1}{a^{2(m+1-2\beta)}}\left(1-\left(\frac{a}{b}\right)^{2(m+1-2\beta)}\right)\\
        &-\frac{2\,\Re\left(a_1a_{-1}\right)}{m-2\beta}\frac{1}{a^{2m-4\beta}}\left(1-\left(\frac{a}{b}\right)^{2m-4\beta}\right)\\
        &\geq \frac{|a_1|^2}{m-1-2\beta}\frac{1}{a^{2(m-1-2\beta)}}\left(1-\left(\frac{a}{b}\right)^{2(m-1-2\beta)}\right)-\frac{(m+1-2\beta)}{(m-2\beta)^2}|a_1|^2\frac{1}{a^{2(m-1-2\beta)}}\frac{\left(1-\left(\frac{a}{b}\right)^{2m-4\beta}\right)^2}{1-\left(\frac{a}{b}\right)^{2(m+1-2\beta)}}\\
        &\geq \frac{|a_1|^2}{1-\left(\frac{a}{b}\right)^{2(m+1-2\beta)}}\frac{1}{a^{2(m-1-2\beta)}}\left(\frac{1}{(m-1-2\beta)(m-2\beta)^2}-\frac{1}{m-1-2\beta}\left(\frac{a}{b}\right)^{2(m-1-2\beta)}\right.\\
        &\left.-\frac{1}{m-1-2\beta}\left(\frac{a}{b}\right)^{2(m+1-2\beta)}\right).
    \end{align*}
    Therefore, we impose the conditions 
    \begin{align*}
        &\frac{1}{m-1-2\beta}\left(\frac{a}{b}\right)^{2(m-1-2\beta)}\leq \frac{1}{4(m-1-2\beta)(m-2\beta)^2}\\
        &\frac{1}{m-1-2\beta}\left(\frac{a}{b}\right)^{2(m+1-2\beta)}\leq \frac{1}{4(m-1-2\beta)(m-2\beta)^2},
    \end{align*}
    which yields
    \begin{align}\label{lm_log_14}
        \log\left(\frac{b}{a}\right)\geq \max\ens{\frac{1}{2(m+1-2\beta)}\log(4),\frac{1}{2(m-1-2\beta)}\log(4)}=\frac{1}{2(m-1-2\beta)}\log(4),
    \end{align}
    and yields 
    \begin{align*}
        &\frac{|a_1|^2}{m-1-2\beta}\frac{1}{a^{2(m-1-2\beta)}}\left(1-\left(\frac{a}{b}\right)^{2(m-1-2\beta)}\right)+\frac{|a_{-1}|^2}{m+1-2\beta}\frac{1}{a^{2(m+1-2\beta)}}\left(1-\left(\frac{a}{b}\right)^{2(m+1-2\beta)}\right)\\
        &-\frac{2\,\Re\left(a_1a_{-1}\right)}{m-2\beta}\frac{1}{a^{2m-4\beta}}\left(1-\left(\frac{a}{b}\right)^{2m-4\beta}\right)\\
        &\geq \frac{|a_1|^2}{2(m-1-2\beta)}\frac{1}{a^{2(m-1-2\beta)}}.
    \end{align*}
    Likewise, we have (with the same condition on the conformal class)
    \begin{align*}
        &\frac{|a_1|^2}{m-1-2\beta}\frac{1}{a^{2(m-1-2\beta)}}\left(1-\left(\frac{a}{b}\right)^{2(m-1-2\beta)}\right)+\frac{|a_{-1}|^2}{m+1-2\beta}\frac{1}{a^{2(m+1-2\beta)}}\left(1-\left(\frac{a}{b}\right)^{2(m+1-2\beta)}\right)\\
        &-\frac{2\,\Re\left(a_1a_{-1}\right)}{m-2\beta}\frac{1}{a^{2m-4\beta}}\left(1-\left(\frac{a}{b}\right)^{2m-4\beta}\right)\\
        &\geq \frac{|a_{-1}|^2}{2(m+1-2\beta)}\frac{1}{a^{2(m+1-2\beta)}},
    \end{align*}
    and finally, we get
    \begin{align}\label{a1_est}
        &\frac{|a_1|^2}{m-1-2\beta}\frac{1}{a^{2(m-1-2\beta)}}\left(1-\left(\frac{a}{b}\right)^{2(m-1-2\beta)}\right)+\frac{|a_{-1}|^2}{m+1-2\beta}\frac{1}{a^{2(m+1-2\beta)}}\left(1-\left(\frac{a}{b}\right)^{2(m+1-2\beta)}\right)\nonumber\\
        &-\frac{2\,\Re\left(a_1a_{-1}\right)}{m-2\beta}\frac{1}{a^{2m-4\beta}}\left(1-\left(\frac{a}{b}\right)^{2m-4\beta}\right)\nonumber\\
        &\geq \frac{|a_1|^2}{4(m-1-2\beta)}\frac{1}{a^{2(m-1-2\beta)}}+\frac{|a_{-1}|^2}{4(m+1-2\beta)}\frac{1}{a^{2(m+1-2\beta)}}. 
    \end{align}
    Now, we estimate the $b_n$ terms. We consider 
    \begin{align*}
        &\frac{|b_m|^2}{2(m+\beta)}b^{4(m+\beta)}\left(1-\left(\frac{a}{b}\right)^{4(m+\beta)}\right)+\frac{|b_{-m}|^2}{2\beta}b^{4\beta}\left(1-\left(\frac{a}{b}\right)^{4\beta}\right)\\
        &+\frac{2\,\Re\left(b_mb_{-m}\right)}{m+2\beta}b^{2m+4\beta}\left(1-\left(\frac{a}{b}\right)^{2m+4\beta}\right)\\
        &\geq \frac{|b_m|^2}{1-\left(\frac{a}{b}\right)^{4\beta}}b^{4(m+\beta)}\left(\frac{m^2}{2(m+\beta)(m+2\beta)^2}-\frac{1}{2(m+\beta)}\left(\frac{a}{b}\right)^{4(m+\beta)}-\frac{1}{2(m+\beta)}\left(\frac{a}{b}\right)^{4\beta}\right).
    \end{align*}
    Therefore, we impose the conditions
    \begin{align*}
        \frac{1}{2(m+\beta)}\left(\frac{a}{b}\right)^{4(m+\beta)}\leq \frac{m^2}{8(m+\beta)(m+2\beta)^2}\\
        \frac{1}{2(m+\beta)}\left(\frac{a}{b}\right)^{4\beta}\leq \frac{m^2}{8(m+\beta)(m+2\beta)^2},
    \end{align*}
    which yields
    \begin{align}\label{lm_log_15}
        \log\left(\frac{b}{a}\right)\geq \max\ens{\frac{1}{2(m+\beta)}\log\left(\frac{2(m+2\beta)}{m}\right),\frac{1}{2\beta}\log\left(\frac{2(m+2\beta)}{m}\right)}=\frac{1}{2\beta}\log\left(\frac{2(m+2\beta)}{m}\right),
    \end{align}
    and we get the estimate
    \begin{align*}
        &\frac{|b_m|^2}{2(m+\beta)}b^{4(m+\beta)}\left(1-\left(\frac{a}{b}\right)^{4(m+\beta)}\right)+\frac{|b_{-m}|^2}{2\beta}b^{4\beta}\left(1-\left(\frac{a}{b}\right)^{4\beta}\right)\\
        &+\frac{2\,\Re\left(b_mb_{-m}\right)}{m+2\beta}b^{2m+4\beta}\left(1-\left(\frac{a}{b}\right)^{2m+4\beta}\right)\geq \frac{m^2}{4(m+\beta)(m+2\beta)^2}|b_m|^2b^{4(m+\beta)}.
    \end{align*}
    Likewise, if
    \begin{align}\label{lm_log_15bis}
        \log\left(\frac{b}{a}\right)\geq \frac{1}{2\beta}\log\left(\frac{2(m+2\beta)}{m}\right),
    \end{align}
    we get
    \begin{align*}
        &\frac{|b_m|^2}{2(m+\beta)}b^{4(m+\beta)}\left(1-\left(\frac{a}{b}\right)^{4(m+\beta)}\right)+\frac{|b_{-m}|^2}{2\beta}b^{4\beta}\left(1-\left(\frac{a}{b}\right)^{4\beta}\right)\\
        &+\frac{2\,\Re\left(b_mb_{-m}\right)}{m+2\beta}b^{2m+4\beta}\left(1-\left(\frac{a}{b}\right)^{2m+4\beta}\right)\geq \frac{m^2}{4\beta(m+2\beta)^2}|b_{-m}|^2b^{4\beta},
    \end{align*}
    and finally, we have
    \begin{align}\label{bm_est1}
        &\frac{|b_m|^2}{2(m+\beta)}b^{4(m+\beta)}\left(1-\left(\frac{a}{b}\right)^{4(m+\beta)}\right)+\frac{|b_{-m}|^2}{2\beta}b^{4\beta}\left(1-\left(\frac{a}{b}\right)^{4\beta}\right)\nonumber\\
        &+\frac{2\,\Re\left(b_mb_{-m}\right)}{m+2\beta}b^{2m+4\beta}\left(1-\left(\frac{a}{b}\right)^{2m+4\beta}\right)\geq \frac{m^2}{8(m+\beta)(m+2\beta)^2}|b_m|^2b^{4(m+\beta)}+\frac{m^2}{8\beta(m+2\beta)^2}|b_{-m}|^2b^{4\beta}.
    \end{align}
    Then, we have
    \begin{align*}
        &\frac{|b_{m-1}|^2}{2m-1+2\beta}b^{2(2m-1+2\beta)}\left(1-\left(\frac{a}{b}\right)^{2(m-1+2\beta)}\right)+\frac{|b_{1-m}|^2}{2\beta+1}b^{2(2\beta+1)}\left(1-\left(\frac{a}{b}\right)^{2(2\beta+1)}\right)\\
        &+\frac{2\,\Re\left(b_{m-1}b_{1-m}\right)}{m+2\beta}b^{2m+4\beta}\left(1-\left(\frac{a}{b}\right)^{2m+4\beta}\right)\\
        &\geq \frac{|b_{m-1}|^2}{2m-1+2\beta}b^{2(2m-1+2\beta)}\left(1-\left(\frac{a}{b}\right)^{2(m-1+2\beta)}\right)-\frac{2\beta+1}{(m+2\beta)^2}|b_{m-1}|^2b^{2(2m-1+2\beta)}\frac{\left(1-\left(\frac{a}{b}\right)^{2m+4\beta}\right)^2}{1-\left(\frac{a}{b}\right)^{2(2\beta+1)}}\\
        &\geq \frac{|b_{m-1}|^2}{1-\left(\frac{a}{b}\right)^{2(2\beta+1)}}b^{2(2m-1+2\beta)}\left(\frac{(m-1)^2}{(2m-1+2\beta)(m+2\beta)^2}-\frac{1}{2m-1+2\beta}\left(\frac{a}{b}\right)^{2(2m-1+2\beta)}\right.\\
        &\left.-\frac{1}{2m-1+2\beta}\left(\frac{a}{b}\right)^{2(2\beta+1)}\right).
    \end{align*}
    Therefore, we impose the conditions 
    \begin{align*}
        &\frac{1}{2m-1+2\beta}\left(\frac{a}{b}\right)^{2(2m-1+2\beta)}\leq \frac{(m-1)^2}{4(2m-1+2\beta)(m+2\beta)^2}\\
        &\frac{1}{2m-1+2\beta}\left(\frac{a}{b}\right)^{2(2\beta+1)}\leq \frac{(m-1)^2}{4(2m-1+2\beta)(m+2\beta)^2}.
    \end{align*}
    Therefore, the conditions becomes
    \begin{align}\label{lm_log_16}
        \log\left(\frac{b}{a}\right)\geq \max\ens{\frac{1}{2m-1+2\beta}\log\left(\frac{2(m+2\beta)}{m-1}\right),\frac{1}{2\beta+1}\log\left(\frac{2(m+2\beta)}{m-1}\right)},%\max\ens{\frac{1}{2(2m-1+2\beta)}\log\left(\frac{4(m+2\beta)^2}{(m-1)^2}\right),\frac{1}{2(2\beta+1)}\log\left(\frac{4(m+2\beta)^2}{(m-1)^2}\right)}.
    \end{align}
    and the estimate is given by 
    \begin{align*}
         &\frac{|b_{m-1}|^2}{2m-1+2\beta}b^{2(2m-1+2\beta)}\left(1-\left(\frac{a}{b}\right)^{2(m-1+2\beta)}\right)+\frac{|b_{1-m}|^2}{2\beta+1}b^{2(2\beta+1)}\left(1-\left(\frac{a}{b}\right)^{2(2\beta+1)}\right)\\
        &+\frac{2\,\Re\left(b_{m-1}b_{1-m}\right)}{m+2\beta}b^{2m+4\beta}\left(1-\left(\frac{a}{b}\right)^{2m+4\beta}\right)\geq \frac{(m-1)^2}{2(2m-1+2\beta)(m+2\beta)^2}|b_{m-1}|^2b^{2(2m-1+2\beta)}.
    \end{align*}
    Likewise, the hypothesis \eqref{lm_log_16} shows that
    \begin{align*}
        &\frac{|b_{m-1}|^2}{2m-1+2\beta}b^{2(2m-1+2\beta)}\left(1-\left(\frac{a}{b}\right)^{2(m-1+2\beta)}\right)+\frac{|b_{1-m}|^2}{2\beta+1}b^{2(2\beta+1)}\left(1-\left(\frac{a}{b}\right)^{2(2\beta+1)}\right)\\
        &+\frac{2\,\Re\left(b_{m-1}b_{1-m}\right)}{m+2\beta}b^{2m+4\beta}\left(1-\left(\frac{a}{b}\right)^{2m+4\beta}\right)\geq \frac{(m-1)^2}{2(2\beta+1)(m+2\beta)^2}b^{2(2\beta+1)},
    \end{align*}
    and finally, we get
    \begin{align}\label{b_m-1_est1}
        &\frac{|b_{m-1}|^2}{2m-1+2\beta}b^{2(2m-1+2\beta)}\left(1-\left(\frac{a}{b}\right)^{2(m-1+2\beta)}\right)+\frac{|b_{1-m}|^2}{2\beta+1}b^{2(2\beta+1)}\left(1-\left(\frac{a}{b}\right)^{2(2\beta+1)}\right)\nonumber\\
        &+\frac{2\,\Re\left(b_{m-1}b_{1-m}\right)}{m+2\beta}b^{2m+4\beta}\left(1-\left(\frac{a}{b}\right)^{2m+4\beta}\right)\geq \frac{(m-1)^2}{4(2m-1+2\beta)(m+2\beta)^2}|b_{m-1}|^2b^{2(2m-1+2\beta)}\nonumber\\
        &+\frac{(m-1)^2}{4(2\beta+1)(m+2\beta)^2}b^{2(2\beta+1)}.
    \end{align}
    Then, we have
    \begin{align*}
        &\frac{|b_1|^2}{m+1+2\beta}b^{2(m+1+2\beta)}\left(1-\left(\frac{a}{b}\right)^{2(m+1+2\beta)}\right)+\frac{|b_{-1}|^2}{m-1+2\beta}b^{2(m-1+2\beta)}\left(1-\left(\frac{a}{b}\right)^{2(m-1+2\beta)}\right)\\
        &+\frac{2\,\Re\left(b_1b_{-1}\right)}{m+2\beta}b^{2m+4\beta}\left(1-\left(\frac{a}{b}\right)^{2m+4\beta}\right)\\
        &\geq \frac{|b_1|^2}{1-\left(\frac{a}{b}\right)^{2(m-1+2\beta)}}b^{2(m+1+2\beta)}\left(\frac{1}{(m+1+2\beta)(m+2\beta)^2}\right.\\
        &\left.-\frac{1}{m+1+2\beta}\left(\frac{a}{b}\right)^{2(m+1+2\beta)}-\frac{1}{m+1+2\beta}\left(\frac{a}{b}\right)^{2(m-1+2\beta)}\right).
    \end{align*}
    Therefore, we impose the conditions
    \begin{align*}
        &\frac{1}{m+1+2\beta}\left(\frac{a}{b}\right)^{2(m+1+2\beta)}\leq \frac{1}{4(m+1+2\beta)(m+2\beta)^2}\\
        &\frac{1}{m-1+2\beta}\left(\frac{a}{b}\right)^{2(m-1+2\beta)}\leq \frac{1}{4(m+1+2\beta)(m+2\beta)^2},
    \end{align*}
    which yields the condition
    \begin{align}\label{lm_log_17}
        \log\left(\frac{b}{a}\right)&\geq \max\ens{\frac{1}{m+1+2\beta}\log\left(2(m+2\beta)\right),\frac{1}{m-1+2\beta}\log\left(2(m+2\beta)\right)}\nonumber\\
        &=\frac{1}{m-1+2\beta}\log\left(2(m+2\beta)\right),
    \end{align}
    and we get the estimate
    \begin{align*}
        &\frac{|b_1|^2}{m+1+2\beta}b^{2(m+1+2\beta)}\left(1-\left(\frac{a}{b}\right)^{2(m+1+2\beta)}\right)+\frac{|b_{-1}|^2}{m-1+2\beta}b^{2(m-1+2\beta)}\left(1-\left(\frac{a}{b}\right)^{2(m-1+2\beta)}\right)\\
        &+\frac{2\,\Re\left(b_1b_{-1}\right)}{m+2\beta}b^{2m+4\beta}\left(1-\left(\frac{a}{b}\right)^{2m+4\beta}\right)
        \geq \frac{1}{2(m+1+2\beta)(m+2\beta)^2}|b_1|^2b^{2(m+1+2\beta)}.
    \end{align*}
    Likewise, if \eqref{lm_log_17} holds, we get
    \begin{align*}
        &\frac{|b_1|^2}{m+1+2\beta}b^{2(m+1+2\beta)}\left(1-\left(\frac{a}{b}\right)^{2(m+1+2\beta)}\right)+\frac{|b_{-1}|^2}{m-1+2\beta}b^{2(m-1+2\beta)}\left(1-\left(\frac{a}{b}\right)^{2(m-1+2\beta)}\right)\\
        &+\frac{2\,\Re\left(b_1b_{-1}\right)}{m+2\beta}b^{2m+4\beta}\left(1-\left(\frac{a}{b}\right)^{2m+4\beta}\right)
        \geq \frac{1}{2(m-1+2\beta)(m+2\beta)^2}|b_{-1}|^2b^{2(m-1+2\beta)},
    \end{align*}
    and finally, we have
    \begin{align}\label{b1_est1}
        &\frac{|b_1|^2}{m+1+2\beta}b^{2(m+1+2\beta)}\left(1-\left(\frac{a}{b}\right)^{2(m+1+2\beta)}\right)+\frac{|b_{-1}|^2}{m-1+2\beta}b^{2(m-1+2\beta)}\left(1-\left(\frac{a}{b}\right)^{2(m-1+2\beta)}\right)\nonumber\\
        &+\frac{2\,\Re\left(b_1b_{-1}\right)}{m+2\beta}b^{2m+4\beta}\left(1-\left(\frac{a}{b}\right)^{2m+4\beta}\right)
        \geq \frac{1}{4(m+1+2\beta)(m+2\beta)^2}|b_1|^2b^{2(m+1+2\beta)}\nonumber\\
        &+\frac{1}{4(m-1+2\beta)(m+2\beta)^2}|b_{-1}|^2b^{2(m-1+2\beta)}.
    \end{align}
    Then, we have
    \begin{align*}
        &\frac{1}{4\beta}|a_m|^2b^{4\beta}+\frac{|a_{-m}|^2}{8(m-\beta)}\frac{1}{a^{4(m-\beta)}}
        +\frac{m^2}{8(m+\beta)(m+2\beta)^2}|b_m|^2b^{4(m+\beta)}+\frac{m^2}{8\beta(m+2\beta)^2}|b_{-m}|^2b^{4\beta}\\
        &+4\,\Re\left(a_mb_{-m}+a_{-m}b_m\right)\log\left(\frac{b}{a}\right).
    \end{align*}

    \textbf{Step 8}. The coefficients $a_n$ for $|n|=m$.
    
    For $n=m$, we will use the $L^2$ norm of $\mathfrak{L}_m$. The relevant quantity is 
    \begin{align*}
        &2(m-1)^4|a_m|^2\log\left(\frac{b}{a}\right)+\frac{1}{2}\frac{\left(m(m+1)-(m-1)\right)^2+(m-1)^2}{2m}|a_{-m}|^2\frac{1}{a^{4m}}\left(1-\left(\frac{a}{b}\right)^{4m}\right)\\
        &-\frac{2(m-1)}{m}\left(m^2-m+1\right)\Re\left(a_ma_{-m}\right)\frac{1}{a^{2m}}\left(1-\left(\frac{a}{b}\right)^{2m}\right)\\
        &+\frac{1}{2}\frac{\left(m(3m-1)+m(m-1)\right)^2+m^2(m-1)^2}{2m}|b_m|^2b^{4m}\left(1-\left(\frac{a}{b}\right)^{4m}\right)\\
        &+m^2(m-1)^2|b_{-m}|^2\log\left(\frac{b}{a}\right)+2m(m-1)(2m-1)\Re\left(b_mb_{-m}\right)b^{2m}\left(1-\left(\frac{a}{b}\right)^{2m}\right)\\
        &+\frac{\left(m(m-1)-(m-1)\right)\left(m(3m-1)+m(m-1)\right)-m(m-1)^2}{m}\Re\left(a_m\bar{b_m}\right)b^{2m}\left(1-\left(\frac{a}{b}\right)^{2m}\right)\\
        &+\frac{\left(m(m+1)-(m-1)\right)\left(m(-m-1)+m(m-1)\right)-m(m-1)^2}{m}\Re\left(a_{-m}\bar{b_{-m}}\right)\frac{1}{a^{2m}}\left(1-\left(\frac{a}{b}\right)^{2m}\right)\\
        &+2\left(m(m-1)\left(m(m-1)-(m-1)\right)+(m-1)\left(m(-m-1)-m(m-1)\right)\right)\Re\left(a_mb_{-m}\right)\log\left(\frac{b}{a}\right)\\
        &+2\left(m(m-1)\left(m(m+1)-(m-1)\right)+(m-1)\left(m(3m+1)-m(m-1)\right)\right)\Re\left(a_{-m}b_{m}\right)\log\left(\frac{b}{a}\right).
    \end{align*}
    Then, we have
    \begin{align*}
        \left(m(m+1)-(m-1)\right)^2+(m-1)^2&=m^4+3m^2-2m+2\\
        \left(m(3m-1)+m(m-1)\right)^2+m^2(m-1)^2&=m^2(17m^2-18m+5)\\
        \left(m(m-1)-(m-1)\right)\left(m(3m-1)+m(m-1)\right)-m(m-1)^2&=m(m-1)^2(4m-3)\\
        \left(m(m+1)-(m-1)\right)\left(m(-m-1)+m(m-1)\right)-m(m-1)^2&=-m(3m^2-2m+3)\\
        m(m-1)\left(m(m-1)-(m-1)\right)+(m-1)\left(m(-m-1)-m(m-1)\right)&=m(m-1)(m^2-4m+1)\\
        m(m-1)\left(m(m+1)-(m-1)\right)+(m-1)\left(m(3m+1)-m(m-1)\right)&=m(m-1)(m^2+2m+3).
    \end{align*}
    Therefore, this quantity can be rewritten as 
    \begin{align*}
        &(m-1)^2\left((m-1)^2+1\right)|a_m|^2\log\left(\frac{b}{a}\right)+\frac{m^4+3m^2-2m+2}{4m}|a_{-m}|^2\frac{1}{a^{4m}}\left(1-\left(\frac{a}{b}\right)^{4m}\right)\\
        &-\frac{2(m-1)}{m}\left(m^2-m+1\right)\Re\left(a_ma_{-m}\right)\frac{1}{a^{2m}}\left(1-\left(\frac{a}{b}\right)^{2m}\right)\\
        &+\frac{m(17m^2-18m+5)}{4}|b_m|^2b^{4m}\left(1-\left(\frac{a}{b}\right)^{4m}\right)
        +m^2(m-1)^2|b_{-m}|^2\log\left(\frac{b}{a}\right)\\
        &+2m(m-1)(2m-1)\Re\left(b_mb_{-m}\right)b^{2m}\left(1-\left(\frac{a}{b}\right)^{2m}\right)\\
        &+m(m-1)^2(4m-3)\Re\left(a_m\bar{b_m}\right)b^{2m}\left(1-\left(\frac{a}{b}\right)^{2m}\right)\\
        &-m(3m^2-2m+3)\Re\left(a_{-m}\bar{b_{-m}}\right)\frac{1}{a^{2m}}\left(1-\left(\frac{a}{b}\right)^{2m}\right)\\
        &+m(m-1)\left(m^2-4m+1\right)\Re\left(a_mb_{-m}\right)\log\left(\frac{b}{a}\right)\\
        &+m(m-1)\left(m^2+2m+3\right)\Re\left(a_{-m}b_m\right)\log\left(\frac{b}{a}\right).
    \end{align*}
    Fix $\epsilon,\delta>0$. We have
    \begin{align*}
        &(m-1)^2\left((m-1)^2+1\right)|a_m|^2\log\left(\frac{b}{a}\right)+\frac{m^4+3m^2-2m+2}{4m}|a_{-m}|^2\frac{1}{a^{4m}}\left(1-\left(\frac{a}{b}\right)^{4m}\right)\\
        &-\frac{2(m-1)}{m}\left(m^2-m+1\right)\Re\left(a_ma_{-m}\right)\frac{1}{a^{2m}}\left(1-\left(\frac{a}{b}\right)^{2m}\right)\\
        &\geq (m-1)^2\left((m-1)^2+1\right)|a_m|^2\log\left(\frac{b}{a}\right)-\frac{(m-1)^2\left(m^2-m+1\right)^2}{m^2\epsilon}|a_m|^2\frac{\left(1-\left(\frac{a}{b}\right)^{2m}\right)^2}{1-\left(\frac{a}{b}\right)^{4m}}\\
        &+\left(\frac{m^4+3m^2-2m+2}{4m}-\epsilon\right)|a_{-m}|^2\frac{1}{a^{4m}}\left(1-\left(\frac{a}{b}\right)^{4m}\right).
    \end{align*}
    Likewise, we have
    \begin{align*}
        &\frac{m(17m^2-18m+5)}{4}|b_m|^2b^{4m}\left(1-\left(\frac{a}{b}\right)^{4m}\right)+m^2(m-1)^2|b_{-m}|^2\log\left(\frac{b}{a}\right)\\
        &+2m(m-1)(2m-1)\Re\left(b_mb_{-m}\right)b^{2m}\left(1-\left(\frac{a}{b}\right)^{2m}\right)\\
        &\geq \left(\frac{m(17m^2-18m+5)}{4}-\delta\right)|b_m|^2b^{4m}\left(1-\left(\frac{a}{b}\right)^{4m}\right)+m^2(m-1)^2|b_{-m}|^2\log\left(\frac{b}{a}\right)\\
        &-\frac{m^2(m-1)^2(2m-1)^2}{\delta}|b_{-m}|^2\frac{\left(1-\left(\frac{a}{b}\right)^{2m}\right)^2}{1-\left(\frac{a}{b}\right)^{4m}}.
    \end{align*}
    Then, we have 
    \begin{align*}
        m(m-1)^2(4m-3)\Re\left(a_m\bar{b_m}\right)b^{2m}\left(1-\left(\frac{a}{b}\right)^{2m}\right)&\geq -\frac{m^2(m-1)^4(4m-3)^2}{2\delta}\frac{\left(1-\left(\frac{a}{b}\right)^{2m}\right)^2}{1-\left(\frac{a}{b}\right)^{4m}}\\
        &-\delta |b_m|^2b^{4m}\left(1-\left(\frac{a}{b}\right)^{4m}\right),
    \end{align*}
    and
    \begin{align*}
        -m(3m^2-2m+3)\Re\left(a_{-m}b_{-m}\right)\frac{1}{a^{2m}}\left(1-\left(\frac{a}{b}\right)^{2m}\right)&\geq -\epsilon |a_{-m}|^2\frac{1}{a^{4m}}\left(1-\left(\frac{a}{b}\right)^{4m}\right)\\
        &-\frac{m^2(3m^2-2m+3)^2}{\epsilon}|b_{-m}|^2\frac{\left(1-\left(\frac{a}{b}\right)^{2m}\right)^2}{1-\left(\frac{a}{b}\right)^{4m}}.
    \end{align*}
    Then, for all $\eta>0$ and $\xi>0$, we have
    \begin{align*}
        m(m-1)\left(m^2-4m+1\right)\Re\left(a_{m}b_{-m}\right)\log\left(\frac{b}{a}\right)&\geq -\frac{\eta}{2}\,m(m-1)\left|m^2-4m+1\right||a_m|^2\log\left(\frac{b}{a}\right)\\
        &-\frac{1}{2\eta}m(m-1)\left|m^2-4m+1\right||b_{-m}|^2\log\left(\frac{b}{a}\right),
    \end{align*}
    and
    \begin{align*}
        &m(m-1)^2(m^2+2m+3)\Re\left(a_{-m}b_m\right)\log\left(\frac{b}{a}\right)\geq -\xi|b_m|^2b^{4m}\left(1-\left(\frac{a}{b}\right)^{4m}\right)\\
        &-\frac{1}{4\xi}m^2(m-1)^2\left(m^2+2m+3\right)^2|a_{-m}|^2\frac{1}{b^{4m}}\log^2\left(\frac{b}{a}\right)\frac{1}{1-\left(\frac{a}{b}\right)^{4m}}.
    \end{align*}
    Finally, we get the estimate
    \begin{align*}
        &(m-1)^2\left((m-1)^2+1\right)|a_m|^2\log\left(\frac{b}{a}\right)+\frac{m^4+3m^2-2m+2}{4m}|a_{-m}|^2\frac{1}{a^{4m}}\left(1-\left(\frac{a}{b}\right)^{4m}\right)\\
        &-\frac{2(m-1)}{m}\left(m^2-m+1\right)\Re\left(a_ma_{-m}\right)\frac{1}{a^{2m}}\left(1-\left(\frac{a}{b}\right)^{2m}\right)\\
        &+\frac{m(17m^2-18m+5)}{4}|b_m|^2b^{4m}\left(1-\left(\frac{a}{b}\right)^{4m}\right)
        +m^2(m-1)^2|b_{-m}|^2\log\left(\frac{b}{a}\right)\\
        &+2m(m-1)(2m-1)\Re\left(b_mb_{-m}\right)b^{2m}\left(1-\left(\frac{a}{b}\right)^{2m}\right)\\
        &+m(m-1)^2(4m-3)\Re\left(a_m\bar{b_m}\right)b^{2m}\left(1-\left(\frac{a}{b}\right)^{2m}\right)\\
        &-m(3m^2-2m+3)\Re\left(a_{-m}\bar{b_{-m}}\right)\frac{1}{a^{2m}}\left(1-\left(\frac{a}{b}\right)^{2m}\right)\\
        &+m(m-1)\left(m^2-4m+1\right)\Re\left(a_mb_{-m}\right)\log\left(\frac{b}{a}\right)\\
        &+m(m-1)\left(m^2+2m+3\right)\Re\left(a_{-m}b_m\right)\log\left(\frac{b}{a}\right)\\
        &\geq |a_m|^2\log\left(\frac{b}{a}\right)\left\{(m-1)^2\left((m-1)^2+1\right)-\frac{\eta}{2}m(m-1)\left|m^2-4m+1\right|\right.\\
        &\left.-\left(\frac{(m-1)^2\left(m^2-m+1\right)^2}{m^2\epsilon}\frac{\left(1-\left(\frac{a}{b}\right)^{2m}\right)^2}{1-\left(\frac{a}{b}\right)^{4m}}+\frac{m^2(m-1)^4(4m-3)^2}{2\delta}\right)\frac{1}{\log\left(\frac{b}{a}\right)}\right\}\\
        &+\frac{|a_{-m}|^2}{1-\left(\frac{a}{b}\right)^{4m}}\frac{1}{a^{4m}}\left(\left(\frac{m^4+3m^2-2m+2}{4m}-2\epsilon\right)\left(1-\left(\frac{a}{b}\right)^{4m}\right)^2\right.\\
        &\left.-\frac{1}{4\xi}m^2(m-1)^2\left(m^2+2m+3\right)^2\left(\frac{a}{b}\right)^{4m}\log^2\left(\frac{b}{a}\right)\right)\\
        &+\left(\frac{m(17m^2-18m+5)}{4}-2\delta-\xi\right)|b_m|^2\left(1-\left(\frac{a}{b}\right)^{4m}\right)\\
        &+|b_{-m}|^2\log\left(\frac{b}{a}\right)\left\{m^2(m-1)^2-\frac{1}{2\eta}m(m-1)\left|m^2-4m+1\right|\right.\\
        &\left.-\left(\frac{m^2(m-1)^2(2m-1)^2}{\delta}\frac{\left(1-\left(\frac{a}{b}\right)^{2m}\right)^2}{1-\left(\frac{a}{b}\right)^{4m}}+\frac{m^2(3m^2-2m+3)^2}{\epsilon}\frac{\left(1-\left(\frac{a}{b}\right)^{2m}\right)^2}{1-\left(\frac{a}{b}\right)^{4m}}\right)\frac{1}{\log\left(\frac{b}{a}\right)}\right\}.
    \end{align*}
    To get a non-trivial estimate, we must impose the following conditions
    \begin{align*}
        &\frac{\eta}{2}m(m-1)|m^2-4m+1|< (m-1)^2\left((m-1)^2+1\right)\\
        &\epsilon<\frac{m^4+3m^2-2m+2}{8m}\\
        &2\delta+\xi<\frac{m(17m^2-18m+5)}{4}\\
        &\frac{1}{2\eta}m(m-1)\left|m^2-4m+1\right|<m^2(m-1)^2.
    \end{align*}
    Fir $m=2+\sqrt{3}$, the first and last conditions are void, and we can easily solve the remaining equations in $\delta,\epsilon,\xi$. Independently of the value of $m$, we can take
    \begin{align}
        \epsilon&=\frac{m^4+3m^2-2m+2}{16m}\\
        \delta&=\frac{m(17m^2-18m+5)}{16}\\
        \xi&=\frac{m(17m^2-18m+5)}{16}.
    \end{align}
    Notice that we already control $b_m$, but not $b_{-m}$, the first and fourth equation cannot be replaced by milder conditions, even if we take $\delta$ arbitrary large. Now, assume that $m\neq 2+\sqrt{3}$. Then, we can find a solution to the first and fourth inequalities provided that 
    \begin{align*}
        \frac{|m^2-4m+1|}{2m(m-1)}<\frac{2(m-1)\left((m-1)^2+1\right)}{m\left|m^2-4m+1\right|},
    \end{align*}
    which is equivalent to 
    \begin{align*}
        4(m-1)^2\left((m-1)^2+1\right)-\left(m^2-4m+1\right)^2=3m^4-8m^3+10m^2-16m+7>0.
    \end{align*}
    Let
    \begin{align*}
        f(x)=3x^4-8x^3+10x^2-16x+7.
    \end{align*}
    We have
    \begin{align*}
        &f'(x)=12x^3-24x^2+20x-16\\
        &f''(x)=36x^2-48x+20=4(9x^2-12x+5)>0
    \end{align*}
    which implies that $f'$ is strictly increasing. Furthermore, we have 
    \begin{align*}
        f'(2)=12\cdot 8-24\cdot 4+40-16=24>0,
    \end{align*}
    which shows in particular that $f$ is strictly increasing on $[2,\infty[$. However, $f(2)=-1<0$, so we will not get the estimate that we need for $m=2$ (recall that we are mostly interested in the case $m\in \N\setminus\ens{0,1}$). If $m_0=2.039423\cdots$ is the largest real root of $f$ (we gave its exact expression in the statement of the theorem), we deduce that for all $m>m_0$, the above system can be solved in $\eta$ and we simply take 
    \begin{align*}
        \eta=\frac{1}{2}\left(\frac{|m^2-4m+1|}{2m(m-1)}+\frac{2(m-1)\left((m-1)^2+1\right)}{m\left|m^2-4m+1\right|}\right).
    \end{align*}
    We have
    \begin{align*}
        \frac{\eta}{2}m(m-1)\left|m^2-4m+1\right|=\frac{(m^2-4m+1)^2}{8}+\frac{1}{2}(m-1)^2\left((m-1)^2+1\right)
    \end{align*}
    Furthermore, since $3x^2-2x-1\leq 0$ for $0\leq x\leq 1$, we get the estimate
    \begin{align*}
        \frac{\left(1-\left(\frac{a}{b}\right)^2\right)^2}{1-\left(\frac{a}{b}\right)^{4m}}\geq 2.
    \end{align*}
    Therefore, we impose the condition
    \begin{align*}
        &\left(\frac{2(m-1)^2\left(m^2-m+1\right)^2}{m^2\epsilon}+\frac{m^2(m-1)^4(4m-3)^2}{2\delta}\right)\frac{1}{\log\left(\frac{b}{a}\right)}\\
        &\leq \frac{1}{2}\left\{(m-1)^2\left((m-1)^2+1\right)-\frac{\eta}{2}m(m-1)\left|m^2-4m+1\right|\right\},
    \end{align*}
    or
    \small
    \begin{align}\label{lm_log_18}
        \log\left(\frac{b}{a}\right)\geq \frac{128}{4(m-1)^2\left((m-1)^2+1\right)-\left(m^2-4m+1\right)^2}\left(\frac{4(m-1)^2\left(m^2-m+1\right)^2}{m(m^4+3m^2-2m+2)}+\frac{m(m-1)^4(4m-3)^2}{17m^2-18m+5}\right).
    \end{align}
    \normalsize
    Then, imposing that 
    \begin{align}\label{lm_log_19}
        \log\left(\frac{b}{a}\right)\geq \frac{1}{4m}\log\left(\frac{\sqrt{2}}{\sqrt{2}-1}\right),
    \end{align}
    we get
    \begin{align*}
        \left(\frac{m^4+3m^2-2m+2}{4m}-2\epsilon\right)\left(1-\left(\frac{a}{b}\right)^{4m}\right)^2\geq \frac{m^4+3m^2-2m+2}{16m}.
    \end{align*}
    Furthermore, we have
    \begin{align*}
        \frac{1}{4\xi}m^2(m-1)^2\left(m^2+2m+3\right)^2\left(\frac{a}{b}\right)^{4m}\log^2\left(\frac{b}{a}\right)\leq \frac{4m(m-1)^2\left(m^2+2m+3\right)^2}{e^2\left(17m^2-18m+5\right)}\left(\frac{a}{b}\right)^{4m-2}. 
    \end{align*}
    Therefore, we impose the condition 
    \begin{align*}
        \frac{4m(m-1)^2\left(m^2+2m+3\right)^2}{e^2\left(17m^2-18m+5\right)}\left(\frac{a}{b}\right)^{4m-2}\leq \frac{m^4+3m^2-2m+2}{32m},
    \end{align*}
    which yields the condition
    \begin{align}\label{lm_log_20}
        \log\left(\frac{b}{a}\right)\geq \frac{1}{2(2m-1)}\log\left(\frac{64m^2(m-1)^2\left(m^2+2m+3\right)^2}{e^2(17m^2-18m+5)(m^4+3m^2-2m+2)}\right).
    \end{align}
    Then, we have
    \begin{align*}
        &\frac{m^2(m-1)^2(2m-1)^2}{\delta}\frac{\left(1-\left(\frac{a}{b}\right)^{2m}\right)^2}{1-\left(\frac{a}{b}\right)^{4m}}+\frac{m^2(3m^2-2m+3)^2}{\epsilon}\frac{\left(1-\left(\frac{a}{b}\right)^{2m}\right)^2}{1-\left(\frac{a}{b}\right)^{4m}}\\
        &\leq \frac{32m(m-1)^2(2m-1)^2}{17m^2-18m+5}+\frac{32m^3\left(3m^2-2m+3\right)^2}{m^4+3m^2-2m+2},
    \end{align*}
    and since
    \begin{align*}
        &m^2(m-1)^2-\frac{1}{2\eta}m(m-1)\left|m^2-4m+1\right|=\frac{2m^3(m-1)\left(m^2-4m+1\right)^2}{\left(m^2-4m+1\right)^2+2(m-1)^2\left((m-1)^2+1\right)}\\
        &=m^2(m-1)^2\frac{4(m-1)^2\left((m-1)^2+1\right)-\left(m^2-4m+1\right)^2}{4(m-1)^2\left((m-1)^2+1\right)+\left(m^2-4m+1\right)^2},
        %=m^2(m-1)^2\frac{4m^2(m-1)^4\left((m-1)^2+1\right)-m^2(m-1)^2\left(m^2-4m+1\right)^2}{4m^2(m-1)^4\left((m-1)^2+1\right)+m^2(m-1)^2\left(m^2-4m+1\right)^2}.
    \end{align*}
    we also impose the condition 
    \begin{align*}
        &\left(\frac{32m(m-1)^2(2m-1)^2}{17m^2-18m+5}+\frac{32m^3\left(3m^2-2m+3\right)^2}{m^4+3m^2-2m+2}\right)\frac{1}{\log\left(\frac{b}{a}\right)}\\
        &\leq \frac{1}{2}m^2(m-1)^2\frac{4(m-1)^2\left((m-1)^2+1\right)-\left(m^2-4m+1\right)^2}{4(m-1)^2\left((m-1)^2+1\right)+\left(m^2-4m+1\right)^2},
    \end{align*}
    which yields
    \begin{align}\label{lm_log_21}
        \log\left(\frac{b}{a}\right)&\geq \frac{2}{m^2(m-1)^2}\frac{4(m-1)^2\left((m-1)^2+1\right)+\left(m^2-4m+1\right)^2}{4(m-1)^2\left((m-1)^2+1\right)-\left(m^2-4m+1\right)^2}\nonumber\\
        &\times\left(\frac{32m(m-1)^2(2m-1)^2}{17m^2-18m+5}+\frac{32m^3\left(3m^2-2m+3\right)^2}{m^4+3m^2-2m+2}\right).
    \end{align}
    Finally, we deduce that 
    \begin{align}\label{bound_bm}
        &(m-1)^2\left((m-1)^2+1\right)|a_m|^2\log\left(\frac{b}{a}\right)+\frac{m^4+3m^2-2m+2}{4m}|a_{-m}|^2\frac{1}{a^{4m}}\left(1-\left(\frac{a}{b}\right)^{4m}\right)\nonumber\\
        &-\frac{2(m-1)}{m}\left(m^2-m+1\right)\Re\left(a_ma_{-m}\right)\frac{1}{a^{2m}}\left(1-\left(\frac{a}{b}\right)^{2m}\right)\nonumber\\
        &+\frac{m(17m^2-18m+5)}{4}|b_m|^2b^{4m}\left(1-\left(\frac{a}{b}\right)^{4m}\right)
        +m^2(m-1)^2|b_{-m}|^2\log\left(\frac{b}{a}\right)\nonumber\\
        &+2m(m-1)(2m-1)\Re\left(b_mb_{-m}\right)b^{2m}\left(1-\left(\frac{a}{b}\right)^{2m}\right)\nonumber\\
        &+m(m-1)^2(4m-3)\Re\left(a_m\bar{b_m}\right)b^{2m}\left(1-\left(\frac{a}{b}\right)^{2m}\right)\nonumber\\
        &-m(3m^2-2m+3)\Re\left(a_{-m}\bar{b_{-m}}\right)\frac{1}{a^{2m}}\left(1-\left(\frac{a}{b}\right)^{2m}\right)\nonumber\\
        &+m(m-1)\left(m^2-4m+1\right)\Re\left(a_mb_{-m}\right)\log\left(\frac{b}{a}\right)\nonumber\\
        &+m(m-1)\left(m^2+2m+3\right)\Re\left(a_{-m}b_m\right)\log\left(\frac{b}{a}\right)\nonumber\\
        &\geq \frac{1}{16}\left(4(m-1)^2\left((m-1)^2+1\right)-\left(m^2-4m+1\right)^2\right)|a_m|^2\log\left(\frac{b}{a}\right)\nonumber\\
        &+\frac{m^4+3m^2-2m+2}{32m}|a_{-m}|^2\frac{1}{a^{4m}}+\frac{m(17m^2-18m+5)}{32}|b_m|^2\nonumber\\
        &+\frac{1}{2}m^2(m-1)^2\frac{4(m-1)^2\left((m-1)^2+1\right)-\left(m^2-4m+1\right)^2}{4(m-1)^2\left((m-1)^2+1\right)+\left(m^2-4m+1\right)^2}|b_{-m}|^2\log\left(\frac{b}{a}\right).
    \end{align}
    \textbf{Step 8.} The estimate of radial components.

    Finally, we need to estimate $b_0$ and the logarithm terms.  Notice that thanks to Parseval's lemma, we have
    \begin{align*}
        \int_{\Omega}\left(\leb_m\mathrm{Rad}(u)\right)^2dx\leq \int_{\Omega}\left(\leb_mu\right)^2dx.
    \end{align*}
    Likewise, we have
    \begin{align*}
        \int_{\Omega}\left|\mathfrak{L}_m\mathrm{Rad}(u)\right|^2|dz|^2\leq \int_{\Omega}\left|\mathfrak{L}_mu\right|^2|dz|^2. 
    \end{align*}
    Now, we have
    \begin{align*}
        \mathrm{Rad}(u)=\alpha_1|z|^{1-m}\log|z|+2\,a_0|z|^{1-m}+\alpha_2|z|^{m+1}\log|z|+2\,b_0|z|^{m+1}.
    \end{align*}
    Recall that $\alpha_1|z|^{1-m}\log|z|+2\,a_0|z|^{1-m}\in \mathrm{Ker}\left(\leb_m\right)$. Then, we have
    \begin{align*}
        &\p{r}\left(r^{m+1}\log(r)\right)=(m+1)r^m\log(r)+r^m\\
        &\p{r}^2\left(r^{m+1}\log(r)\right)=m(m+1)r^{m-1}\log(r)+(2m+1)r^{m-1}.
    \end{align*}
    On the other hand, we have
    \begin{align*}
        \leb_m&=\Delta+2(m-1)\frac{x}{|x|^2}\cdot \D+\frac{(m-1)^2}{|x|^2}\\
        &=\p{r}^2+\frac{2m-1}{r}\p{r}+\frac{(m-1)^2}{r^2}+\frac{1}{r^2}\p{\theta}^2.
    \end{align*}
    Therefore, we have
    \begin{align*}
        &\leb_m\mathrm{Rad}(u)=m(m+1)\alpha_2r^{m-1}\log(r)+(2m+1)\alpha_2r^{m-1}+(2m-1)(m+1)\alpha_2r^{m-1}\log(r)\\
        &+(2m-1)\alpha_2r^m
        +(m-1)^2\alpha_2r^{m-1}\log(r)+2m(m+1)b_0r^{m-1}+2(2m-1)(m+1)b_0r^{m-1}\\
        &+2(m-1)^2b_0r^{m-1}\\
        &=4m^2\alpha_2r^{m-1}\log(r)+4m\alpha_2r^{m-1}+8m^2b_0r^{m-1}\\
        &=4m^2\,r^{m-1}\left(\alpha_2\log(r)+\frac{1}{m}\alpha_2+2b_0\right).
    \end{align*}
    Then, we have
    \begin{align*}
        \p{z}\mathrm{Rad}(u)&=\frac{1-m}{2}\alpha_1\,\z|z|^{-m-1}\log|z|+\frac{\alpha_1}{2}\z|z|^{-m-1}+(1-m)a_0\z|z|^{-1-m}\\
        &+\frac{m+1}{2}\alpha_2\z|z|^{m-1}\log|z|+\frac{\alpha_2}{2}\z|z|^{m-1}+(m+1)b_0\z|z|^{m-1}\\
        \p{z}^2\mathrm{Rad}(u)&=\frac{m^2-1}{4}\alpha_1\,\z^2|z|^{-m-3}\log|z|-\frac{m}{2}\alpha_1\z^2|z|^{-m-3}+\frac{m^2-1}{2}a_0\z^3|z|^{-m-3}\\
        &+\frac{m^2-1}{4}\alpha_2\z^2|z|^{m-3}\log|z|+\frac{m}{2}\alpha_2\z^2|z|^{m-1}+\frac{m^2-1}{2}b_0\z^2|z|^{m-3}.
    \end{align*}
    Therefore, we get
    \begin{align*}
       \mathfrak{L}_m\mathrm{Rad}(u)&=\frac{1}{z^2}\left(|z|^{1-m}\left(\frac{m^2-1}{4}\alpha_1\log|z|-\frac{m}{2}\alpha_1+\frac{m^2-1}{2}a_0\right.\right.\\
        &\left.-\frac{m^2-1}{2}\alpha_1\log|z|+\frac{m-1}{2}\alpha_1-(m^2-1)a_0+\frac{m^2-4m+3}{4}\alpha_1\log|z|+\frac{m^2-4m+3}{2}a_0\right)\\
        &+|z|^{m+1}\left(\frac{m^2-1}{4}\alpha_2\log|z|+\frac{m}{2}\alpha_2+\frac{m^2-1}{4}b_0+\frac{m^2-1}{2}\alpha_2\log|z|+\frac{m-1}{2}\alpha_2+(m^2-1)b_0\right.\\
        &\left.\left.+\frac{m^2-4m+3}{4}\alpha_2\log|z|+\frac{m^2-4m+3}{2}b_0\right)\right)\\
        &=\frac{1}{z^2}|z|^{1-m}\left(-(m-1)\alpha_1\log|z|-\frac{1}{2}\alpha_1-2(m-1)a_0\right)\\
        &+\frac{1}{z^2}|z|^{m+1}\left(m(m-1)\alpha_2\log|z|+\frac{2m-1}{2}\alpha_2+2m(m-1)b_0\right).
    \end{align*}
    Then, we have
    \begin{align*}
        &\p{r}u+\frac{(m-1)}{r}u=(1-m)\alpha_1r^{-m}\log(r)+\alpha_1r^{-m}+2(1-m)a_0r^{-m}+(m+1)\alpha_2r^m\log(r)+\alpha_2r^m\\
        &+2(m+1)b_0r^{m}
        +(m-1)\alpha_1r^{-m}+2(m-1)a_0r^{-m}+(m-1)\alpha_2r^m\log(r)+2(m-1)b_0r^m\\
        &=\alpha_1r^{-m}+2m\,r^{m}\left(\alpha_2\log(r)+\frac{1}{2m}\alpha_2+2b_0\right).
    \end{align*}
    Now, we have
    \begin{align*}
        &4m\frac{z^2}{|z|^2}\mathfrak{L}_m\mathrm{Rad}(u)-(m-1)\leb_m\mathrm{Rad}(u)=-4m(m-1)r^{-m-1}\left(\alpha_1\log(r)+\frac{\alpha_1}{2(m-1)}+2a_0\right)\\
        &+r^{m-1}\left(4m^2(m-1)\alpha_2\log(r)+2(m-1)(2m-1)\alpha_2+8m(m-1)b_0\right)\\
        &-r^{m-1}\left(4m^2(m-1)\alpha_2\log(r)+4m(m-1)\alpha_2+8m^2(m-1)b_0\right)\\
        &=-4m(m-1)r^{-m-1}\left(\alpha_1\log(r)+\frac{\alpha_1}{2(m-1)}+2a_0\right)-2(m-1)\alpha_2r^{m-1}.
    \end{align*}
    Therefore, we have
    \begin{align*}
        \frac{1}{2}\leb_m\mathrm{Rad}(u)-\frac{2m}{m-1}\frac{z^2}{|z|^2}\mathfrak{L}_m\mathrm{Rad}(u)=\alpha_2r^{m-1}+2m\,r^{-m-1}\left(\alpha_1\log(r)+\frac{\alpha_1}{2(m-1)}+2a_0\right).
    \end{align*}
    Then, we have for $\alpha\neq -1$
    \begin{align*}
        \int r^{\alpha}\log(r)dr&=\frac{r^{\alpha+1}}{\alpha+1}\log(r)-\frac{1}{\alpha+1}\int r^{\alpha}dr=\frac{r^{\alpha+1}}{\alpha+1}\log(r)-\frac{1}{(\alpha+1)^2}r^{\alpha+1}\\
        \int r^{\alpha}\log^2(r)dr%&=\frac{r^{\alpha+1}}{\alpha+1}\log^2(r)-\frac{2}{\alpha+1}\int r^{\alpha}\log(r)dr\\
        &=\frac{r^{\alpha+1}}{\alpha+1}\log^2(r)-\frac{2}{(\alpha+1)^2}r^{\alpha+1}\log(r)+\frac{2}{(\alpha+1)^3}r^{\alpha+1},
    \end{align*}
    and for all $\beta>-1$
    \begin{align*}
        \int \frac{\log^{\beta}(r)}{r}dr=\frac{1}{\beta+1}\log^{\beta+1}(r).
    \end{align*}
    In particular, for $\alpha=2\gamma-1$, we have
    \begin{align*}
        \int r^{2\gamma-1}\log(r)dr&=-\frac{1}{2\gamma}\left(r^{2\gamma}\log\left(\frac{1}{r}\right)+\frac{1}{2\gamma}r^{2\gamma}\right)\\
        \int r^{2\gamma-1}\log^2(r)dr%&=\frac{r^{\alpha+1}}{\alpha+1}\log^2(r)-\frac{2}{\alpha+1}\int r^{\alpha}\log(r)dr\\
        &=\frac{1}{2\gamma}\left(r^{2\gamma}\log^2\left(\frac{1}{r}\right)+\frac{1}{\gamma}r^{2\gamma}\log\left(\frac{1}{r}\right)+\frac{1}{2\gamma^2}r^{2\gamma}\right),
    \end{align*}
    We deduce that
    \begin{align*}
        &\int_{\Omega}\left(\frac{1}{2}\leb_m\mathrm{Rad}(u)-\frac{2m}{m-1}\frac{z^2}{|z|^2}\mathfrak{L}_m\mathrm{Rad}(u)\right)^2dx=2\pi\int_{a}^b\left(\alpha_2^2r^{2m-1}+4m^2r^{-2m-1}\log^2(r)\right.\\
        &\left.+4m^2\left(\frac{\alpha_1}{2(m-1)}+2a_0\right)^2r^{-2m-1}+8m^2\left(\frac{\alpha_1}{2(m-1)}+2a_0\right)\alpha_1\,r^{-2m-1}\log(r)\right.\\
        &\left.+4m\,\alpha_1\alpha_2r^{-1}\log(r)+4m\left(\frac{\alpha_1}{2(m-1)}+2a_0\right)\alpha_2r^{-1}\right)dr\\
        &=2\pi\left\{\frac{\alpha_2^2}{2m}b^{2m}\left(1-\left(\frac{a}{b}\right)^{2m}\right)+2m\,\alpha_1^2\left(\frac{1}{a^{2m}}\log^2\left(\frac{1}{a}\right)-\frac{1}{m}\frac{1}{a^{2m}}\log\left(\frac{1}{a}\right)+\frac{1}{2m^2}\frac{1}{a^{2m}}\right.\right.\\
        &\left.-\left(\frac{1}{b^{2m}}\log^2\left(\frac{1}{b}\right)-\frac{1}{m}\frac{1}{b^{2m}}\log\left(\frac{1}{b}\right)+\frac{1}{2m^2}\frac{1}{b^{2m}}\right)\right)+2m\left(\frac{\alpha_1}{2(m-1)}+2a_0\right)^2\frac{1}{a^{2m}}\left(1-\left(\frac{a}{b}\right)^{2m}\right)\\
        &-4m\left(\frac{\alpha_1}{2(m-1)}+2a_0\right)\alpha_1\left(\frac{1}{a^{2m}}\log\left(\frac{1}{a}\right)-\frac{1}{2m}\frac{1}{a^{2m}}-\left(\frac{1}{b^{2m}}\log\left(\frac{1}{b}\right)-\frac{1}{2m}\frac{1}{b^{2m}}\right)\right)\\
        &\left.+2m\,\alpha_1\alpha_2\left(\log^2(b)-\log^2(a)\right)+2m\left(\frac{\alpha_1}{2(m-1)}+2a_0\right)\alpha_2\log\left(\frac{b}{a}\right)\right\}.
    \end{align*}
    Likewise, we get
    \begin{align*}
        &\int_{\Omega}\left(\leb_m\mathrm{Rad}(u)\right)^2dx\\
        &=32\pi\,m^2\int_{a}^b\left(\alpha_2^2r^{2m-1}\log^2(r)+\left(\frac{1}{m}\alpha_2+2b_0\right)^2r^{2m-1}+2\left(\frac{1}{m}\alpha_2+2b_0\right)\alpha_2\,r^{2m-1}\log(r)\right)dr\\
        &=32\pi\,m^4\left\{\frac{1}{2m}\alpha_2^2\left(b^{2m}\log^2\left(\frac{1}{b}\right)+\frac{1}{m}b^{2m}\log\left(\frac{1}{b}\right)+\frac{1}{2m^2}b^{2m}\right.\right.\\
        &\left.-\left(a^{2m}\log^2\left(\frac{1}{a}\right)+\frac{1}{m}a^{2m}\log\left(\frac{1}{a}\right)+\frac{1}{2m^2}a^{2m}\right)\right)
        +\frac{1}{2m}\left(\frac{1}{m}\alpha_2+2b_0\right)^{2}b^{2m}\left(1-\left(\frac{a}{b}\right)^{2m}\right)\\
        &\left.-\frac{1}{m}\left(\frac{1}{m}\alpha_2+2b_0\right)\alpha_2\left(b^{2m}\log\left(\frac{1}{b}\right)+\frac{1}{2m}b^{2m}-\left(a^{2m}\log\left(\frac{1}{a}\right)+\frac{1}{2m}a^{2m}\right)\right)\right\}.
    \end{align*}
    Now, we have
    \begin{align*}
        &\mathrm{Rad}(u)^2=\left(\alpha_1r^{1-m}\log(r)+2a_0r^{1-m}+\alpha_2r^{m+1}\log(r)+2b_0r^{m+1}\right)^2\\
        &=\alpha_1^2r^{2-2m}\log^2(r)+4a_0^2r^{2-2m}+\alpha_2^2r^{2m+2}\log^2(r)+4b_0^2r^{2m+2}+4\alpha_1a_0r^{2-2m}\log(r)
        +4\alpha_2b_0r^{2m+2}\log(r)\\
        &+2\alpha_1\alpha_2r^{2}\log^2(r)
        +4(\alpha_1b_0+\alpha_2a_0)r^2\log(r)+8a_0b_0r^2,
    \end{align*}
    which yields for all $1/2<\beta<m/2$
    \begin{align*}
        &\int_{\Omega}\frac{\mathrm{Rad}(u)^2}{|x|^4}\left(\left(\frac{|x|}{b}\right)^{4\beta}+\left(\frac{a}{|x|}\right)^{4\beta}\right)dx
        =\frac{2\pi}{b^{4\beta}}\int_{a}^b\left(\alpha_1^2r^{-2m-1+4\beta}\log^2(r)+4a_0^2r^{-2m-1+4\beta}\right.\\
        &\left.+\alpha_2^2r^{2m-1+4\beta}\log^2(r)+4b_0^2r^{2m-1+4\beta}+4\alpha_1a_0r^{-2m-1+4\beta}\log(r)+2\alpha_1\alpha_2r^{-1+4\beta}\log^2(r)\right.\\
        &\left.+4(\alpha_1b_0+\alpha_2a_0)r^{-1+4\beta}\log(r)+8a_0b_0r^{-1+4\beta}\right)dr\\
        &+2\pi\,a^{4\beta}\int_{a}^b\left(\alpha_1^2r^{-2m-1-4\beta}\log^2(r)+4a_0^2r^{-2m-1-4\beta}+\alpha_2^2r^{2m-1-4\beta}\log^2(r)+4b_0^2r^{2m-1-4\beta}\right.\\
        &\left.+4\alpha_1a_0r^{-2m-1-4\beta}\log(r)+2\alpha_1\alpha_2r^{-1-4\beta}\log(r)+4(\alpha_1b_0+\alpha_2a_0)r^{-1-4\beta}\log^2(r)+8a_0b_0r^{-1-4\beta}\right)dr\\
        &=2\pi\bigg\{\frac{\alpha_1^2}{2m-4\beta}\left(\left(\frac{a}{b}\right)^{4\beta}\left(\frac{1}{a^{2m}}\log^2\left(\frac{1}{a}\right)-\frac{1}{m-2\beta}\frac{1}{a^{2m}}\log\left(\frac{1}{a}\right)+\frac{1}{2(m-2\beta)^2}\frac{1}{a^{2m}}\right)\right.\\
        &\left.-\left(\frac{1}{b^{2m}}\log^2\left(\frac{1}{b}\right)-\frac{1}{m-2\beta}\frac{1}{b^{2m}}\log\left(\frac{1}{b}\right)+\frac{1}{2(m-2\beta)^2}\frac{1}{b^{2m}}\right)\right)\\
        &+\frac{2}{m-2\beta}a_0^2\left(\frac{a}{b}\right)^{4\beta}\frac{1}{a^{4m}}\left(1-\left(\frac{a}{b}\right)^{4m-2\beta}\right)\\
        &+\frac{\alpha_2^2}{2m+4\beta}\left(b^{2m}\log^2\left(\frac{1}{b}\right)+\frac{1}{m+2\beta}b^{2m}\log\left(\frac{1}{b}\right)+\frac{1}{2(m+2\beta)^2}b^{2m}\right.\\
        &\left.-\left(\frac{a}{b}\right)^{4\beta}\left(a^{2m}\log^2\left(\frac{1}{a}\right)+\frac{1}{m+2\beta}a^{2m}\log\left(\frac{1}{a}\right)+\frac{1}{2(m+2\beta)}a^{2m}\right)\right)\\
        &+\frac{2}{m+2\beta}b_0^2b^{2m}\left(1-\left(\frac{a}{b}\right)^{2m+4\beta}\right)-\frac{2\alpha_1a_0}{m-2\beta}\left(\left(\frac{a}{b}\right)^{4\beta}\left(\frac{1}{a^{2m}}\log\left(\frac{1}{a}\right)-\frac{1}{2m-4\beta}\frac{1}{a^{2m}}\right)\right.\\
        &\left.-\left(\frac{1}{b^{2m}}\log\left(\frac{1}{b}\right)-\frac{1}{2m-4\beta}\frac{1}{b^{2m}}\right)\right)\\
        &-\frac{2\alpha_2b_0}{m+2\beta}\left(b^{2m}\log\left(\frac{1}{b}\right)+\frac{1}{2m+4\beta}b^{2m}-\left(\frac{a}{b}\right)^{4\beta}\left(a^{2m}\log\left(\frac{1}{a}\right)+\frac{1}{2m+4\beta}a^{2m}\right)\right)\\
        &+\frac{\alpha_1\alpha_2}{4\beta}\left(\log^2\left(\frac{1}{b}\right)+\frac{1}{2\beta}\log\left(\frac{1}{b}\right)+\frac{1}{8\beta^2}-\left(\frac{a}{b}\right)^{4\beta}\left(\log^2\left(\frac{1}{a}\right)+\frac{1}{2\beta}\log\left(\frac{1}{a}\right)+\frac{1}{8\beta^2}\right)\right)\\
        &-\frac{\alpha_1b_0+\alpha_2a_0}{\beta}\left(\log\left(\frac{1}{b}\right)+\frac{1}{8\beta}-\left(\frac{a}{b}\right)^{4\beta}\left(\log\left(\frac{1}{a}\right)+\frac{1}{8\beta}\right)\right)+\frac{2a_0b_0}{\beta}\left(1-\left(\frac{a}{b}\right)^{4\beta}\right)\\
        &+\frac{\alpha_1^2}{2m+4\beta}\left(\frac{1}{a^{2m}}\log^2\left(\frac{1}{a}\right)-\frac{1}{m+2\beta}\frac{1}{a^{2m}}\log\left(\frac{1}{a}\right)+\frac{1}{2(m+2\beta)^2}\frac{1}{a^{2m}}\right.\\
        &\left.-\left(\frac{a}{b}\right)^{4\beta}\left(\frac{1}{b^{2m}}\log^2\left(\frac{1}{b}\right)\right)-\frac{1}{m+2\beta}\log\left(\frac{1}{b}\right)+\frac{1}{2(m+2\beta)}\frac{1}{b^{2m}}\right)\\
        &+\frac{2}{m+2\beta}a_0^2\frac{1}{a^{2m}}\left(1-\left(\frac{a}{b}\right)^{2m+4\beta}\right)\\
        &+\frac{\alpha_2^2}{m-2\beta}\left(\left(\frac{a}{b}\right)^{4\beta}\left(b^{2m}\log^2\left(\frac{1}{b}\right)+\frac{1}{m-2\beta}b^{2m}\log\left(\frac{1}{b}\right)+\frac{1}{2(m-2\beta)^2}b^{2m}\right)\right.\\
        &\left.-\left(a^{2m}\log^2\left(\frac{1}{a}\right)+\frac{1}{m-2\beta}a^{2m}\log\left(\frac{1}{a}\right)\right)+\frac{1}{2(m-2\beta)^2}a^{2m}\right)\\
        &+\frac{2}{m-2\beta}\left(\frac{a}{b}\right)^{4\beta}b^{2m}\left(1-\left(\frac{a}{b}\right)^{2m-4\beta}\right)-\frac{2\alpha_1a_0}{m+2\beta}\left(\frac{1}{a^{2m}}\log\left(\frac{1}{a}\right)-\frac{1}{2m+4\beta}\frac{1}{a^{2m}}\right.\\
        &\left.-\left(\frac{a}{b}\right)^{4\beta}\left(\frac{1}{b^{2m}}\log\left(\frac{1}{b}\right)-\frac{1}{2m+4\beta}\frac{1}{b^{2m}}\right)\right)-\frac{\alpha_1\alpha_2}{2\beta}\left(\log\left(\frac{1}{a}\right)-\frac{1}{4\beta}-\left(\frac{a}{b}\right)^{4\beta}\left(\log\left(\frac{1}{b}\right)-\frac{1}{4\beta}\right)\right)\\
        &+\frac{\alpha_1b_0+\alpha_2a_0}{\beta}\left(\log^2\left(\frac{1}{a}\right)-\frac{1}{2\beta}\log\left(\frac{1}{a}\right)+\frac{1}{8\beta^2}-\left(\frac{a}{b}\right)^{4\beta}\left(\log^2\left(\frac{1}{b}\right)-\frac{1}{4\beta}\log\left(\frac{1}{b}\right)+\frac{1}{8\beta^2}\right)\right)\\
        &+\frac{2a_0b_0}{\beta}\left(1-\left(\frac{a}{b}\right)^{4\beta}\right)\bigg\}.
    \end{align*}
    Finally, we have 
    \begin{align*}
        &\frac{1}{r^2}\left|\p{r}\mathrm{Rad}(u)+\frac{m-1}{r}\mathrm{Rad}(u)\right|^2=\alpha_1r^{-2m-2}\\
        &+4m^2\,r^{2m-2}\left(\alpha_2^2\log^2(r)+\left(\frac{1}{2m}\alpha_2+2b_0\right)^2+2\left(\frac{1}{2m}\alpha_2+2b_0\right)\alpha_2\log(r)\right)\\
        &+4m\alpha_1\alpha_2r^{-2}\log(r)+4m\alpha_1\left(\frac{1}{2m}\alpha_2+2b_0\right)r^{-2}.
    \end{align*}
    We deduce that
    \begin{align*}
        &\int_{\Omega}\frac{1}{|x|^2}\left|\p{r}\mathrm{Rad}(u)+\frac{m-1}{r}\mathrm{Rad}(u)\right|^2\left(\left(\frac{|x|}{b}\right)^{2\gamma}+\left(\frac{a}{|x|}\right)^{2\gamma}\right)dx\\
        &=\frac{2\pi}{b^{2\gamma}}\int_{a}^b\bigg\{\alpha_1^2r^{-2m-1+2\gamma}+4m^2\alpha_2^2r^{2m-1+2\gamma}\log^2(r)+4m^2\left(\frac{1}{2m}\alpha_2+2b_0\right)^2r^{2m-1+2\gamma}\\
        &+8m^2\left(\frac{1}{2m}\alpha_2+2b_0\right)\alpha_2r^{2m-1+2\gamma}\log(r)+4m\,\alpha_1\alpha_2r^{-1+2\gamma}\log(r)+4m\,\alpha_1\left(\frac{1}{2m}\alpha_2+2b_0\right)r^{-1+2\gamma}\bigg\}dr\\
        &+2\pi\,a^{2\gamma}\int_{a}^b\bigg\{\alpha_1^2r^{-2m-1-2\gamma}+4m^2\alpha_2^2r^{2m-1-2\gamma}\log^2(r)+4m^2\left(\frac{1}{2m}\alpha_2+2b_0\right)^2r^{2m-1-2\gamma}\\
        &+8m^2\left(\frac{1}{2m}\alpha_2+2b_0\right)\alpha_2\,r^{2m-1-2\gamma}\log(r)+4m\,\alpha_1\alpha_2r^{-1-2\gamma}\log(r)+4m\,\alpha_1\left(\frac{1}{2m}\alpha_2+2b_0\right)r^{-1-2\gamma}\bigg\}dr\\
        &=2\pi\bigg\{\frac{\alpha_1^2}{2(m-\gamma)}\left(\frac{a}{b}\right)^{2\gamma}\frac{1}{a^{2m}}\left(1-\left(\frac{a}{b}\right)^{2(m-\gamma)}\right)\\
        &+\frac{2m^2}{m+\gamma}\left(b^{2m}\log^2\left(\frac{1}{b}\right)+\frac{1}{m+\gamma}b^{2m}\log\left(\frac{1}{b}\right)+\frac{1}{2(m+\gamma)^2}b^{2m}\right.\\
        &\left.-\left(\frac{a}{b}\right)^{2\gamma}\left(a^{2m}\log^2\left(\frac{1}{a}\right)+\frac{1}{m+\gamma}a^{2m}\log\left(\frac{1}{a}\right)+\frac{1}{2(m+\gamma)^2}a^{2m}\right)\right)\\
        &+\frac{2m^2}{m+\gamma}\left(\frac{1}{m}\alpha_2+2b_0\right)^2b^{2m}\left(1-\left(\frac{a}{b}\right)^{2(m+\gamma)}\right)\\
        &-\frac{4m^2}{m+\gamma}\left(\frac{1}{2m}\alpha_2+2b_0\right)\alpha_2\left(b^{2m}\log\left(\frac{1}{b}\right)-\frac{1}{2(m+\gamma)}b^{2m}-\left(\frac{a}{b}\right)^{2\gamma}\left(a^{2m}\log\left(\frac{1}{a}\right)-\frac{1}{2(m+\gamma)}a^{2m}\right)\right)\\
        &-\frac{2m}{\gamma}\alpha_1\alpha_2\left(\log\left(\frac{1}{b}\right)-\frac{1}{2\gamma}-\left(\frac{a}{b}\right)^{2\gamma}\left(\log\left(\frac{1}{a}\right)-\frac{1}{2\gamma}\right)\right)+\frac{2m}{\gamma}\alpha_1\left(\frac{1}{2m}\alpha_2+2b_0\right)\left(1-\left(\frac{a}{b}\right)^{2\gamma}\right)\\
        &+\frac{\alpha_1^2}{2(m+\gamma)}\frac{1}{a^{2m}}\left(1-\left(\frac{a}{b}\right)^{2(m+\gamma)}\right)\\
        &+\frac{2m^2}{m-\gamma}\left(\left(\frac{a}{b}\right)^{2\gamma}\left(b^{2m}\log^2\left(\frac{1}{b}\right)+\frac{1}{m-\gamma}b^{2m}\log\left(\frac{1}{b}\right)+\frac{1}{2(m-\gamma)^2}b^{2m}\right)\right.\\
        &\left.-\left(a^{2m}\log^2\left(\frac{1}{a}\right)+\frac{1}{m-\gamma}a^{2m}\log\left(\frac{1}{a}\right)+\frac{1}{2(m-\gamma)^2}a^{2m}\right)\right)\\
        &+\frac{2m^2}{m-\gamma}\left(\frac{1}{2m}\alpha_2+2b_0\right)^2\left(\frac{a}{b}\right)^{2\gamma}b^{2m}\left(1-\left(\frac{a}{b}\right)^{2(m-\gamma)}\right)\\
        &-\frac{4m^2}{m-\gamma}\left(\frac{1}{m}\alpha_2+2b_0\right)\alpha_2\left(\left(\frac{a}{b}\right)^{2\gamma}\left(b^{2m}\log\left(\frac{1}{b}\right)-\frac{1}{2(m-\gamma)}b^{2m}\right)-\left(a^{2m}\log\left(\frac{1}{a}\right)-\frac{1}{2(m-\gamma)}a^{2m}\right)\right)\\
        &-\frac{2m}{\gamma}\alpha_1\alpha_2\left(\log\left(\frac{1}{a}\right)-\frac{1}{2\gamma}-\left(\frac{a}{b}\right)^{2\gamma}\left(\log\left(\frac{1}{b}\right)-\frac{1}{2\gamma}\right)\right)+\frac{2m}{\gamma}\alpha_1\left(\frac{1}{2m}\alpha_2+2b_0\right)\left(1-\left(\frac{a}{b}\right)^{2\gamma}\right)\bigg\}.
    \end{align*}

    Recalling that
    \begin{align*}
        \mathrm{Rad}(u)=\alpha_1r^{1-m}\log(r)+2a_0r^{1-m}+\alpha_2r^{m+1}\log(r)+2b_0r^{m+1}
    \end{align*}
    and
    \begin{align}\label{identity_lm}
        \leb_m\mathrm{Rad}(u)=4m^2r^{m-1}\left(\alpha_2\log(r)+\dfrac{1}{m}\alpha_2+2b_0\right),
    \end{align}
    we apply the previous estimate \eqref{improved_cauchy_schwarz2} to $f_1=r^{m-1}\log(r)$ $f_2=\left(\dfrac{1}{m}\alpha_2+2b_0\right)r^{m-1}$, and $(\lambda_1,\lambda_2)=\left(\alpha_2,\dfrac{1}{m}\alpha_2+2b_0\right)$. Notice that we can assume without generality (otherwise, the estimate would be trivial as the integral of a sum of positive functions) that
    \begin{align*}
        \lambda_1\lambda_2=\left(\frac{1}{m}\alpha_2+2b_0\right)\alpha_2>0.
    \end{align*}
    Thanks to \eqref{rest_cs}, we have
    \begin{align*}
        \int_{\Omega}|f_1(x)f_2(x)-f_1(y)f_2(y)|^2dx\,dy&=2\pi\int_{a}^{b}\int_{a}^br^{2m-1}s^{2m-1}\log^2\left(\frac{r}{s}\right)dr\,ds\\
        &=\frac{\pi}{4m^4}b^{4m}\left\{\left(1-\left(\frac{a}{b}\right)^{2m}\right)^2-4m^2\left(\frac{a}{b}\right)^{2m}\log^2\left(\frac{b}{a}\right)\right\}.
    \end{align*}
    On the other hand, we have
    \begin{align*}
        &\int_{\Omega}|f_1|^2dx=2\pi\int_{a}^br^{2m-1}\log^2(r)dr\\
        &=\frac{\pi}{m}\left(b^{2m}\log^2\left(\frac{1}{b}\right)-a^{2m}\log^2\left(\frac{1}{a}\right)+\frac{1}{m}\left(b^{2m}\log\left(\frac{1}{b}\right)-a^{2m}\log\left(\frac{1}{a}\right)\right)+\frac{1}{2m^2}b^{2m}\left(1-\left(\frac{a}{b}\right)^{2m}\right)\right),
    \end{align*}
    while
    \begin{align*}
        \int_{\Omega}|f_2|^2dx=2\pi\int_{a}^br^{2m-1}dr=\frac{\pi}{m}b^{2m}\left(1-\left(\frac{a}{b}\right)^{2m}\right).
    \end{align*}
    Then, assuming that $b\leq e^{-1}$, we trivially have
    \begin{align*}
        \int_{\Omega}|f_1|^2dx\leq \frac{\pi}{m}\left(1+\frac{1}{m}+\frac{1}{2m^2}\right)b^{2m}\log^2\left(\frac{1}{b}\right),
    \end{align*}
    which yields the estimate
    \begin{align*}
        \left(\int_{\Omega}|f_1|^2dx\right)\left(\int_{\Omega}|f_2|^2dx\right)\leq \frac{\pi^2}{2m^2}\left(1+\frac{1}{m}+\frac{1}{2m^2}\right)b^{4m}\log^2\left(\frac{1}{b}\right)\leq \frac{2\pi^2}{m^2}b^{4m}\log^2\left(\frac{1}{b}\right).
    \end{align*}
    Therefore, we get
    \begin{align*}
        &\frac{1}{4m^3\sqrt{2}}\left(\frac{1}{m}\alpha_2+2b_0\right)\alpha_2\frac{b^{2m}}{\log\left(\frac{1}{b}\right)}\left\{\left(1-\left(\frac{a}{b}\right)^{2m}\right)^2-4m^2\left(\frac{a}{b}\right)^{2m}\log^2\left(\frac{b}{a}\right)\right\}\\
        &\leq \int_{\Omega}\left(\alpha_2\log(r)+\frac{1}{2}\alpha_2+2b_0\right)^2r^{2m-1}dr=\frac{1}{16\pi m^4}\int_{\Omega}\left(\leb_m\mathrm{Rad}(u)\right)^2dx,
    \end{align*}
    which implies that 
    \begin{align*}
        \left(\frac{1}{m}\alpha_2+2b_0\right)\alpha_2\frac{b^{2m}}{\log\left(\frac{1}{b}\right)}\leq \frac{1}{2m\pi\sqrt{2}}\frac{1}{\left(1-\left(\frac{a}{b}\right)^{2m}\right)^2-4m^2\left(\frac{a}{b}\right)^{2m}\log^2\left(\frac{b}{a}\right)}\int_{\Omega}\left(\leb_m\mathrm{Rad}(u)\right)^2dx.
    \end{align*}
    Therefore, using the identity 
    \begin{align*}
        &\frac{1}{16\pi m^3}\int_{\Omega}\left(\leb_m\mathrm{Rad}(u)\right)^2dx=\alpha_2^2\left(b^{2m}\log^2\left(\frac{1}{b}\right)-a^{2m}\log^2\left(\frac{1}{a}\right)+\frac{1}{m}\left(b^{2m}\log\left(\frac{1}{b}\right)-a^{2m}\log\left(\frac{1}{a}\right)\right)\right.\\
        &\left.+\frac{1}{2m^2}b^{2m}\left(1-\left(\frac{a}{b}\right)^{2m}\right)\right)
        +\left(\frac{1}{m}\alpha_2+2b_0\right)^{2}b^{2m}\left(1-\left(\frac{a}{b}\right)^{2m}\right)\\
        &-2\left(\frac{1}{m}\alpha_2+2b_0\right)\alpha_2\left(b^{2m}\log\left(\frac{1}{b}\right)+\frac{1}{2m}b^{2m}-\left(a^{2m}\log\left(\frac{1}{a}\right)+\frac{1}{2m}a^{2m}\right)\right),
    \end{align*}
    we deduce that 
    \begin{align*}
        &\alpha_2^2\left(b^{2m}\log^2\left(\frac{1}{b}\right)-a^{2m}\log^2\left(\frac{1}{a}\right)+\frac{1}{m}\left(b^{2m}\log\left(\frac{1}{b}\right)-a^{2m}\log\left(\frac{1}{a}\right)\right)\right.\\
        &\left.+\frac{1}{2m^2}b^{2m}\left(1-\left(\frac{a}{b}\right)^{2m}\right)\right)
        +\left(\frac{1}{m}\alpha_2+2b_0\right)^{2}b^{2m}\left(1-\left(\frac{a}{b}\right)^{2m}\right)\\
        &\leq \left(\frac{1}{16\pi m^3}+\frac{\sqrt{2}}{\pi\,m}\frac{\log^2\left(\frac{1}{b}\right)}{\left(1-\left(\frac{a}{b}\right)^{2m}\right)^2-4m^2\left(\frac{a}{b}\right)^{2m}\log^2\left(\frac{b}{a}\right)}\right)\int_{\Omega}\left(\leb_m\mathrm{Rad}(u)\right)^2dx,
    \end{align*}
    which implies that 
    \begin{align}\label{radial_estimate_alpha2}
        \alpha_2^2b^{2m}\leq \frac{1}{1-\left(\frac{a}{b}\right)^{2m}\frac{\log^2\left(\frac{1}{a}\right)}{\log^2\left(\frac{1}{b}\right)}}\left(\frac{1}{16\pi m^3}+\frac{\sqrt{2}}{\pi\,m}\frac{1}{\left(1-\left(\frac{a}{b}\right)^{2m}\right)^2-4m^2\left(\frac{a}{b}\right)^{2m}\log^2\left(\frac{b}{a}\right)}\right)\int_{\Omega}\left(\leb_m\mathrm{Rad}(u)\right)^2dx.
    \end{align} 
    On can check that using Cauchy--Schwarz inequality, the estimate \emph{cannot} be bootstrapped to recover the $\log^2(1/b)$ factor. There seems to be no way to get estimates for $a_0$ and $b_0$ using the weighted integral of $u^2$ to get the missing logarithm factors, so we will use another argument and estimate $\alpha_1^2a^{-2m}$ in a similar fashion. First, notice that the identity \eqref{rest_cs} holds for all $\gamma\in \C$ by analytic continuation (one needs to compute the limit when $\gamma\rightarrow 0$, but we are only interested in non-zero values of $\gamma$, and since the left-hand side is non-singular in $\gamma$, we actually get the analytic continuation on $\C$ immediately). Notice that for $\gamma=0$, we have
    \begin{align*}
        \int_{a}^b\int_{a}^b\log^2\left(\frac{r}{s}\right)\frac{dr}{r}\frac{ds}{s}&=2\left(\int_{a}^{b}\log^2(r)\frac{dr}{r}\right)\left(\int_{a}^b\frac{ds}{s}\right)-2\left(\int_{a}^{b}\log(t)\frac{dt}{t}\right)^2\\
        &=\frac{2}{3}\log^3\left(\frac{b}{a}\right)\times \log\left(\frac{b}{a}\right)-2\left(\frac{1}{2}\log^2\left(\frac{b}{a}\right)\right)^2=\frac{1}{6}\log^4\left(\frac{b}{a}\right),
    \end{align*}
    which yields
    \begin{align*}
        \lim_{\gamma\rightarrow 0}\frac{1}{8\gamma^4}\left\{\left(1-\left(\frac{a}{b}\right)^{2\gamma}\right)^2-4\gamma^2\left(\frac{a}{b}\right)^{2\gamma}\log^2\left(\frac{b}{a}\right)\right\}=\frac{1}{6}\log^4\left(\frac{b}{a}\right).
    \end{align*}
    Therefore, thanks to \eqref{elementary}, we deduce that for all $\gamma>0$, we have
    \begin{align*}
        \int_{a}^b\log^2\left(\frac{r}{s}\right)\frac{dr}{r^{2\gamma+1}}\frac{ds}{s^{2\gamma+1}}=\frac{1}{8\gamma^4}\frac{1}{a^{4\gamma}}\left\{\left(1-\left(\frac{a}{b}\right)^{2\gamma}\right)^2-4\gamma^2\left(\frac{a}{b}\right)^{2\gamma}\log^2\left(\frac{b}{a}\right)\right\}.
    \end{align*}
    Now, recall that 
    \begin{align*}
        &\frac{1}{2\pi}\int_{\Omega}\left(\frac{1}{2}\leb_m\mathrm{Rad}(u)-\frac{2m}{m-1}\frac{z^2}{|z|^2}\mathfrak{L}_m\mathrm{Rad}(u)\right)^2dx=\frac{\alpha_2^2}{2m}b^{2m}\left(1-\left(\frac{a}{b}\right)^{2m}\right)\\
        &+2m\,\alpha_1^2\left(\frac{1}{a^{2m}}\log^2\left(\frac{1}{a}\right)-\frac{1}{m}\frac{1}{a^{2m}}\log\left(\frac{1}{a}\right)+\frac{1}{2m^2}\frac{1}{a^{2m}}\right.\\
        &\left.-\left(\frac{1}{b^{2m}}\log^2\left(\frac{1}{b}\right)-\frac{1}{m}\frac{1}{b^{2m}}\log\left(\frac{1}{b}\right)+\frac{1}{2m^2}\frac{1}{b^{2m}}\right)\right)\\
        &+2m\left(\frac{\alpha_1}{2(m-1)}+2a_0\right)^2\frac{1}{a^{2m}}\left(1-\left(\frac{a}{b}\right)^{2m}\right)\\
        &-4m\left(\frac{\alpha_1}{2(m-1)}+2a_0\right)\alpha_1\left(\frac{1}{a^{2m}}\log\left(\frac{1}{a}\right)-\frac{1}{2m}\frac{1}{a^{2m}}-\left(\frac{1}{b^{2m}}\log\left(\frac{1}{b}\right)-\frac{1}{2m}\frac{1}{b^{2m}}\right)\right)\\
        &+2m\,\alpha_1\alpha_2\left(\log^2(b)-\log^2(a)\right)+2m\left(\frac{\alpha_1}{2(m-1)}+2a_0\right)\alpha_2\log\left(\frac{b}{a}\right).
    \end{align*}
    In fact, there is an easier way to get the needed estimate since we already have a control on $\alpha_2$. Also write for simplicity
    \begin{align*}
        \mathfrak{D}_m\mathrm{Rad}(u)&=\frac{1}{2}\leb_m\mathrm{Rad}(u)-\frac{2m}{m-1}\frac{z^2}{|z|^2}\mathfrak{L}_m\mathrm{Rad}(u)\\
        &=\alpha_2\,r^{m-1}+2m\,r^{-m-1}\left(\alpha_1\log(r)+\frac{\alpha_1}{2(m-1)}+2a_0\right).
    \end{align*}
    We have
    \begin{align*}
        &\int_{\Omega}\left(\mathfrak{D}_m\mathrm{Rad}(u)\right)^2dx\geq 2m^2\int_{a}^br^{-2m-2}\left(\alpha_1\log(r)+\frac{\alpha_1}{2(m-1)}+2a_0\right)^2r\,dr-\alpha_2^2\int_{a}^br^{2m-1}dr\\
        &=2m^2\int_{a}^br^{-2m-2}\left(\alpha_1\log(r)+\frac{\alpha_1}{2(m-1)}+2a_0\right)^2r\,dr-\frac{\alpha_2^2}{2m}b^{2m}\left(1-\left(\frac{a}{b}\right)^{2m}\right).
    \end{align*}
    Therefore, we get the estimate
    \begin{align}
        &2m^2\int_{a}^br^{-2m-2}\left(\alpha_1\log(r)+\frac{\alpha_1}{2(m-1)}+2a_0\right)^2r\,dr\leq \frac{\alpha_2
        ^2}{2m}b^{2m}\left(1-\left(\frac{a}{b}\right)^{2m}\right)+\int_{\Omega}\left(\mathfrak{D}_m\mathrm{Rad}(u)\right)^2dx\nonumber\\
        &\leq \left(\frac{1}{2}+\frac{1-\left(\frac{a}{b}\right)^{2m}}{1-\left(\frac{a}{b}\right)^{2m}\frac{\log^2\left(\frac{1}{a}\right)}{\log^2\left(\frac{a}{b}\right)}}\left(\frac{1}{16\pi m^3}+\frac{\sqrt
        2}{\pi\,m}\frac{1}{\left(1-\left(\frac{a}{b}\right)^{2m}\right)^2-4m^2\left(\frac{a}{b}\right)^{2m}\log^2\left(\frac{b}{a}\right)}\right)\right)\nonumber\\
        &\times\int_{\Omega}\left(\leb_m\mathrm{Rad}(u)\right)^2dx+\frac{8m^2}{(m-1)^2}\int_{\Omega}\left|\mathfrak{L}_m\mathrm{Rad}(u)\right|^2|dz|^2.
    \end{align}
    Therefore, we are reduced to the previous analysis. If 
    \begin{align*}
        f(x)=|x|^{-(m+1)}\left(\alpha_1\log(r)+\frac{\alpha_1}{2(m-1)}+2a_0\right),
    \end{align*}
    we have
    \begin{align*}
        &\frac{1}{2\pi}\int_{\Omega}|f(x)|^2dx=\alpha_1^2\int_{a}^br^{-2m-1}\log^2(r)dr+\left(\frac{\alpha_1}{2(m-1)}+2a_0\right)^2\int_{a}^br^{-2m-1}dr\\
        &+2\alpha_1\left(\frac{\alpha_1}{2(m-1)}+2a_0\right)\int_{a}^br^{-2m-1}\log(r)dr\\
        &=\alpha_1^2\left[-\frac{1}{2m}r^{-2m}\log^2(r)-\frac{1}{4m^2}r^{-2m}\log(r)-\frac{1}{8m^3}r^{-2m}\right]_{a}^b\\
        &+\left(\frac{\alpha_1}{2(m-1)}+2a_0\right)^2\frac{1}{2m}\frac{1}{a^{2m}}\left(1-\left(\frac{a}{b}\right)^{2m}\right)\\
        &+2\alpha_1\left(\frac{\alpha_1}{2(m-1)}+2a_0\right)\left[-\frac{1}{2m}r^{-2m}\log(r)-\frac{1}{4m^2}r^{-2m}\right]_a^b\\
        &=\frac{\alpha_1^2}{2m}\frac{1}{a^{2m}}\log^2\left(\frac{1}{a}\right)\left(1-\left(\frac{a}{b}\right)^{2m}\frac{\log^2\left(\frac{1}{b}\right)}{\log^2\left(\frac{1}{a}\right)}\right)-\frac{\alpha_1^2}{4m^2}\frac{1}{a^{2m}}\log\left(\frac{1}{a}\right)\left(1-\left(\frac{a}{b}\right)^{2m}\frac{\log\left(\frac{1}{b}\right)}{\log\left(\frac{1}{a}\right)}\right)\\
        &+\frac{\alpha_1^2}{8m^3}\frac{1}{a^{2m}}\left(1-\left(\frac{a}{b}\right)^{2m}\right)
        -\frac{\alpha_1}{m}\left(\frac{\alpha_1}{2(m-1)}+2a_0\right)\frac{1}{a^{2m}}\log\left(\frac{1}{a}\right)\left(1-\left(\frac{a}{b}\right)^{2m}\frac{\log\left(\frac{1}{b}\right)}{\log\left(\frac{1}{a}\right)}\right)\\
        &+\frac{\alpha_1}{2m^2}\left(\frac{\alpha_1}{2(m-1)}+2a_0\right)\frac{1}{a^{2m}}\left(1-\left(\frac{a}{b}\right)^{2m}\right).
    \end{align*}
    As previously, the estimates are trivial if
    \begin{align*}
        \alpha_1\left(\frac{\alpha_1}{2(m-1)}+2a_0\right)\leq 0,
    \end{align*}
    so we assume that
    \begin{align*}
        \alpha_1\left(\frac{\alpha_1}{2(m-1)}+2a_0\right)>0.
    \end{align*}
    Therefore, we will get an estimate thanks to \eqref{improved_cauchy_schwarz2}. We readily estimate
    \begin{align*}
        \int_{a}^br^{-2m-1}\log^2(r)dr\int_{a}^bs^{-2m-1}ds\leq \frac{1}{2m^2}\frac{1}{a^{4m}}\log^2\left(\frac{1}{a}\right).
    \end{align*}
    Therefore, we deduce that
    \begin{align*}
        &\frac{1}{4m^3\sqrt{2}}\alpha_1\left(\frac{\alpha_1}{2(m-1)}+2a_0\right)\frac{1}{a^{2m}}\frac{1}{\log\left(\frac{1}{a}\right)}\left\{\left(1-\left(\frac{a}{b}\right)^{2m}\right)^{2}-4m^2\left(\frac{a}{b}\right)^{2m}\log^2\left(\frac{b}{a}\right)\right\}\\
        &\leq \int_{a}^b\left(\alpha_1\log(r)+\frac{1}{2(m-1)}\alpha_1+2a_0\right)^2r^{-(2m+1)}dr=\frac{1}{2\pi}\int_{\Omega}|f(x)|^2dx,
    \end{align*}
    and we get
    \begin{align*}
        \alpha_1\left(\frac{1}{2(m-1)}\alpha_1+2a_0\right)\frac{1}{a^{2m}}\frac{1}{\log\left(\frac{1}{a}\right)}\leq \frac{2m^3\sqrt{2}}{\left(1-\left(\frac{a}{b}\right)^{2m}\right)^2-4m^2\left(\frac{a}{b}\right)^{2m}\log^2\left(\frac{b}{a}\right)}\int_{\Omega}|f(x)|^2dx.
    \end{align*}
    Therefore, we deduce that
    \begin{align*}
        &\alpha_1^2\frac{1}{a^{2m}}\log^2\left(\frac{1}{a}\right)\left(1-\frac{1}{2m}\frac{1}{\log\left(\frac{1}{a}\right)}\right)\left(1-\left(\frac{a}{b}\right)^{2m}\frac{\log^2\left(\frac{1}{b}\right)}{\log^2\left(\frac{1}{a}\right)}\right)\\
        &\leq \left(\frac{m}{\pi}+\frac{4m^3\sqrt{2}}{\pi}\log^2\left(\frac{1}{a}\right)\frac{1-\left(\frac{a}{b}\right)^{2m}\frac{\log\left(\frac{1}{b}\right)}{\log\left(\frac{1}{a}\right)}}{\left(1-\left(\frac{a}{b}\right)^{2m}\right)^2-4m^2\left(\frac{a}{b}\right)^{2m}\log^2\left(\frac{b}{a}\right)}\right)\int_{\Omega}|f(x)|^2dx.
    \end{align*}
    Since $a<b\leq e^{-1}$ by hypothesis, we deduce that 
    \begin{align}\label{estimate_alpha1_radial}
        &\alpha_1^2\frac{1}{a^{2m}}\leq \frac{1}{1-\left(\frac{a}{b}\right)^{2m}\frac{\log^2\left(\frac{1}{b}\right)}{\log^2\left(\frac{1}{a}\right)}}\left(\frac{2m}{\pi}+\frac{8m^3\sqrt{2}}{\pi}\frac{1-\left(\frac{a}{b}\right)^{2m}\frac{\log\left(\frac{1}{b}\right)}{\log\left(\frac{1}{a}\right)}}{\left(1-\left(\frac{a}{b}\right)^{2m}\right)^2-4m^2\left(\frac{a}{b}\right)^{2m}\log^2\left(\frac{b}{a}\right)}\right)\int_{\Omega}|f(x)|^2dx\nonumber\\
        &\leq \frac{1}{1-\left(\frac{a}{b}\right)^{2m}\frac{\log^2\left(\frac{1}{b}\right)}{\log^2\left(\frac{1}{a}\right)}}\left(\frac{2m}{\pi}+\frac{8m^3\sqrt{2}}{\pi}\frac{1-\left(\frac{a}{b}\right)^{2m}\frac{\log\left(\frac{1}{b}\right)}{\log\left(\frac{1}{a}\right)}}{\left(1-\left(\frac{a}{b}\right)^{2m}\right)^2-4m^2\left(\frac{a}{b}\right)^{2m}\log^2\left(\frac{b}{a}\right)}\right)\nonumber\\
        &\times \left\{\left(\frac{1}{4m^2}+\frac{1-\left(\frac{a}{b}\right)^{2m}}{1-\left(\frac{a}{b}\right)^{2m}\frac{\log^2\left(\frac{1}{a}\right)}{\log^2\left(\frac{a}{b}\right)}}\left(\frac{1}{32\pi m^4}+\frac{\sqrt
        2}{2\pi\,m^2}\frac{1}{\left(1-\left(\frac{a}{b}\right)^{2m}\right)^2-4m^2\left(\frac{a}{b}\right)^{2m}\log^2\left(\frac{b}{a}\right)}\right)\right)\right.\nonumber\\
        &\left.\times \int_{\Omega}\left(\leb_m\mathrm{Rad}(u)\right)^2dx
        +\frac{2}{(m-1)^2}\int_{\Omega}\left|\mathfrak{L}_m\mathrm{Rad}(u)\right|^2|dz|^2\right\}.
    \end{align}
    We can now conclude the proof of the radial estimate. We rewrite and use the elementary estimate
    \begin{align}\label{square_comparison}
        \left(\sum_{i=1}^{d}a_i\right)^2\leq d\sum_{i=1}^{d}a_i^2\qquad\forall d\in \N,\;\forall \ens{a_1,\cdots,a_d}\subset\R, 
    \end{align}
    to show that
    \begin{align*}
        &\frac{1}{r^2}\left|\p{r}\mathrm{Rad}(u)+\frac{m-1}{r}\mathrm{Rad}(u)\right|^2=\frac{1}{r^2}\left|\alpha_1\,r^{-m}+2m\,r^m\left(\alpha_2\log(r)+\frac{1}{2m}\alpha_2+2b_0\right)\right|^2\\
        &=\frac{1}{r^2}\left|\alpha_1\,r^{-m}-\alpha_2\,r^m+2m\,r^m\left(\alpha_2\log(r)+\frac{1}{m}\alpha_2+2b_0\right)\right|^2\\
        &\leq \frac{3}{r^2}\left(\alpha_1^2r^{-2m}+\alpha_2^2r^{2m}+4m^2r^{2m}\left(\alpha_2^2\log^2(r)+\left(\frac{1}{m}\alpha_2+2b_0\right)^2+2\alpha_2\left(\frac{1}{m}\alpha_2+2b_0\right)\log(r)\right)\right).
    \end{align*}
    Therefore, we have by \eqref{identity_lm}, \eqref{radial_estimate_alpha2}, and \eqref{estimate_alpha1_radial}
    \begin{align}\label{radial_estimate_final}
        &\int_{\Omega}\frac{1}{|x|^2}\left|\p{r}\mathrm{Rad}(u)+\frac{m-1}{r}\mathrm{Rad}(u)\right|^2dx\leq 6\pi\int_{a}^b\alpha_1^2\,r^{-(2m+1)}dr+6\pi\int_{a}^{b}\alpha_2^2r^{2m-1}dr\nonumber\\
        &+\frac{1}{4m^2}\int_{\Omega}\left(\leb_m\mathrm{Rad}(u)\right)^2dx\nonumber\\
        &=\frac{3\pi}{m}\alpha_1^2\frac{1}{a^{2m}}\left(1-\left(\frac{a}{b}\right)^{2m}\right)+\frac{3\pi}{m}\alpha_2^2\left(1-\left(\frac{a}{b}\right)^{2m}\right)+\frac{1}{4m^2}\int_{\Omega}\left(\leb_m\mathrm{Rad}(u)\right)^2dx\nonumber\\
        &\leq \frac{1-\left(\frac{a}{b}\right)^{2m}}{1-\left(\frac{a}{b}\right)^{2m}\frac{\log^2\left(\frac{1}{b}\right)}{\log^2\left(\frac{1}{a}\right)}}\left(6+24m^2\sqrt{2}\frac{1-\left(\frac{a}{b}\right)^{2m}\frac{\log\left(\frac{1}{b}\right)}{\log\left(\frac{1}{a}\right)}}{\left(1-\left(\frac{a}{b}\right)^{2m}\right)^2-4m^2\left(\frac{a}{b}\right)^{2m}\log^2\left(\frac{b}{a}\right)}\right)\nonumber\\
        &\times \left\{\left(\frac{1}{4m^2}+\frac{1-\left(\frac{a}{b}\right)^{2m}}{1-\left(\frac{a}{b}\right)^{2m}\frac{\log^2\left(\frac{1}{a}\right)}{\log^2\left(\frac{1}{b}\right)}}\left(\frac{1}{32\pi m^4}+\frac{\sqrt
        2}{2\pi\,m^2}\frac{1}{\left(1-\left(\frac{a}{b}\right)^{2m}\right)^2-4m^2\left(\frac{a}{b}\right)^{2m}\log^2\left(\frac{b}{a}\right)}\right)\right)\right.\nonumber\\
        &\left.\times \int_{\Omega}\left(\leb_m\mathrm{Rad}(u)\right)^2dx
        +\frac{2}{(m-1)^2}\int_{\Omega}\left|\mathfrak{L}_m\mathrm{Rad}(u)\right|^2|dz|^2\right\}\nonumber\\
        &+\frac{1-\left(\frac{a}{b}\right)^{2m}}{1-\left(\frac{a}{b}\right)^{2m}\frac{\log^2\left(a\right)}{\log^2(b)}}\left(\frac{3}{16 m^4}+\frac{3\sqrt{2}}{m^2}\frac{1}{\left(1-\left(\frac{a}{b}\right)^{2m}\right)^2-4m^2\left(\frac{a}{b}\right)^{2m}\log^2\left(\frac{b}{a}\right)}\right)\int_{\Omega}\left(\leb_m\mathrm{Rad}(u)\right)^2dx\nonumber\\
        &+\frac{1}{4m^2}\int_{\Omega}\left(\leb_m\mathrm{Rad}(u)\right)^2dx.
    \end{align}

    \textbf{Step 9.} The final estimates.  
    Therefore, we now need only control the angular part of
    \begin{align*}
        \int_{\Omega}\frac{1}{|x|^2}\left|\D u+(m-1)\frac{x}{|x|^2}u\right|^2\left(\frac{a}{|x|}\right)^{2\gamma}dx.
    \end{align*}
    However, this proceeds in the exact same way as the estimate of 
    \begin{align*}
        \int_{\Omega}\frac{1}{|x|^2}\left|\D u+(m-1)\frac{x}{|x|^2}u\right|^2\left(\frac{a}{|x|}\right)^{2\gamma}dx
    \end{align*}
    and we skip the details due to length considerations. This concludes the proof of the theorem. 
    \end{proof}

    For technical reasons that only become apparent in the involved argument involving the Bochner identity to show that eigenvalues associated to non-trivial eigenspaces are bounded from below, we would have to refine this estimate so that it involves a weight on the two elliptic operators on the right-hand side of the inequality of Theorem \ref{lm_last_lemma}.

However, the estimate for functions in $W^{2,2}_0(\Omega)$ does not hold true when one adds weights to the $L^2$ norm of $\leb_mu$ (see Theorem \ref{poincare_weight_m}), so we will have to replace the operator $\leb_m$ by a new fourth-order elliptic operator with regular singularities that corresponds to the integral
\begin{align*}
    \int_{\Omega_k(\alpha)}|z|^{2-2m}|\p{z}^2u|^2|dz|^2
\end{align*}
before the change of variable. The advantage is that for $W^{2,2}_0(\Omega)$ function, the needed doubly weighted estimate does hold true (see Theorem \ref{poincare_weight_m_2}).

    \section{A New Family of Fourth-Order Operators with Regular Singularities}

    The main issue in those estimates is that if we want to have a weighted estimate for $u$ that involved the weight on the right-hand side, we have to use the $L^2$ norm of $\mathfrak{L}_m$ instead of $\leb_m$. We therefore have to study the (real-valued) solutions to the equation $\bar{\mathfrak{L}_m}^{\ast}\mathfrak{L}_mu$, or equivalently, $\Re\left(\bar{\mathfrak{L}_m}^{\ast}\mathfrak{L}_m\right)u=0$. Recall that the Cauchy-Riemann operator is defined by
    \begin{align*}
        \p{z}=\frac{1}{2}\left(\p{x}-i\,\p{y}\right).
    \end{align*}
    Furthermore, we have
    \begin{align*}
    \left\{\begin{alignedat}{4}
        \p{r}&u=&&\cos(\theta)\p{x}u\,&&+&&\sin(\theta)\p{y}u\\
        \frac{1}{r}\p{\theta}u&=-&&\sin(\theta)\p{x}u\,&&+&&\cos(\theta)\p{y}u,
        \end{alignedat}\right.
    \end{align*}
    which shows that
    \begin{align*}
        \begin{pmatrix}
            \p{x}u\\
            \p{y}u
        \end{pmatrix}&=\begin{pmatrix}
            \cos(\theta) & -\sin(\theta)\\
            \sin(\theta) & \cos(\theta)
        \end{pmatrix}\begin{pmatrix}
            \p{r}u
            \vspace{0.3em}\\
            \dfrac{1}{r}\p{\theta}u
        \end{pmatrix}=\begin{pmatrix}
            \cos(\theta)\p{r}u-\sin(\theta)\dfrac{1}{r}\p{\theta}u
            \vspace{0.3em}\\
            \sin(\theta)\p{r}u+\cos(\theta)\dfrac{1}{r}\p{\theta}u
        \end{pmatrix}.
    \end{align*}
    Therefore, we have
    \begin{align}\label{elementary_cauchy_riemann}
        \p{z}u&=\frac{1}{2}(\p{x}-i\,\p{y}u)=\frac{1}{2}\left(\left(\cos(\theta)\p{r}u-\sin(\theta)\frac{1}{r}\p{\theta}u\right)-i\left(\sin(\theta)\p{r}u+\cos(\theta)\frac{1}{r}\p{\theta}u\right)\right)\nonumber\\
        &=\frac{1}{2}\left(e^{-i\theta}\p{r}u-i\,e^{-i\theta}\frac{1}{r}\p{\theta}u\right),
    \end{align}
    which shows that
    \begin{align}\label{zpz}
        \Re\left(z\,\p{z}u\right)=\frac{1}{2}r\,\p{r}u=\frac{1}{2}x\cdot \D u.
    \end{align}
    Likewise, we compute
    \begin{align*}
        \p{x}^2u&=\cos(\theta)\p{r}(\p{x}u)-\sin(\theta)\frac{1}{r}\p{\theta}(\p{x}u)\\
        &=\cos(\theta)\p{r}\left(\cos(\theta)\p{r}u-\sin(\theta)\frac{1}{r}\p{\theta}u\right)-\sin(\theta)\frac{1}{r}\p{\theta}\left(\cos(\theta)\p{r}u-\sin(\theta)\frac{1}{r}\p{\theta}u\right)\\
        &=\cos^2(\theta)\p{r}^2u-\cos(\theta)\sin(\theta)\frac{1}{r}\p{r\theta}^2u+\cos(\theta)\sin(\theta)\frac{1}{r^2}\p{\theta}u\\
        &-\cos(\theta)\sin(\theta)\frac{1}{r}\p{r\theta}^2u+\frac{\sin^2(\theta)}{r}\p{r}u+\frac{\sin^2(\theta)}{r^2}\p{\theta}^2+\cos(\theta)\sin(\theta)\frac{1}{r^2}\p{\theta}u\\
        &=\cos^2(\theta)\p{\theta}^2u+\frac{\sin^2(\theta)}{r}\p{r}u+\frac{\sin^2(\theta)}{r^2}\p{\theta}^2-2\cos(\theta)\sin(\theta)\frac{1}{r}\p{r\theta}^2u+2\cos(\theta)\sin(\theta)\frac{1}{r^2}\p{\theta}u,
    \end{align*}
    and
    \begin{align*}
        \p{y}^2u&=\sin(\theta)\p{r}(\p{y}u)+\cos(\theta)\frac{1}{r}\p{\theta}(\p{y}u)\\
        &=\sin(\theta)\p{r}\left(\sin(\theta)\p{r}u+\cos(\theta)\frac{1}{r}\p{\theta}u\right)+\cos(\theta)\frac{1}{r}\p{\theta}\left(\sin(\theta)\p{r}u+\cos(\theta)\frac{1}{r}\p{\theta}u\right)\\
        &=\sin^2(\theta)\p{r}^2u+\cos(\theta)\sin(\theta)\frac{1}{r}\p{r\theta}^2u-\cos(\theta)\sin(\theta)\frac{1}{r^2}\p{\theta}u\\
        &+\cos(\theta)\sin(\theta)\frac{1}{r}\p{r\theta}^2u+\frac{\cos^2(\theta)}{r}\p{r}u+\frac{\cos^2(\theta)}{r^2}\p{\theta}^2u-\cos(\theta)\sin(\theta)\frac{1}{r^2}\p{\theta}u\\
        &=\sin^2(\theta)\p{r}^2u+\frac{\cos^2(\theta)}{r}\p{r}u+\frac{\cos^2(\theta)}{r^2}\p{\theta}^2u+2\cos(\theta)\sin(\theta)\frac{1}{r}\p{r\theta}^2u-2\cos(\theta)\sin(\theta)\frac{1}{r^2}\p{\theta}u.
    \end{align*}
    In particular, we recover the identity
    \begin{align}\label{classical_polar}
        \Delta=\p{x}^2+\p{y}^2=\p{r}^2+\frac{1}{r}\p{r}+\frac{1}{r^2}\p{\theta}^2.
    \end{align}
    Then, we have
    \begin{align*}
        &\p{xy}^2u=\cos(\theta)\p{r}(\p{y}u)-\sin(\theta)\frac{1}{r}\p{\theta}(\p{y}u)\\
        &=\cos(\theta)\p{r}\left(\sin(\theta)\p{r}u+\cos(\theta)\frac{1}{r}\p{\theta}u\right)-\sin(\theta)\frac{1}{r}\p{\theta}\left(\sin(\theta)\p{r}u+\cos(\theta)\frac{1}{r}\p{\theta}u\right)\\
        &=\cos(\theta)\sin(\theta)\p{r}^2u+\frac{\cos^2(\theta)}{r}\p{r\theta}^2u-\frac{\cos^2(\theta)}{r^2}\p{\theta}u\\
        &-\frac{\sin^2(\theta)}{r}\p{r\theta}^2u-\cos(\theta)\sin(\theta)\frac{1}{r}\p{r}u-\cos(\theta)\sin(\theta)\frac{1}{r^2}\p{\theta}^2+\frac{\sin^2(\theta)}{r^2}\p{\theta}u\\
        &=\cos(\theta)\sin(\theta)\p{r}^2u-\cos(\theta)\sin(\theta)\frac{1}{r}\p{r}u-\cos(\theta)\sin(\theta)\frac{1}{r^2}\p{\theta}^2+\frac{\cos^2(\theta)-\sin^2(\theta)}{r}\p{r\theta}^2u\\
        &-\frac{\cos^2(\theta)-\sin^2(\theta)}{r^2}\p{\theta}u.
    \end{align*}
    We also compute
    \begin{align*}
        &\p{yx}^2u=\sin(\theta)\p{r}(\p{x}u)+\cos(\theta)\frac{1}{r}\p{\theta}\left(\p{x}u\right)\\
        &=\sin(\theta)\p{r}\left(\cos(\theta)\p{r}u-\sin(\theta)\frac{1}{r}\p{\theta}u\right)+\cos(\theta)\frac{1}{r}\p{\theta}\left(\cos(\theta)\p{r}u-\sin(\theta)\frac{1}{r}\p{\theta}u\right)\\
        &=\cos(\theta)\sin(\theta)\p{r}^2u-\frac{\sin^2(\theta)}{r}\p{r\theta}^2u+\frac{\sin^2(\theta)}{r}\p{\theta}u\\
        &+\frac{\cos^2(\theta)}{r}\p{r\theta}^2u-\frac{\cos(\theta)\sin(\theta)}{r}\p{r}u-\frac{\cos(\theta)\sin(\theta)}{r^2}\p{\theta}^2u-\frac{\cos^2(\theta)}{r^2}\p{\theta}u\\
        &=\cos(\theta)\sin(\theta)\p{r}^2u-\frac{\cos(\theta)\sin(\theta)}{r}\p{r}u-\frac{\cos(\theta)\sin(\theta)}{r^2}\p{\theta}^2u+\frac{\cos^2(\theta)-\sin^2(\theta)}{r}\p{r\theta}^2u\\
        &-\frac{\cos^2(\theta)-\sin^2(\theta)}{r^2}\p{\theta}u.
    \end{align*}
    Then, we have
    \begin{align}\label{z2pz2}
        &\Re\left(z^2\,\p{z}^2\right)=\Re\left((x+i\,y)^2\left(\frac{1}{2}\left(\p{x}-i\,\p{y}\right)\right)^2\right)=\frac{1}{4}\Re\left(\left(x^2-y^2+2i\,xy\right)\left(\p{x}^2-\p{y}^2-2i\,\p{xy}^2\right)\right)\nonumber\\
        &=\frac{1}{4}\Big((x^2-y^2)\left(\p{x}^2-\p{y}^2\right)+4\,xy\,\p{xy}^2\Big)\nonumber\\
        &=\frac{r^2}{4}\left(\left(\cos^2(\theta)-\sin^2(\theta)\right)\left(\left(\cos^2(\theta)-\sin^2(\theta)\right)\p{r}^2-\frac{\cos^2(\theta)-\sin^2(\theta)}{r}\p{r}-\frac{\cos^2(\theta)-\sin^2(\theta)}{r^2}\p{\theta}^2\right.\right.\nonumber\\
        &\left.\left.-\frac{4\cos(\theta)\sin(\theta)}{r}\p{r\theta}^2+\frac{4\cos(\theta)\sin(\theta)}{r^2}\p{\theta}\right)+4\cos(\theta)\sin(\theta)\left(\cos(\theta)\sin(\theta)\p{r}^2-\frac{\cos(\theta)\sin(\theta)}{r}\p{r}\right.\right.\nonumber\\
        &\left.\left.-\frac{\cos(\theta)\sin(\theta)}{r^2}\p{\theta}^2+\frac{\cos^2(\theta)-\sin^2(\theta)}{r}\p{r\theta}^2-\frac{\cos^2(\theta)-\sin^2(\theta)}{r^2}\p{\theta}\right)\right)\nonumber\\
        &=\frac{r^2}{4}\left(\p{r}^2-\frac{1}{r}\p{r}-\frac{1}{r^2}\p{\theta}^2\right)=\frac{r^2}{4}\left(2\,\p{r}^2-\Delta\right)\nonumber\\
        &=\frac{|x|^2}{4}\left(2\left(\frac{x}{|x|}\right)\cdot \D^2(\,\cdot\,)\cdot \left(\frac{x}{|x|}\right)-\Delta\right)=\frac{|x|^4}{2}\left(\frac{x}{|x|^2}\right)\cdot \D^2(\,\cdot\,)\cdot \left(\frac{x}{|x|^2}\right)-\frac{|x|^2}{4}\Delta.
    \end{align}
    Now, notice that for all $u,v\in C^{\infty}_c(\Omega)$, we have
    \begin{align*}
        \int_{\Omega}v\frac{1}{z}\p{z}u|dz|^2=-\int_{\Omega}u\,\p{z}\left(\frac{1}{z}v\right)|dz|^2=-\int_{\Omega}u\left(\frac{1}{z}\p{z}v-\frac{1}{z^2}v\right)|dz|^2.
    \end{align*}
    Therefore, we deduce that 
    \begin{align*}
        \left(\frac{1}{z}\p{z}\right)^{\ast}=-\left(\frac{1}{z}\p{z}-\frac{1}{z^2}\right).
    \end{align*}
    Therefore, we have
    \begin{align*}
        \mathfrak{L}_m^{\ast}&=\left(\p{z}^2+\frac{(m-1)}{z}\p{z}+\frac{(m-1)(m-3)}{4z^2}\right)^{\ast}=\left(\p{z}^2-\frac{(m-1)}{z}\p{z}+\frac{(m-1)(m-3)}{4z^2}+\frac{(m-1)}{z^2}\right)\\
        &=\left(\p{z}^2-\frac{(m-1)}{z}\p{z}+\frac{(m-1)^2}{z^2}+\frac{(m^2-1)}{4z^2}\right).
    \end{align*}
    Now, we have
    \begin{align*}
        &\bar{\mathfrak{L}_m}^{\ast}\mathfrak{L}_m=\left(\p{\z}^2-\frac{(m-1)}{\z}\p{\z}+\frac{(m-1)(m-3)}{4\z^2}+\frac{(m-1)}{\z^2}\right)\left(\p{z}^2+\frac{(m-1)}{z}\p{z}+\frac{(m-1)(m-3)}{4z^2}\right)\\
        &=\frac{1}{16}\Delta^2+\frac{(m-1)}{4z}\p{\z}\Delta+\frac{(m-1)(m-3)}{4|z|^4}\z^2\p{\z}^2-\frac{(m-1)}{4\z}\p{z}\Delta-\frac{(m-1)^2}{4|z|^2}\Delta-\frac{(m-1)^2(m-3)}{4|z|^4}\z\,\p{\z}\\
        &+\frac{(m-1)(m-3)}{4|z|^4}z^2\p{z}^2+\frac{(m-1)^2(m-3)}{4|z|^4}z\,\p{z}+\frac{(m-1)^2(m-3)^2}{16|z|^4}\\
        &+\frac{(m-1)}{|z|^4}z^2\p{z}^2+\frac{(m-1)^2}{|z|^4}z\,\p{z}+\frac{(m-1)^2(m-3)}{4|z|^4}\\
        &=\frac{1}{16}\Delta^2-\frac{i}{2}\frac{(m-1)}{|z|^2}\Im\left(z\,\p{z}\Delta\right)+\frac{i}{2}\frac{(m-1)^2(m-3)}{4|z|^4}\Im\left(z\,\p{z}\right)-\frac{(m-1)^2}{4|z|^2}\Delta \\
        &+\frac{(m-1)(m-3)}{2|z|^4}\Re\left(z^2\p{z}^2\right)+\frac{(m-1)}{|z|^4}z^2\p{z}^2+\frac{(m-1)^2}{|z|^4}z\,\p{z}+\frac{(m+1)(m-1)^2(m-3)}{16|z|^2}.
    \end{align*}
    Therefore, we deduce that
    \begin{align*}
        \Re\left(\bar{\mathfrak{L}_m}^{\ast}\mathfrak{L}_m\right)&=\frac{1}{16}\Delta^2-\frac{(m-1)^2}{4|z|^2}\Delta+\frac{(m-1)^2}{2|z|^4}\Re\left(z^2\p{z}^2\right)+\frac{(m-1)^2}{|z|^4}\Re\left(z\,\p{z}\right)\\
        &+\frac{(m+1)(m-1)^2(m-3)}{16|z|^4}.
    \end{align*}
    Using \eqref{zpz} and \eqref{z2pz2}, we deduce that 
    \begin{align*}
        \Re\left(\bar{\mathfrak{L}_m}^{\ast}\mathfrak{L}_m\right)&=\frac{1}{16}\Delta^2-\frac{(m-1)^2}{4|x|^2}\Delta+\frac{(m-1)^2}{2}\left(\frac{1}{2}\left(\frac{x}{|x|^2}\right)\cdot \D^2(\,\cdot\,)\cdot \left(\frac{x}{|x|^2}\right)-\frac{1}{4|x|^2}\Delta\right)\\
        &+\frac{(m-1)^2}{2|x|^2}\frac{x}{|x|^2}\cdot \D u+\frac{(m+1)(m-1)^2(m-3)}{16|x|^4}\\
        &=\frac{1}{16}\Delta^2-\frac{3(m-1)^2}{8|x|^2}\Delta+\frac{(m-1)^2}{4}\left(\frac{x}{|x|^2}\right)\cdot \D^2(\,\cdot\,)\cdot\left(\frac{x}{|x|^2}\right)+\frac{(m-1)^2}{2|x|^2}\frac{x}{|x|^2}\cdot \D u\\
        &+\frac{(m+1)(m-1)^2(m-3)}{16|x|^4}.
    \end{align*}
    And finally, we have
    \begin{align*}
        \mathfrak{D}_m&=16\,\Re\left(\bar{\mathfrak{L}_m}^{\ast}\mathfrak{L}_m\right)=\Delta^2-\frac{6(m-1)^2}{|x|^2}\Delta +4(m-1)^2\left(\frac{x}{|x|^2}\right)\cdot \D^2(\,\cdot\,)\cdot \left(\frac{x}{|x|^2}\right)\\
        &+\frac{8(m-1)^2}{|x|^2}\frac{x}{|x|^2}\cdot \D u+\frac{(m+1)(m-1)^2(m-3)}{|x|^4}
    \end{align*}
    Recall that in comparison, we have
    \begin{align*}
        \leb_m^{\ast}\leb_m=\Delta^2+\frac{2(m^2-1)}{|x|^2}\Delta -4(m^2-1)\left(\frac{x}{|x|^2}\right)^t\cdot\D^2(\,\cdot\,)\cdot\left(\frac{x}{|x|^2}\right)+\frac{(m^2-1)^2}{|x|^4}.
    \end{align*}
    Let us now find a basis of solutions of $\mathfrak{D}_mu=0$. We have (classically or by the above computation \eqref{classical_polar})
    \begin{align*}
        \Delta=\p{r}^2+\frac{1}{r}\p{r}+\frac{1}{r^2}\p{\theta}^2.
    \end{align*}
    Therefore, we get
    \begin{align*}
        \p{r}\Delta&=\p{r}^3+\frac{1}{r}\p{r}^2-\frac{1}{r^2}\p{r}+\frac{1}{r^2}\p{r}\p{\theta}^2-\frac{2}{r^3}\p{\theta}^2\\
        \p{r}^2\Delta&=\p{r}^4+\frac{1}{r}\p{r}^3-\frac{2}{r^2}\p{r}^2+\frac{2}{r^3}\p{r}+\frac{1}{r^2}\p{r}^2\p{\theta}^2-\frac{4}{r^3}\p{r}\p{\theta}^2+\frac{6}{r^4}\p{\theta}^2.
    \end{align*}
    Therefore, we find that
    \begin{align*}
        \Delta^2&=\p{r}^4+\frac{1}{r}\p{r}^3-\frac{2}{r^2}\p{r}^2+\frac{2}{r^3}\p{r}+\frac{1}{r^2}\p{r}^2\p{\theta}^2-\frac{4}{r^3}\p{r}\p{\theta}^2+\frac{6}{r^4}\p{\theta}^2\\
        &+\frac{1}{r}\p{r}^3+\frac{1}{r}\p{r}^2-\frac{1}{r^3}\p{r}+\frac{1}{r^3}\p{r}\p{\theta}^2-\frac{2}{r^4}\p{\theta}^2\\
        &+\frac{1}{r^2}\p{r}^2\p{\theta}^2+\frac{1}{r^3}\p{r}\p{\theta}^2+\frac{1}{r^4}\p{\theta}^4\\
        &=\p{r}^4+\frac{2}{r}\p{r}^3-\frac{1}{r}\p{r}^2+\frac{1}{r}\p{r}+\frac{2}{r^2}\p{r}^2\p{\theta}^2-\frac{2}{r^3}\p{r}\p{\theta}^2+\frac{4}{r^4}\p{\theta}^2+\frac{1}{r^4}\p{\theta}^2.
    \end{align*}
    Therefore, we have
    \begin{align*}
        &\mathfrak{D}_m=\p{r}^4+\frac{2}{r}\p{r}^3-\frac{1}{r}\p{r}^2+\frac{1}{r}\p{r}+\frac{2}{r^2}\p{r}^2\p{\theta}^2-\frac{2}{r^3}\p{r}\p{\theta}^2+\frac{4}{r^4}\p{\theta}^2+\frac{1}{r^4}\p{\theta}^4\\
        &-6(m-1)^2\left(\p{r}^2+\frac{1}{r}\p{r}+\frac{1}{r^2}\p{\theta}^2\right)+4(m-1)^2\p{r}^2+\frac{8(m-1)^2}{r^3}\p{r}+\frac{(m+1)(m-1)^2(m-3)}{r^4}\\
        &=\p{r}^4+\frac{2}{r}\p{r}^3-\frac{2(m-1)^2+1}{r^2}\p{r}^2+\frac{2(m-1)^2+1}{r}\p{r}+\frac{2}{r^2}\p{r}^2\p{\theta}^2-\frac{2}{r^3}\p{r}\p{\theta}^2+\frac{4-6(m-1)^2}{r^4}\p{\theta}^2+\frac{1}{r^4}\p{\theta}^4\\
        &+\frac{(m+1)(m-1)^2(m-3)}{r^4}.
    \end{align*}
    Therefore, we deduce that
    \begin{align*}
        \Pi_{n^2}(\mathfrak{D}_m)&=\p{r}^4+\frac{2}{r}\p{r}^3-\frac{2(m-1)^2+2n^2+1}{r^2}\p{r}^2+\frac{2(m-1)^2+2n^2+1}{r^3}\p{r}\\
        &+\frac{n^4+n^2(6(m-1)^2-4)+(m+1)(m-1)^2(m-3)}{r^4}.
    \end{align*}
    Make the change of variable $u(r)=Y(\log(r))$. It yields
    \begin{align*}
        \left\{\begin{alignedat}{1}
            \p{r}u&=\frac{1}{r}Y'\\
            \p{r}^2u&=\frac{1}{r^2}(Y''-Y')\\
            \p{r}^3u&=\frac{1}{r^3}(Y'''-Y''-2(Y''-Y'))=\frac{1}{r^3}(Y'''-3Y''+2Y')\\
            \p{r}^4u&=\frac{1}{r^4}(Y''''-3Y'''+2Y''-3(Y'''-3Y''+2Y''))=\frac{1}{r^4}(Y''''-6Y'''+11Y''-6Y').
        \end{alignedat}\right.
    \end{align*}
    Therefore, we get
    \begin{align*}
        r^4\Pi_{n^2}(\mathfrak{D}_m)u&=Y''''-6Y'''+11Y''-6Y'+2\left(Y'''-3Y''+2Y'\right)-(2(m-1)^2+2n^2+1)\left(Y''-Y'\right)\\
        &+(2(m-1)^2+2n^2+1)Y'+\left(n^4+n^2(6(m-1)^2-4)+(m+1)(m-1)^2(m-3)\right)Y\\
        &=Y''''-4Y'''-\left(2(m-1)^2+2n^2-4\right)Y''+(4(m-1)^2+4n^2)Y'\\
        &+\left(n^4+n^2(6(m-1)^2-4)+(m+1)(m-1)^2(m-3)\right)Y.
    \end{align*}
    Write for simplicity $\lambda_{m,n}=n^4+n^2(6(m-1)^2-4)+(m+1)(m-1)^2(m-3)$.     The characteristic polynomial of this ordinary differential equation is given by 
    \begin{align*}
        P(X)=X^4-4X^3-(2(m-1)^2+2n^2-4)X^2+(4(m-1)^2+4n^2)X+\lambda_{m,n}.
    \end{align*}
    Therefore, we have
    \begin{align*}
        Q(X)&=P(X+1)\\
        &=X^4+4X^3+6X^2+4X+1-4(X^3+3X^2+3X+1)-(2(m-1)^2+2n^2-4)(X^2+2X+1)\\
        &+(4(m-1)^2+4n^2)(X+1)+\lambda_{m,n}\\
        &=X^4-2((m-1)^2+n^2+1)X^2+2(m-1)^2+2n^2+1+\lambda_{m,n},
    \end{align*}
    which is a biquadratic equation. The discriminant of the associated quadratic polynomial is equal to
    \begin{align*}
        D_{m,n}=16((m-1)^2-((m-1)^2-1)n^2).
    \end{align*}
    Therefore, for $n=0$, we get $D_{m,0}=16(m-1)^2$, which shows that the roots of $Q(\sqrt{X})$ are equal to 
    \begin{align*}
        r_{0,1}=(m-1)^2+1+2(m-1)=m^2\quad r_{0,2}=(m-1)^2+1-2(m-1)=(m-2)^2,
    \end{align*}
    which shows that a basis of solutions of $\Pi_0(\mathscr{D}_m)u=0$ is given by (assuming that $m>2$)
    \begin{align*}
        r^{m+1}, r^{1-m}, r^{m-1}, r^{3-m}.
    \end{align*}
    For $n=\pm 1$, we have $D_{m,\pm 1}=16$, which yields the roots
    \begin{align*}
        r_{1,1}=(m-1)^2+2+2=(m-1)^2+4\qquad r_{1,2}=(m-1)^2+2-2=(m-1)^2.
    \end{align*}
    As $m>2$, a basis of solutions of $\Pi_{\pm 1}(\mathfrak{D}_m)u=0$ is given by
    \begin{align*}
        r^{1+\sqrt{(m-1)^2+4}}, r^{1-\sqrt{(m-1)^2+4}}, r^{m}, r^{2-m}.
    \end{align*}
    Notice that a real-valued function $v$ satisfies $\p{z}^2v=0$ if and only if $v$ belongs to the $4$-dimensional linear space spanned by
    \begin{align*}
        1,r\cos(\theta),r\sin(\theta),r^2.
    \end{align*}
    Indeed, since $\p{z}v$ is anti-holomorphic, there exists $\ens{a_n}_{n\in \Z}$ such that
    \begin{align*}
        \p{z}v=\sum_{n\in \Z}a_n\z^n.
    \end{align*}
    Therefore, we deduce that there exists $\ens{b_n}_{n\in \Z}$ such that
    \begin{align*}
        v(z)=\sum_{n\in \Z}a_nz\z^n+\sum_{n\in \Z}b_n\z^n.
    \end{align*}
    Since $v$ must be real valued, we deduce that $b_0\in \R$, that $b_1=\bar{a_0}$, and that $a_1\in\R$. All other coefficients must vanish, which shows in the end that
    \begin{align*}
        v(z)=b_0+2\,\Re\left(a_0z\right)+a_1|z|^2.
    \end{align*}
    Now, if $v=|z|^{m-1}u$, this implies that
    \begin{align*}
        u(z)&=b_0|z|^{1-m}+2\,|z|^{1-m}\,\Re(a_0z)+a_1|z|^{3-m}\\
        &=b_0\,r^{1-m}+a_1r^{3-m}+2\,\Re(a_0)r^{2-m}\cos(\theta)-2\,\Im(a_0)r^{2-m}\sin(\theta),
    \end{align*}
    which shows that the Kernel of $\mathfrak{L}_m$ is spanned by the functions
    \begin{align*}
        r^{1-m},r^{3-m},r^{2-m}\cos(\theta),r^{2-m}\sin(\theta).
    \end{align*}
    The result is coherent with the roots found above since we have $r^{1-m},r^{3-m}\in \mathrm{Ker}(\Pi_0(\mathfrak{D}_m))$ and $r^{2-m}\cos(\theta),r^{2-m}\sin(\theta)\in \mathrm{Ker}(\Pi_{\pm 1}(\mathfrak{D}_m))$. Indeed, as those functions belong to $\mathrm{Ker}(\Pi_0(\mathfrak{L}_m)$ and\\
    $\mathrm{Ker}(\Pi_{\pm 1}(\mathfrak{L}_m))$, they must \emph{a fortiori} belong to $\mathrm{Ker}(\Pi_0(\Re(\bar{\mathfrak{L}_m}^{\ast}\mathfrak{L}_m)))$ and $\mathrm{Ker}(\Pi_{\pm 1}(\Re(\bar{\mathfrak{L}_m}^{\ast}\mathfrak{L}_m)))$.
    
    Now, if $n\geq 2$, we have 
    \begin{align*}
        ((m-1)^2-1)n^2-(m-1)^2\geq 3(m-1)^2-4>0\Longleftrightarrow m>1+\frac{2}{\sqrt{3}}=2.154700\cdots .
    \end{align*}
    Therefore, assuming that $m>1+\frac{2}{\sqrt{3}}$, we have $D_{m,n}<0$ for all $|n|\geq 2$. This implies that the roots of $Q(\sqrt{X})$ are given by
    \begin{align*}
        r_{n,1}&=(m-1)^2+n^2+1+2i\sqrt{((m-1)^2-1)n^2-(m-1)^2}\\
        r_{n,2}&=(m-1)^2+n^2+1-2i\sqrt{((m-1)^2-1)n^2-(m-1)^2}.
    \end{align*}
    Recall that if $x>0$, then for all $y>0$, the roots of $x+i\,y$ are given by $\pm(a+i\,b)$, where
    \begin{align*}
        a=\sqrt{\frac{\sqrt{x^2+y^2}+x}{2}}\quad b=\sqrt{\frac{\sqrt{x^2+y^2}-x}{2}},
    \end{align*}
    while for $y<0$, we have
    \begin{align*}
        a=\sqrt{\frac{\sqrt{x^2+y^2}+x}{2}}\quad b=-\sqrt{\frac{\sqrt{x^2+y^2}-x}{2}}
    \end{align*}
    Here, we have $x=(m-1)^2+n^2+1$ and $y=\pm 2\sqrt{((m-1)^2-1)n^2-(m-1)^2}$. 
    
    Since the real part of the roots is positive, we deduce that the roots of $Q$ are given by
    \begin{align*}
    \left\{\begin{alignedat}{1}
        r_{m,n,1}&=\sqrt{\frac{\sqrt{m^2(m-2)^2+n^2(n^2+6m(m-2)+4)}+(m-1)^2+n^2+1}{2}}\\
        &+i\sqrt{\frac{\sqrt{m^2(m-2)^2+n^2(n^2+6m(m-2)+4)}-((m-1)^2+n^2+1)}{2}}\\
        r_{m,n,2}&=-\sqrt{\frac{\sqrt{m^2(m-2)^2+n^2(n^2+6m(m-2)+4)}+(m-1)^2+n^2+1}{2}}\\
        &+i\sqrt{\frac{\sqrt{m^2(m-2)^2+n^2(n^2+6m(m-2)+4)}-((m-1)^2+n^2+1)}{2}}\\
        r_{m,n,3}&=\sqrt{\frac{\sqrt{m^2(m-2)^2+n^2(n^2+6m(m-2)+4)}+(m-1)^2+n^2+1}{2}}\\
        &-i\sqrt{\frac{\sqrt{m^2(m-2)^2+n^2(n^2+6m(m-2)+4)}-((m-1)^2+n^2+1)}{2}}\\
        r_{m,n,4}&=-\sqrt{\frac{\sqrt{m^2(m-2)^2+n^2(n^2+6m(m-2)+4)}+(m-1)^2+n^2+1}{2}}\\
        &-i\sqrt{\frac{\sqrt{m^2(m-2)^2+n^2(n^2+6m(m-2)+4)}-((m-1)^2+n^2+1)}{2}}.
    \end{alignedat}\right.
    \end{align*}
    Finally, the roots of $P$ are given by 
    \begin{align*}
        \left\{\begin{alignedat}{1}
        \lambda_{m,n,1}&=1+\sqrt{\frac{\sqrt{m^2(m-2)^2+n^2(n^2+6m(m-2)+4)}+(m-1)^2+n^2+1}{2}}\\
        &+i\sqrt{\frac{\sqrt{m^2(m-2)^2+n^2(n^2+6m(m-2)+4)}-((m-1)^2+n^2+1)}{2}}\\
        \lambda_{m,n,2}&=1-\sqrt{\frac{\sqrt{m^2(m-2)^2+n^2(n^2+6m(m-2)+4)}+(m-1)^2+n^2+1}{2}}\\
        &+i\sqrt{\frac{\sqrt{m^2(m-2)^2+n^2(n^2+6m(m-2)+4)}-((m-1)^2+n^2+1)}{2}}\\
        \lambda_{m,n,3}&=1+\sqrt{\frac{\sqrt{m^2(m-2)^2+n^2(n^2+6m(m-2)+4)}+(m-1)^2+n^2+1}{2}}\\
        &-i\sqrt{\frac{\sqrt{m^2(m-2)^2+n^2(n^2+6m(m-2)+4)}-((m-1)^2+n^2+1)}{2}}\\
        \lambda_{m,n,4}&=1-\sqrt{\frac{\sqrt{m^2(m-2)^2+n^2(n^2+6m(m-2)+4)}+(m-1)^2+n^2+1}{2}}\\
        &-i\sqrt{\frac{\sqrt{m^2(m-2)^2+n^2(n^2+6m(m-2)+4)}-((m-1)^2+n^2+1)}{2}}.
    \end{alignedat}\right.
    \end{align*}
    Noting for simplicity $\mu_{m,n}=\sqrt{m^2(m-2)^2+n^2(n^2+6m(m-2)+4)}$, we deduce that a basis of real solutions of the equation $\Pi_{n^2}(\mathfrak{D}_m)u=0$ for $|n|\geq 2$ is given by 
    \begin{align*}
        &r^{1+\sqrt{\frac{\mu_{m,n}+(m-1)^2+n^2+1}{2}}}\cos\left(\sqrt{\frac{\mu_{m,n}-((m-1)^2+n^2+1)}{2}}r\right)\\
        &r^{1+\sqrt{\frac{\mu_{m,n}+(m-1)^2+n^2+1}{2}}}\sin\left(\sqrt{\frac{\mu_{m,n}-((m-1)^2+n^2+1)}{2}}r\right)\\
        &r^{1-\sqrt{\frac{\mu_{m,n}+(m-1)^2+n^2+1}{2}}}\cos\left(\sqrt{\frac{\mu_{m,n}-((m-1)^2+n^2+1)}{2}}r\right)\\
        &r^{1-\sqrt{\frac{\mu_{m,n}+(m-1)^2+n^2+1}{2}}}\sin\left(\sqrt{\frac{\mu_{m,n}-((m-1)^2+n^2+1)}{2}}r\right).
    \end{align*}
    To simplify further the notations, write 
    \begin{align*}
        &\alpha_{m,n}=\sqrt{\frac{\mu_{m,n}+(m-1)^2+n^2+1}{2}}\\
        &\beta_{m,n}=\sqrt{\frac{\mu_{m,n}-((m-1)^2+n^2+1)}{2}}.
    \end{align*}
    Then the basis of solutions is simply written as
    \begin{align*}
        r^{1+\alpha_{m,n}}\cos(\beta_{m,n}r),r^{1+\alpha_{m,n}}\sin(\beta_{m,n}r),r^{1-\alpha_{m,n}}\cos(\beta_{m,n}r),r^{1-\alpha_{m,n}}\sin(\beta_{m,n}r).
    \end{align*}
    Notice that we obviously have $\alpha_{m,-n}=\alpha_{m,n}$ and $\beta_{m,n}=\beta_{m,-n}$ for all $n\in \Z\setminus\ens{-1,0,1}$. Finally, if $u$ solves the equation $\mathfrak{D}_mu=0$, it admits an expansion of the form (where $\alpha_0,\alpha_1,\alpha_2,\alpha_3\in \R$ and $\ens{a_n}_{n\in \Z},\ens{b_n}_{n\in \Z},\ens{c_n},\ens{d_n}\in \C$)
    \begin{align*}
        &u(r,\theta)=\alpha_0\,r^{m+1}+\alpha_1r^{m-1}+\alpha_2r^{1-m}+\alpha_3r^{3-m}\\
        &+r^{1+\sqrt{(m-1)^2+4}}\Re(a_1e^{i\theta})+r^{1-\sqrt{(m-1)^2+4}}\Re(b_1e^{i\theta})+r^{m}\Re(c_1e^{i\theta})+r^{2-m}\Re(d_1e^{i\theta})\\
        &+\sum_{n=2}^{\infty}r^{1+\alpha_{m,n}}\cos(\beta_{m,n}r)\Re\left(a_ne^{in\theta}\right)+\sum_{n=2}^{\infty}r^{1+\alpha_{m,n}}\sin(\beta_{m,n}r)\Re\left(b_ne^{in\theta}\right)\\
        &+\sum_{n=2}^{\infty}r^{1-\alpha_{m,n}}\cos(\beta_{m,n}r)\Re\left(c_ne^{in\theta}\right)+\sum_{n=2}^{\infty}r^{1-\alpha_{m,n}}\sin(\beta_{m,n}r)\Re\left(d_ne^{in\theta}\right).
    \end{align*}
    Now, the first issue is that for all $\alpha\in \C\setminus\N$, the function $r^{\alpha}\cos(r)$ does not admit a primitive expressible with standard functions, 
    Notice that
    \begin{align*}
        \alpha_{m,n}&\underset{|n|\rightarrow \infty}{\sim}|n|\\
        \beta_{m,n}&=\frac{4(((m-1)^2-1)n^2-(m-1)^2)}{\sqrt{m^2(m-2)^2+n^2(n^2+6m(m-2)+4)}+(m-1)^2+n^2+1}\\
        &\conv{|n|\rightarrow \infty}2((m-1)^2-1)=2m(m-2).
    \end{align*}
    Now, consider for $a,b,c>0$ the function
    \begin{align*}
        f(x)=\sqrt{x^2+ax+b}-(x+c).
    \end{align*}
    We have
    \begin{align*}
        f'(x)=\frac{2x+a}{\sqrt{x^2+ax+b}}-1\geq 0\Longleftrightarrow (2x+a)^2-(x^2+ax+b)\geq 0 \Longleftrightarrow 3x^2+ax+a^2-b\geq 0.
    \end{align*}
    The discriminant of $3X^2+aX+a^2-b$ is equal to $D=-11a^2+12b$. In our case of interest, we have $a=6m(m-2)+4$, $b=m^2(m-2)^2$ and $c=(m-1)^2+1$, which yields
    \begin{align*}
        D&=-11\left(6m(m-2)+4\right)^2+12m^2(m-2)^2=-16(24m^2(m-2)^2+33m(m-2)+11)<0
    \end{align*}
    for all $m\geq 2$. Therefore, $f$ is increasing and 
    \begin{align*}
        f(x)\geq f(1)=\sqrt{4+2a+b}-(c+2)\quad \text{for all}\;\, x\geq 2.
    \end{align*}
    and we have
    \begin{align*}
        4+2a+b-(c+2)^2=m^{4} - 4 \, m^{3} + 15 \, m^{2} - 22 \, m + 8>0
    \end{align*}
    for all
    \begin{align*}
        m>1+\sqrt{\frac{\sqrt{89}-9}{2}}=1.465822\cdots.
    \end{align*}
    Therefore, we have 
    \begin{align*}
        &\sqrt{m^2(m-2)^2+12m(m-2)+12}-((m-1)^2+2)\\
        &=\frac{m^{4} - 4 \, m^{3} + 15 \, m^{2} - 22 \, m}{\sqrt{m^2(m-2)^2+12m(m-2)+12}+((m-1)^2+2)}
        \leq \beta_{m,n}\leq 2m(m-2)
    \end{align*}
    for all $|n|\geq 2$. By the elementary inequality $\sin(x)\leq x$ for all $x\geq 0$, we deduce that
    \begin{align*}
        \cos(x)-1+\frac{x^2}{2}\geq 0\qquad \text{for all}\;\,x\geq 0,
    \end{align*}
    which shows that 
    \begin{align*}
        \sin(x)\geq x-\dfrac{x^3}{6}\quad \text{for all}\;\, x\geq 0.
    \end{align*}
    Therefore, for all $0\leq x\leq 1$, we have
    \begin{align*}
        &\frac{1}{2}\leq \cos(x)\leq 1\\
        &\frac{x}{2}\leq \sin(x)\leq x.
    \end{align*}
    Then, we need a more precise expansion of $\alpha_{m,n}$. We have
    \begin{align*}
        \mu_{m,n}&=\sqrt{m^2(m-2)^2+n^2(n^2+6m(m-2)+4)}=n^2\sqrt{1+\frac{2(3m(m-2)+2)}{n^2}+\frac{m^2(m-2)^2}{n^4}}\\
        &=n^2+3m(m-2)+2+O\left(\frac{1}{n^2}\right), 
    \end{align*}
    and this implies that
    \begin{align*}
        \alpha_{m,n}&=\sqrt{\frac{\mu_{m,n}+(m-1)^2+n^2+1}{2}}=
        \sqrt{n^2+2m(m-2)+2+O\left(\frac{1}{n^2}\right)}\\
        &=|n|+m(m-2)+1+O\left(\frac{1}{n^2}\right).
    \end{align*}
    This means that there exists $0<\gamma_{m}<\infty$ such that $0\leq r\leq \frac{1}{2m(m-2)}$
    \begin{align*}
        &r^{2+m(m-2)+n-\frac{\gamma_m}{n^2}}\leq r^{1+\alpha_{m,n}}\cos(\beta_{m,n}r)\leq r^{2+m(m-2)+n+\frac{\gamma_m}{n^2}}\\
        &r^{-m(m-2)+n-\frac{\gamma_m}{n^2}}\leq r^{1-\alpha_{m,n}}\cos(\beta_{m,n}r)\leq r^{-m(m-2)+n+\frac{\gamma_m}{n^2}}\\
        &\beta_{m,2}r^{3+m(m-2)+n-\frac{\gamma_m}{n^2}}\leq r^{1+\alpha_{m,n}}\sin(\beta_{m,n}r)\leq 2m(m-2)r^{3+m(m-2)+n+\frac{\gamma_m}{n^2}}\\
        &\beta_{m,2}r^{1-m(m-2)+n-\frac{\gamma_m}{n^2}}\leq r^{1-\alpha_{m,n}}\sin(\beta_{m,n}r)\leq \beta_{m}r^{1-m(m-2)+n+\frac{\gamma_m}{n^2}}.
    \end{align*}
    It means that when we integrate $u$ and its derivatives on $\Omega=B_b\setminus\bar{B}_a(0)$ (provided that $b\leq \dfrac{1}{2m(m-2)}$), $u$ behaves like the function 
    \begin{align*}
        v(z)&=\alpha_0|z|^{m+1}+\alpha_1|z|^{m-1}+\alpha_2|z|^{1-m}+\alpha_3|z|^{3-m}\\
        &+r^{\sqrt{(m-1)^2+4}}\Re(a_1z)+r^{-\sqrt{(m-1)^2+4}}\Re(b_1z)+r^{m-1}\Re(c_1z)+r^{1-m}\Re(d_1z)\\
        &+|z|^{m(m-2)+2}\sum_{n=2}^{\infty}\Re(a_nz^n)+|z|^{m(m-2)+3}\sum_{n=2}^{\infty}\Re(b_nz^n)\\
        &+|z|^{-m(m-2)}\sum_{n=2}^{\infty}\Re\left(\frac{c_n}{z^n}\right)+|z|^{1-m(m-2)}\sum_{n=2}^{\infty}\Re\left(\frac{d_n}{z^n}\right)
    \end{align*}
    which is much more reminiscent than the expression for solutions of $\leb^{\ast}_m\leb_mw=0$, whose solutions can be expanded products by harmonic functions multiplied by a fixed radial function of $m$:
    \begin{align*}
        w(z)=|z|^{1-m}\left(\alpha_0\log|z|+\Re\left(\sum_{n\in \Z}a_nz^n\right)\right)+|z|^{m+1}\left(\alpha_1\log|z|+\Re\left(\sum_{n\in \Z}b_nz^n\right)\right).
    \end{align*}
    Now, let us go back to the expression
    \begin{align*}
        &u(r,\theta)=\alpha_0\,r^{m+1}+\alpha_1r^{m-1}+\alpha_2r^{1-m}+\alpha_3r^{3-m}\\
        &+r^{1+\sqrt{(m-1)^2+4}}\Re(a_1e^{i\theta})+r^{1-\sqrt{(m-1)^2+4}}\Re(b_1e^{i\theta})+r^{m}\Re(c_1e^{i\theta})+r^{2-m}\Re(d_1e^{i\theta})\\
        &+\sum_{n=2}^{\infty}r^{1+\alpha_{m,n}}\cos(\beta_{m,n}r)\Re\left(a_ne^{in\theta}\right)+\sum_{n=2}^{\infty}r^{1+\alpha_{m,n}}\sin(\beta_{m,n}r)\Re\left(b_ne^{in\theta}\right)\\
        &+\sum_{n=2}^{\infty}r^{1-\alpha_{m,n}}\cos(\beta_{m,n}r)\Re\left(c_ne^{in\theta}\right)+\sum_{n=2}^{\infty}r^{1-\alpha_{m,n}}\sin(\beta_{m,n}r)\Re\left(d_ne^{in\theta}\right).
    \end{align*}
    We can now state the main theorem of this section.
    \begin{theorem}\label{weighted_estimate_frak_lm}
        Let $3\leq m<\infty$ and define 
        \begin{align*}
            \mathfrak{L}_m=\p{z}^2+\frac{m-1}{z}\p{z}+\frac{(m-1)(m-3)}{4z^2}.
        \end{align*}
        For all $0<\gamma<1$, for all $0<\beta<1$, and for all $0<\alpha<2$, there exists a universal constant $\Gamma_{m,\alpha,\beta,\gamma}<\infty$ and $C_{\alpha,\beta,\gamma}$ with the following property.
        Let $0<a<b<\infty$ and $\Omega=B_b\setminus\bar{B}_a(0)$. Then, provided that 
        \begin{align}
            \log\left(\frac{b}{a}\right)\geq \Gamma_{m,\alpha,\beta,\gamma},
        \end{align}
        for all $u\in W^{2,2}(\Omega)$ such that $\Re\left(\bar{\mathfrak{L}}_m^{\ast}\mathfrak{L}_m\right)u=0$, or equivalently
        \begin{align*}
            &\Delta^2u-\frac{6(m^2-1)}{|x|^2}\Delta u+4(m^2-1)\left(\frac{x}{|x|^2}\right)^t\cdot \D^2u\cdot \left(\frac{x}{|x|^2}\right)+\frac{8(m-1)^2}{|x|^2}\frac{x}{|x|^2}\cdot \D u\\
            &+\frac{(m+1)(m-1)^2(m-3)}{|x|^4}u=0,
        \end{align*}
        we have
        \begin{align*}
            &\int_{\Omega}\left|\D u+(m-1)\frac{x}{|x|^2}u\right|^2\left(\left(\frac{|x|}{b}\right)^{2\gamma}+\left(\frac{a}{|x|}\right)^{2\gamma}\right)dx\leq C_{\alpha,\beta,\gamma}\left(\int_{\Omega}\frac{u^2}{|x|^4}\left(\left(\frac{|x|}{b}\right)^{4\beta}+\left(\frac{a}{|x|}\right)^{4\beta}\right)dx\right.\\
            &\left.+\int_{\Omega}\left(\leb_mu\right)^2\left(\left(\frac{|x|}{b}\right)^{2\alpha}+\left(\frac{a}{|x|}\right)^{2\alpha}\right)dx+\int_{\Omega}\left|\mathfrak{L}_mu\right|^2\left(\left(\frac{|x|}{b}\right)^{2\alpha}+\left(\frac{a}{|x|}\right)^{2\alpha}\right)dx\right).
        \end{align*}
    \end{theorem}
    \begin{rem}
        To compare to the previous theorem, recall that the solutions of $\leb_m^{\ast}\leb_mu=0$ satisfy the equation
        \begin{align*}
            \Delta^2u+\frac{2(m^2-1)}{|x|^2}\Delta u-4(m^2-1)\left(\frac{x}{|x|^2}\right)^t\cdot \D^2u\cdot \left(\frac{x}{|x|^2}\right)+\frac{(m-1)^2}{|x|^4}u=0.
        \end{align*}
        Although formally similar, the Taylor expansion of the solution of the former system are significantly more involved. 
    \end{rem}
    \begin{proof}
    \textbf{Step 1.} Expansion and first integral. 
    
        Recall that $u$ admits the following expansion
        \begin{align*}
        &u(r,\theta)=\alpha_0\,r^{m+1}+\alpha_1\,r^{m-1}+\alpha_2\,r^{1-m}+\alpha_3\,r^{3-m}\\
        &+r^{1+\sqrt{(m-1)^2+4}}\,\Re(a_1e^{i\theta})+r^{1-\sqrt{(m-1)^2+4}}\,\Re(b_1e^{i\theta})+r^{m}\,\Re(c_1e^{i\theta})+r^{2-m}\,\Re(d_1e^{i\theta})\\
        &+\sum_{n=2}^{\infty}r^{1+\alpha_{m,n}}\cos(\beta_{m,n}r)\Re\left(a_ne^{in\theta}\right)+\sum_{n=2}^{\infty}r^{1+\alpha_{m,n}}\sin(\beta_{m,n}r)\Re\left(b_ne^{in\theta}\right)\\
        &+\sum_{n=2}^{\infty}r^{1-\alpha_{m,n}}\cos(\beta_{m,n}r)\Re\left(c_ne^{in\theta}\right)+\sum_{n=2}^{\infty}r^{1-\alpha_{m,n}}\sin(\beta_{m,n}r)\Re\left(d_ne^{in\theta}\right).
    \end{align*}
    Furthermore, the following estimates hold true, provided that $0\leq r\leq b\leq \dfrac{1}{2m(m-2)}$:
    \begin{align*}
        &\frac{1}{2}\leq \cos(\beta_{m,n}r)\leq 1\\
        &\frac{\beta_{m,n}r}{2}\leq \sin(\beta_{m,n}r)\leq \beta_{m,n}r.
    \end{align*}
    We now compute
    \begin{align*}
        &\p{r}u=(m+1)\alpha_0\,r^m+(m-1)\alpha_1\,r^{m-2}+(1-m)\alpha_2\,r^{-m}+(3-m)\alpha_3\,r^{2-m}\\
        &+\left(1+\sqrt{(m-1)^2+4}\right)r^{\sqrt{(m-1)^4+4}}\,\Re(a_1e^{i\theta})+\left(1-\sqrt{(m-1)^2+4}\right)r^{-\sqrt{(m-1)^2+4}}\,\Re\left(b_1e^{i\theta}\right)\\
        &+m\,r^{m-1}\,\Re\left(c_1e^{i\theta}\right)+(2-m)r^{1-m}\,\Re\left(d_1e^{i\theta}\right)\\
        &+\sum_{n=2}^{\infty}\left((1+\alpha_{m,n})r^{\alpha_{m,n}}\cos(\beta_{m,n}r)-\beta_{m,n}r^{1+\alpha_{m,n}}\sin(\beta_{m,n}r)\right)\Re\left(a_ne^{in\theta}\right)\\
        &+\sum_{n=2}^{\infty}\left((1+\alpha_{m,n})r^{\alpha_{m,n}}\sin(\beta_{m,n}r)+\beta_{m,n}r^{1+\alpha_{m,n}}\cos(\beta_{m,n}r)\right)\Re\left(b_ne^{in\theta}\right)\\
        &+\sum_{n=2}^{\infty}\left((1-\alpha_{m,n})r^{-\alpha_{m,n}}\cos(\beta_{m,n}r)-\beta_{m,n}r^{1-\alpha_{m,n}}\sin(\beta_{m,n}r)\right)\Re\left(c_ne^{in\theta}\right)\\
        &+\sum_{n=2}^{\infty}\left((1-\alpha_{m,n})r^{-\alpha_{m,n}}\sin(\beta_{m,n}r)+\beta_{m,n}r^{1-\alpha_{m,n}}\cos(\beta_{m,n}r)\right)\Re\left(d_ne^{in\theta}\right).
    \end{align*}
    Therefore, we have
    \begin{align*}
        &\p{r}u+\frac{m-1}{r}u=2m\,\alpha_0r^m+2(m-1)\alpha_1r^{m-2}+2\,\alpha_3r^{2-m}\\
        &+\left(m+\sqrt{(m-1)^2+4}\right)r^{\sqrt{(m-1)^4+4}}\,\Re\left(a_1e^{i\theta}\right)+\left(m-\sqrt{(m-1)^2+4}\right)r^{-\sqrt{(m-1)^4+4}}\,\Re\left(b_1e^{i\theta}\right)\\
        &+(2m-1)r^{m-1}\Re\left(c_1e^{i\theta}\right)+r^{1-m}\Re\left(d_1e^{i\theta}\right)\\
        &+\sum_{n=2}^{\infty}\left((m+\alpha_{m,n})r^{\alpha_{m,n}}\cos(\beta_{m,n}r)-\beta_{m,n}r^{1+\alpha_{m,n}}\sin(\beta_{m,n}r)\right)\Re\left(a_ne^{in\theta}\right)\\
        &+\sum_{n=2}^{\infty}\left((m+\alpha_{m,n})r^{\alpha_{m,n}}\sin(\beta_{m,n}r)+\beta_{m,n}r^{1+\alpha_{m,n}}\cos(\beta_{m,n}r)\right)\Re\left(b_ne^{in\theta}\right)\\
        &+\sum_{n=2}^{\infty}\left((m-\alpha_{m,n})r^{-\alpha_{m,n}}\cos(\beta_{m,n}r)-\beta_{m,n}r^{1-\alpha_{m,n}}\sin(\beta_{m,n}r)\right)\Re\left(c_ne^{in\theta}\right)\\
        &+\sum_{n=2}^{\infty}\left((m-\alpha_{m,n})r^{-\alpha_{m,n}}\sin(\beta_{m,n}r)+\beta_{m,n}r^{1-\alpha_{m,n}}\cos(\beta_{m,n}r)\right)\Re\left(d_ne^{in\theta}\right).
    \end{align*}
    Notice that for $m=\dfrac{5}{2}$, $m-\sqrt{(m-1)^2+4}=0$, so we will not have to estimate this term in this specific case. Thanks to the elementary estimate
    \begin{align*}
        \left(\sum_{i=1}^pa_i\right)^2\leq p\sum_{i=1}^pa_i^2\qquad \text{for all}\;\, (a_1,\cdots,a_p)\in \R^p,
    \end{align*}
    and Parseval's identity from Fourier series, we deduce that 
    \begin{align}\label{norm_op_u}
        &\int_{\Omega}\frac{1}{|x|^2}\left|\p{r}u+\frac{m-1}{r}u\right|^2|x|^{2\gamma}dx\leq 6\pi\left(4m^2\alpha_0^2\int_{a}^br^{2m-1+2\gamma}dr+4(m-1)^2\alpha_1^2\int_{a}^br^{2m-5+2\gamma}\right.\nonumber\\
        &\left.+4\alpha_3^2\int_{a}^{b}r^{3-2m+2\gamma}dr\right)\nonumber\\
        &+4\pi\left(\left(m+\sqrt{(m-1)^2+4}\right)^2|a_1|^2\int_{a}^br^{2\sqrt{(m-1)^2+4}-1+2\gamma}dr\right.\nonumber\\
        &\left.+\left(m-\sqrt{(m-1)^2+4}\right)^2|b_1|^2\int_{a}^br^{-1-2\sqrt{(m-1)^2+4}+2\gamma}dr\right.\nonumber\\
        &\left.+(2m-1)^2|c_1|^2\int_{a}^br^{2m-3+2\gamma}dr+|d_1|^2\int_{a}^br^{1-2m+2\gamma}dr\right)\nonumber\\
        &+8\pi\left(\sum_{n=2}^{\infty}|a_n|^2\int_{a}^b\left((m+\alpha_{m,n})^2r^{2\alpha_{m,n}-1+2\gamma}\cos^2(\beta_{m,n}r)+\beta_{m,n}^2r^{2\alpha_{m,n}+1+2\gamma}\sin^2(\beta_{m,n}r)\right)dr\right.\nonumber\\
        &+\sum_{n=2}^{\infty}|b_n|^2\int_{a}^b\left((m+\alpha_{m,n})^2r^{2\alpha_{m,n}-1+2\gamma}\sin^2(\beta_{m,n}r)+\beta_{m,n}^2r^{2\alpha_{m,n}+1+2\gamma}\cos^2(\beta_{m,n}r)\right)dr\nonumber\\
        &+\sum_{n=2}^{\infty}|c_n|^2\int_{a}^b\left((m+\alpha_{m,n})^2r^{-2\alpha_{m,n}-1+2\gamma}\cos^2(\beta_{m,n}r)+\beta_{m,n}^2r^{-2\alpha_{m,n}+1+2\gamma}\sin^2(\beta_{m,n}r)\right)dr\nonumber\\
        &\left.+\sum_{n=2}^{\infty}|d_n|^2\int_{a}^b\left((m+\alpha_{m,n})^2r^{-2\alpha_{m,n}-1+2\gamma}\sin^2(\beta_{m,n}r)+\beta_{m,n}^2r^{-2\alpha_{m,n}+1+2\gamma}\cos^2(\beta_{m,n}r)\right)dr\right)\nonumber\\
        &\leq 6\pi\left(\frac{2m^2\alpha_0^2}{m+\gamma}b^{2(m+\gamma)}\left(1-\left(\frac{a}{b}\right)^{2(m+\gamma)}\right)+\frac{2(m-1)^2\alpha_1^2}{m-2+\gamma}b^{2(m-2+\gamma)}\left(1-\left(\frac{a}{b}\right)^{2(m-2+\gamma)}\right)\right.\nonumber\\
        &\left.+\frac{2\alpha_3^2}{m-2-\gamma}\frac{1}{a^{2(m-2-\gamma)}}\left(1-\left(\frac{a}{b}\right)^{2(m-2-\gamma)}\right)\right)\nonumber\\
        &+2\pi\left(\frac{\left(m+\sqrt{(m-1)^2+4}\right)^2}{\sqrt{(m-1)^2+4}+\gamma}|a_1|^2b^{2\left(\sqrt{(m-1)^2+4}+\gamma\right)}\left(1-\left(\frac{a}{b}\right)^{2\left(\sqrt{(m-1)^2+4}+\gamma\right)}\right)\right.\nonumber\\
        &+\frac{\left(m-\sqrt{(m-1)^2+4}\right)^2}{\sqrt{(m-1)^2+4}-\gamma}|b_1|^2\frac{1}{a^{2\left(\sqrt{(m-1)^2+4}-\gamma\right)}}\left(1-\left(\frac{a}{b}\right)^{2\left(\sqrt{(m-1)^2+4}-\gamma\right)}\right)\nonumber\\
        &+\frac{(2m-1)^2}{m-1+\gamma}|c_1|^2b^{2(m-1+\gamma)}\left(1-\left(\frac{a}{b}\right)^{2(m-1+\gamma)}\right)+\frac{|d_1|^2}{m-1-\gamma}\frac{1}{a^{2(m-1-\gamma)}}\left(1-\left(\frac{a}{b}\right)^{2(m-1-\gamma)}\right)\nonumber\\
        &+4\pi\left(\sum_{n=2}^{\infty}|a_n|^2\left(\frac{(m+\alpha_{m,n})^2}{\alpha_{m,n}+\gamma}b^{2(\alpha_{m,n}+\gamma)}\left(1-\left(\frac{a}{b}\right)^{2(\alpha_{m,n}+\gamma)}\right)\right.\right.\nonumber\\
        &\left.+\frac{\beta_{m,n}^4}{\alpha_{m,n}+1+\gamma}b^{2(\alpha_{m,n}+1+\gamma)}\left(1-\left(\frac{a}{b}\right)^{2(\alpha_{m,n}+1+\gamma)}\right)\right)\nonumber\\
        &+\sum_{n=2}^{\infty}\frac{(m+\alpha_{m,n})^2+\beta_{m,n}^4}{\alpha_{m,n}+1+\gamma}|b_n|^2b^{2(\alpha_{m,n}+1+\gamma)}\left(1-\left(\frac{a}{b}\right)^{2(\alpha_{m,n}+1+\gamma)}\right)\nonumber\\
        &+\sum_{n=2}^{\infty}|c_n|^2\left(\frac{(m+\alpha_{m,n})^2}{\alpha_{m,n}-\gamma}\frac{1}{a^{2(\alpha_{m,n}-\gamma)}}\left(1-\left(\frac{a}{b}\right)^{2(\alpha_{m,n}-\gamma)}\right)\right.\nonumber\\
        &\left.+\frac{\beta_{m,n}^4}{\alpha_{m,n}-2-\gamma}\frac{1}{a^{2(\alpha_{m,n}-2-\gamma)}}\left(1-\left(\frac{a}{b}\right)^{2(\alpha_{m,n}-2-\gamma)}\right)\right)\nonumber\\
        &\left.+\sum_{n=2}^{\infty}\frac{(m+\alpha_{m,n})^2+\beta_{m,n}^4}{\alpha_{m,n}-2-\gamma}|d_n|^2\frac{1}{a^{2\left(\alpha_{m,n}-2-\gamma\right)}}\left(1-\left(\frac{a}{b}\right)^{2(\alpha_{m,n}-2-\gamma)}\right)\right).
    \end{align}

    \textbf{Step 2.} The radial component.

    Now, make the decomposition $u=\mathrm{Rad}(u)+u_0$, where $\mathrm{Rad}(u)$ is the radial component. 
    We have 
    \begin{align*}
        \mathrm{Rad}(u)=\alpha_0r^{m+1}+\alpha_1r^{m-1}+\alpha_2r^{1-m}+\alpha_3r^{3-m}.
    \end{align*}
    We have
    \begin{align*}
        &\p{r}\mathrm{Rad}(u)=(m+1)\alpha_0r^{m}+(m-1)\alpha_1r^{m-2}+(1-m)\alpha_2r^{-m}+(3-m)\alpha_3r^{2-m}\\
        &\frac{1}{r}\left(\p{r}\mathrm{Rad}(u)+\frac{(m-1)}{r}u\right)=2\left(m\,\alpha_0r^{m-1}+(m-1)\alpha_1r^{m-3}+\alpha_3r^{1-m}\right).
    \end{align*}
    On the other hand, we have for all $\alpha\in \R$ 
    \begin{align*}
        \p{z}r^{\alpha}&=\p{r}(z\z)^{\frac{\alpha}{2}}=\frac{\alpha}{2}\z|z|^{\alpha-2}=\frac{\alpha}{2z}r^{\alpha}\\
        \p{z}^2r^{\alpha}&=\frac{\alpha(\alpha-2)}{4}\z^2|z|^{\alpha-4}=\frac{\alpha(\alpha-2)}{4z^2}r^{\alpha}.
    \end{align*}
    Therefore, we have
    \begin{align*}
        \mathfrak{L}_mr^{\alpha}&=\left(\p{z}^2+\frac{m-1}{z}\p{z}+\frac{(m-1)(m-3)}{4z^2}\right)=\left(\frac{\alpha(\alpha-2)}{4}+\frac{(m-1)\alpha}{2}+\frac{(m-1)(m-3)}{4}\right)\frac{r^{\alpha}}{z^2}\\
        &=\frac{1}{4z^2}\left(\alpha+m-1\right)\left(\alpha+m-3\right)r^{\alpha},
    \end{align*}
    which yields
    \begin{align*}
        \dfrac{4z^2}{|z|^2}\mathfrak{L}_mr^{\alpha}=(\alpha+m-1)(\alpha+m-3)r^{\alpha-2}.
    \end{align*}
    Therefore, $r^{1-m},r^{3-m}\in \mathrm{Ker}(\mathfrak{L}_m)$ and we have
    \begin{align*}
        \dfrac{4z^2}{|z|^2}\mathfrak{L}_m\mathrm{Rad}(u)=(m-1)\left(m\,\alpha_0r^{m-1}+(m-2)\alpha_1r^{m-3}\right).
    \end{align*}
    On the other hand, we have
    \begin{align*}
        \leb_m=\Delta+2(m-1)\frac{x}{|x|^2}\cdot \D+\frac{(m-1)^2}{|x|^2}=\p{r}^2+\frac{(2m-1)}{r}\p{r}+\frac{(m-1)^2}{r^2}+\frac{1}{r^2}\p{\theta}^2.
    \end{align*}
    Therefore, we have
    \begin{align*}
        \leb_mr^{\alpha}&=\left(\alpha(\alpha-1)+(2m-1)\alpha+(m-1)\right)^2r^{\alpha-2}\\
        &=\left(\alpha^2+2(m-1)\alpha+(m-1)^2\right)r^{\alpha-2}=(\alpha+m-1)^2r^{\alpha-2}.
    \end{align*}
    Therefore, we have
    \begin{align*}
        \leb_m\mathrm{Rad}(u)=4\left(m^2\alpha_0r^{m-1}+(m-1)^2\alpha_1r^{m-3}+\alpha_3r^{1-m}\right).
    \end{align*}
    First, we estimate
    \begin{align}\label{estimate_rad_nabla_u}
        &\int_{\Omega}\frac{1}{|x|^2}\left|\p{r}\mathrm{Rad}(u)+\frac{m-1}{r}\mathrm{Rad}(u)\right|^2dx\leq 24\pi\int_{a}^b\left(m^2\alpha_0^2r^{2m-2}+(m-1)^2\alpha_1^2r^{2m-6}+\alpha_3^2r^{2-2m}\right)r\,dr\nonumber\\
        &=12\pi\left(m\alpha_0^2b^{2m}\left(1-\left(\frac{a}{b}\right)^{2m}\right)+(m-1)\alpha_1^2b^{2(m-2)}\left(1-\left(\frac{a}{b}\right)^{2(m-2)}\right)\right.\nonumber\\
        &\left.+\alpha_3^2\frac{1}{a^{2(m-2)}}\left(1-\left(\frac{a}{b}\right)^{2(m-2)}\right)\right).
    \end{align}
    Now, we have
    \begin{align*}
        &\int_{\Omega}\left|\mathfrak{L}_m\mathrm{Rad}(u)\right|^2\left(\frac{|x|}{b}\right)^{2\alpha}dx\\
        &=2\pi(m-1)^2\int_{a}^b\left(m^2\alpha_0^2r^{2m-2}+(m-2)^2\alpha_1^2r^{2m-6}+2m(m-2)\alpha_0\alpha_1r^{2m-4}\right)r^{1+2\alpha}\frac{dx}{b^{2\alpha}}\\
        &=\pi(m-1)^2\left(\frac{m^2}{m+\alpha}\alpha_0^2b^{2m}\left(1-\left(\frac{a}{b}\right)^{2(m+\alpha)}\right)+\frac{(m-2)^2}{m-2+\alpha}\alpha_1^2b^{2(m-2)}\left(1-\left(\frac{a}{b}\right)^{2(m-2+\alpha)}\right)\right.\\
        &\left.+\frac{2m(m-2)}{m-1+\alpha}\alpha_0\alpha_1b^{2(m-1)}\left(1-\left(\frac{a}{b}\right)^{2(m-1+\alpha)}\right)\right).
    \end{align*}
    We estimate
    \begin{align*}
        &\frac{m^2}{m+\alpha}\alpha_0^2b^{2m}\left(1-\left(\frac{a}{b}\right)^{2(m+\alpha)}\right)+\frac{(m-2)^2}{m-2+\alpha}\alpha_1^2b^{2(m-2)}\left(1-\left(\frac{a}{b}\right)^{2(m-2+\alpha)}\right)\\
        &+\frac{2m(m-2)}{m-1+\alpha}\alpha_0\alpha_1b^{2(m-1)}\left(1-\left(\frac{a}{b}\right)^{2(m-1+\alpha)}\right)\\
        &\geq \frac{m^2}{m+\alpha}\alpha_0^2b^{2m}\left(1-\left(\frac{a}{b}\right)^{2(m+\alpha)}\right)-\frac{m^2(m-2)(m-2+\alpha)}{(m-1+\alpha)^2}\alpha_0^2b^{2m}\frac{\left(1-\left(\frac{a}{b}\right)^{2(m-1+\alpha)}\right)^2}{1-\left(\frac{a}{b}\right)^{2(m-2+\alpha)}}\\
        &=\frac{1}{1-\left(\frac{a}{b}\right)^{2(m-2+\alpha)}}\frac{\alpha_0^2}{(m+\alpha)(m-1+\alpha)^2}b^{2m}\left(m^2-2m^2(m-1+\alpha)^2\left(\frac{a}{b}\right)^{2(m-1+\alpha)}\right).
    \end{align*}
    Therefore, we impose the following constraint on the conformal class:
    \begin{align*}
        \log\left(\frac{b}{a}\right)\geq \frac{1}{m-1+\alpha}\log\left(\frac{m-1+\alpha}{2}\right),
    \end{align*}
    which yields
    \begin{align*}
        &\frac{m^2}{m+\alpha}\alpha_0^2b^{2m}\left(1-\left(\frac{a}{b}\right)^{2(m+\alpha)}\right)+\frac{(m-2)^2}{m-2+\alpha}\alpha_1^2b^{2(m-2)}\left(1-\left(\frac{a}{b}\right)^{2(m-2+\alpha)}\right)\\
        &+\frac{2m(m-2)}{m-1+\alpha}\alpha_0\alpha_1b^{2(m-1)}\left(1-\left(\frac{a}{b}\right)^{2(m-1+\alpha)}\right)\geq \frac{m^2\alpha_0^2}{2(m+\alpha)(m-1+\alpha)^2}b^{2m}.
    \end{align*}
    Likewise, if
    \begin{align*}
        \log\left(\frac{b}{a}\right)\geq \frac{1}{m-1+\alpha}\log\left(\frac{m-2+\alpha}{2}\right),
    \end{align*}
    we get
    \begin{align*}
        &\frac{m^2}{m+\alpha}\alpha_0^2b^{2m}\left(1-\left(\frac{a}{b}\right)^{2(m+\alpha)}\right)+\frac{(m-2)^2}{m-2+\alpha}\alpha_1^2b^{2(m-2)}\left(1-\left(\frac{a}{b}\right)^{2(m-2+\alpha)}\right)\\
        &+\frac{2m(m-2)}{m-1+\alpha}\alpha_0\alpha_1b^{2(m-1)}\left(1-\left(\frac{a}{b}\right)^{2(m-1+\alpha)}\right)\geq \frac{(m-2)^2\alpha_1^2}{2(m-2+\alpha)(m-1+\alpha)^2}b^{2(m-2)}
    \end{align*}
    and finally
    \begin{align}\label{frak_lm_1}
        \int_{\Omega}\left|\mathfrak{L}_m\mathrm{Rad}(u)\right|^2\left(\frac{|x|}{b}\right)^{2\alpha}dx\geq \frac{\pi \,m^2(m-1)^2\alpha_0^2}{4(m+\alpha)(m-1+\alpha)^2}b^{2m}+\frac{\pi(m-1)^2(m-2)^2\alpha_1^2}{4(m-2+\alpha)(m-1+\alpha)^2}b^{2(m-2)}.
    \end{align}
    Now, using the elementary inequality
    \begin{align*}
        2ab\leq \epsilon\,a^2+\frac{1}{\epsilon}b^2\qquad \forall\;\, a,b\in \R\,\;\forall\epsilon>0,
    \end{align*}
    we deduce that $(a+b)^2\geq \dfrac{1}{2}a^2-b^2$, which yields
    \begin{align*}
        \left(\leb_m\mathrm{Rad}(u)\right)^2\geq 16\left(\frac{1}{2}\alpha_3^2r^{2-2m}-2m^4\alpha_0^2r^{2m-2}-2(m-1)^4\alpha_1^2r^{2m-6}\right),
    \end{align*}
    and we get
    \begin{align}\label{frak_lm_2}
        &\int_{\Omega}\left(\leb_m\mathrm{Rad}(u)\right)^2\left(\frac{a}{|x|}\right)^{2\alpha}dx\nonumber\\
        &\geq 16\pi\,a^{2\alpha}\int_{a}^b\alpha_3^2r^{3-2m-2\alpha}dr-64\pi\int_{a}^b\left(m^4\alpha_0^2r^{2m-1}+(m-1)^4\alpha_1^2r^{2m-5}\right)dr\nonumber\\
        &=\frac{8\pi\alpha_3^2}{m-2+\alpha}\frac{1}{a^{2(m-2)}}\left(1-\left(\frac{a}{b}\right)^{2(m-2+\alpha)}\right)-32\pi\,m^3\alpha_0^2b^{2m}\left(1-\left(\frac{a}{b}\right)^{2m}\right)\nonumber\\
        &-32\pi(m-1)^3\alpha_1^2b^{2(m-2)}\left(1-\left(\frac{a}{b}\right)^{2(m-2)}\right)\nonumber\\
        &\geq \frac{8\pi\alpha_3^2}{m-2+\alpha}\frac{1}{a^{2(m-2)}}\left(1-\left(\frac{a}{b}\right)^{2(m-2+\alpha)}\right)\nonumber\\
        &-128(m-1+\alpha)^2\left(\frac{m(m+\alpha)}{(m-1)^2}+\frac{(m-1)(m-2+\alpha)}{(m-2)^2}\right)\int_{\Omega}\left|\mathfrak{L}_m\mathrm{Rad}(u)\right|^2\left(\frac{|x|}{b}\right)^{2\alpha}dx,
    \end{align}
    where we used \eqref{frak_lm_1}. Therefore, we impose
    \begin{align*}
        \log\left(\frac{b}{a}\right)\geq \frac{\log(2)}{2(m-2+\alpha)}
    \end{align*}
    to get
    \begin{align}\label{frak_lm_3}
        &\frac{4\pi\alpha_3^2}{m-2+\alpha}\frac{1}{a^{2(m-2)}}\leq \int_{\Omega}\left(\leb_m\mathrm{Rad}(u)\right)^2\left(\frac{a}{|x|}\right)^{2\alpha}dx\nonumber\\
        &+128(m-1+\alpha)^2\left(\frac{m(m+\alpha)}{(m-1)^2}+\frac{(m-1)(m-2+\alpha)}{(m-2)^2}\right)\int_{\Omega}\left|\mathfrak{L}_m\mathrm{Rad}(u)\right|^2\left(\frac{|x|}{b}\right)^{2\alpha}dx.
    \end{align}
    Therefore, we deduce by \eqref{estimate_rad_nabla_u}, \eqref{frak_lm_1}, and \eqref{frak_lm_3} that
    \begin{align}\label{frak_lm_4}
        &\int_{\Omega}\frac{1}{|x|^2}\left|\p{r}\mathrm{Rad}(u)+\frac{m-1}{r}\mathrm{Rad}(u)\right|^2dx\leq 3(m-2+\alpha)\int_{\Omega}\left(\leb_m\mathrm{Rad}(u)\right)^2\left(\frac{a}{|x|}\right)^{2\alpha}dx\nonumber\\
        &+384(m-2+\alpha)(m-1+\alpha)^2\left(\frac{m(m+\alpha)}{(m-1)^2}+\frac{(m-1)(m-2+\alpha)}{(m-2)^2}\right)\int_{\Omega}\left|\mathfrak{L}_m\mathrm{Rad}(u)\right|^2\left(\frac{|x|}{b}\right)^{2\alpha}dx\nonumber\\
        &+48(m-1+\alpha)^2\left(\frac{m+\alpha}{m(m-1)^2}+\frac{m-2+\alpha}{(m-1)(m-2)^2}\right)\int_{\Omega}\left|\mathfrak{L}_m\mathrm{Rad}(u)\right|^2\left(\frac{|x|}{b}\right)^{2\alpha}dx.
    \end{align}

    \textbf{Step 3.} The first Fourier mode.
    
    Now, let us estimate the first frequency since it also has a special status. The difficulty here is that $r^{2-m}\cos(\theta),r^{2-m}\sin(\theta)\in \mathrm{Ker}(\leb_m)\cap\mathrm{Ker}(\mathfrak{L}_m)$, so we will also need the weighted $L^2$ norm of $u$ to control all components.  

    Now, recalling \eqref{elementary_cauchy_riemann}, we have
    \begin{align*}
        \p{z}=\frac{1}{2}\left(e^{-i\theta}\p{r}-i\,e^{-i\theta}\frac{1}{r}\p{\theta}\right).
    \end{align*}
    Therefore, for all $z\in \Z$, we have
    \begin{align*}
        &\p{z}\left(r^{\alpha}e^{in\theta}\right)=\frac{\alpha+n}{2}r^{\alpha-1}e^{i(n-1)\theta}\\
        &\p{z}^2\left(r^{\alpha}e^{in\theta}\right)=\frac{(\alpha+n)(\alpha-2+n)}{4}r^{\alpha-2}e^{i(n-2)\theta},
    \end{align*}
    which yields
    \begin{align*}
        &\mathfrak{L}_m\left(r^{\alpha}e^{in\theta}\right)=\left(\p{z}^2+\frac{m-1}{z}\p{z}+\frac{(m-1)(m-3)}{4z^2}\right)\left(r^{\alpha}e^{in\theta}\right)\\
        &=\frac{(\alpha+n)(\alpha-2+n)}{4}r^{\alpha-2}e^{i(n-2)\theta}+\frac{(m-1)(\alpha+n)}{2}r^{\alpha-2}e^{i(n-2)\theta}+\frac{(m-1)(m-3)}{4}r^{\alpha-2}e^{i(n-2)\theta}\\
        &=\frac{(\alpha+m+n-1)(\alpha+m+n-3)}{4}r^{\alpha-2}e^{i(n-2)\theta},
    \end{align*}
    where we used that
    \begin{align*}
        &(\alpha+n)(\alpha+n-2)+2(m-1)(\alpha+n)+(m-1)(m-3)\\
        &=(\alpha+n)(\alpha+n-2+(m-1))+(m-1)((\alpha+n)+(m-3))=(\alpha+m+n-1)(\alpha+m+n-3).
    \end{align*}
    For Fourier consideration, it is useful to consider instead $e^{2i\theta}\mathfrak{L}_m$, and we compute
    \begin{align*}
        8\frac{z^2}{|z|^2}\mathfrak{L}_m\left(r^{\alpha}\Re\left(a\,e^{in\theta}\right)\right)
        &=(\alpha+m+n-1)(\alpha+m+n-3)a\,r^{\alpha-2}e^{in\theta}\\
        &+(\alpha+m-n-1)(\alpha+m-n-3)\bar{a}\,r^{\alpha-2}e^{-in\theta}.
    \end{align*}
    We see in particular that no element belongs to the Kernel provided that $|n|\geq 2$. Write $u_1$ the function associated with the $1$st Fourier frequency. Explicitly, we have
    \begin{align*}
        u_1(r,\theta)=r^{1+\sqrt{(m-1)^2+4}}\,\Re\left(a_1\,e^{i\theta}\right)+r^{1-\sqrt{(m-1)^2+4}}\,\Re\left(b_1\,e^{i\theta}\right)+r^m\,\Re\left(c_1\,e^{i\theta}\right)+r^{2-m}\,\Re\left(d_1\,e^{i\theta}\right).
    \end{align*}
    Recall that 
    \begin{align*}
        \leb_m&=\p{r}^2+\frac{(2m-1)}{r}\p{r}+\frac{(m-1)^2}{r^2}+\frac{1}{r^2}\p{\theta}^2\\
        \leb_m\,r^{\alpha}&=(\alpha+m-1)^2\,r^{\alpha-2}.
    \end{align*}
    Therefore, we get
    \begin{align*}
         \leb_m\left(r^{\alpha}\Re\left(a\,e^{in\theta}\right)\right)&=(\alpha+m-1)^2\,r^{\alpha-2}\Re\left(a\,e^{in\theta}\right)-n^2\,r^{\alpha-2}\Re\left(a\,e^{in\theta}\right)\\
         &=(\alpha+(m-1)+n)(\alpha+(m-1)-n)r^{\alpha-2}\Re\left(a\,e^{in\theta}\right).
    \end{align*}
    We deduce that 
    \begin{align*}
        \leb_mu_1&=\left(m+1+\sqrt{(m-1)^2+4}\right)\left(m-1+\sqrt{(m-1)^2+4}\right)r^{\sqrt{(m-1)^2+4}-1}\,\Re\left(a_1\,e^{i\theta}\right)\\
        &+\left(\sqrt{(m-1)^2+4}-(m+1)\right)\left(\sqrt{(m-1)^2+4}-(m-1)\right)r^{-1-\sqrt{(m-1)^2+4}}\,\Re\left(b_1\,e^{i\theta}\right)\\
        &+(m^2-1)r^{m-2}\,\Re\left(c_1\,e^{i\theta}\right)\\
        &=\left(m+1+\sqrt{(m-1)^2+4}\right)\left(m-1+\sqrt{(m-1)^2+4}\right)r^{\sqrt{(m-1)^2+4}-1}\,\Re\left(a_1\,e^{i\theta}\right)\\
        &-\frac{16(m-1)}{\left(\sqrt{(m-1)^2+4}+(m-1)^2\right)\left(\sqrt{(m-1)^2+4}+(m+1)^2\right)}r^{-1-\sqrt{(m-1)^2+4}}\,\Re\left(b_1\,e^{i\theta}\right)\\
        &+(m^2-1)\,r^{m-2}\,\Re\left(c_1\,e^{i\theta}\right).
    \end{align*}
    We deduce that
    \small
    \begin{align*}
        &\int_{\Omega}\left(\leb_mu_1\right)^2\left(\frac{|x|}{b}\right)^{2\alpha}dx\\
        &=\pi\left(m+1+\sqrt{(m-1)^2+4}\right)^2\left(m-1+\sqrt{(m-1)^2+4}\right)^2|a_1|^2\int_{a}^br^{2\sqrt{(m-1)^2+4}-1+2\alpha}\frac{dr}{b^{2\alpha}}\\
        &+\frac{256\pi(m-1)^2|b_1|^2}{\left(\sqrt{(m-1)^2+4}+(m-1)^2\right)^2\left(\sqrt{(m-1)^2+4}+(m+1)^2\right)^2}\int_{a}^{b}r^{-1-2\sqrt{(m-1)^2+4}+2\alpha}\frac{dr}{b^{2\alpha}}\\
        &+\pi(m^2-1)^2|c_1|^2\int_{a}^{b}r^{2m-3+2\alpha}\frac{dr}{b^{2\alpha}}\\
        &-32\pi\frac{(m-1)\left(m+1+\sqrt{(m-1)^2+4}\right)\left(m-1+\sqrt{(m-1)^2+4}\right)}{\left(\sqrt{(m-1)^2+4}+(m-1)^2\right)\left(\sqrt{(m-1)^2+4}+(m+1)^2\right)}\Re\left(a_1\bar{b_1}\right)\int_{a}^{b}r^{-1+2\alpha}\frac{dr}{b^{2\alpha}}\\
        &+2\pi(m^2-1)\left(m+1+\sqrt{(m-1)^2+4}\right)\left(m-1+\sqrt{(m-1)^2+4}\right)\Re\left(a_1\bar{c_1}\right)\int_{a}^br^{m-2+2\alpha+\sqrt{(m-1)^2+4}}\frac{dr}{b^{2\alpha}}\\
        &-\frac{32\pi(m-1)(m^2-1)}{\left(\sqrt{(m-1)^2+4}+(m-1)^2\right)\left(\sqrt{(m-1)^2+4}+(m+1)^2\right)}\Re\left(b_1\bar{c_1}\right)\int_{a}^{b}r^{m-2+2\alpha-\sqrt{(m-1)^2+4}}\frac{dr}{b^{2\alpha}}\\
        &=\frac{\pi\left(m+1+\sqrt{(m-1)^2+4}\right)^2\left(m-1+\sqrt{(m-1)^2+4}\right)^2}{2\left(\sqrt{(m-1)^2+4}+\alpha\right)}|a_1|^2b^{2\sqrt{(m-1)^2+4}}\left(1-\left(\frac{a}{b}\right)^{2\left(\sqrt{(m-1)^2+4}+\alpha\right)}\right)\\
        &+\frac{128\pi(m-1)^2|b_1|^2}{\left(\sqrt{(m-1)^2+4}+(m-1)^2\right)^2\left(\sqrt{(m-1)^2+4}+(m+1)^2\right)^2\left(\sqrt{(m-1)^2+4}-\alpha\right)}\\
        &\times \left(\frac{a}{b}\right)^{2\alpha}\frac{1}{a^{2\sqrt{(m-1)^2+4}}}\left(1-\left(\frac{a}{b}\right)^{2\left(\sqrt{(m-1)^2+4}-\alpha\right)}\right)
        +\frac{\pi(m^2-1)^2}{2(m-1+\alpha)}|c_1|^2b^{2(m-1)}\left(1-\left(\frac{a}{b}\right)^{2(m-1+\alpha)}\right)\\
        &-\frac{16\pi}{\alpha}\frac{(m-1)\left(m+1+\sqrt{(m-1)^2+4}\right)\left(m-1+\sqrt{(m-1)^2+4}\right)}{\left(\sqrt{(m-1)^2+4}+(m-1)^2\right)\left(\sqrt{(m-1)^2+4}+(m+1)^2\right)}\Re\left(a_1\bar{b_1}\right)\left(\frac{a}{b}\right)^{2\alpha}\left(1-\left(\frac{a}{b}\right)^{2\alpha}\right)\\
        &+\frac{2\pi\left(m^2-1\right)\left(m+1+\sqrt{(m-1)^2+4}\right)\left(m-1+\sqrt{(m-1)^2+4}\right)}{m-1+2\alpha+\sqrt{(m-1)^2+4}}\Re\left(a_1\bar{c_1}\right)b^{m-1+\sqrt{(m-1)^2+4}}\\
        &\times\left(1-\left(\frac{a}{b}\right)^{m+2\alpha+\sqrt{(m-1)^2+4}}\right)\\
        &-\frac{32\pi(m-1)\left(m^2-1\right)\Re\left(b_1\bar{c_1}\right)}{\left(\sqrt{(m-1)^2+4}+(m-1)^2\right)\left(\sqrt{(m-1)^2+4}+(m+1)^2\right)\left((m-1+2\alpha)-\sqrt{(m-1)^2+4}\right)}\\
        &\times \frac{1}{b^{\sqrt{(m-1)^2+4}-(m-1)}}\left(1-\left(\frac{a}{b}\right)^{2\alpha}\right),
    \end{align*}
    \normalsize
    where we used that $\sqrt{(m-1)^2+4}<m-1+2\alpha$ for all
    \begin{align*}
        \alpha>f(m)=\frac{2}{\sqrt{(m-1)^2+4}+(m-1)^2}.
    \end{align*}
    Since this latter function $f$ of $m$ is strictly decreasing (and converges to $0$ as $m\rightarrow \infty$), we can always find a value $0<\alpha<1$ for which the estimate will be satisfied since
    \begin{align*}
        f(3)=\frac{1}{2+\sqrt{2}}=0.2928932\cdots<1.
    \end{align*}
    Now, we consider the following quantity
    \begin{align*}
        &\frac{\left(m+1+\sqrt{(m-1)^2+4}\right)^2\left(m-1+\sqrt{(m-1)^2+4}\right)^2}{2\left(\sqrt{(m-1)^2+4}+\alpha\right)}|a_1|^2b^{2\sqrt{(m-1)^2+4}}\left(1-\left(\frac{a}{b}\right)^{2\left(\sqrt{(m-1)^2+4}+\alpha\right)}\right)\\
        &+\frac{(m^2-1)}{2(m-1+\alpha)}|c_1|^2b^{2(m-1)}\left(1-\left(\frac{a}{b}\right)^{2(m-1+\alpha)}\right)\\
        &+\frac{2\left(m^2-1\right)\left(m+1+\sqrt{(m-1)^2+4}\right)\left(m-1+\sqrt{(m-1)^2+4}\right)}{m-1+2\alpha+\sqrt{(m-1)^2+4}}\Re\left(a_1\bar{c_1}\right)b^{m-1+\sqrt{(m-1)^2+4}}\\
        &\times\left(1-\left(\frac{a}{b}\right)^{m+2\alpha+\sqrt{(m-1)^2+4}}\right).%\\
        %&\geq \frac{\left(m+1+\sqrt{(m-1)^2+4}\right)\left(m-1+\sqrt{(m-1)^2+4}\right)}{2\left(\sqrt{(m-1)^2+4}+\alpha\right)}|a_1|^2b^{2\sqrt{(m-1)^2+4}}\left(1-\left(\frac{a}{b}\right)^{2\left(\sqrt{(m-1)^2+4}+\alpha\right)}\right)\\
        %&-\frac{2\left(m+1+\sqrt{(m-1)^2+4}\right)^2\left(m-1+\sqrt{(m-1)^2+4}\right)^2}{\left(m-1+2\alpha+\sqrt{(m-1)^2+4}\right)^2}.
    \end{align*}
    In fact, to avoid redoing multiple times the same estimate, let $\alpha_1,\alpha_2>-1$, $a_1,a_2\in \C$, and 
    $n\in \Z^{\ast}$, and define
    \begin{align*}
        f(r,\theta)&=r^{\alpha_1}\Re\left(\mathfrak{a}_1e^{in\theta}\right)+r^{\alpha_2}\Re\left(\mathfrak{a}_2\,e^{in\theta}\right)=\Re\left(\left(\mathfrak{a}_1r^{\alpha_1}+\mathfrak{a}_2r^{\alpha_2}\right)e^{in\theta}\right)\\
        &=\frac{1}{2}\left(\mathfrak{a}_1r^{\alpha_1}+\mathfrak{a}_2r^{\alpha_2}\right)e^{in\theta}+\frac{1}{2}\left(\bar{\mathfrak{a}_1}r^{\alpha_1}+\bar{\mathfrak{a}_2}r^{\alpha_2}\right)e^{-in\theta}.
    \end{align*}
    We have
    \begin{align*}
        |f(r,\theta)|^2&=\frac{1}{2}|\mathfrak{a}_1r^{\alpha_1}+\mathfrak{a}_2r^{\alpha_2}|^2+\frac{1}{2}\Re\left(\left(\mathfrak{a}_1r^{\alpha_1}+\mathfrak{a}_2r^{\alpha_2}\right)e^{2in\theta}\right)\\
        &=\frac{1}{2}\left(|\mathfrak{a}_1|^2r^{2\alpha_1}+|\mathfrak{a}_2|^2r^{2\alpha_2}+2\,\Re\left(\mathfrak{a}_1\bar{\mathfrak{a}_2}\right)r^{\alpha_1+\alpha_2}\right)+\frac{1}{2}\Re\left(\left(\mathfrak{a}_1r^{\alpha_1}+\mathfrak{a}_2r^{\alpha_2}\right)e^{2in\theta}\right).
    \end{align*}
    As $n\in \Z^{\ast}$, we have 
    \begin{align*}
        &\int_{\Omega}|f(x)|^2dx=\pi\int_{a}^b\left(|\mathfrak{a}_1|^2r^{2\alpha_1+1}+|\mathfrak{a}_2|^2r^{2\alpha_2+1}+2\,\Re\left(\mathfrak{a}_1\bar{\mathfrak{a}_2}\right)r^{\alpha_1+\alpha_2+1}\right)dr\\
        &=\pi\left(\frac{|\mathfrak{a}_1|^2}{2(\alpha_1+1)}b^{2(\alpha_1+1)}\left(1-\left(\frac{a}{b}\right)^{2(\alpha_1+1)}\right)+\frac{|\mathfrak{a}_2|^2}{2(\alpha_2+1)}b^{2(\alpha_2+1)}\left(1-\left(\frac{a}{b}\right)^{2(\alpha_2+1)}\right)\right.\\
        &\left.
        +\frac{2\,\Re\left(\mathfrak{a}_1\bar{\mathfrak{a}_2}\right)}{\alpha_1+\alpha_2+2}b^{\alpha_1+\alpha_2+2}\left(1-\left(\frac{a}{b}\right)^{\alpha_1+\alpha_2+2}\right)\right).
    \end{align*}
    Using the elementary inequality
    \begin{align*}
        2ab\geq -\dfrac{1}{2}a^2-2b^2,
    \end{align*}
    we deduce that 
    \begin{align*}
        &\frac{|\mathfrak{a}_1|^2}{2(\alpha_1+1)}b^{2(\alpha_1+1)}\left(1-\left(\frac{a}{b}\right)^{2(\alpha_1+1)}\right)+\frac{|\mathfrak{a}_2|^2}{2(\alpha_2+1)}b^{2(\alpha_2+1)}\left(1-\left(\frac{a}{b}\right)^{2(\alpha_2+1)}\right)\\
        &
        +\frac{2\,\Re\left(\mathfrak{a}_1\bar{\mathfrak{a}_2}\right)}{\alpha_1+\alpha_2+2}b^{\alpha_1+\alpha_2+2}\left(1-\left(\frac{a}{b}\right)^{\alpha_1+\alpha_2+2}\right)\\
        &\geq \frac{|\mathfrak{a}_1|^2}{2(\alpha_1+1)}b^{2(\alpha_1+1)}\left(1-\left(\frac{a}{b}\right)^{2(\alpha_1+1)}\right)-\frac{2(\alpha_2+1)}{(\alpha_1+\alpha_2+2)^2}|\mathfrak{a}_1|^2b^{2(\alpha_1+1)}\frac{\left(1-\left(\frac{a}{b}\right)^{\alpha_1+\alpha_2+2}\right)^2}{1-\left(\frac{a}{b}\right)^{2(\alpha_2+1)}}.
    \end{align*}
    Now, we have
    \begin{align*}
        \frac{1}{2(\alpha_1+1)}-\frac{2(\alpha_2+1)}{(\alpha_1+\alpha_2+2)}=\frac{(\alpha_1+\alpha_2+2)^2-4(\alpha_1+1)(\alpha_2+1)}{2(\alpha_1+1)(\alpha_1+\alpha_2+2)}=\frac{(\alpha_1-\alpha_2)^2}{2(\alpha_1+1)(\alpha_1+\alpha_2+1)^2}.
    \end{align*}
    Therefore, we get
    \begin{align*}
        &\frac{|\mathfrak{a}_1|^2}{2(\alpha_1+1)}b^{2(\alpha_1+1)}\left(1-\left(\frac{a}{b}\right)^{2(\alpha_1+1)}\right)-\frac{2(\alpha_2+1)}{(\alpha_1+\alpha_2+2)^2}|\mathfrak{a}_1|^2b^{2(\alpha_1+1)}\frac{\left(1-\left(\frac{a}{b}\right)^{\alpha_1+\alpha_2+2}\right)^2}{1-\left(\frac{a}{b}\right)^{2(\alpha_2+1)}}\\
        &\geq \frac{1}{1-\left(\frac{a}{b}\right)^{2(\alpha_2+1)}}\frac{|\mathfrak{a}_1|^2}{2(\alpha_1+1)}b^{2(\alpha_1+1)}\left(\frac{(\alpha_1-\alpha_2)^2}{(\alpha_1+\alpha_2+2)^2}-\left(\frac{a}{b}\right)^{2(\alpha_1+1)}-\left(\frac{a}{b}\right)^{2(\alpha_2+1)}\right).
    \end{align*}
    Assuming that $\alpha_1\neq \alpha_2$ (in which case the estimate is trivial since $\Re(\mathfrak{a}_1\bar{\mathfrak{a}_2})=|\mathfrak{a}_1|^2\geq 0$) and that
    \begin{align}\label{est_conf_frak_m}
        \log\left(\frac{b}{a}\right)\geq \frac{1}{\min\ens{\alpha_1+1,\alpha_2+1}}\log\left(\frac{2(\alpha_1+\alpha_2+2)}{|\alpha_1-\alpha_2|}\right),
    \end{align}
    we get
    \begin{align*}
        &\frac{|\mathfrak{a}_1|^2}{2(\alpha_1+1)}b^{2(\alpha_1+1)}\left(1-\left(\frac{a}{b}\right)^{2(\alpha_1+1)}\right)+\frac{|\mathfrak{a}_2|^2}{2(\alpha_2+1)}b^{2(\alpha_2+1)}\left(1-\left(\frac{a}{b}\right)^{2(\alpha_2+1)}\right)\\
        &
        +\frac{2\,\Re\left(\mathfrak{a}_1\bar{\mathfrak{a}_2}\right)}{\alpha_1+\alpha_2+2}b^{\alpha_1+\alpha_2+2}\left(1-\left(\frac{a}{b}\right)^{\alpha_1+\alpha_2+2}\right)\\
        &\geq \frac{(\alpha_1-\alpha_2)^2}{4(\alpha_1+1)(\alpha_1+\alpha_2+2)^2}|\mathfrak{a}_1|^2b^{2(\alpha_1+1)}.
    \end{align*}
    Therefore, using a symmetrical estimate for $\mathfrak{a}_2$, we deduce that provided that provided that \eqref{est_conf_frak_m} holds, we have
    \begin{align}\label{est_gen_f}
        \int_{\Omega}|f(x)|^2dx\geq  \frac{\pi(\alpha_1-\alpha_2)^2}{8(\alpha_1+1)(\alpha_1+\alpha_2+2)^2}|\mathfrak{a}_1|^2b^{2(\alpha_1+1)}+\frac{\pi(\alpha_1-\alpha_2)^2}{8(\alpha_2+1)(\alpha_1+\alpha_2+2)^2}|\mathfrak{a}_2|^2b^{2(\alpha_2+1)}.
    \end{align}
    Applying this estimate to 
    \begin{align}\label{def_frak_a1}
    \left\{\begin{alignedat}{2}
        &\alpha_1=\sqrt{(m-1)^2+4}-1+\alpha\qquad &&\mathfrak{a}_1=\left(m+1+\sqrt{(m-1)^2+4}\right)\left(m-1+\sqrt{(m-1)^2+4}\right)a_1\\
        &\alpha_2=m-2+\alpha\qquad &&\mathfrak{a_2}=(m^2-1)c_1
    \end{alignedat}\right.
    \end{align}
    and $n=1$, we deduce that 
    \begin{align}\label{new_frak_lm1}
        &\frac{\left(m+1+\sqrt{(m-1)^2+4}\right)^2\left(m-1+\sqrt{(m-1)^2+4}\right)^2}{2\left(\sqrt{(m-1)^2+4}+\alpha\right)}|a_1|^2b^{2\sqrt{(m-1)^2+4}}\left(1-\left(\frac{a}{b}\right)^{2\left(\sqrt{(m-1)^2+4}+\alpha\right)}\right)\nonumber\\
        &+\frac{(m^2-1)}{2(m-1+\alpha)}|c_1|^2b^{2(m-1)}\left(1-\left(\frac{a}{b}\right)^{2(m-1+\alpha)}\right)\nonumber\\
        &+\frac{2\left(m^2-1\right)\left(m+1+\sqrt{(m-1)^2+4}\right)\left(m-1+\sqrt{(m-1)^2+4}\right)}{m-1+2\alpha+\sqrt{(m-1)^2+4}}\Re\left(a_1\bar{c_1}\right)b^{m-1+\sqrt{(m-1)^2+4}}\nonumber\\
        &\times\left(1-\left(\frac{a}{b}\right)^{m+2\alpha+\sqrt{(m-1)^2+4}}\right)\nonumber\\
        &\geq \frac{\left(\sqrt{(m-1)^2+4}-(m-1)\right)^2\left(m+1+\sqrt{(m-1)^2+4}\right)\left(m-1+\sqrt{(m-1)^2+4}\right)}{8\left(\sqrt{(m-1)^2+4}+\alpha\right)(\sqrt{(m-1)^2+4}+(m-1)+2\alpha)}|a_1|^2b^{2\sqrt{(m-1)^2+4}}\nonumber\\
        &+\frac{(m^2-1)\left(\sqrt{(m-1)^2+4}-(m-1)\right)^2}{8(m-1+\alpha)\left(\sqrt{(m-1)^2+4}+(m-1)+2\alpha\right)}|c_1|^2b^{2(m-1)}.
    \end{align}
    The other crossed terms are easily estimated. Define
    \begin{align}\label{def_frak_a2}
        \alpha_3=-1-\sqrt{(m-1)^2+4}+\alpha\qquad \mathfrak{a}_3=\left(\sqrt{(m-1)^2+4}-(m+1)\right)\left(\sqrt{(m-1)^2}-(m-1)\right)b_1.
    \end{align}
    We estimate
    \begin{align*}
        &\frac{2\,\Re\left(\mathfrak{a}_1\mathfrak{a}_3\right)}{\alpha}\left(\frac{a}{b}\right)^{2\alpha}\geq -\frac{|\mathfrak{a}_1|^2}{\delta\alpha^2}\left(\frac{a}{b}\right)^{2\alpha}a^{2(-\alpha_3-1)}-\delta|\mathfrak{a}_3|^2\left(\frac{a}{b}\right)^{2\alpha}\frac{1}{a^{2(-\alpha_3-1)}},
    \end{align*}
    while
    \begin{align*}
        &\frac{2\,\Re\left(\mathfrak{a}_2\mathfrak{a}_3\right)}{(m-1+2\alpha)-\sqrt{(m-1)^2+4}}b^{m-1}\frac{1}{b^{\sqrt{(m-1)^2+4}}}\\
        &\geq -\epsilon|\mathfrak{a}_2|^2b^{2(m-1)}-\frac{|\mathfrak{a}_3|^2}{\epsilon\left(m-1+2\alpha-\sqrt{(m-1)^2+4}\right)^2}\frac{1}{b^{2\sqrt{(m-1)^2+4}}}. 
    \end{align*}
    We can take 
    \begin{align*}
        \epsilon=\frac{(\alpha_1-\alpha_2)^2}{16(\alpha_2+1)(\alpha_1+\alpha_2+2)},
    \end{align*}
    and then we simply estimate 
    \begin{align*}
        &\frac{|\mathfrak{a}_3|^2}{2\left(\sqrt{(m-1)^2+4}-\alpha\right)}\left(\frac{a}{b}\right)^{2\alpha}\frac{1}{a^{2(-\alpha_3-1)}}-\delta|\mathfrak{a}_3|^2\left(\frac{a}{b}\right)^{2\alpha}\frac{1}{a^{2(-\alpha_3-1)}}\\
        &-\frac{|\mathfrak{a}_3|^2}{\epsilon\left(m-1+2\alpha-\sqrt{(m-1)^2+4}\right)^2}\frac{1}{b^{2\sqrt{(m-1)^2+4}}}\\
        &\geq \frac{|a_3|^2}{4\left(\sqrt{(m-1)^2+4}-\alpha\right)}\left(\frac{a}{b}\right)^{2\alpha}\frac{1}{a^{2(-\alpha_3-1)}}\left(1-4\left(\sqrt{(m-1)^2+4}-\alpha\right)\left(\frac{a}{b}\right)^{2(-\alpha_3-1-\alpha)}\right)\\
        &\geq \frac{\left(\sqrt{(m-1)^2+4}-(m+1)\right)^2\left(\sqrt{(m-1)^2+4}-(m-1)\right)^2}{8\left(\sqrt{(m-1)^2+4}-\alpha\right)}|b_1|^2\left(\frac{a}{b}\right)^{2\alpha}\frac{1}{a^{2\sqrt{(m-1)^2+4}}}
    \end{align*}
    if we take
    \begin{align*}
        \delta=\frac{1}{4\left(\sqrt{(m-1)^2+4}-\alpha\right)},
    \end{align*}
    and impose
    \begin{align*}
        \log\left(\frac{b}{a}\right)\geq \frac{1}{2\left(\sqrt{(m-1)^2+4}-\alpha\right)}\log\left(8\left(\sqrt{(m-1)^2+4}-\alpha\right)\right).
    \end{align*}
    Notice that this imposes the condition $\alpha<\sqrt{(m-1)^2+4}$, which is satisfied since $m\geq 3$ provided that $\alpha<2\sqrt{2}$. Finally we   estimate
    \begin{align*}
        \frac{(\alpha_1-\alpha_2)^2}{8(\alpha_1+1)(\alpha_1+\alpha_2+2)^2}|\mathfrak{a}_1|^2b^{2(\alpha_1+1)}-\frac{|\mathfrak{a}_1|^2}{\delta\alpha^2}\left(\frac{a}{b}\right)^{2\alpha}a^{2(\alpha_1+1)}
        \geq \frac{(\alpha_1-\alpha_2)^2}{16(\alpha_1+1)(\alpha_1+\alpha_2+2)^2}|\mathfrak{a}_1|^2b^{2(\alpha_1+1)}
    \end{align*}
    provided that 
    \begin{align*}
        \log\left(\frac{b}{a}\right)\geq \frac{1}{2\sqrt{(m-1)^2+4}}\log\left(\frac{8(\alpha_1+1)(\alpha_1+\alpha_2+2)^2}{\delta\alpha^2(\alpha_1-\alpha_2)^2}\right).
    \end{align*}
    Finally, we deduce that (assuming that all previous conditions on the conformal class are satisfied)
    \begin{align}\label{frak_lm_fourier_1}
        &\int_{\Omega}\left(\leb_mu_1\right)^2\left(\frac{|x|}{b}\right)^{2\alpha}dx\nonumber\\
        &\geq 
        \frac{\pi\left(\sqrt{(m-1)^2+4}-(m-1)\right)^2\left(m+1+\sqrt{(m-1)^2+4}\right)\left(m-1+\sqrt{(m-1)^2+4}\right)}{8\left(\sqrt{(m-1)^2+4}+\alpha\right)(\sqrt{(m-1)^2+4}+(m-1)+2\alpha)}\nonumber\\
        &\times |a_1|^2b^{2\sqrt{(m-1)^2+4}}\nonumber\\
        &+\frac{\pi\left(\sqrt{(m-1)^2+4}-(m+1)\right)^2\left(\sqrt{(m-1)^2+4}-(m-1)\right)^2}{8\left(\sqrt{(m-1)^2+4}-\alpha\right)}|b_1|^2\left(\frac{a}{b}\right)^{2\alpha}\frac{1}{a^{2\sqrt{(m-1)^2+4}}}\nonumber\\
        &+\frac{\pi(m^2-1)\left(\sqrt{(m-1)^2+4}-(m-1)\right)^2}{8(m-1+\alpha)\left(\sqrt{(m-1)^2+4}+(m-1)+2\alpha\right)}|c_1|^2b^{2(m-1)}.
    \end{align}
    Writing to simplify
    \begin{align*}
        \leb_mu_1=r^{\sqrt{(m-1)^2+4}-1}\,\Re\left(\mathfrak{a}_1e^{i\theta}\right)+r^{-1-\sqrt{(m-1)^2+4}}\,\Re\left(\mathfrak{a}_3e^{i\theta}\right)+r^{m-2}\,\Re\left(\mathfrak{a}_2e^{i\theta}\right).
    \end{align*}
    Therefore, we get
    \begin{align*}
        &\int_{\Omega}\left(\leb_mu_1\right)^2\left(\frac{a}{|x|}\right)^{2\alpha}dx\geq \frac{1}{2}\int_{\Omega}\left(r^{-1-\sqrt{(m-1)^2+4}}\,\Re\left(\mathfrak{a}_3e^{i\theta}\right)\right)^2\left(\frac{a}{|x|}\right)^{2\alpha}dx\\
        &-4\int_{\Omega}\left(r^{\sqrt{(m-1)^2+4}-1}\,\Re\left(\mathfrak{a}_1e^{i\theta}\right)\right)^2dx-4\int_{\Omega}\left(r^{m-2}\,\Re\left(\mathfrak{a}_2e^{i\theta}\right)\right)^2dx\\
        &\geq \frac{\pi|\mathfrak{a}_3|^2}{4\left(\sqrt{(m-1)^2+4}-\alpha\right)}\frac{1}{a^{2\sqrt{(m-1)^2+4}}}\left(1-\left(\frac{a}{b}\right)^{2\left(\sqrt{(m-1)^2+4}-\alpha\right)}\right)\\
        &-\frac{2\pi|\mathfrak{a}_1|^2}{\sqrt{(m-1)^2+4}}b^{2\sqrt{(m-1)^2}}-\frac{2\pi|\mathfrak{a}_2|^2}{m-1}b^{2(m-1)}.
    \end{align*}
    Therefore, using the estimate of \eqref{frak_lm_fourier_1}, we deduce that there exists a constant $\Gamma_{1,m}<\infty$ such that 
    \begin{align}\label{frak_lm_fourier_2}
        &\frac{\pi\left(\sqrt{(m-1)^2+4}-(m-1)\right)^2\left(m+1+\sqrt{(m-1)^2+4}\right)\left(m-1+\sqrt{(m-1)^2+4}\right)}{8\left(\sqrt{(m-1)^2+4}+\alpha\right)(\sqrt{(m-1)^2+4}+(m-1)+2\alpha)}|a_1|^2b^{2\sqrt{(m-1)^2+4}}\nonumber\\
        &+\frac{\pi\left(\sqrt{(m-1)^2+4}-(m+1)\right)^2\left(\sqrt{(m-1)^2+4}-(m-1)\right)^2}{8\left(\sqrt{(m-1)^2+4}-\alpha\right)}|b_1|^2\frac{1}{a^{2\sqrt{(m-1)^2+4}}}\nonumber\\
        &+\frac{\pi(m^2-1)\left(\sqrt{(m-1)^2+4}-(m-1)\right)^2}{8(m-1+\alpha)\left(\sqrt{(m-1)^2+4}+(m-1)+2\alpha\right)}|c_1|^2b^{2(m-1)}\nonumber\\
        &\leq \Gamma_{1,m}\int_{\Omega}\left(\leb_mu_1\right)^2\left(\left(\frac{|x|}{b}\right)^{2\alpha}+\left(\frac{a}{|x|}\right)^{2\alpha}\right)dx.
    \end{align}
    Then, we have for all $\beta>0$ the estimate
    \begin{align*}
        \int_{\Omega}\frac{1}{|x|^4}\left(r^{2-m}\Re\left(d_1e^{i\theta}\right)\right)^2\left(\frac{a}{|x|}\right)^{4\beta}dx=\frac{\pi}{2(m-1+2\beta)}|d_1|^2\frac{1}{a^{2(m-1)}}\left(1-\left(\frac{a}{b}\right)^{2(m-1+2\beta)}\right)dx.
    \end{align*}
    Therefore, using the previous estimate \eqref{frak_lm_fourier_2}, we deduce that there exists $\Gamma_{2,m}<\infty$ such that for all $0<\alpha<\sqrt{(m-1)^2+4}$ and $\beta>0$
    \begin{align*}
        \int_{\Omega}\frac{1}{|x|^2}\left|\p{r}u_1+\frac{m-1}{r}u_1\right|^2dx\leq  \Gamma_{2,m}\int_{\Omega}\left(\leb_mu_1\right)^2\left(\left(\frac{|x|}{b}\right)^{2\alpha}+\left(\frac{a}{|x|}\right)^{2\alpha}\right)dx+\int_{\Omega}\frac{u_1^2}{|x|^4}\left(\frac{a}{|x|}\right)^{4\beta}dx.
    \end{align*}

    \textbf{Step 4.} The remaining Fourier modes for $|n|\geq 2$. 
    For all $n\geq 2$, we have
    \begin{align}
        &u_n=r^{1+\alpha_{m,n}}\cos\left(\beta_{m,n}r\right)\Re\left(a_ne^{in\theta}\right)+r^{1+\alpha_{m,n}}\sin\left(\beta_{m,n}r\right)\Re\left(b_ne^{in\theta}\right)\nonumber\\
        &+r^{1-\alpha_{m,n}}\cos\left(\beta_{m,n}r\right)\Re\left(c_ne^{in\theta}\right)+r^{1-\alpha_{m,n}}\sin\left(\beta_{m,n}r\right)\Re\left(d_ne^{in\theta}\right).
    \end{align}
    We deduce that
    \begin{align}\label{id_rad_un_square}
        &2\,\mathrm{Rad}(u_n^2)=r^{2+2\alpha_{m,n}}\cos^2(\beta_{m,n}r)|a_n|^2+r^{2+2\alpha_{m,n}}\sin^2(\beta_{m,n}r)|b_n|^2\nonumber\\
        &+r^{2-2\alpha_{m,n}}\cos^2(\beta_{m,n}r)|c_n|^2+r^{2-2\alpha_{m,n}}\sin^2(\beta_{m,n}r)|d_n|^2\nonumber\\
        &+2\,r^{2+2\alpha_{m,n}}\cos(\beta_{m,n}r)\sin(\beta_{m,n}r)\Re\left(a_n\bar{b_n}\right)
        +2\,r^2\cos^2(\beta_{m,n}r)\Re\left(a_n\bar{c_n}\right)+2\,r^{2}\cos(\beta_{m,n}r)\sin(\beta_{m,n}r)\nonumber\\
        &+2\,r^2\cos(\beta_{m,n}r)\sin(\beta_{m,n}r)\Re\left(b_n\bar{c_n}\right)+2\,r^{2}\sin^2(\beta_{m,n}r)\Re\left(b_n\bar{d_n}\right)\nonumber\\
        &+r^{2-2\alpha_{m,n}}\cos(\beta_{m,n}r)\sin(\beta_{m,n}r)\Re\left(c_n\bar{d_n}\right).
    \end{align}
    Therefore, it comes down to estimating the $L^2$ norm of the following function
    \begin{align*}
        f(r,\theta)=r^{\alpha}\left(\mathfrak{a}_1\cos(\beta r)+\mathfrak{a}_2\sin(\beta r)\right)
    \end{align*}
    which is given by
    \begin{align*}
        |f(r,\theta)|^2&=r^{2\alpha}\left|\mathfrak{a}_1\cos(\beta r)+\mathfrak{a}_2\sin(\beta r)\right|^2\\
        &=r^{2\alpha}\left(|\mathfrak{a}_1|^2\cos^2(\beta r)+|\mathfrak{a}_2|^2\sin^2(\beta r)+2\,\Re\left(\mathfrak{a}_1\bar{\mathfrak{a}_2}\right)\cos(\beta r)\sin(\beta r)\right)
    \end{align*}
    on $\Omega=B_b\setminus\bar{B}_a(0)$, where we can take $b$ as small as we want. In such a way that
    \begin{align*}
        &\dfrac{1}{2}\leq \cos(\beta r)\leq 1\\
        &\dfrac{\beta r}{2}\leq \sin(\beta r)\leq \beta r.
    \end{align*}
    The difficulty is that the functions $r^{2\alpha}\cos(\beta r),r^{2\alpha}\sin(\beta r)$ do not possess primitives in terms of elementary functions. Now, we will use ideas coming from the estimates involving solutions of $\leb_m^{\ast}\leb_mu=0$. Assuming without loss of generality that $\mathfrak{a_1}\mathfrak{a}_2\neq 0$, write
    \begin{align*}
        &f_1(r,\theta)=\frac{\mathfrak{a_1}}{|\mathfrak{a}_1|}r^{\alpha}\cos(\beta\,r)\\
        &f_2(r,\theta)=\frac{\mathfrak{a_2}}{|\mathfrak{a}_2|}r^{\alpha}\sin(\beta\,r).
    \end{align*}
    We deduce by Lemma \ref{estimée_somme_Cauchy--Schwarz_theorem} that
    \begin{align*}
        &\frac{|\mathfrak{a}_1||\mathfrak{a}_2|}{2}\int_{\Omega\times \Omega}\left|\det\begin{pmatrix}
            f_1(x) & f_2(x)\\
            f_1(y) & f_2(y)
        \end{pmatrix}\right|^2dx\,dy\leq \np{f_1}{2}{\Omega}\np{f_2}{2}{\Omega}\int_{\Omega}||\mathfrak{a}_1|f_1+|\mathfrak{a}_2|f_2|^2dx\\
        &=\np{f_1}{2}{\Omega}\np{f_2}{2}{\Omega}\int_{\Omega}|f|^2dx.
    \end{align*}
    If $\alpha>-1$, we have
    \begin{align*}
        \np{f_1}{2}{\Omega}^2=2\pi\int_{a}^br^{2\alpha+1}\cos^2(\beta r)dr\leq \frac{\pi}{\alpha+1}b^{2(\alpha+1)}\left(1-\left(\frac{a}{b}\right)^{2(\alpha+1)}\right).
    \end{align*}
    On the other hand, we have
    \begin{align*}
        \np{f_2}{2}{\Omega}^2=2\pi\int_{a}^br^{2\alpha+1}\sin^2(\beta r)dr\leq \frac{\pi \beta^2}{\alpha+2}b^{2(\alpha+2)}\left(1-\left(\frac{a}{b}\right)^{2(\alpha+2)}\right).
    \end{align*}
    Therefore, we get
    \begin{align*}
        \np{f_1}{2}{\Omega}\np{f_2}{2}{\Omega}\leq \frac{\pi\,\beta}{\sqrt{(\alpha+1)(\alpha+2)}}b^{2\alpha+3}.
    \end{align*}
    Then, we have
    \begin{align}\label{cross_term_est1}
        \int_{\Omega\times \Omega}\left|\det\begin{pmatrix}
            f_1(x) & f_2(x)\\
            f_1(y) & f_2(y)
        \end{pmatrix}\right|^2dx\,dy&=4\pi^2\int_{a}^b\int_{a}^b\left(r^{\alpha}\cos(\beta r)s^{\alpha}\sin(\beta s)-s^{\alpha}\cos(\beta s)r^{\alpha}\sin(\beta r)\right)^2rs\,dr\,ds\nonumber\\
        &=4\pi^2\int_{a}^b\int_{a}^b\left(\cos(\beta r)\sin(\beta s)-\cos(\beta s)\sin(\beta r)\right)^2r^{2\alpha+1}s^{2\alpha+1}dr\,ds\nonumber\\
        &=4\pi^2\int_{a}^b\int_{a}^b\sin^2\left(\beta(r-s)\right)r^{2\alpha+1}s^{2\alpha+1}dr\,ds\nonumber\\
        &\geq \pi^2\beta^2\int_{a}^b\int_{a}^b(r-s)^2r^{2\alpha+1}s^{2\alpha+1}dr\,ds.
    \end{align}
    We have by Fubini's theorem
    \begin{align*}
        &\int_{a}^b\int_{a}^b(r-s)^2r^{2\alpha+1}s^{2\alpha+1}dr\,ds=2\left(\int_{a}^br^{2\alpha+3}dr\right)\left(\int_{a}^{b}s^{2\alpha+1}ds\right)-2\left(\int_{a}^br^{2\alpha+2}dr\right)\left(\int_{a}^bs^{2\alpha+2}ds\right)\\
        &=2\left(\frac{b^{2(\alpha+2)}}{2(\alpha+2)}\left(1-\left(\frac{a}{b}\right)^{2( \alpha+2)}\right)\right)\left(\frac{b^{2(\alpha+1)}}{2(\alpha+1)}\left(1-\left(\frac{a}{b}\right)^{2(\alpha+1)}\right)\right)-2\left(\frac{b^{2\alpha+3}}{2\alpha+3}\left(1-\left(\frac{a}{b}\right)^{2\alpha+3}\right)\right)^2\\
        &\geq \frac{1}{2(\alpha+1)(\alpha+2)(2\alpha+3)}b^{2(2\alpha+3)}\left(1-2(2\alpha+3)\left(\frac{a}{b}\right)^{2(\alpha+1)}\right).
    \end{align*}
    Assuming that
    \begin{align*}
        \log\left(\frac{b}{a}\right)\geq \frac{1}{2(\alpha+1)}\log\left(4(2\alpha+3)\right),
    \end{align*}
    we deduce that
    \begin{align*}
        \int_{\Omega\times \Omega}\left|\det\begin{pmatrix}
            f_1(x) & f_2(x)\\
            f_1(y) & f_2(y)
        \end{pmatrix}\right|^2dx\,dy\geq \frac{\pi^2\beta^2}{4(\alpha+1)(\alpha+2)(2\alpha+3)}b^{2(2\alpha+3)}.
    \end{align*}
    In the end, we obtain the estimate
    \begin{align*}
        \frac{|\mathfrak{a_1}||\mathfrak{a}_2|}{2}\frac{\pi^2\beta^2}{4(\alpha+1)(\alpha+2)(2\alpha+3)}b^{2(2\alpha+3)}\leq \frac{\pi\beta}{\sqrt{(\alpha+1)(\alpha+2)}}b^{2\alpha+3}\int_{\Omega}|f|^2dx,
    \end{align*}
    which finally yields
    \begin{align}\label{cauchy_schwarz_fort}
        |\mathfrak{a}_1||\mathfrak{a}_2|b^{2\alpha+3}\leq \frac{8\sqrt{(\alpha+1)(\alpha+2)}(2\alpha+3)}{\pi\beta}\int_{\Omega}|f|^2dx.
    \end{align}
    Then, we have
    \begin{align*}
        &\int_{\Omega}|f|^2dx=2\pi\int_{a}^b\left(|\mathfrak{a}_1|^2r^{2\alpha+1}\cos^2(\beta r)+|\mathfrak{a}_2|^2r^{2\alpha+1}\sin^2(\beta r)+2\,\Re(\mathfrak{a}_1\bar{\mathfrak{a}_2})r^{2\alpha+1}\cos(\beta r)\sin(\beta r)\right)dr\\
        &\geq \pi\int_{a}^b|\mathfrak{a}_1|^2r^{2\alpha+1}dr+2\pi\beta^2\int_{a}^b|\mathfrak{a}_2|^2r^{2\alpha+3}dr-2\left|\Re\left(\mathfrak{a}_1\bar{\mathfrak{a}_2}\right)\right|\int_{a}^br^{2\alpha+1}\cos(\beta r)\sin(\beta r)dr\\
        &\geq \pi\left(\int_{a}^b|\mathfrak{a}_1|^2r^{2\alpha+1}dr+2\beta^2\int_{a}^b|\mathfrak{a}_2|^2r^{2\alpha+3}dr-2\beta\left|\Re\left(\mathfrak{a}_1\bar{\mathfrak{a}_2}\right)\right|\int_{a}^br^{2\alpha+2}dr\right)\\
        &=\pi\left(\frac{1}{2(\alpha+1)}|\mathfrak{a}_1|^2b^{2(\alpha+1)}\left(1-\left(\frac{a}{b}\right)^{2(\alpha+1)}\right)+\frac{\beta^2}{\alpha+2}|\mathfrak{a}_2|^2b^{2(\alpha+2)}\left(1-\left(\frac{a}{b}\right)^{2(\alpha+2)}\right)\right.\\
        &\left.-\frac{2\beta\,\Re\left(\mathfrak{a}_1\bar{\mathfrak{a}_2}\right)}{2\alpha+3}b^{2\alpha+3}\left(1-\left(\frac{a}{b}\right)^{2\alpha+3}\right)\right).
    \end{align*}
    Therefore, \eqref{cauchy_schwarz_fort} implies that 
    \begin{align}\label{non_algebrique_dépassé}
        &\frac{\pi}{2(\alpha+1)}|\mathfrak{a}_1|^2b^{2(\alpha+1)}\left(1-\left(\frac{a}{b}\right)^{2(\alpha+1)}\right)+\frac{\pi\beta^2}{\alpha+2}|\mathfrak{a}_2|^2b^{2(\alpha+2)}\left(1-\left(\frac{a}{b}\right)^{2(\alpha+2)}\right)\nonumber\\
        &\leq \int_{\Omega}|f|^2dx
        +\frac{2\pi\beta\,\Re\left(\mathfrak{a}_1\bar{\mathfrak{a}_2}\right)}{2\alpha+3}b^{2\alpha+3}\left(1-\left(\frac{a}{b}\right)^{2\alpha+3}\right)\nonumber\\
        &\leq \int_{\Omega}|f|^2dx+\frac{2\pi\beta|\mathfrak{a}_1||\mathfrak{a}_2|}{2\alpha+3}b^{2\alpha+3}\leq \left(1+16\sqrt{(\alpha+1)(\alpha+2)}\right)\int_{\Omega}|f|^2dx.
    \end{align}
    Now, for $\alpha<-2$, we have
    \begin{align*}
        &\int_{a}^{b}\int_{a}^{b}(r-s)^2r^{-2\alpha+1}s^{-2\alpha+1}drds=2\left(\frac{1}{2(\alpha-2)}\frac{1}{a^{2(\alpha-2)}}\left(1-\left(\frac{a}{b}\right)^{2(\alpha-2)}\right)\right)\\
        &\times \left(\frac{1}{2(\alpha-1)}\frac{1}{a^{2(\alpha-1)}}\left(1-\left(\frac{a}{b}\right)^{2(\alpha-1)}\right)\right)-2\left(\frac{1}{2\alpha-3}\frac{1}{a^{2\alpha-3}}\left(1-\left(\frac{a}{b}\right)^{2\alpha-3}\right)\right)^2\\
        &\geq \frac{1}{2(\alpha-1)(\alpha-2)(2\alpha-3)}\frac{1}{a^{2(2\alpha-3)}}\left(1-2(2\alpha-3)\left(\frac{a}{b}\right)^{2(\alpha-2)}\right),
    \end{align*}
    and this implies that if
    \begin{align*}
        g(r,\theta)&=r^{-\alpha}\left(\mathfrak{a}_1\cos(\beta r)+\mathfrak{a}_2\sin(\beta r)\right)\\
        &=|\mathfrak{a}_1|g_1(r,\theta)+|\mathfrak{a}_2|g_2(r,\theta),
    \end{align*}
    then
    \begin{align*}
        \np{g_1}{2}{\Omega}^2=2\pi\int_{a}^br^{-2\alpha+1}\cos^2(\beta r)dr\leq \frac{\pi}{\alpha-1}\frac{1}{a^{2(\alpha-1)}}\left(1-\left(\frac{a}{b}\right)^{2(\alpha-1)}\right)
    \end{align*}
    and
    \begin{align*}
        \np{g_2}{2}{\Omega}^2\leq \frac{\pi\beta^2}{\alpha-2}\frac{1}{a^{2(\alpha-2)}}\left(1-\left(\frac{a}{b}\right)^{2(\alpha-1)}\right)
    \end{align*}
    and we get
    \begin{align}\label{cauchy_schwarz_précisé_alpha_négatif}
        |\mathfrak{a}_1||\mathfrak{a}_2|\frac{1}{a^{2\alpha-3}}\leq \frac{8\sqrt{(\alpha-1)(\alpha-2)}(2\alpha-3)}{\pi\beta}\int_{\Omega}|g|^2dx.
    \end{align}
    Therefore, we deduce by the same argument as in \eqref{non_algebrique_dépassé}
    \begin{align}\label{cauchy_schwarz_alpha_négatif}
        &\frac{\pi}{2(\alpha-1)}|\mathfrak{a}_1|^2\frac{1}{a^{2(\alpha-1)}}\left(1-\left(\frac{a}{b}\right)^{2(\alpha-1)}\right)+\frac{\pi\beta^2}{\alpha-2}|\mathfrak{a}_2|^2b^{2(\alpha-2)}\left(1-\left(\frac{a}{b}\right)^{2(\alpha-2)}\right)\nonumber\\
        &\leq \left(1+16\sqrt{(\alpha-1)(\alpha-2)}\right)\int_{\Omega}|g|^2dx.
    \end{align}
    The argument allows us to overcome the fact that $r^{\alpha}\cos(\beta r)$ has no primitive expressible in terms of elementary functions for all $\alpha\neq 0$. Notice that a trivial argument using the inequality $\cos^2(\beta r)\geq 1-\beta^2r^2$ would not be of much help since we would get a factor of the form $1-\beta^2b^2$ which would make the direct application of Cauchy's inequality fail. 
    Now, we rewrite identity \eqref{id_rad_un_square} as
    \begin{align*}
        2\,\mathrm{Rad}(u_n^2)&=\left|r^{1+\alpha_{m,n}}\left(a_n\cos(\beta_{m,n}r)+b_n\sin(\beta_{m,n}r)\right)+r^{1-\alpha_{m,n}}\left(c_n\,\cos(\beta_{m,n}r)+d_n\sin(\beta_{m,n}r)\right)\right|^2\\
        &=\left|r^{1+\alpha_{m,n}}\left(a_n\cos(\beta_{m,n}r)+b_n\sin(\beta_{m,n}r)\right)\right|^2+\left|r^{1-\alpha_{m,n}}\left(c_n\,\cos(\beta_{m,n}r)+d_n\sin(\beta_{m,n}r)\right)\right|^2\\
        &+2\,r^2\cos^2(\beta_{m,n}r)\Re\left(a_n\bar{c_n}\right)+2\,r^{2}\cos(\beta_{m,n}r)\sin(\beta_{m,n}r)\nonumber\\
        &+2\,r^2\cos(\beta_{m,n}r)\sin(\beta_{m,n}r)\Re\left(b_n\bar{c_n}\right)+2\,r^{2}\sin^2(\beta_{m,n}r)\Re\left(b_n\bar{d_n}\right).
    \end{align*}
    Recall that 
    \begin{align*}
        \alpha_{m,n}=\sqrt{\frac{\mu_{m,n}+(m-1)^2+n^2+1}{2}},
    \end{align*}
    where
    \begin{align*}
        \mu_{m,n}&=\sqrt{n^4+\left(6m(m-2)+4\right)n^2+m^2(m-2)^2}\\
        &> \sqrt{n^4+2m(m-2)n^2+m^2(m-2)^2}=n^2+m(m-2).
    \end{align*}
    Therefore, we get
    \begin{align*}
        \alpha_{m,n}>\sqrt{\frac{n^2+m(m-2)+(m-1)^2+n^2+1}{2}}=\sqrt{n^2+(m-1)^2}.
    \end{align*}
    We have (replacing respectively $\alpha$ by $\alpha_{m,n}-1+\alpha$ and $-\alpha$ by $-\alpha_{m,n}-1+\alpha$
    \begin{align*}
        &\int_{\Omega}\frac{u_n^2}{|x|^4}|x|^{2\alpha}dx=2\pi\int_{a}^b\bigg\{\left|r^{\alpha_{m,n}-1+\alpha}\left(a_n\cos(\beta_{m,n}r)+b_n\sin(\beta_{m,n}r)\right)\right|^2\\
        &+\left|r^{\alpha_{m,n}-1+\alpha}\left(c_n\,\cos(\beta_{m,n}r)+d_n\sin(\beta_{m,n}r)\right)\right|^2\\
        &+2\,r^{-2+2\alpha}\cos^2(\beta_{m,n}r)\Re\left(a_n\bar{c_n}\right)+2\,r^{-2+2\alpha}\cos(\beta_{m,n}r)\sin(\beta_{m,n}r)\nonumber\\
        &+2\,r^{-2+2\alpha}\cos(\beta_{m,n}r)\sin(\beta_{m,n}r)\Re\left(b_n\bar{c_n}\right)+2\,r^{-2+2\alpha}\sin^2(\beta_{m,n}r)\Re\left(b_n\bar{d_n}\right)\bigg\}r\,dr\\
        &\geq \frac{1}{1+16\sqrt{(\alpha_{m,n}+\alpha)(\alpha_{m,n}+1+\alpha)}}\left(\frac{\pi}{2(\alpha_{m,n}+\alpha)}b^{2(\alpha_{m,n}+\alpha)}|a_n|^2\left(1-\left(\frac{a}{b}\right)^{2(\alpha_{m,n}+\alpha)}\right)\right.\\
        &\left.+\frac{\pi\beta_{m,n}^2}{\alpha_{m,n}+1+\alpha}|b_n|^2b^{2(\alpha_{m,n}+1+\alpha)}\left(1-\left(\frac{a}{b}\right)^{2(\alpha_{m,n}+1+\alpha)}\right)\right)\\
        &+\frac{1}{1+16\sqrt{(\alpha_{m,n}-\alpha)}(\alpha_{m,n}-1-\alpha)}\left(\frac{\pi}{2(\alpha_{m,n}-\alpha)}|c_n|^2\frac{1}{a^{2(\alpha_{m,n}-\alpha)}}\left(1-\left(\frac{a}{b}\right)^{2(\alpha_{m,n}-\alpha)}\right)\right.\\
        &\left.+\frac{\pi\beta_{m,n}^2}{\alpha_{m,n}-1-\alpha}|d_n|^2\frac{1}{a^{2(\alpha_{m,n}-1-\alpha)}}\left(1-\left(\frac{a}{b}\right)^{2(\alpha_{m,n}-1-\alpha)}\right)\right)\\
        &-\frac{2\pi}{|\alpha|}\left(\left|\mathrm{Re}(a_n\bar{c_n})\right|+\left|\mathrm{Re}(a_n\bar{d_n})\right|+\left|\mathrm{Re}(b_n\bar{c_n})\right|+\left|\mathrm{Re}(b_n\bar{d_n})\right|\right)\left|b^{2\alpha}-a^{2\alpha}\right|.
    \end{align*}
    The crossed terms are easily estimated by Cauchy's inequality as in the proof of Theorem \ref{lm_last_lemma}. Let us only treat the case $\alpha>0$. Then, we have
    \begin{align*}
       &\frac{2}{\alpha}\left|\mathrm{Re}(a_n\bar{c_n})\right|b^{2\alpha}\left(1-\left(\frac{a}{b}\right)^{2\alpha}\right)\leq \epsilon|a_n|^2 b^{2(\alpha_{m,n}+\alpha)}+\frac{1}{\epsilon\alpha^2}|c_n|^2\frac{1}{b^{2(\alpha_{m,n}-\alpha)}}
    \end{align*}
    so we see that for all $0<\epsilon<\infty$ and $0<\delta<1$, there exists $M=M(\epsilon,\alpha)<\infty$ (independent of $|n|\geq 2$) such that 
    if
    \begin{align*}
        \log\left(\frac{b}{a}\right)\geq M(\epsilon,\alpha)
    \end{align*}
    then
    \begin{align*}
        &\frac{1}{1+16\sqrt{(\alpha_{m,n}+\alpha)(\alpha_{m,n}+1+\alpha)}}\frac{\pi}{2(\alpha_{m,n}+\alpha)}b^{2(\alpha_{m,n}+\alpha)}|a_n|^2\left(1-\left(\frac{a}{b}\right)^{2(\alpha_{m,n}+\alpha)}\right)\\
        &+\frac{1}{1+16\sqrt{(\alpha_{m,n}-\alpha)}(\alpha_{m,n}-1-\alpha)}\frac{\pi}{2(\alpha_{m,n}-\alpha)}|c_n|^2\frac{1}{a^{2(\alpha_{m,n}-\alpha)}}\left(1-\left(\frac{a}{b}\right)^{2(\alpha_{m,n}-\alpha)}\right)\\
        &-\frac{2}{\alpha}\left|\mathrm{Re}(a_n\bar{c_n})\right|b^{2\alpha}\left(1-\left(\frac{a}{b}\right)^{2\alpha}\right)\\
        &\geq \left(\frac{1}{1+16\sqrt{(\alpha_{m,n}+\alpha)(\alpha_{m,n}+1+\alpha)}}\frac{\pi}{2(\alpha_{m,n}+\alpha)}-\epsilon\right)b^{2(\alpha_{m,n}+\alpha)}|a_n|^2\left(1-\left(\frac{a}{b}\right)^{2(\alpha_{m,n}+\alpha)}\right)\\
        &+\left(\frac{1}{1+16\sqrt{(\alpha_{m,n}-\alpha)}(\alpha_{m,n}-1-\alpha)}-\delta\right)\frac{\pi}{2(\alpha_{m,n}-\alpha)}|c_n|^2\frac{1}{a^{2(\alpha_{m,n}-\alpha)}}\left(1-\left(\frac{a}{b}\right)^{2(\alpha_{m,n}-\alpha)}\right)
    \end{align*}
    Applying the same argument to the three remaining terms, we finally deduce that there exists $A=A(\alpha)$ such that 
    \begin{align*}
        \log\left(\frac{b}{a}\right)\geq A(\alpha),
    \end{align*}
    then
    \begin{align}\label{est_un_frak_lm}
        &\int_{\Omega}\frac{u_n^2}{|x|^4}|x|^{2\alpha}dx\geq \frac{1}{1+16\sqrt{(\alpha_{m,n}+\alpha)(\alpha_{m,n}+1+\alpha)}}\left(\frac{\pi}{2(\alpha_{m,n}+\alpha)}b^{2(\alpha_{m,n}+\alpha)}|a_n|^2\left(1-\left(\frac{a}{b}\right)^{2(\alpha_{m,n}+\alpha)}\right)\right.\nonumber\\
        &\left.+\frac{\pi\beta_{m,n}^2}{\alpha_{m,n}+1+\alpha}|b_n|^2b^{2(\alpha_{m,n}+1+\alpha)}\left(1-\left(\frac{a}{b}\right)^{2(\alpha_{m,n}+1+\alpha)}\right)\right)\nonumber\\
        &+\frac{1}{1+16\sqrt{(\alpha_{m,n}-\alpha)}(\alpha_{m,n}-1-\alpha)}\left(\frac{\pi}{2(\alpha_{m,n}-\alpha)}|c_n|^2\frac{1}{a^{2(\alpha_{m,n}-\alpha)}}\left(1-\left(\frac{a}{b}\right)^{2(\alpha_{m,n}-\alpha)}\right)\right.\nonumber\\
        &\left.+\frac{\pi\beta_{m,n}^2}{\alpha_{m,n}-1-\alpha}|d_n|^2\frac{1}{a^{2(\alpha_{m,n}-1-\alpha)}}\left(1-\left(\frac{a}{b}\right)^{2(\alpha_{m,n}-1-\alpha)}\right)\right).
    \end{align}
    Therefore, we need only estimate the weighted $L^2$ norm of $\leb_mu$ or $\mathfrak{L}_mu$ for $|n|$ large enough since we need to capture the linear growth in $|n|$ that appears in the Fourier coefficient of the weighted norm of 
    \begin{align*}
        \D u+(m-1)\frac{x}{|x|^2}u.
    \end{align*}

    \textbf{Step 5.} The asymptotic estimate as $|n|\rightarrow \infty$.
    Now, recalling that
    \begin{align*}
        &\leb_m=\p{r}^2+\frac{(2m-1)}{r}\p{r}+\frac{(m-1)^2}{r^2}+\frac{1}{r^2}\p{\theta}^2\\
        &\leb_m\left(r^{\alpha}\Re\left(a\,e^{in\theta}\right)\right)=\left(\alpha+(m-1)+n\right)\left(\alpha+(m-1)-n\right)r^{\alpha-2}\Re\left(a\,e^{in\theta}\right).
    \end{align*}
    Therefore, we deduce that
    \begin{align*}
        &\leb_m\left(r^{\alpha}\cos(\beta r)\Re\left(a\,e^{in\theta}\right)\right)=\bigg(\left(\alpha+(m-1)+n\right)\left(\alpha+(m-1)-n\right)r^{\alpha-2}\cos(\beta r)\nonumber\\
        &+2\,\p{r}\left(r^{\alpha}\right)\p{r}\left(\cos(\beta r)\right)+r^{\alpha}\p{r}^2\left(\cos(\beta r)\right)+r^{\alpha}\frac{2m-1}{r}\p{r}\left(\cos(\beta r)\right)\bigg)\Re\left(a\,e^{in\theta}\right)\\
        &=\Big(\left(\left(\alpha+(m-1)+n\right)\left(\alpha+(m-1)-n\right)-\beta^2\right)r^{\alpha-2}\cos(\beta r)\\
        &+\beta\left(2\alpha+2m-1\right)r^{\alpha-1}\sin(\beta r)\Big)\Re\left(a\,e^{in\theta}\right).
    \end{align*}
    Likewise, we have 
    \begin{align*}
        \leb_m\left(r^{\alpha}\sin(\beta r)\Re\left(a\,e^{in\theta}\right)\right)&=\Big(\left(\left(\alpha+(m-1)+n\right)\left(\alpha+(m-1)-n\right)-\beta^2\right)r^{\alpha-2}\sin(\beta r)\\
        &+\beta\left(2\alpha+2m-1\right)r^{\alpha-1}\cos(\beta r)\Big)\Re\left(a\,e^{in\theta}\right).
    \end{align*}
    Since
    \begin{align*}
        &u_n=r^{1+\alpha_{m,n}}\cos\left(\beta_{m,n}r\right)\Re\left(a_ne^{in\theta}\right)+r^{1+\alpha_{m,n}}\sin\left(\beta_{m,n}r\right)\Re\left(b_ne^{in\theta}\right)\nonumber\\
        &+r^{1-\alpha_{m,n}}\cos\left(\beta_{m,n}r\right)\Re\left(c_ne^{in\theta}\right)+r^{1-\alpha_{m,n}}\sin\left(\beta_{m,n}r\right)\Re\left(d_ne^{in\theta}\right),
    \end{align*}
    we deduce that 
    \begin{align*}
        &\leb_mu_n=\Big\{\big(\left(\alpha_{m,n}+m+n\right)\left(\alpha_{m,n}+m-n\right)-\beta_{m,n}^2\big)r^{\alpha_{m,n}-1}\cos(\beta_{m,n}r)\\
        &+\beta_{m,n}\left(2\alpha_{m,n}+2m+1\right)r^{\alpha_{m,n}}\sin(\beta_{m,n}r)\Big\}
         \Re\left(a_n\,e^{in\theta}\right)\\
         &+\Big\{\big(\left(\alpha_{m,n}+m+n\right)\left(\alpha_{m,n}+m-n\right)-\beta_{m,n}^2\big)r^{\alpha_{m,n}-1}\sin(\beta_{m,n}r)\\
        &+\beta_{m,n}\left(2\alpha_{m,n}+2m+1\right)r^{\alpha_{m,n}}\cos(\beta_{m,n}r)\Big\}
         \Re\left(b_n\,e^{in\theta}\right)\\
         &+\Big\{\big(\left(-\alpha_{m,n}+m+n\right)\left(-\alpha_{m,n}+m-n\right)-\beta_{m,n}^2\big)r^{-\alpha_{m,n}-1}\cos(\beta_{m,n}r)\\
        &+\beta_{m,n}\left(-2\alpha_{m,n}+2m+1\right)r^{\alpha_{m,n}}\sin(\beta_{m,n}r)\Big\}
         \Re\left(c_n\,e^{in\theta}\right)\\
         &+\Big\{\big(\left(-\alpha_{m,n}+m+n\right)\left(-\alpha_{m,n}+m-n\right)-\beta_{m,n}^2\big)r^{-\alpha_{m,n}-1}\sin(\beta_{m,n}r)\\
        &+\beta_{m,n}\left(-2\alpha_{m,n}+2m+1\right)r^{-\alpha_{m,n}}\cos(\beta_{m,n}r)\Big\}
         \Re\left(d_n\,e^{in\theta}\right).
    \end{align*}
    On the other hand, we have $\beta_{m,n}\leq 2m(m-2)$, so we get
    \begin{align*}
        &\left(\alpha_m+m+n\right)\left(\alpha_{m,n}+m-n\right)-\beta_{m,n}^2=\left(\alpha_{m,n}+m\right)^2-\left(n^2+\beta_{m,n}^2\right)\\
        &=\alpha_{m,n}^2+2m\,\alpha_{m,n}+m^2-\left(n^2+\beta_{m,n}^2\right)\\
        &=\frac{\mu_{m,n}+(m-1)^2+n^2+1}{2}+2m\,\alpha_{m,n}+m^2-\left(n^2+\frac{\mu_{m,n}-\left((m-1)^2+n^2+1\right)}{2}\right)\\
        &=2m\,\alpha_{m,n}+m^2+(m-1)^2+1>2m\sqrt{n^2+(m-1)^2}+2(m(m-1)+1)>2m|n|+2(m(m-1)+1).
    \end{align*}
    On the other hand, we have
    \begin{align*}
        &\left(-\alpha_{m,n}+m+n\right)\left(-\alpha_{m,n}+m-n\right)-\beta_{m,n}^2=\alpha_{m,n}^2-2m\alpha_{m,n}+m^2-\left(n^2+\beta_{m,n}^2\right)\\
        &=\frac{\mu_{m,n}+(m-1)^2+n^2+1}{2}-2m\,\alpha_{m,n}+m^2-\left(n^2+\frac{\mu_{m,n}-\left((m-1)^2+n^2+1\right)}{2}\right)\\
        &=-2m\,\alpha_{m,n}+m^2+(m-1)^2+1=\frac{-4m^2\alpha_{m,n}^2+\left(m^2+(m-1)^2+1\right)^2}{4m^2\alpha_{m,n}+m^2+(m-1)^2+1}.
    \end{align*}
    Then, we have
    \begin{align*}
        &-4m^2\alpha_{m,n}^2+\left(m^2+(m-1)^2+1\right)^2%=-2m^2\left(\mu_{m,n}+(m-1)^2+n^2+1\right)+m^4+(m-1)^4+1+2m^2(m-1)^2+2m^2+2(m-1)^2\\
        =-2m^2\mu_{m,n}-2m^2n^2+m^4+(m-1)^4+2(m-1)^2.
    \end{align*}
    Therefore, we have $-4m^2\alpha_{m,n}^2+\left(m^2+(m-1)^2+1\right)^2<0$ if and only if
    \begin{align*}
        2m^2(\mu_{m,n}+n^2)>m^4+(m-1)^4+(m-1)^2.
    \end{align*}
    Using the inequality $\mu_{m,n}>n^2+m(m-2)$, we get 
    \begin{align*}
        2m^2(\mu_{m,n}+n^2)>4m^2n^2+2m^3(m-2)\geq m^4+(m-1)^4+(m-1)^2\Longleftrightarrow n^2>2-\frac{2}{m}+\frac{3}{4m^2}.
    \end{align*}
    Notice that the right-hand side is always lower than $2$ (since $m\geq 3/2$), and since $|n|\geq 2$, we actually get
    \begin{align*}
        -4m^2\alpha_{m,n}^2+\left(m^2+(m-1)^2+1\right)^2<-8m^2<0.
    \end{align*}
    On the other hand, we have
    \begin{align*}
        2(1-\alpha_{m,n})+2m-1=-2\alpha_{m,n}+2m+1<0 \Longleftrightarrow 2\left(\mu_{m,n}+(m-1)^2+n^2+1\right)>(2m+1)^2
    \end{align*}
    and this holds if and only if
    \begin{align*}
        2\mu_{m,n}+2n^2>2m^2+8m-3.
    \end{align*}
    Since $\mu_{m,n}>n^2+m(m-2)$, this inequality is true in particular if 
    \begin{align*}
        4n^2>2m^2+8m-3-2m(m-2)=12m-3
    \end{align*}
    which only holds for $|n|$ large enough. Instead, if $2m^2+8m-3-2n^2>0$, the inequality becomes
    \begin{align*}
        &4\left(n^4+(6m(m-2)+4)n^2+m^2(m-2)^2\right)>(2m^2+8m-3-2n^2)^2 \\
        &\Longleftrightarrow 32 \, m^{2} n^{2} - 48 \, m^{3} - 16 \, m n^{2} - 36 \, m^{2} + 4 \, n^{2} + 48 \, m - 9>0\\
        &\Longleftrightarrow 4(8m^2-4m+1)n^2-12(4m^3+3m^2-4m)-9>0\\
        &\Longleftrightarrow n^2>\frac{12(4m^3+3m^2-4m)+9}{4(8m^2-4m+1)}.
    \end{align*}
    Since $n^2<m^2+4m-\frac{3}{2}$, this inequality holds for some $n$ provided that 
    \begin{align*}
        \frac{12(4m^3+3m^2-4m)+9}{4(8m^2-4m+1)}<m^2+4m-\frac{3}{2}\Longleftrightarrow 32 \, m^{4} + 64 \, m^{3} - 144 \, m^{2} + 88 \, m - 15 > 0
    \end{align*}
    which is always true. Now, it is apparent that we must control the following quantity
    \begin{align*}
        g(r,\theta)&=\mathfrak{a}_1\left(r^{\alpha}\cos(\beta r)+\lambda\,r^{\alpha+1}\sin(\beta r)\right)+\mathfrak{a}_2\left(r^{\alpha}\sin(\beta r)+\lambda\,r^{\alpha+1}\cos(\beta r)\right)\\
        &=r^{\alpha}\left(\mathfrak{a}_1\cos(\beta r)+\mathfrak{a}_2\sin(\beta r)\right)+\lambda\,r^{\alpha+1}\left(\mathfrak{a}_1\sin(\beta r)+\mathfrak{a}_2\cos(\beta r)\right).
    \end{align*}
    where $\lambda\in \R$. First, we want to refine \eqref{non_algebrique_dépassé}. For all $0<\epsilon<1$ and $0<\beta<\infty$, there exists $0<b=b(\epsilon,\beta)<\infty$ such that 
    \begin{align*}
    \left\{\begin{alignedat}{2}
        &1-\epsilon\leq \cos(\beta r)\leq 1\qquad &&\text{for all}\;\, 0<r<b\\
        &(1-\epsilon)\beta r\leq \sin(\beta r)\leq \sin(\beta r)\qquad &&\text{for all}\;\, 0<r<b.
    \end{alignedat}\right.
    \end{align*}
    Then, \eqref{cross_term_est1} becomes
    \begin{align*}
        &\int_{\Omega\times \Omega}\left|\det\begin{pmatrix}
            f_1(x) & f_2(x)\\
            f_1(y) & f_2(y)
        \end{pmatrix}\right|^2dx\,dy=4\pi^2\int_{a}^{b}\int_{a}^{b}\sin^2\left(\beta(r-s)\right)r^{2\alpha+1}s^{\alpha+1}dr\,ds\\
        &\geq 4\pi^2(1-\epsilon)^2\beta^2\int_{a}^{b}\int_{a}^{b}(r-s)^2r^{2\alpha+1}s^{2\alpha+1}dr\,ds. 
    \end{align*}
    The next estimate must also be handled more carefully. We have
    \begin{align*}
        &\frac{1}{2}\int_{a}^{b}\int_{a}^{b}(r-s)^2r^{2\alpha+1}s^{2\alpha+1}dr\,ds=\frac{b^{2(2\alpha+3)}}{4(\alpha+1)(\alpha+2)}\left(1-\left(\frac{a}{b}\right)^{2(\alpha+1)}-\left(\frac{a}{b}\right)^{2(\alpha+2)}+\left(\frac{a}{b}\right)^{2(2\alpha+3)}\right)\\
        &-\frac{b^{2(2\alpha+3)}}{(2\alpha+3)^2}\left(1-2\left(\frac{a}{b}\right)^{2\alpha+3}+\left(\frac{a}{b}\right)^{2(2\alpha+3)}\right)\\
        &=\frac{1}{4(\alpha+1)(\alpha+2)(2\alpha+3)^2}b^{2(2\alpha+3)}\\
        &\times \left(1-(2\alpha+3)^2\left(\left(\frac{a}{b}\right)^{2(\alpha+1)}+\left(\frac{a}{b}\right)^{2(\alpha+2)}\right)+8(\alpha+1)(\alpha+2)\left(\frac{a}{b}\right)^{2\alpha+3}+\left(\frac{a}{b}\right)^{2(2\alpha+3)}\right)
    \end{align*}
    Since we computed the integral of a function positive on a set of full measure (since it only vanishes on the diagonal $\Delta=\Omega\times \Omega\cap\ens{(x,y):x=y}$), we deduce in particular that for all $0<a<b<\infty$, we have 
    \begin{align*}
        \Lambda\left(\alpha,\log\left(\frac{b}{a}\right)\right)&=1-(2\alpha+3)^2\left(\left(\frac{a}{b}\right)^{2(\alpha+1)}+\left(\frac{a}{b}\right)^{2(\alpha+2)}\right)+8(\alpha+1)(\alpha+2)\left(\frac{a}{b}\right)^{2\alpha+3}+\left(\frac{a}{b}\right)^{2(2\alpha+3)}\\
        &>0,
    \end{align*}
    where 
    \begin{align*}
        \Lambda(\alpha,t)=1-(2\alpha+3)^2\left(e^{-2(\alpha+1)t}+e^{-2(\alpha+2)t}\right)+8(\alpha+1)(\alpha+2)e^{-(2\alpha+3)t}+e^{-2(2\alpha+3)t}, \quad t> 0.
    \end{align*}
    Therefore, we deduce that
    \begin{align*}
        \int_{\Omega\times \Omega}\left|\det\begin{pmatrix}
            f_1(x) & f_2(x)\\
            f_1(y) & f_2(y)
        \end{pmatrix}\right|^2dx\,dy&=4\pi^2\beta^2(1-\epsilon)^2\times \frac{1}{2(\alpha+1)(\alpha+2)(2\alpha+3)}b^{2(2\alpha+3)}\Lambda\left(\alpha,\log\left(\frac{b}{a}\right)\right)\\
        &=\frac{2\pi^2\beta^2(1-\epsilon)^2}{(\alpha+1)(\alpha+2)(2\alpha+3)}b^{2(2\alpha+3)}\Lambda\left(\alpha,\log\left(\frac{b}{a}\right)\right),
    \end{align*}
    while
    \begin{align*}
        \np{f_1}{2}{\Omega}\np{f_2}{2}{\Omega}\leq \frac{\pi\beta }{\alpha+1}b^{2\alpha+3}\sqrt{\left(1-\left(\frac{a}{b}\right)^{2(\alpha+1)}\right)\left(1-\left(\frac{a}{b}\right)^{2(\alpha+2)}\right)},
    \end{align*}
    and we finally get
    \begin{align*}
        &\frac{\pi^2\beta^2(1-\epsilon)^2}{(\alpha+1)(\alpha+2)(2\alpha+3)}|\mathfrak{a}_1||\mathfrak{a_2}|b^{2(2\alpha+3)}\Lambda\left(\alpha,\log\left(\frac{b}{a}\right)\right)\\
        &\leq \frac{\pi\beta }{\sqrt{(\alpha+1)(\alpha+2)}}b^{2\alpha+3}\sqrt{\left(1-\left(\frac{a}{b}\right)^{2(\alpha+1)}\right)\left(1-\left(\frac{a}{b}\right)^{2(\alpha+2)}\right)}\int_{\Omega}|f|^2dx,
    \end{align*}
    which becomes
    \begin{align}
        &|\mathfrak{a}_1||\mathfrak{a}_2|b^{2(2\alpha+3)}\label{precised_crossed_non-algébrique}\\
        &\leq \frac{\beta\sqrt{(\alpha+1)(\alpha+2)}(2\alpha+3)\sqrt{\left(1-\left(\frac{a}{b}\right)^{2(\alpha+1)}\right)\left(1-\left(\frac{a}{b}\right)^{2(\alpha+2)}\right)}}{\pi(1-\epsilon)^2\left(1-(2\alpha+3)^2\left(\left(\frac{a}{b}\right)^{2(\alpha+1)}+\left(\frac{a}{b}\right)^{2(\alpha+2)}\right)+8(\alpha+1)(\alpha+2)\left(\frac{a}{b}\right)^{2\alpha+3}+\left(\frac{a}{b}\right)^{2(2\alpha+3)}\right)}\int_{\Omega}|f|^2dx.\nonumber
    \end{align}
    
    Likewise, if $\alpha>2$ and
    \begin{align*}
        g(r,\theta)=r^{-\alpha}\mathfrak{a}_1\cos(\beta r)+r^{-\alpha}\mathfrak{a}_2\sin(\beta r),
    \end{align*}
    we have
    \begin{align}\label{precised_crossed_non-algébrique_2}
        &|\mathfrak{a}_1||\mathfrak{a}_2|\frac{1}{a^{2(2\alpha-3)}}
        \\
        &\leq \frac{\beta\sqrt{(\alpha-1)(\alpha-2)}(2\alpha-3)\sqrt{\left(1-\left(\frac{a}{b}\right)^{2(\alpha-1)}\right)\left(1-\left(\frac{a}{b}\right)^{2(\alpha-2)}\right)}}{\pi(1-\epsilon)^2\left(1-(2\alpha-3)\left(\left(\frac{a}{b}\right)^{2(\alpha-1)}+\left(\frac{a}{b}\right)^{2(\alpha-2)}\right)+8(\alpha-1)(\alpha-2)\left(\frac{a}{b}\right)^{2\alpha-3}\right)+\left(\frac{a}{b}\right)^{2(2\alpha-3)}}\int_{\Omega}|g|^2dx.\nonumber
    \end{align}
    
    Now, for $\lambda>0$, we can actually make a direct estimate since the crossed terms yield an almost positive contribution. Indeed, they are equal to
    \begin{align}
    	&\lambda\,r^{2\alpha+1}\Re\left(\left(\mathfrak{a}_1\cos(\beta r)+\mathfrak{a}_2\sin(\beta r)\right)\left(\bar{\mathfrak{a}_1}\sin(\beta r)+\bar{\mathfrak{a}_2}\cos(\beta r)\right)\right)\nonumber\\
    	&=\lambda\,r^{2\alpha+1}\left(\left(|\mathfrak{a}_1|^2+|\mathfrak{a}_2|^2\right)\cos(\beta r)\sin(\beta r)+\mathrm{Re}\left(\mathfrak{a}_1\bar{\mathfrak{a}_2}\right)\right)\nonumber\\
    	&\geq (1-\epsilon)\beta\lambda\,r^{2\alpha+2}+\lambda\,\mathrm{Re}\left(\mathfrak{a}_1\bar{\mathfrak{a}_2}\right)r^{2\alpha+1}.
    \end{align}
    However, we will not get the growth in $|n|$ in the coefficient so we will instead use $\mathfrak{L}_mu_n$. Recall that
    \begin{align*}
    	8\frac{z^2}{|z|^2}\mathfrak{L}_m\left(r^{\alpha}\Re\left(a\,e^{in\theta}\right)\right)
    	&=(\alpha+m+n-1)(\alpha+m+n-3)a\,r^{\alpha-2}e^{in\theta}\\
    	&+(\alpha+m-n-1)(\alpha+m-n-3)\bar{a}\,r^{\alpha-2}e^{-in\theta}.
    \end{align*}
    Recall also that
    \begin{align}
    	\p{z}=\frac{1}{2}\left(e^{-i\theta}\p{r}-i\,e^{-i\theta}\frac{1}{r}\p{\theta}\right).
    \end{align}
    Therefore, we have
    \begin{align}
    	&\p{z}(r^{\alpha}\cos(\beta r)e^{in\theta})=\frac{1}{2}\left(\left(\alpha+n\right)r^{\alpha-1}\cos(\beta r)-\beta\,r^{\alpha}\sin(\beta r)\right)e^{i(n-1)\theta}\nonumber\\
    	&\p{z}^2\left(r^{\alpha}\cos(\beta r)e^{in\theta}\right)=\frac{1}{4}\left((\alpha+n)(\alpha-1)r^{\alpha-2}\cos(\beta r)+\beta(\alpha+n)r^{\alpha-1}\cos(\beta r)-\alpha\beta r^{\alpha-1}\sin(\beta r)\right.\nonumber\\
    	&\left.-\beta^2r^{\alpha}\cos(\beta r)-n\left((\alpha+n)r^{\alpha-1}\cos(\beta r)-\beta \,r^{\alpha}\sin(\beta r)\right)\right)e^{i(n-2)\theta}\nonumber\\
        &=\frac{1}{4}\left((\alpha+n)(\alpha+n-2)r^{\alpha-2}\cos(\beta r)-\beta(2\alpha+2n-1)\sin(\beta r)-\beta^2r^{\alpha}\cos(\beta r)\right)e^{i(n-2)\theta}.
    \end{align}
    Therefore, we get
    \begin{align*}
        \mathfrak{L}_m\left(r^{\alpha}\cos(\beta r)e^{in\theta}\right)&=\frac{1}{4}\left((\alpha+m+n-2)(\alpha+m+n-3)r^{\alpha-2}\cos(\beta r)\right.\\
        &\left.-\beta(2\alpha+2m+2n-3)r^{\alpha-1}\sin(\beta r)-\beta^2r^{\alpha}\cos(\beta r)\right)e^{i(n-2)\theta}.
    \end{align*}
    Likewise, we have
    \begin{align*}
        \mathfrak{L}_m\left(r^{\alpha}\sin(\beta r)e^{in\theta}\right)&=\frac{1}{4}\left((\alpha+m+n-2)(\alpha+m+n-3)r^{\alpha-2}\sin(\beta r)\right.\\
        &\left.+\beta(2\alpha+2m+2n-3)r^{\alpha-1}\cos(\beta r)-\beta^2r^{\alpha}\sin(\beta r)\right)e^{i(n-2)\theta},
    \end{align*}
    and finally
    \begin{align*}
        &8\frac{z^2}{|z|^2}\mathfrak{L}_m\left(r^{\alpha}\cos(\beta r)\Re\left(a\,e^{in\theta}\right)\right)=\left((\alpha+m+n-1)(\alpha+m+n-3)\cos(\beta r)\right.\\
        &\left.-\beta(2\alpha+2m+2n-3)r\sin(\beta r)-\beta^2r^2\cos(\beta r)\right)a\,r^{\alpha-2}e^{in\theta}\\
        &+\left((\alpha+m-n-1)(\alpha+m-n-3)\cos(\beta r)\right.\\
        &\left.-\beta(2\alpha+2m-2n-3)r\sin(\beta r)-\beta^2r^2\cos(\beta r)\right)a\,r^{\alpha-2}e^{-in\theta}
    \end{align*}
    and
    \begin{align*}
        &8\frac{z^2}{|z|^2}\mathfrak{L}_m\left(r^{\alpha}\sin(\beta r)\Re\left(a\,e^{in\theta}\right)\right)=\left((\alpha+m+n-1)(\alpha+m+n-3)\sin(\beta r)\right.\\
        &\left.+\beta(2\alpha+2m+2n-3)r\cos(\beta r)-\beta^2r^2\sin(\beta r)\right)a\,r^{\alpha-2}e^{in\theta}\\
        &+\left((\alpha+m-n-1)(\alpha+m-n-3)\sin(\beta r)\right.\\
        &\left.+\beta(2\alpha+2m-2n-3)r\cos(\beta r)-\beta^2r^2\sin(\beta r)\right)a\,r^{\alpha-2}e^{-in\theta}.
    \end{align*}
    Therefore, we see that the main coefficient behaves likes $n^2$, which will yield a $n^3$ growth at infinity if it were not for the crossed terms. Taking them into account and using inequalities, as $\alpha_{m,n}=|n|+m(m-2)+1+O\left(\frac{1}{n^2}\right)$, we see by inequalities \eqref{precised_crossed_non-algébrique} and \eqref{precised_crossed_non-algébrique_2} (see also \eqref{non_algebrique_dépassé}) that we will only get an inequality of the type (where $\Gamma$ is a universal constant)
    \begin{align*}
        |a_n||b_n|\leq \Gamma\,n^2\int_{\Omega}|\mathfrak{L}_mu_n|^2\left(\left(\frac{|x|}{b}\right)^{2\alpha}+\left(\frac{a}{|x|}\right)^{2\alpha}\right)dx,
    \end{align*}
    which in turns will imply by the same argument in the \textbf{Step 4} that 
    \begin{align*}
        &\pi|n|\left(|a_n|^2b^{2\alpha_{m,n}}+|b_n|^2b^{2(\alpha_{m,n}+1)}+|c_n|^2\frac{1}{a^{2\alpha_{m,n}}}+|d_n|^2\frac{1}{a^{2(\alpha_{m,n}-1)}}\right)\\
        &\leq \Gamma\int_{\Omega}|\mathfrak{L}_mu_n|^2\left(\left(\frac{|x|}{b}\right)^{2\alpha}+\left(\frac{a}{|x|}\right)^{2\alpha}\right)dx. 
    \end{align*}
    Therefore, the $n^3$ growth turns into a growth in $|n|$, and this matches the growth found in \eqref{norm_op_u}. This estimate for $\dfrac{1}{r}\p{\theta}u$ follows also and this concludes the proof of the theorem. 
    \end{proof}

\section{Bochner's Identity and the Eigenvalue Estimate}

\begin{cor}\label{gagliardo_nirenberg_frak_m}
    Fix $3\leq m<\infty$. For all $0<\gamma<1$ and for all $\gamma/2\leq \beta<\min\ens{\dfrac{m-1}{4},1}$, there exists a constant $\Lambda_{\gamma}<\infty$ such that for all $0<a<b<\infty$ such that
    \begin{align*}
        \log\left(\frac{b}{a}\right)\geq \Lambda_{\gamma},
    \end{align*}
    the following property is verified. Define $\Omega=B_b\setminus\bar{B}_a(0)\subset\R^2$. There exists a constant $C_{m,\beta,\gamma}<\infty $ such that for all $u\in W^{2,2}(\Omega)$, 
    \begin{align}\label{gagliardo_nirenberg_frak_m_ineq}
        &\int_{\Omega}\frac{1}{|x|^2}\left|\D u+(m-1)\frac{x}{|x|^2}u\right|^2\left(\left(\frac{|x|}{b}\right)^{2\gamma}+\left(\frac{a}{|x|}\right)^{2\gamma}\right)dx\leq C_{m,\beta,\gamma}\left(\int_{\Omega}\frac{u^2}{|x|^4}\left(\left(\frac{|x|}{b}\right)^{4\beta}+\left(\frac{a}{|x|}\right)^{4\beta}\right)dx\right.\nonumber\\
        &\left.+\int_{\Omega}\left(\leb_m u\right)^2\left(\left(\frac{|x|}{b}\right)^{2\gamma}+\left(\frac{a}{|x|}\right)^{2\gamma}\right)dx+\int_{\Omega}\left|\mathfrak{L}_mu\right|^2\left(\left(\frac{|x|}{b}\right)^{2\gamma}+\left(\frac{a}{|x|}\right)^{2\gamma}\right)dx\right).
    \end{align}
\end{cor}
\begin{proof}
    Make the decomposition $u=u_1+u_2$, where $u_1$ solves the elliptic system
    \begin{align}\label{elliptic_shift}
        \left\{\begin{alignedat}{2}
            \mathrm{Re}\left(\bar{\mathfrak{L}_m}^{\ast}\left(\left(\frac{|x|}{b}\right)^{2\gamma}+\left(\frac{a}{|x|}\right)^{2\gamma}\right)\mathfrak{L}_m\right)u_1&=0\qquad&&\text{in}\;\,\Omega\\
            u_1&=u\qquad&&\text{on}\;\,\partial \Omega
        \end{alignedat}\right.
    \end{align}
    that we obtain by direct minimisation of 
    \begin{align*}
        \mathfrak{E}_m(v)=\int_{\Omega}|\mathfrak{L}_mv|^2\left(\left(\frac{|x|}{b}\right)^{2\gamma}+\left(\frac{a}{|x|}\right)^{2\gamma}\right)dx
    \end{align*}
    on $W^{2,2}_u(\Omega)$. In particular, $u_2\in W^{2,2}_0(\Omega)$, which implies by Theorem \ref{poincare_weight_m_2} that for all $0<\gamma<1$ (since $\dfrac{m-1}{2}\geq 1$), we have
    \begin{align}\label{final_frak_lm_1}
        &\int_{\Omega}\frac{1}{|x|^2}\left|\D u_2+(m-1)\frac{x}{|x|^2}u_2\right|^2\left(\left(\frac{|x|}{b}\right)^{2\gamma}+\left(\frac{a}{|x|}\right)^{2\gamma}\right)dx\nonumber\\
        &\leq C_{m,\gamma}\int_{\Omega}|\mathfrak{L}_mu_2|^2\left(\left(\frac{|x|}{b}\right)^{2\gamma}+\left(\frac{a}{|x|}\right)^{2\gamma}\right)dx.
    \end{align}
    Furthermore, since $u_1$ minimises $\mathfrak{E}_m$ on $W^{2,2}_u(\Omega)$, we have
    \begin{align*}
        \int_{\Omega}|\mathfrak{L}_mu_1|^2\left(\left(\frac{|x|}{b}\right)^{2\gamma}+\left(\frac{a}{|x|}\right)^{2\gamma}\right)dx\leq \int_{\Omega}|\mathfrak{L}_mu|^2\left(\left(\frac{|x|}{b}\right)^{2\gamma}+\left(\frac{a}{|x|}\right)^{2\gamma}\right)dx.
    \end{align*}
    Therefore, we have
    \begin{align}\label{final_frak_lm_extra}
        &\int_{\Omega}\left|\mathfrak{L}_mu_2\right|^2\left(\left(\frac{|x|}{b}\right)^{2\gamma}+\left(\frac{a}{|x|}\right)^{2\gamma}\right)dx\leq 2\int_{\Omega}\left|\mathfrak{L}_mu_1\right|^2\left(\left(\frac{|x|}{b}\right)^{2\gamma}+\left(\frac{a}{|x|}\right)^{2\gamma}\right)dx\nonumber\\
        &+2\int_{\Omega}\left|\mathfrak{L}_mu\right|^2\left(\left(\frac{|x|}{b}\right)^{2\gamma}+\left(\frac{a}{|x|}\right)^{2\gamma}\right)dx
        \leq 4\int_{\Omega}\left|\mathfrak{L}_mu\right|^2\left(\left(\frac{|x|}{b}\right)^{2\gamma}+\left(\frac{a}{|x|}\right)^{2\gamma}\right)dx.
    \end{align}
    Furthermore, the proof of Theorem \ref{weighted_estimate_frak_lm} shows that the theorem also holds for solutions of \eqref{elliptic_shift} (the explanation is that adding the weight \emph{a priori} in the PDE has the same effect to shift the frequencies as if we added them \emph{a posteriori}) 
    so we estimate for all $0<\gamma<1$ and $0<\beta<\mathrm{min}\ens{\dfrac{m-1}{4},1}$
    \begin{align}\label{final_frak_lm_2}
        &\int_{\Omega}\frac{1}{|x|^2}\left|\D u_1+(m-1)\frac{x}{|x|^2}u_1\right|^2\left(\left(\frac{|x|}{b}\right)^{2\gamma}+\left(\frac{a}{|x|}\right)^{2\gamma}\right)dx\leq C_{m,\gamma}\left(\int_{\Omega}\frac{u_1^2}{|x|^4}\left(\left(\frac{|x|}{b}\right)^{4\beta}+\left(\frac{a}{|x|}\right)^{4\beta}\right)dx\right.\nonumber\\
        &\left.+\int_{\Omega}(\leb_mu_1)^2\left(\left(\frac{|x|}{b}\right)^{2\gamma}+\left(\frac{a}{|x|}\right)^{2\gamma}\right)dx+\int_{\Omega}\left|\mathfrak{L}_mu_1\right|^2\left(\left(\frac{|x|}{b}\right)^{2\gamma}+\left(\frac{a}{|x|}\right)^{2\gamma}\right)dx\right).
    \end{align}
    We first have
    \begin{align*}
        &\int_{\Omega}\frac{u_1^2}{|x|^4}\left(\left(\frac{|x|}{b}\right)^{4\beta}+\left(\frac{a}{|x|}\right)^{4\beta}\right)dx\\
        &\leq 2\int_{\Omega}\frac{u^2}{|x|^4}\left(\left(\frac{|x|}{b}\right)^{4\beta}+\left(\frac{a}{|x|}\right)^{4\beta}\right)dx+2\int_{\Omega}\frac{u_2^2}{|x|^4}\left(\left(\frac{|x|}{b}\right)^{4\beta}+\left(\frac{a}{|x|}\right)^{4\beta}\right)dx\\
        &\leq 2\int_{\Omega}\frac{u^2}{|x|^4}\left(\left(\frac{|x|}{b}\right)^{4\beta}+\left(\frac{a}{|x|}\right)^{4\beta}\right)dx+C_{m,\beta}\int_{\Omega}|\mathfrak{L}_mu_2|^2\left(\left(\frac{|x|}{b}\right)^{4\beta}+\left(\frac{a}{|x|}\right)^{4\beta}\right)dx\\
        &\leq 2\int_{\Omega}\frac{u^2}{|x|^4}\left(\left(\frac{|x|}{b}\right)^{4\beta}+\left(\frac{a}{|x|}\right)^{4\beta}\right)dx+C_{m,\beta}\int_{\Omega}\left|\mathfrak{L}_mu_2\right|^2\left(\left(\frac{|x|}{b}\right)^{2\gamma}+\left(\frac{a}{|x|}\right)^{2\gamma}\right)dx\\
        &\leq 2\int_{\Omega}\frac{u^2}{|x|^4}\left(\left(\frac{|x|}{b}\right)^{4\beta}+\left(\frac{a}{|x|}\right)^{4\beta}\right)dx+4\,C_{m,\beta}\int_{\Omega}\left|\mathfrak{L}_mu\right|^2\left(\left(\frac{|x|}{b}\right)^{2\gamma}+\left(\frac{a}{|x|}\right)^{2\gamma}\right)dx
    \end{align*}
    provided that $4\beta\geq 2\gamma$, where we used \eqref{final_frak_lm_extra}.
    Then, notice that for all $v\in W^{2,2}_0(\Omega)$
    \begin{align*}
        &\int_{\Omega}|\p{z}^2v|^2|z|^{2\alpha}|dz|^2=\int_{\Omega}z^{\alpha}\z^{\alpha}\p{z}^2v\cdot \p{\z}^2v|dz|^2=-\int_{\Omega}z^{\alpha}\z^{\alpha}\p{z}(\p{z\z}^2v)\cdot \p{\z}v\,|dz|^2\\
        &-\alpha\int_{\Omega}z^{\alpha}\z^{\alpha-1}\p{z}^2v\cdot \p{\z}v\,|dz|^2\\
        &=\int_{\Omega}|z|^{2\alpha}|\p{z\z}^2v|^2|dz|^2+\alpha\int_{\Omega}z^{\alpha-1}\z^{\alpha}\p{z\z}^2v\cdot \p{\z}u\,|dz|^2-\alpha\int_{\Omega}z^{\alpha}\z^{\alpha-1}\p{z}^2v\cdot \p{\z}v|dz|^2.
    \end{align*}
    Therefore, we deduce that
    \begin{align*}
        \int_{\Omega}|z|^{2\alpha}|\p{z\z}^2v|^2|dz|^2\leq \frac{1}{2}\int_{\Omega}|z|^{2\alpha}|\p{z\z}^2v|^2|dz|^2+\frac{3\alpha^2}{2}\int_{\Omega}\frac{|\p{z}v|^2}{|z|^2}|z|^{2\alpha}|dz|^2+\frac{3}{2}\int_{\Omega}|z|^{2\alpha}|\p{z}^2v|^2|dz|^2,
    \end{align*}
    and this implies that
    \begin{align*}
        \int_{\Omega}(\Delta v)^2|x|^{2\alpha}dx\leq 12\alpha^2\int_{\Omega}|\p{z}v|^2|z|^{2\alpha}|dz|^2+12\int_{\Omega}|\p{z}^2v|^2|dz|^2,
    \end{align*}
    or
    \begin{align*}
        &\int_{\Omega}\left(\leb_mu_2\right)^2\left(\left(\frac{|x|}{b}\right)^{2\gamma}+\left(\frac{a}{|x|}\right)^{2\gamma}\right)dx\\
        &\leq 12\gamma^2\int_{\Omega}\frac{1}{|x|^2}\left|\D u_2+(m-1)\frac{x}{|x|^2}u_2\right|^2\left(\left(\frac{|x|}{b}\right)^{2\gamma}+\left(\frac{a}{|x|}\right)^{2\gamma}\right)dx\\
        &+12\int_{\Omega}|\mathfrak{L}_mu_2|^2\left(\left(\frac{|x|}{b}\right)^{2\gamma}+\left(\frac{a}{|x|}\right)^{2\gamma}\right)dx\\
        &\leq C_{m,\gamma}\int_{\Omega}\left|\mathfrak{L}_mu\right|^2\left(\left(\frac{|x|}{b}\right)^{2\gamma}+\left(\frac{a}{|x|}\right)^{2\gamma}\right)dx.
    \end{align*}
    Therefore, we have
    \begin{align}\label{final_frak_lm_3}
        &\int_{\Omega}(\leb_mu_1)^2\left(\left(\frac{|x|}{b}\right)^{2\gamma}+\left(\frac{a}{|x|}\right)^{2\gamma}\right)dx\leq 2\int_{\Omega}\left(\leb_mu\right)^2\left(\left(\frac{|x|}{b}\right)^{2\gamma}+\left(\frac{a}{|x|}\right)^{2\gamma}\right)dx\nonumber\\
        &+2\int_{\Omega}\left(\leb_mu_2\right)^2\left(\left(\frac{|x|}{b}\right)^{2\gamma}+\left(\frac{a}{|x|}\right)^{2\gamma}\right)dx\leq 2\int_{\Omega}\left(\leb_mu\right)^2\left(\left(\frac{|x|}{b}\right)^{2\gamma}+\left(\frac{a}{|x|}\right)^{2\gamma}\right)dx\nonumber\\
        &+C_{m,\gamma}\int_{\Omega}\left|\mathfrak{L}_mu\right|^2\left(\left(\frac{|x|}{b}\right)^{2\gamma}+\left(\frac{a}{|x|}\right)^{2\gamma}\right)dx.
    \end{align}
    Therefore, the estimate follows from \eqref{final_frak_lm_1}, \eqref{final_frak_lm_2}, and \eqref{final_frak_lm_3}.
\end{proof}

\section{Morse Triviality of Variations Located in the Neck}

In this section, generalising \cite[Theorem 1.4.1 p. 52]{willmore_scs} (that corresponds to the case $m=1$ in the theorem), we will show that there cannot be negative variations whose support is located in the neck regions. 

\begin{theorem}\label{lower_bound_main_theorem}
    Let $\{\phi_k\}_{k\in\N}$ be a sequence of Willmore immersions such that
    \begin{align*}
        \limsup_{k\rightarrow \infty}W(\phi_k)<\infty.
    \end{align*}
    Let $\Omega_k(\alpha)=B_{\alpha}\setminus\bar{B}_{\alpha^{-1}}(0)$ be a neck region, and let $m\geq 2$ be the integer multiplicity of the branch point of the limiting bubble. Then, for all $0<\beta_1,\beta_2<1$ such that $0<\beta_1+\dfrac{\beta_2}{2}<\min\ens{\dfrac{m-1}{4},1}$ and $\beta_1\neq \dfrac{1}{2}$ and $\beta_2\neq 1$, there exists a universal constant $\lambda_0=\lambda_0(\beta_1,\beta_2)>0$ such that for $k$ large enough, we have for all $u_k\in W^{2,2}_0(\Omega_k(\alpha))$
    \begin{align*}
        Q_{\phi_k}(u_k)&\geq \lambda_0\int_{\Omega_k(\alpha)}\frac{u_k^2}{|x|^{2m+2}}\left(\left(\frac{|x|}{\alpha}\right)^{4\beta_1}+\left(\frac{\alpha^{-1}\rho_k}{|x|}\right)^{4\beta_1}+\frac{1}{\log^2\left(\frac{\alpha^2}{\rho_k}\right)}\right)dx\\
        &+\lambda_0\int_{\Omega_k(\alpha)}\frac{|\D u_k|^2}{|x|^{2m}}\left(\left(\frac{|x|}{\alpha}\right)^{2\beta_2}+\left(\frac{\alpha^{-1}\rho_k}{|x|}\right)^{2\beta_2}+\frac{1}{\log^2\left(\frac{\alpha^2}{\rho_k}\right)}\right)dx
    \end{align*}
\end{theorem}
\begin{proof}
    Recall that by \cite{pointwise}, there exists a bounded function $\mu_k\in C^0(\Omega_k(\alpha))$ such that
    \begin{align}\label{conformal_expansion}
        e^{\lambda_k}=e^{\mu_k}|x|^{m-1}\qquad\text{for all}\;\, x\in \Omega_k(\alpha).
    \end{align}
    In particular, there exists a constant $A\in \R$ such that
    \begin{align*}
        e^{-A}|x|^{m-1}\leq e^{\lambda_k}\leq e^{A}|x|^{m-1}.
    \end{align*}
    Therefore, we deduce in particular that for all $u\in W^{2,2}(\Omega_k(\alpha))$
    \begin{align*}
        \int_{\Omega_k(\alpha)}(\Delta_{g_k}u)^2\,d\mathrm{vol}_{g_k}=\int_{\Omega_k(\alpha)}e^{-2\lambda_k}(\Delta u)^2dx\geq e^{-A}\int_{\Omega_k(\alpha)}|x|^{2-2m}(\Delta u)^2dx.
    \end{align*}
    Letting $v=|x|^{m-1}u\in W^{2,2}_0(\Omega_k(\alpha))$, we get
    \begin{align*}
        \int_{\Omega_k(\alpha)}|x|^{2-2m}(\Delta u_k)^2dx=\int_{\Omega_k(\alpha)}\left(\Delta v+2(m-1)\frac{x}{|x|^2}\cdot \D v+\frac{(m-1)^2}{|x|^2}v\right)^2dx.
    \end{align*}
    Now, thanks to \cite[Corollary 2.1.2 p. 93, Corollary 2.1.5 p. 94]{willmore_scs}, for $k$ large enough, we have
    \begin{align*}
        \int_{\Omega_k(\alpha)}\left(\Delta v+2(m-1)\frac{x}{|x|^2}\cdot \D v+\frac{(m-1)^2}{|x|^2}v\right)^2dx\geq \left(4m^2+\frac{\pi^2}{\log^2\left(\frac{\alpha^2}{\rho_k}\right)}\right)\frac{\pi^2}{\log^2\left(\frac{\alpha^2}{\rho_k}\right)}\int_{\Omega_k(\alpha)}\frac{v^2}{|x|^4}dx
    \end{align*}
    and
    \begin{align*}
        \int_{\Omega_k(\alpha)}\left(\Delta v+2(m-1)\frac{x}{|x|^2}\cdot \D v+\frac{(m-1)^2}{|x|^2}v\right)^2dx\geq \frac{4m^2+\dfrac{\pi^2}{\log^2\left(\frac{\alpha^2}{\rho_k}\right)}}{4(m^2+1)+\dfrac{2\pi^2}{\log^2\left(\frac{\alpha^2}{\rho_k}\right)}}\frac{\pi^2}{\log^2\left(\frac{\alpha^2}{\rho_k}\right)}\int_{\Omega_k(\alpha)}\frac{|\D v|^2}{|x|^2}dx.
    \end{align*}
    On the other hand, Theorem \ref{poincare_weight_m} shows that 
    \begin{align*}
        \int_{\Omega_k(\alpha)}(\leb_m v)^2dx\geq C_{m,\beta_1,1}^{-1}\int_{\Omega_k(\alpha)}\frac{v^2}{|x|^4}\left(\left(\frac{|x|}{\alpha}\right)^{4\beta_1}+\left(\frac{\alpha^{-1}\rho_k}{|x|}\right)^{4\beta_1}+\frac{1}{\log^2\left(\frac{\alpha^2}{\rho_k}\right)}\right)dx
    \end{align*}
    and
    \begin{align*}
        \int_{\Omega_k(\alpha)}(\leb_m v)^2dx\geq C_{m,\beta_2,2}^{-1}\int_{\Omega_k(\alpha)}\frac{|\D v|^2}{|x|^2}\left(\left(\frac{|x|}{\alpha}\right)^{2\beta_2}+\left(\frac{\alpha^{-1}\rho_k}{|x|}\right)^{2\beta_2}+\frac{1}{\log^2\left(\frac{\alpha^2}{\rho_k}\right)}\right)dx.
    \end{align*}
    Noticing that
    \begin{align*}
        \D v=\D\left(|x|^{1-m}u\right)=\frac{\D u_k}{|x|^{m-1}}-(m-1)u\frac{x}{|x|^{m+1}},
    \end{align*}
    all inequalities finally show that for some $C_0>0$, we have
    \begin{align}\label{positive_neck1}
        \int_{\Omega_k(\alpha)}(\Delta_{g_k}u)^2\,d\mathrm{vol}_{g_k}&\geq C_0\int_{\Omega_k(\alpha)}\frac{u^2}{|x|^{2m+2}}\left(\left(\frac{|x|}{\alpha}\right)^{4\beta_1}+\left(\frac{\alpha^{-1}\rho_k}{|x|}\right)^{4\beta_1}+\frac{1}{\log^2\left(\frac{\alpha^2}{\rho_k}\right)}\right)dx \nonumber\\
        &+C_0\int_{\Omega_k(\alpha)}\frac{|\D u|^2}{|x|^{2m}}\left(\left(\frac{|x|}{\alpha}\right)^{2\beta_2}+\left(\frac{\alpha^{-1}\rho_k}{|x|}\right)^{2\beta_2}+\frac{1}{\log^2\left(\frac{\alpha^2}{\rho_k}\right)}\right)dx.
    \end{align}
    Furthermore, inequality \cite[(1.2.14) p. 18]{willmore_scs} shows that for all $u\in W^{2,2}(\Sigma)$, we have
    \begin{align}\label{second_derivative_lower_bound}
        Q_{\phi_k}(u)&\geq \frac{1}{4}\int_{\Sigma}(\Delta_{g_k}u)^2\,d\mathrm{vol}_{g_k}-2\int_{\Sigma}|d u|_{g_k}^2|A_k|^2\,d\mathrm{vol}_{g_k}-2\int_{\Sigma}|A_k|^4u^2\,d\mathrm{vol}_{g_k}\nonumber\\
            &-16\int_{\Sigma}\s{\partial u\otimes \partial H}{h_0}_{\mathrm{WP}}u\,d\mathrm{vol}_{g_k}.
    \end{align}
    Now, for every $0<\beta<1$, and for all $k\in\N$, thanks to \cite[Theorem 1.3.30 p. 51]{willmore_scs} there exists $C_l(\beta)<\infty$ such that for small enough $0<\alpha<1$ and large enough $k\in \N$, we have
    \begin{align*}
        |x|^l|\D^l\n_k(x)|\leq C_k(\beta)\left(\left(\frac{|x|}{\alpha}\right)^{\beta}+\left(\frac{\alpha^{-1}\rho_k}{|x|}\right)^{\beta}+\frac{1}{\log\left(\frac{\alpha^2}{\rho_k}\right)}\right)\np{\D \n_k}{2}{\Omega_k(2\alpha)}\qquad\text{for all}\;\, x\in \Omega_k(\alpha).
    \end{align*}
    Therefore, fix $0<\beta_1,\beta_2<1$ satisfying the bounds of the theorem. We have 
    \begin{align*}
        &|A_k|\leq \frac{C_0'(\beta_1)}{|x|^{m}}\left(\left(\frac{|x|}{\alpha}\right)^{\beta_1}+\left(\frac{\alpha^{-1}\rho_k}{|x|}\right)^{\beta_1}+\frac{1}{\log\left(\frac{\alpha^2}{\rho_k}\right)}\right)\np{\D \n_k}{2}{\Omega_k(2\alpha)}\\
        &|\D A_k|\leq \frac{C_1'(\beta_2)}{|x|^{m+1}}\left(\left(\frac{|x|}{\alpha}\right)^{\beta_2}+\left(\frac{\alpha^{-1}\rho_k}{|x|}\right)^{\beta_2}+\frac{1}{\log\left(\frac{\alpha^2}{\rho_k}\right)}\right)\np{\D \n_k}{2}{\Omega_k(2\alpha)}.
    \end{align*}
    Therefore, we get the following estimates for all $u\in \Omega_k(\alpha)$
    \begin{align}\label{positive_neck2}
        \int_{\Sigma}|d u|_{g_k}^2|A_k|^2\,d\mathrm{vol}_{g_k}&\leq C\np{\D\n_k}{2}{\Omega_k(2\alpha)}^2\int_{\Omega_k(\alpha)}\frac{|\D u_k|^2}{|x|^{2m}}\left(\left(\frac{|x|}{\alpha}\right)^{2\beta_2}+\left(\frac{\alpha^{-1}\rho_k}{|x|}\right)^{2\beta_2}+\frac{1}{\log^2\left(\frac{\alpha^2}{\rho_k}\right)}\right)dx\nonumber\\
        \int_{\Sigma}|A_k|^4u^2\,d\mathrm{vol}_{g_k}&\leq C\np{\D\n_k}{2}{\Omega_k(2\alpha)}^4\int_{\Omega_k(\alpha)}\frac{u^2}{|x|^{2m}}\left(\left(\frac{|x|}{\alpha}\right)^{4\beta_1}+\left(\frac{\alpha^{-1}\rho_k}{|x|}\right)^{4\beta_1}+\frac{1}{\log^4\left(\frac{\alpha^2}{\rho_k}\right)}\right)dx\nonumber\\
        &\leq C\np{\D\n_k}{2}{\Omega_k(2\alpha)}^4\int_{\Omega_k(\alpha)}\frac{u^2}{|x|^{2m}}\left(\left(\frac{|x|}{\alpha}\right)^{4\beta_1}+\left(\frac{\alpha^{-1}\rho_k}{|x|}\right)^{4\beta_1}+\frac{1}{\log^2\left(\frac{\alpha^2}{\rho_k}\right)}\right)dx.
    \end{align}
    Then, we get by Cauchy--Schwarz inequality
    \begin{align}\label{positive_neck3}
        &\left|\int_{\Sigma}\s{\partial u\otimes \partial H_k}{h_0}_{\mathrm{WP}}u\,d\mathrm{vol}_{g_k}\right|\nonumber\\
        &\leq C\np{\D\n_k}{2}{\Omega_k(2\alpha)}^2\int_{\Sigma}\frac{|\D u||u|}{|x|^{2m+1}}\left(\left(\frac{|x|}{\alpha}\right)^{2\beta}+\left(\frac{\alpha^{-1}\rho_k}{|x|}\right)^{2\beta}+\frac{1}{\log^2\left(\frac{\alpha^2}{\rho_k}\right)}\right)dx\nonumber\\
        &\leq C\np{\D\n_k}{2}{\Omega_k(2\alpha)}^2\left(\int_{\Sigma}\frac{u^2}{|x|^{2m+2}}\left(\left(\frac{|x|}{\alpha}\right)^{4(1-\epsilon)\beta}+\left(\frac{\alpha^{-1}\rho_k}{|x|}\right)^{4(1-\epsilon)\beta}+\frac{1}{\log^2\left(\frac{\alpha^2}{\rho_k}\right)}\right)dx\right)^{\frac{1}{2}}\nonumber\\
        &\times\left(\int_{\Omega_k(\alpha)}\frac{|\D u|^2}{|x|^{2m}}\left(\left(\frac{|x|}{\alpha}\right)^{4\epsilon\beta}+\left(\frac{\alpha^{-1}\rho_k}{|x|}\right)^{4\epsilon\beta}+\frac{1}{\log^2\left(\frac{\alpha^2}{\rho_k}\right)}\right)dx\right)^{\frac{1}{2}}.
    \end{align}
    Now, we need to choose $0<\beta<1$ and $0<\epsilon<1$ such that $(1-\epsilon)\beta\geq \beta_1$ and $2\epsilon\beta\geq \beta_2$, which yields
    \begin{align*}
        &\beta\geq \beta_1+\frac{\beta_2}{2}\\
        &\frac{\beta_2}{2\beta}\leq \epsilon\leq 1-\frac{\beta_1}{\beta}.  
    \end{align*}
    Therefore, we can find a solution provided that $0<\beta_1+\dfrac{\beta_2}{2}<\beta<1$. 
    Therefore, we deduce by \eqref{positive_neck1}, \eqref{positive_neck2}, and \eqref{positive_neck3} that 
    \begin{align}\label{positive_neck4}
        Q_{\phi_k}(u)&\geq \left(\frac{1}{4}C_0-C\np{\D\n_k}{2}{\Omega_k(2\alpha)}^2\left(1+\np{\D\n_k}{2}{\Omega_k(2\alpha)}^2\right)\right)\\
        &\times \int_{\Omega_k(\alpha)}\frac{u^2}{|x|^{2m+2}}\left(\left(\frac{|x|}{\alpha}\right)^{4\beta_1}+\left(\frac{\alpha^{-1}\rho_k}{|x|}\right)^{4\beta_1}+\frac{1}{\log^2\left(\frac{\alpha^2}{\rho_k}\right)}\right)dx\\
        &+\left(\frac{1}{4}C_0-C\np{\D\n_k}{2}{\Omega_k(2\alpha)}\right)\int_{\Omega_k(\alpha)}\frac{|\D u|^2}{|x|^{2m}}\left(\left(\frac{|x|}{\alpha}\right)^{2\beta_2}+\left(\frac{\alpha^{-1}\rho_k}{|x|}\right)^{2\beta_2}+\frac{1}{\log^2\left(\frac{\alpha^2}{\rho_k}\right)}\right)dx.
    \end{align}
    Thanks to the quantization of energy (\cite{quanta}, notice that the previous Hölder-type pointwise bounds follow from the improved energy quantization of \cite{pointwise}), the result follows. %Indeed, notice that as $\epsilon\rightarrow \dfrac{1}{4\beta}$, $\beta_2\rightarrow 1/2$, and as $\epsilon\rightarrow 1-\dfrac{1}{2\beta}$, $\beta_1\rightarrow 1/2$, so the inequality is verified for all $1/2<\beta_1,\beta_2<1$ (and need not satisfy any relation between each other).
\end{proof}
Then, recall the formula for the second derivative of the Willmore energy:
\begin{align}\label{formula_second_derivative}
    Q_{\phi}(u)&=\frac{1}{2}\int_{\Sigma}\left(\Delta_gu+|A|^2u\right)^2d\vg+\int_{\Sigma}\left(|du|_g^2+4|h_0|_{\mathrm{WP}}u^2\right)H^2d\vg-8\int_{\Sigma}\s{\partial u\otimes \partial u}{h_0}_{\mathrm{WP}}H\,d\vg\nonumber\\
    &-16\int_{\Sigma}\s{\partial u\otimes \partial H}{h_0}_{\mathrm{WP}}u\,d\vg.
\end{align}
As in \cite[p. 56]{willmore_scs}, fix $0<\beta_1<\min\ens{\dfrac{m-1}{4},1}$ and introduce the weight
\begin{align*}
    \omega_{1,\alpha,k}(x)=\left\{\begin{alignedat}{2}
        &\frac{1}{\alpha^{4m}}\left(1+\left(\frac{\alpha^{-1}\delta_k}{\alpha}\right)^{4\beta_1}+\frac{1}{\log^2\left(\frac{\alpha^2}{\rho_k}\right)}\right)\qquad&&\forall x\in \Sigma\setminus\bar{B}(0,\alpha)\\
        &\frac{1}{|x|^{4m}}\left(\left(\frac{|x|}{\alpha}\right)^{4\beta_1}+\left(\frac{\alpha^{-1}\rho_k}{|x|}\right)^{4\beta_1}+\frac{1}{\log^2\left(\frac{\alpha^2}{\rho_k}\right)}\right)\qquad&&\forall x\in \Omega_k(\alpha)\\
        &\frac{\alpha^{4m}}{\rho_k^{4m}}\left(\frac{1}{\alpha^{4m+4}}\frac{(1+\alpha^2)^{2m+2}}{\left(1+\left(\frac{|x|}{\rho_k}\right)^{2}\right)^{2m+2}}+\left(\frac{\alpha^{-1}\rho_k}{\alpha}\right)^{4\beta_1}+\frac{1}{\log^2\left(\frac{\alpha^2}{\rho_k}\right)}\right)\qquad&&\forall x\in B(0,\alpha^{-1}\rho_k).
    \end{alignedat}\right.
\end{align*}
Notice that $\omega_{1,\alpha,k}\in C^0(\Sigma)$ for all $0<\alpha<1$ and $k\in \N$. 
Introducing the operator
\begin{align*}
    \mathcal{L}_{g_k}&=\Delta_{g_k}^2+|A_k|^2\Delta_{g_k}+2\s{d|A_k|}{d(\,\cdot\,)}_{g_k}+2\,d^{\ast_{g_k}}\left(H_k^2\,d(\,\cdot\,)\right)+\left(|A_k|^4+\Delta_{g_k}|A_k|^2+24|H_k|^2|h_{0,k}|_{\mathrm{WP}}^2\right)\\
    &+16\,\ast_{g_k}\,d\,\Re\left(g_k^{-1}\otimes h_{0,k}\otimes \partial\left(H_k\,\cdot\,\right)\right)-16\s{\partial(\,\cdot\,)\otimes \partial H_k}{h_{0,k}}_{\mathrm{WP}}.
\end{align*}
Then, $\mathcal{L}_{g_k}$ is a self-adjoint operator and the following identity holds for all $u\in W^{2,2}(\Sigma)$
\begin{align*}
    Q_{\phi_k}(u)=\frac{1}{2}\int_{\Sigma}u\,\mathcal{L}_{g_k}u\,d\mathrm{vol}_{g_k}.
\end{align*}
In particular, if we define $\mathcal{L}_{\alpha,k}=\omega_{1,\alpha,k}^{-1}\mathcal{L}_{g_k}$, we deduce that $\mathcal{L}_{\alpha,k}$ is a self-adjoint operator with respect to the weight $\omega_{1,\alpha,k}$ and that
\begin{align*}
    Q_{\phi_k}(u)=\frac{1}{2}\int_{\Sigma}u\,\mathcal{L}_{\alpha,k}u\,\omega_{1,\alpha,k}\,d\mathrm{vol}_{g_k}.
\end{align*}
The next lemma readily follows from the general analysis given in \cite[Lemma IV.3]{riviere_morse_scs}.
\begin{lemme}
    For all $\lambda\in \R$, $0<\alpha<1$, and $k\in \N$ let
    \begin{align*}
        \mathscr{E}_{\alpha,k}(\lambda)=W^{2,2}(\Sigma)\cap \ens{u:\mathcal{L}_{\alpha,k}u=\lambda\,u}.
    \end{align*}
    Then, for all $0<\alpha<1$ and $k\in \N$, the following formula holds true
    \begin{align*}
        \mathrm{Ind}_{W}(\phi_k)=\mathrm{dim}\bigoplus_{\lambda<0}\mathscr{E}_{\alpha,k}(\lambda).
    \end{align*}
\end{lemme}
We also need to make precise the Bochner identity in conformal coordinates.
\begin{lemme}\label{bochner_metric}
    Let $\Omega_k(\alpha)$ be the neck region as above. Then, for all $u\in W^{2,2}(\Omega_k(\alpha))$, if $g_k=\phi_k^{\ast}g_{\R^3}=e^{2\lambda_k}dx$ is the induced metric, we have
    \begin{align}\label{hessian_pointwise}
        |\mathrm{Hess}(u)|_{g_k}^2=e^{-4\lambda_k}\left(|\D^2u|^2+2(\D u\cdot \D\lambda_k)\Delta u-2\,\D\lambda_k\cdot \D|\D u|^2+2|\D\lambda_k|^2|\D u|^2\right).
    \end{align}
\end{lemme}
\begin{proof}
    Since the metric is conformal, we have $g^{i,j}=e^{-2\lambda_k}\delta_{i,j}$, which yields 
    \begin{align*}
        |\mathrm{Hess}(u)|_{g_k}^2=\sum_{i,j=1}^2e^{-4\lambda_k}\left(\frac{\partial^2u}{\partial x_i\partial x_j}-\sum_{l=1}^2\Gamma_{i,j}^l\frac{\partial u}{\partial x_l}\right)^2,
    \end{align*}
    where $\Gamma_{i,j}^l$ are the Christoffel symbols. Explicitly, we have
    \begin{align*}
        \Gamma_{i,j}^l&=\frac{1}{2}g^{1,l}\left(\frac{\partial g_{j,1}}{\partial x_i}+\frac{\partial g_{1,i}}{\partial x_j}-\frac{\partial g_{i,j}}{\partial x_1}\right)+\frac{1}{2}g^{2,l}\left(\frac{\partial g_{j,2}}{\partial x_i}+\frac{\partial g_{2,i}}{\partial x_j}-\frac{\partial g_{i,j}}{\partial x_2}\right)\\
        &=\delta_{1,l}\left(\delta_{1,j}\frac{\partial}{\partial x_i}+\delta_{1,i}\frac{\partial }{\partial x_j}-\delta_{i,j}\frac{\partial}{\partial x_1}\right)\lambda_k+\delta_{2,l}\left(\delta_{2,j}\frac{\partial}{\partial x_i}+\delta_{i,2}\frac{\partial}{\partial x_j}-\delta_{i,j}\frac{\partial}{\partial x_2}\right)\lambda_k.
    \end{align*}
    Therefore, we successively compute
    \begin{align*}
    \left\{\begin{alignedat}{1}
        \Gamma_{1,1}^1&=\frac{\partial \lambda_k}{\partial x_1}\\
        \Gamma_{1,1}^2&=-\frac{\partial \lambda_k}{\partial x_2}\\
        \Gamma_{1,2}^1&=\Gamma_{2,1}^1=\frac{\partial \lambda_k}{\partial x_2}\\
    \end{alignedat}\right.
    \qquad 
        \left\{\begin{alignedat}{1}
        \Gamma_{1,2}^2&=\Gamma_{2,1}^2=\frac{\partial \lambda_k}{\partial x_1}\\
        \Gamma_{2,2}^1&=-\frac{\partial \lambda_k}{\partial x_1}\\
        \Gamma_{2,2}^2&=\frac{\partial \lambda_k}{\partial x_2}.
    \end{alignedat}\right.
    \end{align*}
    Therefore, we deduce (removing the index $k$ on the right-hand side for simplicity of notation) that
    \begin{align*}
        &e^{4\lambda_k}|\mathrm{Hess}(u)|_{g_k}^2=\left(\partial_{x_1}^2u-\partial_{x_1}\lambda\,\partial_{x_1}u+\partial_{x_2}\lambda\,\partial_{x_2}u\right)^2
        +2\left(\partial_{x_1,x_2}^2u-\partial_{x_2}\lambda\,\partial_{x_1}u-\partial_{x_1}\lambda\,\partial_{x_2}u\right)^2\\
        &+\left(\partial_{x_2}^2u+\partial_{x_1}\lambda\,\partial_{x_1}u-\partial_{x_2}\lambda\,\partial_{x_2}u\right)^2\\
        &=|\D^2u|^2-2\,\partial_{x_1}\lambda\,\partial_{x_1}^2u\,\partial_{x_1}u+2\,\partial_{x_2}\lambda\,\partial_{x_1}^2u\,\partial_{x_2}u
        -4\,\partial_{x_2}\lambda\,\partial_{x_1,x_2}^2u\,\partial_{x_1}u-4\,\partial_{x_1}\lambda\,\partial_{x_1,x_2}^2u\,\partial_{x_2}u\\
        &+2\,\partial_{x_1}\lambda\,\partial_{x_2}^2u\,\partial_{x_1}u-2\,\partial_{x_2}\lambda\,\partial_{x_2}^2u\,\partial_{x_2}u+(\p{x_1}\lambda)^2(\p{x_1}u)^2+(\p{x_2}\lambda)^2(\p{x_2}u)^2-\colorcancel{2(\p{x_1}\lambda)(\p{x_2}\lambda)(\p{x_1}u)(\p{x_2}u)}{red}\\
        &+2(\p{x_2}\lambda)(\p{x_1}u)^2+2(\p{x_1}\lambda)^2(\p{x_2}u)^2+\colorcancel{4(\p{x_1}\lambda)(\p{x_2}\lambda)(\p{x_1}u)(\p{x_2}u)}{red}\\
        &+(\p{x_1}\lambda)^2(\p{x_1}u)^2+(\p{x_2}\lambda)^2(\p{x_2}u)^2-\colorcancel{2(\p{x_1}\lambda)(\p{x_2}\lambda)(\p{x_1}u)(\p{x_2}u)}{red}\\
        &=|\D^2u|^2-\partial_{x_1}\lambda\,\partial_{x_1}|\partial_{x_1}u|^2+2\,\partial_{x_2}\lambda\,\Delta u\,\partial_{x_2}u-\partial_{x_2}\lambda\,\partial_{x_2}|\partial_{x_2}u|^2
        -2\,\partial_{x_2}\lambda\,\partial_{x_2}|\partial_{x_1}u|^2-2\,\partial_{x_1}\lambda\,\partial_{x_1}|\partial_{x_2}u|^2\\
        &+2\,\partial_{x_1}\lambda\,\Delta u\,\partial_{x_1}u-\partial_{x_1}\lambda\,\partial_{x_1}|\partial_{x_1}u|^2-\partial_{x_2}\lambda\, \partial_{x_2}|\partial_{x_2}u|^2+2|\D\lambda|^2|\D u|^2\\
        &=|\D^2u|^2+2(\D u\cdot \D \lambda)\Delta u-2\D\lambda\cdot \D |\D u|^2+2|\D\lambda|^2|\D u|^2.
    \end{align*}
\end{proof}

%Next, we generalise \cite[Lemma 1.4.4]{willmore_scs} to the case of branch points. This is where our Gagliardo--Nirenberg inequality proves crucial.
%\begin{cor}\label{gagliardo_nirenberg_frak_m2}
 %   Fix $3\leq m<\infty$. For all $0<\gamma<1$ and for all $\gamma/2\leq \beta<\min\ens{\dfrac{m-1}{4},1}$, there exists a constant $\Lambda_{\gamma}<\infty$ such that for all $0<a<b<\infty$ such that
   % \begin{align*}
  %      \log\left(\frac{b}{a}\right)\geq \Lambda_{\gamma},
    %%\end{align*}
    %the following property is verified. Define $\Omega=B_b\setminus\bar{B}_a(0)\subset\R^2$. There exists a constant $C_{m,\beta,\gamma}<\infty $ such that for all $u\in W^{2,2}(\Omega)$, 
    %\begin{align}\label{gagliardo_nirenberg_frak_m_ineq2}
     %   &\int_{\Omega}\frac{1}{|x|^2}\left|\D u+(m-1)\frac{x}{|x|^2}u\right|^2\left(\left(\frac{|x|}{b}\right)^{2\gamma}+\left(\frac{a}{|x|}\right)^{2\gamma}\right)dx\leq C_{m,\beta,\gamma}\left(\int_{\Omega}\frac{u^2}{|x|^4}\left(\left(\frac{|x|}{b}\right)^{4\beta}+\left(\frac{a}{|x|}\right)^{4\beta}\right)dx\right.\nonumber\\
      %  &\left.+\int_{\Omega}\left(\leb_m u\right)^2\left(\left(\frac{|x|}{b}\right)^{2\gamma}+\left(\frac{a}{|x|}\right)^{2\gamma}\right)dx+\int_{\Omega}\left|\mathfrak{L}_mu\right|^2\left(\left(\frac{|x|}{b}\right)^{2\gamma}+\left(\frac{a}{|x|}\right)^{2\gamma}\right)dx\right).
   % \end{align}
%\end{cor}
The next result contains the most technical proof of the article. Indeed, this is where our weighted Gagliardo--Nirenberg inequality Corollary \ref{gagliardo_nirenberg_frak_m} is used. 
\begin{theorem}\label{theorem_eigenvalue_bounded}
    There exists $0<\alpha_0<\infty$, $0<\mu_0<\infty$, and a family of constants $\ens{\mu_{\alpha,k}}_{0<\alpha<\alpha_0,k\in \N}\subset (0,\mu_0)$ such that
    \begin{align*}
        \lim_{\alpha\rightarrow 0}\limsup_ {k\rightarrow \infty}\mu_{\alpha,k}=0,
    \end{align*}
    with the following property: for all $\lambda\in \R$, 
    \begin{align*}
        \mathrm{dim}\,\mathscr{E}_{\alpha,k}(\lambda)>0\implies \lambda\geq -\mu_{\alpha,k}\geq -\mu_0.
    \end{align*}
\end{theorem}
\begin{proof}
    \textbf{Part 1: Estimates on the Second Fundamental Form and the Metric.}

    Assume without loss of generality (\cite{quantamoduli}) that the multiplicity of the limiting branch point is $m\geq 3$.
    Thanks to the estimate of \cite[Theorem 1.3.30]{willmore_scs}, and the expansion \eqref{conformal_expansion}, we deduce that for all $x\in \Omega_k(\alpha)$, we have
    \begin{align*}
        |A_k(x)|\leq \frac{C(\beta_1)}{|x|^m}\left(\left(\frac{|x|}{\alpha}\right)^{\beta_1}+\left(\frac{\alpha^{-1}\rho_k}{|x|}\right)^{\beta_1}+\frac{1}{\log\left(\frac{\alpha^2}{\rho_k}\right)}\right)\np{\D\n_k}{\alpha}{\Omega_k(2\alpha)}.
    \end{align*}
    Therefore, we deduce that
    \begin{align*}
        \lim_{\alpha\rightarrow 0}\limsup_{k\rightarrow \infty}\np{\frac{A_k}{\omega_{1,\alpha,k}^{\frac{1}{4}}}}{\infty}{\Omega_k(\alpha)}\leq C'(\beta_1)\lim_{\alpha\rightarrow 0}\limsup_{k\rightarrow \infty}\np{\D\n_k}{2}{\Omega_k(2\alpha)}=0
    \end{align*}
    thanks to the no-neck energy property. Then, by the strong convergence on $\Sigma\setminus\bar{B}(0,\alpha)$, we deduce that 
    \begin{align*}
        \limsup_{k\rightarrow \infty}\np{\frac{A_k}{\omega_{1,\alpha,k}^{\frac{1}{4}}}}{\infty}{\Sigma\setminus\bar{B}(0,\alpha)}=\alpha^m\np{A_{\infty}}{\infty}{\Sigma}. 
    \end{align*}
    Now, since $\phi_{\infty}$ has a branch point of order $m$ at $0$, Bernard--Rivière's Taylor expansion (\cite{beriviere}) shows that there exists an integer $\theta_0\leq m-1$ such that
    \begin{align*}
        \H_{\infty}=\Re\left(\frac{\vec{C}_0}{z^{\theta_0}}\right)+O(|z|^{1-\theta_0-\epsilon})\qquad\text{for all}\;\, \epsilon>0.
    \end{align*}
    Likewise, we have $\h_{0,\infty}=O(|z|^{m-1})$, which implies that $A_{\infty}(x)=O(|x|^{1-m})$, so that (as $A_{\infty}$ is smooth in $\Sigma\setminus\ens{0}$)
    \begin{align*}
        \lim_{\alpha\rightarrow 0}\alpha^m\np{A_{\infty}}{\infty}{\Sigma\setminus\bar{B}(0,\alpha)}=0.
    \end{align*}
    Therefore, we need only prove the estimate on the bubble region. The bubble-neck decomposition implies that the sequence of immersions $\vec{\Psi}_k:B(0,\alpha^{-1})\rightarrow \R^3$ such that
    \begin{align*}
        \vec{\Psi}_k(x)=\rho_k^{-m}\left(\phi_k(\rho_kx)-\phi_k(0)\right)
    \end{align*}
    is such that 
    \begin{align*}
        \vec{\Psi}_k\conv{k\rightarrow \infty}\vec{\Psi}_{\infty}\qquad\text{in}\;\, C^l_{\mathrm{loc}}(\widehat{\C})\;\, \text{for all}\;\, l\in \N.
    \end{align*}
    Notice that the choice of scaling in the definition of $\vec{\Psi}_k$ is justified by the following inequalities (\cite{pointwise}; it will be justified at length in proof of the main Theorem \ref{main_theorem} that follows the proof of Theorem \ref{theorem_eigenvalue_bounded})
    \begin{align}\label{bound_lambda_bubble}
        e^{-A}\left(\alpha^{-1}\rho_k\right)^{m-1}\leq e^{\lambda_k(\alpha^{-1}\rho_k)}\leq e^A\left(\alpha^{-1}\rho_k\right)^{m-1}.
    \end{align}
    Now, let $\phi:B(0,1)\rightarrow \R^3$ be a smooth conformal immersion, $\rho,\mu>0$ and $\vec{\Psi}:B(0,\rho^{-1})\rightarrow \R^3$ be such that
    \begin{align*}
        \vec{\Psi}(x)=e^{-\mu}\phi(\rho\,x)\qquad\text{for all}\;\, x\in B(0,\rho^{-1}).
    \end{align*}
    The conformal factor of $\vec{\Psi}$ is given by 
    \begin{align*}
        e^{2\lambda_{\vec{\Psi}}(x)}=\frac{1}{2}|\D \vec{\Psi}(x)|^2=e^{-2\mu}\rho^2e^{2\lambda_{\phi}(\rho\, x)}.
    \end{align*}
    Then, we have
    \begin{align*}
        \Delta\vec{\Psi}(x)=e^{-\mu}\rho^2\Delta\phi(\rho\, x).
    \end{align*}
    Therefore, its mean curvature is given by
    \begin{align*}
        \H_{\vec{\Psi}}(x)=\frac{1}{2}e^{-2\lambda_{\vec{\Psi}}(x)}\Delta\vec{\Psi}(x)=\frac{1}{2}e^{2\mu}\rho^{-2}e^{-2\lambda_{\phi}(\rho\,x)}\,e^{-\mu}\rho^2\Delta\phi(\rho\,x)=e^{\mu}\,\H_{\phi}(\rho\,x),
    \end{align*}
    which shows in particular that
    \begin{align*}
        \int_{B(0,\rho^{-1})}|\H_{\vec{\Psi}}(x)|^2\,d\mathrm{vol}_{g_{\vec{\Psi}}}(x)=\int_{B(0,\rho^{-1})}e^{2\mu}|\vec{H}_{\phi}(\rho\,x)|^2e^{-2\mu}\rho^2\,d\mathrm{vol}_{g_{\phi}}(\rho\,x)=\int_{B(0,1)}|\H_{\phi}|^2\,d\mathrm{vol}_{g_{\phi}},
    \end{align*}
    that also comes from the conformal invariance of the Willmore energy. Therefore, we deduce that in the example above, we have 
    \begin{align*}
        \H_{\vec{\Psi}_k}(x)=\rho_k^m\H_{\phi_{k}}(\rho_kx).
    \end{align*}
    Therefore, we get by the same argument that $A_{\vec{\Psi}_k}(x)=\rho_k^mA_k(\rho_kx)$, which implies that
    \begin{align*}
        \frac{A_k(\rho_kx)}{\omega_{1,\alpha,k}^{\frac{1}{4}}}&=\rho_k^{-m}A_{\vec{\Psi}_k}(x)\left(\frac{1}{\alpha^4\rho_k^{4m}}\left(\frac{(1+\alpha^2)^{2m+2}}{(1+|x|^2)^{2m+2}}+\alpha^{-4m-4}(\alpha^{-2}\rho_k)^{4\beta_1}+\frac{\alpha^{-4m-4}}{\log^2\left(\frac{\alpha^2}{\rho_k}\right)}\right)\right)^{\frac{-1}{4}}\\
        &=\alpha\, A_{\vec{\Psi}_k}(x)\frac{(1+|x|^2)^{\frac{m+1}{2}}}{(1+\alpha^2)^{\frac{m+1}{2}}}\left(1+o_k(1)\right)\conv{k\rightarrow \infty}\alpha\, A_{\vec{\Psi}_{\infty}}(x)\frac{(1+|x|^2)^{\frac{m+1}{2}}}{(1+\alpha^2)^{\frac{m+1}{2}}}.
    \end{align*}
    Now, since $\vec{\Psi}_{\infty}:\C\rightarrow \R^3$ completes as a branched Willmore sphere $\vec{\Xi}_{\infty}=\dfrac{\vec{\Psi}_{\infty}}{|\vec{\Psi}_{\infty}|^2}:S^2\rightarrow \R^3$ (assuming without loss of generality---due to the conformal invariance of the Willmore energy---that $0\notin \vec{\Psi}_{\infty}(\C)$), using the Taylor expansion of Bernard--Rivière \cite{beriviere}, we deduce that as $|x|\rightarrow 0$, we have 
    \begin{align*}
        |A_{\vec{\Xi}_{\infty}}(x)|\leq \frac{C}{|x|^{m-1}}. 
    \end{align*}
    Here, we used a result of \cite{pointwise} that asserts that the end of $\vec{\Psi}_{\infty}$ has multiplicity $m$, which yields a branch point of multiplicity $m$ for $\vec{\Xi}_{\infty}$. Since (\cite[Lemma 10.8]{index4})
    \begin{align*}
    \left\{\begin{alignedat}{1}
        H_{\vec{\Xi}_{\infty}}&=-|\vec{\Psi}_{\infty}|^2H_{\vec{\Psi}_{\infty}}-2\s{\n_{\vec{\Psi}_{\infty}}}{\vec{\Psi}_{\infty}}\\
        h_{\vec{\Xi}_{\infty}}^0&=-\frac{1}{|\vec{\Psi}_{\infty}|}h^0_{\vec{\Psi}_{\infty}}
    \end{alignedat}\right.
    \end{align*}
    we deduce that there exists $A>0$ and $C_0<\infty$ such that for all we have for $|x|>A$ 
    \begin{align*}
        |A_{\vec{\Psi}_{\infty}}(x)|\leq \frac{C_0}{|x|^{m+1}}.
    \end{align*}
    Since $A_{\vec{\Psi}_{\infty}}$ is smooth in $\C$, this implies that there exists $C_1<\infty$ such that for all $x\in \C$, we have
    \begin{align}\label{bound_A_Psi}
        |A_{\vec{\Psi}_{\infty}}(x)|\leq \frac{C_1}{1+|x|^{m+1}}\leq \frac{2^{\frac{m}{2}}C_1}{(1+|x|^2)^{\frac{m+1}{2}}}
    \end{align}
    where we used that
    \begin{align*}
        (1+|x|^2)^{m+1}=1+\sum_{k=1}^mC_{m+1}^k|x|^{2k}+|x|^{2m+2}
        &\leq 1+\sum_{k=1}^mC_{m+1}^k\left(\frac{m-k+1}{m+1}+\frac{k}{m+1}|x|^{2m+2}\right)+|x|^{2m+2}\\
        &\leq (2^{m}-1)(1+|x|^{2m+2})\leq (2^{m}-1)(1+|x|^{m})^2.
    \end{align*}
    Therefore, we have
    \begin{align*}
        \alpha\, |A_{\vec{\Psi}_{\infty}}(x)|\frac{(1+|x|^2)^{\frac{m+1}{2}}}{(1+\alpha^2)^{\frac{m+1}{2}}}\leq 2^{\frac{m}{2}}C_1\frac{\alpha}{(1+\alpha^2)^{\frac{m+1}{2}}}\conv{\alpha\rightarrow 0}0. 
    \end{align*}
    Therefore, we deduce that
    \begin{align}\label{weight_bounded1}
        \lim_{\alpha\rightarrow 0}\limsup_{k\rightarrow \infty}\np{\frac{A_k}{\omega_{1,\alpha,k}^{\frac{1}{4}}}}{\infty}{\Sigma}=0.
    \end{align}
    Notice that this result can be slightly refined. Define  
    \begin{align*}
    \widetilde{\omega}_{1,\alpha,k}(x)=\left\{\begin{alignedat}{2}
        &\frac{1}{\alpha^{m}}\left(1+\left(\frac{\alpha^{-1}\delta_k}{\alpha}\right)^{\beta_1}+\frac{1}{\log\left(\frac{\alpha^2}{\rho_k}\right)}\right)\qquad&&\forall x\in \Sigma\setminus\bar{B}(0,\alpha)\\
        &\frac{1}{|x|^{m}}\left(\left(\frac{|x|}{\alpha}\right)^{\beta_1}+\left(\frac{\alpha^{-1}\rho_k}{\alpha}\right)^{\beta_1}+\frac{1}{\log\left(\frac{\alpha^2}{\rho_k}\right)}\right)\qquad&&\forall x\in \Omega_k(\alpha)\\
        &\frac{1}{(\alpha^{-1}\rho_k)^{m}}\left(\frac{1}{\alpha^{m+1}}\frac{(1+\alpha^2)^{\frac{m+1}{2}}}{\left(1+\left(\frac{|x|}{\rho_k}\right)^{2}\right)^{\frac{m+1}{2}}}+\left(\frac{\alpha^{-1}\rho_k}{\alpha}\right)^{\beta_1}+\frac{1}{\log\left(\frac{\alpha^2}{\rho_k}\right)}\right)\qquad&&\forall x\in B(0,\alpha^{-1}\rho_k).
    \end{alignedat}\right.
\end{align*}
    and 
    \begin{align}\label{muk}
        \mu_{\alpha,k}=\np{\frac{A_k}{\widetilde{\omega}_{1,\alpha,k}^{\frac{1}{4}}}}{\infty}{\Sigma}.
    \end{align}
    The same argument as above shows that there exists $\alpha_0>0$ and $\mu_0<\infty$ such that 
    \begin{align}\label{mu_bound}
        \left\{\begin{alignedat}{1}
            &\lim_{\alpha\rightarrow 0}\limsup_{k\rightarrow \infty}\mu_{\alpha,k}=0\\
            &\text{For all}\;\, k\in \N,\;\,\text{for all}\;\, 0<\alpha<\alpha_0,\;\,\text{we have}\;\, 0<\mu_{\alpha,k}<\mu_0. 
        \end{alignedat}\right.
    \end{align}
    Furthermore, the elementary inequality $(a+b+c)^2\leq 3(a^2+b^2+c^2)$ applied twice yields $\widetilde{\omega}_{1,\alpha,k}^4\leq 27\,\omega_{1,\alpha,k}$.
    Now, define a second weight $\omega_{2,\alpha,k}\in C^0(\Sigma)$ by
    \begin{align}\label{second_weight}
            \omega_{2,\alpha,k}(x)&=\left\{\begin{alignedat}{2}
                &\frac{1}{\alpha^{4m-2}}\left(1+\left(\frac{\alpha^{-1}\delta_k}{\alpha}\right)^{2\beta_2}+\frac{1}{\log^2\left(\frac{\alpha^2}{\rho_k}\right)}\right)\qquad&&\forall x\in \Sigma\setminus\bar{B}(0,\alpha)\\
                &\frac{1}{|x|^{4m-2}}\left(\left(\frac{|x|}{\alpha}\right)^{2\beta_2}+\left(\frac{\alpha^{-1}\rho_k}{\alpha}\right)^{2\beta_2}+\frac{1}{\log^2\left(\frac{\alpha^2}{\rho_k}\right)}\right)\qquad&&\forall x\in \Omega_k(\alpha)\\
                &\frac{1}{(\alpha^{-1}\rho_k)^{4m-2}}\left(1+\left(\frac{\alpha^{-1}\rho_k}{\alpha}\right)^{2\beta_2}+\frac{1}{\log^2\left(\frac{\alpha^2}{\rho_k}\right)}\right)\qquad&&\forall x\in B(0,\alpha^{-1}\rho_k).
            \end{alignedat}\right.
    \end{align}
    We will use this weight later to estimate the third trivial component of 
    \begin{align*}
        Q_{\phi_k}(u)&=\frac{1}{2}\int_{\Sigma}\left(\Delta_{g_k}u+|A_k|^2u\right)^2\,d\mathrm{vol}_{g_k}+\int_{\Sigma}\left(|du|_{g_k}^2+4|h_{0,k}|^2_{\mathrm{WP}}\right)H^2_k\,d\mathrm{vol}_{g_k}\\
        &-8\int_{\Sigma}\bs{\partial u\otimes \partial u}{h_{0,k}}_{\mathrm{WP}}H_k\,d\mathrm{vol}_{g_k}
        -16\int_{\Sigma}\s{\partial u\otimes \partial H_k}{h_{0,k}}_{\mathrm{WP}}\,u\,d\mathrm{vol}_{g_k}.
    \end{align*}
    We first estimate the last component. Fix $0<\beta_3<1$ and introduce the weight
    \begin{align}\label{third_weight}
        \omega_{3,\alpha,k}=\left\{\begin{alignedat}{2}
                &\frac{1}{\alpha^{2m+1}}\left(1+\left(\frac{\alpha^{-1}\delta_k}{\alpha}\right)^{2\beta_3}+\frac{1}{\log^2\left(\frac{\alpha^2}{\rho_k}\right)}\right)\quad&&\forall x\in \Sigma\setminus\bar{B}(0,\alpha)\\
                &\frac{1}{|x|^{2m+1}}\left(\left(\frac{|x|}{\alpha}\right)^{2\beta_3}+\left(\frac{\alpha^{-1}\rho_k}{\alpha}\right)^{2\beta_3}+\frac{1}{\log^2\left(\frac{\alpha^2}{\rho_k}\right)}\right)\quad&&\forall x\in \Omega_k(\alpha)\\
                &\frac{1}{(\alpha^{-1}\rho_k)^{2m+1}}\left(\frac{1}{\alpha^{2m+3}}\frac{(1+\alpha^2)^{\frac{2m+3}{2}}}{\left(1+\left(\frac{|x|}{\rho_k}\right)^{2}\right)^{\frac{2m+3}{2}}}+\left(\frac{\alpha^{-1}\rho_k}{\alpha}\right)^{2\beta_3}+\frac{1}{\log^2\left(\frac{\alpha^2}{\rho_k}\right)}\right)\quad&&\forall x\in B(0,\alpha^{-1}\rho_k).
            \end{alignedat}\right.
    \end{align}
    Now, let us prove that
    \begin{align}\label{weight_bounded2}
        \lim_{\alpha\rightarrow 0}\limsup_{k\rightarrow \infty}\np{\frac{|A_k||\D A_k|}{\omega_{3,\alpha,k}}}{\infty}{\Sigma}=0.
    \end{align}
    First, on the neck region, we have (by \cite{willmore_scs}) for all $0<\beta_3<1$ the estimates for all $x\in \Omega_k(\alpha)$
    \begin{align*}
        |A_k(x)|&\leq \frac{C_{\beta_3}}{|x|^m}\left(\left(\frac{|x|}{\alpha}\right)^{\beta_3}+\left(\frac{\alpha^{-1}\rho_k}{|x|}\right)^{\beta_3}+\frac{1}{\log\left(\frac{\alpha^2}{\rho_k}\right)}\right)\np{\D\n_k}{2}{\Omega_k(2\alpha)}\\
        |\D A_k(x)|&\leq \frac{C_{\beta_3}}{|x|^{m+1}}\left(\left(\frac{|x|}{\alpha}\right)^{\beta_3}+\left(\frac{\alpha^{-1}\rho_k}{|x|}\right)^{\beta_3}+\frac{1}{\log\left(\frac{\alpha^2}{\rho_k}\right)}\right)\np{\D\n_k}{2}{\Omega_k(2\alpha)}.
    \end{align*}
    Therefore, we have for all $x\in \Omega_k(\alpha)$
    \begin{align*}
        \frac{|A_k(x)||\D A_k(x)|}{\omega_{3,\alpha,k}}\leq C_{\beta}'\np{\D\n_k}{2}{\Omega_k(\alpha)}.
    \end{align*}
    Therefore, the no-neck energy property implies that
    \begin{align}\label{second_weight_step1}
        \lim_{\alpha\rightarrow 0}\limsup_{k\rightarrow \infty}\np{\frac{|A_k||\D A_k|}{\omega_{3,\alpha,k}}}{\infty}{\Omega_k(\alpha)}=0.
    \end{align}
    Then, we for all $x\in \Sigma\setminus\bar{B}(0,\alpha)$ a uniform convergence
    \begin{align*}
        \frac{|A_k(x)||\D A_k(x)|}{\omega_{3,\alpha,k}(x)}\conv{k\rightarrow \infty}\alpha^{2m+1}|A_{\infty}(x)||\D A_{\infty}(x)|.
    \end{align*}
    By the previously mentioned second residue analysis of Bernard--Rivière \cite{beriviere}, we have $A_{\infty}(x)=O(|x|^{1-m})$ and $\D A_{\infty}(x)=O(|x|^{-m})$, which shows that $\alpha^{2m+1}|A_{\infty}||\D A_{\infty}|=O(\alpha^2)$ and finally that 
    \begin{align}\label{second_weight_step2}
        \lim_{\alpha\rightarrow 0}\limsup_{k\rightarrow \infty}\np{\frac{|A_k||\D A_k|}{\omega_{3,\alpha,k}}}{\infty}{\Omega_k(\alpha)}=0.
    \end{align}
    Finally, we investigate the behaviour on the bubble domain. We have
    \begin{align*}
        \D\H_{\vec{\Psi}_k}(x)=\rho_k^{m+1}\D\H_{\phi_k}(\rho_kx),
    \end{align*}
    which shows that for all $x\in B(0,\alpha^{-1})$ 
    \begin{align*}
        \frac{|A_k(\rho_kx)||\D A_k(\rho_kx)|}{\omega_{3,\alpha,k}(\rho_k x)}&=\rho_k^{-(2m+1)}|A_{\vec{\Psi}_k}(x)||\D A_{\vec{\Psi}_k}(x)|\\
        &\times \left(\frac{1}{\alpha^2\rho_k^{2m+1}}\left(\left(\frac{(1+\alpha^2)^{\frac{2m+3}{2}}}{\left(1+|x|^2\right)^{\frac{2m+3}{2}}}\right)+\alpha^{-(2m+3)}\left(\frac{\alpha^{-1}\rho_k}{\alpha}\right)^{4\beta_2}+\frac{\alpha^{-(2m+3)}}{\log^2\left(\frac{\alpha^2}{\rho_k}\right)}\right)\right)^{-1}\\
        &=\alpha^2|A_{\vec{\Psi}_k}(x)||\D A_{\vec{\Psi}_k}(x)|\frac{(1+|x|^2)^{\frac{2m+3}{2}}}{\left(1+\alpha^2\right)^{\frac{2m+3}{2}}}\left(1+o_k(1)\right)\\
        &\conv{k\rightarrow \infty}\alpha^2|A_{\vec{\Psi}_{\infty}}(x)||\D A_{\vec{\Psi}_{\infty}}(x)|\frac{(1+|x|^2)^{\frac{2m+3}{2}}}{\left(1+\alpha^2\right)^{\frac{2m+3}{2}}}.
    \end{align*}
    Once more, Bernard--Rivière's estimates show that there exists a constant $C<\infty$ such that for all $|x|$ small enough, 
    \begin{align*}
        |\D A_{\vec{\Xi}_{\infty}}(x)|\leq \frac{C}{|x|^m},
    \end{align*}
    which implies by smoothness of $\vec{\Psi}_{\infty}:\C\rightarrow \R^3$ that there exists $C_2<\infty$ such that
    \begin{align*}
        |\D A_{\vec{\Psi}_{\infty}}(x)|\leq \frac{C_2}{1+|x|^{m+2}}\leq \frac{2^{\frac{m+1}{2}}}{\left(1+|x|^2\right)^{\frac{m+2}{2}}}.
    \end{align*}
    Therefore, we deduce that
    \begin{align*}
        \alpha^2|A_{\vec{\Psi}_{\infty}}(x)||\D A_{\vec{\Psi}_{\infty}}(x)|\frac{(1+|x|^2)^{\frac{2m+3}{2}}}{\left(1+\alpha^2\right)^{\frac{2m+3}{2}}}\leq \frac{2^{\frac{2m+3}{2}}C_1C_2\alpha^2}{(1+\alpha^2)^{\frac{2m+3}{2}}}\conv{\alpha\rightarrow 0}0.
    \end{align*}
    Therefore, we deduce that 
    \begin{align}\label{second_weight_step3}
        \lim_{\alpha\rightarrow 0}\limsup_{k\rightarrow \infty}\np{\frac{|A_k||\D A_k|}{\omega_{3,\alpha,k}}}{\infty}{B(0,\alpha^{-1}\rho_k)}=0
    \end{align}
    and the limit \eqref{weight_bounded2} is established. Therefore, if 
    \begin{align*}
        \nu_{\alpha,k}=\np{\frac{|A_k||\D A_k|}{\omega_{3,\alpha,k}}}{\infty}{B(0,\alpha^{-1}\rho_k)},
    \end{align*}
    the following properties hold true
    \begin{align}\label{kappa_bound}
        \left\{\begin{alignedat}{1}
            &\lim_{\alpha\rightarrow 0}\limsup_{k\rightarrow \infty}\nu_{\alpha,k}=0\\
            &\text{For all}\;\, k\in \N,\;\,\text{for all}\;\, 0<\nu<\nu_0,\;\,\text{we have}\;\, 0<\nu_{\alpha,k}<\nu_0. 
        \end{alignedat}\right.
    \end{align}
    \textbf{Part 2: Estimate of the First Two Components of the Second Derivative.}
    
    Now, recall  formula \eqref{formula_second_derivative}. By Cauchy's inequality, for all $u\in W^{2,2}(\Sigma)$ we have the elementary inequality
    \begin{align*}
        \frac{1}{2}\int_{\Sigma}\left(\Delta_{g_k}u+|A_k|^2\right)^2\,d\mathrm{vol}_{g_k}&=\frac{1}{2}\int_{\Sigma}(\Delta_{g_k}u)^2\,d\mathrm{vol}_{g_k}+\int_{\Sigma}(\Delta_{g_k}u)^2|A_k|^2u\,d\mathrm{vol}_{g_k}+\frac{1}{2}\int_{\Sigma}|A_k|^4u^2\,d\mathrm{vol}_{g_k}\\
        &\geq \frac{1}{4}\int_{\Sigma}(\Delta_{g_k}u)^2\,d\mathrm{vol}_{g_k}-\frac{1}{2}\int_{\Sigma}|A_k|^4u^2\,d\mathrm{vol}_{g_k}.
    \end{align*}
    We estimate directly 
    \begin{align*}
        \int_{\Sigma}|A_k|^4u^2\,d\mathrm{vol}_{g_k}\leq \mu_{\alpha,k}^4\int_{\Sigma}u^2\omega_{1,\alpha,k}\,d\mathrm{vol}_{g_k}.
    \end{align*}
    Therefore, we get
    \begin{align}\label{first_estimate}
        &Q_{\phi_k}(u)\geq \frac{1}{4}\int_{\Sigma}(\Delta_{g_k}u)^2\,d\mathrm{vol}_{g_k}-\frac{1}{2}\mu_{\alpha,k}^4\int_{\Sigma}u^2\,\omega_{1,\alpha,k}\,d\mathrm{vol}_{g_k}\nonumber\\
        &-8\int_{\Sigma}\s{\partial u\otimes \partial u}{h_{0,k}}_{\mathrm{WP}}\,H_k\,d\mathrm{vol}_{g_k}-16\int_{\Sigma}\s{\partial u\otimes \partial H_k}{h_{0,k}}_{\mathrm{WP}}\,u\,d\mathrm{vol}_{g_k}.
    \end{align}
    Likewise, if $0<\gamma<1$ and
    \begin{align*}
    \bar{\omega}_{1,\alpha,k}(x)=\left\{\begin{alignedat}{2}
        &\frac{1}{\alpha^{2m}}\left(1+\left(\frac{\alpha^{-1}\delta_k}{\alpha}\right)^{2\gamma}+\frac{1}{\log^2\left(\frac{\alpha^2}{\rho_k}\right)}\right)\quad&&\forall x\in \Sigma\setminus\bar{B}(0,\alpha)\\
        &\frac{1}{|x|^{2m}}\left(\left(\frac{|x|}{\alpha}\right)^{2\gamma}+\left(\frac{\alpha^{-1}\rho_k}{\alpha}\right)^{2\gamma}+\frac{1}{\log^2\left(\frac{\alpha^2}{\rho_k}\right)}\right)\quad&&\forall x\in \Omega_k(\alpha)\\
        &\frac{1}{(\alpha^{-1}\rho_k)^{2m}}\left(\frac{1}{\alpha^{2m+2}}\frac{(1+\alpha^2)^{m+1}}{\left(1+\left(\frac{|x|}{\rho_k}\right)^{2}\right)^{m+1}}+\left(\frac{\alpha^{-1}\rho_k}{\alpha}\right)^{2\gamma}+\frac{1}{\log^2\left(\frac{\alpha^2}{\rho_k}\right)}\right)\quad&&\forall x\in B(0,\alpha^{-1}\rho_k).
    \end{alignedat}\right.
\end{align*}
    and 
    \begin{align}\label{muk_bis}
        \kappa_{\alpha,k}=\np{\frac{A_k}{\bar{\omega}_{1,\alpha,k}^{\frac{1}{2}}}}{\infty}{\Sigma},
    \end{align}
    The same argument as above shows that there exists $0<\alpha_2<\alpha_1$ and $\kappa_0<\infty$ such that 
    \begin{align}\label{kappa_bound}
        \left\{\begin{alignedat}{1}
            &\lim_{\alpha\rightarrow 0}\limsup_{k\rightarrow \infty}\kappa_{\alpha,k}=0\\
            &\text{For all}\;\, k\in \N,\;\,\text{for all}\;\, 0<\kappa<\kappa_0,\;\,\text{we have}\;\, 0<\kappa_{\alpha,k}<\kappa_0. 
        \end{alignedat}\right.
    \end{align}
    %Likewise, we have
    %\begin{align*}
     %   \left|\int_{\Sigma}\bs{\partial u\otimes \partial H_k}{h_{0,k}}_{\mathrm{WP}}\,u\,d\mathrm{vol}_{g_k}\right|\leq \int_{\Sigma} 
    %\end{align*}
    \textbf{Part 3: The Estimate on the Third Component of the Second Derivative.}
    
    %The Weighted Gagliardo-Nirenberg Inequality.}
    We have
    \begin{align*}
        \left|\int_{\Sigma}\s{\partial u\otimes \partial u}{h_{0,k}}_{\mathrm{WP}}H_k\,d\mathrm{vol}_{g_k}\right|\leq \kappa_{\alpha,k}^2\int_{\Sigma}|du|_{g_k}^2\bar{\omega}_{1,\alpha,k}\,d\mathrm{vol}_{g_k}.
    \end{align*}
  %  Now comes the most technical part of the proof. Let us summarise the main ideas. The weighted $L^2$ norm of $u$ is bounded thanks to our renormalisation while the $L^2$ norm of the metric Laplacian of $u$ is also bounded, so we need only show that the weighted $L^2$ norm of the gradient of $u$ is bounded by the previous two quantities. In other words, we have to establish a weighted Gagliardo--Nirenberg inequality on $\Sigma$. Using that $\Sigma$ is closed, once we integrate by parts, a single additional term involving the gradient of the weight in the neck region (notice that the weight is constant on $\Sigma\setminus\bar{\Omega_k(\alpha)}$, so that only the term in the neck region remains) will appear, and this is this quantity that we will have to bound. The first major ingredient is the weighted Gagliardo--Nirenberg in the neck region $\Omega_k(2\alpha)$ already established in Corollary \eqref{gagliardo_nirenberg_frak_m2}. Using the Bochner identity on the closed surface $\Sigma$, we can replace the estimate involving the Hessian by the Laplacian and additional terms that can be estimated. 
  %   However, this is not exactly straightforward since we have to relate the flat Hessian to the metric Hessian and the additional terms are quite delicate to estimate. 
    Integrating by parts, we get
    \begin{align}\label{ipp_poids_1}
        \int_{\Sigma}|du|_{g_k}^2\bar{\omega}_{1,\alpha,k}\,d\mathrm{vol}_{g_k}&=-\int_{\Sigma}u\,(\Delta_{g_k}u)\,\bar{\omega}_{1,\alpha,k}\,d\mathrm{vol}_{g_k}-\int_{\Omega_k(\alpha)}u\,\D u\cdot \D\bar{\omega}_{1,\alpha,k}\,dx\nonumber\\
        &-\int_{B(0,\alpha^{-1}\rho_k)}u\,\D u\cdot \D\bar{\omega}_{1,\alpha,k}\,dx.
    \end{align}
    We directly estimate by Cauchy--Schwarz inequality (assuming that $\gamma\geq \beta_1$)
    \begin{align}\label{ipp_poids_2}
        \left|\int_{\Sigma}u\,\Delta_{g_k}u\,\bar{\omega}_{1,\alpha,k}\,d\mathrm{vol}_{g_k}\right|&\leq \left(\int_{\Sigma}u^2\,\bar{\omega}_{1,\alpha,k}^2\,d\mathrm{vol}_{g_k}\right)^{\frac{1}{2}}\left(\int_{\Sigma}(\Delta_{g_k}u)^2\,d\mathrm{vol}_{g_k}\right)^{\frac{1}{2}}\nonumber\\
        &\leq \sqrt{3}\left(\int_{\Sigma}u^2\,{\omega}_{1,\alpha,k}\,d\mathrm{vol}_{g_k}\right)^{\frac{1}{2}}\left(\int_{\Sigma}(\Delta_{g_k}u)^2\,d\mathrm{vol}_{g_k}\right)^{\frac{1}{2}}
    \end{align}
    The second term is much more delicate to estimate, so we keep it for later and estimate the component in the bubble region $B(0,\alpha^{-1}\rho_k)$. 

    \textbf{Step 1: Estimate on the Bubble Domain.} First, recall that
    \begin{align*}
        \vec{\Psi}_k(z)=\rho_k^{-m}\left(\phi_k(\rho_kz)-\phi_k(0)\right)\qquad \text{for all}\;\, z\in \C.
    \end{align*}
    Here, we mean that the value of $\vec{\Psi}_k$ at $z\in \C$ is well defined for $k$ large enough. This identity implies in particular that
    \begin{align}\label{conformal_bubbling}
        \lambda_{\vec{\Psi}_k}(z)=\lambda_k(\rho_kz)-(m-1)\log(\rho_k).
    \end{align}
    As the convergence of $\vec{\Psi}_k$ to $\vec{\Psi}_{\infty}$ is smooth in $\C$, we deduce that for all fixed $0<A<\infty$, there exists $N(A)\in \N$ such that for all $k\in \N$ and for all $|z|\leq A$, we have
    \begin{align*}
        \frac{1}{2}e^{2\lambda_{\vec{\Psi}_{\infty}}(z)}\leq e^{2\lambda_{\vec{\Psi}_k}(z)}\leq 2\,e^{2\lambda_{\vec{\Psi}_{\infty}}(z)}.
    \end{align*}
    Here, we used that the conformal parameter of $\vec{\Psi}_{\infty}:\C\rightarrow \R^3$ never vanishes in $\C$ due to the absence of branch points. Since the end of $\vec{\Psi}_{\infty}$ is of multiplicity $m$ by the analysis of \cite{pointwise} (equivalently, $\vec{\Xi}_{\infty}:S^2\rightarrow \R^3$ has a unique branch point of order $m$), there exists $\zeta\in C^0\cap L^{\infty}(\C)$ such that
    \begin{align*}
        e^{2\vec{\Psi}_{\infty}(z)}=e^{\zeta(z)}\left(1+|z|^2\right)^{m-1}\qquad \text{for all}\;\,\C.
    \end{align*}
    Therefore, there exists a universal constant $0<\Gamma<\infty$ independent of $A$ such that for all $|z|\leq A$ and $k\geq N(A)$, we have
    \begin{align*}
        e^{-2\,\Gamma}\left(1+|z|^2\right)^{m-1}\leq e^{2\lambda_{\vec{\Psi}_k}(z)}\leq e^{2\,\Gamma}\left(1+|z|^2\right)^{m-1}.
    \end{align*}
    This implies by the above formula \eqref{conformal_bubbling} that for all $|z|\leq A$ and $k\geq N(A)$, we have
    \begin{align}\label{bubbling_conformal0}
        e^{-2\,\Gamma}\rho_k^{2(m-1)}\left(1+|z|^2\right)^{m-1}\leq e^{2\lambda_k(\rho_kz)}\leq e^{2\,\Gamma}\rho_k^{2(m-1)}\left(1+|z|^2\right)^{m-1},
    \end{align}
    that we finally rewrite as
    \begin{align}\label{bubbling_conformal}
        e^{-2\,\Gamma}\rho_k^{2(m-1)}\left(1+\left(\frac{|x|}{\rho_k}\right)^2\right)^{m-1}\leq e^{2\lambda_k(x)}\leq e^{2\,\Gamma}\rho_k^{2(m-1)}\left(1+\left(\frac{|x|}{\rho_k}\right)^2\right)^{m-1}\qquad\forall x\in B(0,\alpha^{-1}\rho_k).
    \end{align}
    For all $x\in B(0,\alpha^{-1}\rho_k)$, we have 
    \begin{align*}
        \D \bar{\omega}_{1,\alpha,k}=-2(m+1)x\,\frac{1}{\alpha^2\rho_k^{2m+2}}\frac{(1+\alpha^2)^{m+1}}{\left(1+\left(\frac{|x|}{\rho_k}\right)^2\right)^{m+2}}.
    \end{align*}
    Therefore, we get
    \begin{align}\label{bubbling_conformal2}
        &-\int_{B(0,\alpha^{-1}\rho_k)}u\D u\cdot \D\bar{\omega}_{1,\alpha,k}dx=2(m+1)\int_{B(0,\alpha^{-1}\rho_k)}u\,(\D u\cdot x)\frac{1}{\alpha^2\rho_k^{2m+2}}\frac{(1+\alpha^2)^{m+1}}{\left(1+\left(\frac{|x|}{\rho_k}\right)^2\right)^{m+2}}dx\nonumber\\
        &\leq 2(m+1)\left(\int_{B(0,\alpha^{-1}\rho_k)}u^2\frac{1}{\alpha^4\rho_k^{2m+2}}\frac{\left(1+\alpha^2\right)^{2m+2}}{\left(1+\left(\frac{|x|}{\rho_k}\right)^2\right)^{m+3}}dx\right)^{\frac{1}{2}}\nonumber\\
        &\times \left(\int_{B(0,\alpha^{-1}\rho_k)}|\D u|^2\frac{|x|^2}{\rho_k^{2m+2}}\frac{dx}{\left(1+\left(\frac{|x|}{\rho_k}\right)^{2}\right)^{m+1}}\right)^{\frac{1}{2}}.
    \end{align}
    Thanks to \eqref{bubbling_conformal}, we get
    \begin{align}\label{bubbling_conformal3}
        &\int_{B(0,\alpha^{-1}\rho_k)}u^2\frac{1}{\alpha^4\rho_k^{2m+2}}\frac{\left(1+\alpha^2\right)^{2m+2}}{\left(1+\left(\frac{|x|}{\rho_k}\right)^2\right)^{m+3}}\nonumber\\
        &=\int_{B(0,\alpha^{-1}\rho_k)}u^2\frac{1}{\alpha^4\rho_k^{4m}}\frac{\left(1+\alpha^2\right)^{2m+2}}{\left(1+\left(\frac{|x|}{\rho_k}\right)^2\right)^{2m+2}}\times \rho_k^{2m-2}\left(1+\left(\frac{|x|}{\rho_k}\right)^2\right)^{m-1}dx\nonumber\\
        &\leq e^{2\,\Gamma}\int_{B(0,\alpha^{-1}\rho_k)}u^2\,\omega_{1,\alpha,k}\,d\mathrm{vol}_{g_k}.
    \end{align}
    On the other hand, we trivially have 
    \begin{align}\label{bubbling_conformal4}
        &\int_{B(0,\alpha^{-1}\rho_k)}|\D u|^2\frac{|x|^2}{\rho_k^{2m+2}}\frac{dx}{\left(1+\left(\frac{|x|}{\rho_k}\right)^{2}\right)^{m+1}}\leq \int_{B(0,\alpha^{-1}\rho_k)}|\D u|^2\frac{1}{\alpha^2\rho_k^{2m}}\frac{dx}{\left(1+\left(\frac{|x|}{\rho_k}\right)^{2}\right)^{m+1}}\nonumber\\
        &\leq \frac{1}{(1+\alpha^2)^{m+1}}\int_{B(0,\alpha^{-1}\rho_k)}|\D u|^2\bar{\omega}_{1,\alpha,k}\,dx=\frac{1}{(1+\alpha^2)^{m+1}}\int_{B(0,\alpha^{-1}\rho_k)}|du|_{g_k}^2\,\bar{\omega}_{1,\alpha,k}\,d\mathrm{vol}_{g_k}.
    \end{align}
    Thanks to \eqref{bubbling_conformal2}, \eqref{bubbling_conformal3}, and \eqref{bubbling_conformal4}, we have
    \begin{align}\label{bubbling_conformal5}
        &-\int_{B(0,\alpha^{-1}\rho_k)}u\D u\cdot \D\bar{\omega}_{1,\alpha,k}\,dx\nonumber\\
        &\leq \frac{2(m+1)}{\left(1+\alpha^2\right)^{\frac{m+1}{2}}}\left(\int_{B(0,\alpha^{-1}\rho_k)}u^2\,\omega_{1,\alpha,k}\,d\mathrm{vol}_{g_k}\right)^{\frac{1}{2}}\left(\int_{B(0,\alpha^{-1}\rho_k)}|du|_{g_k}^2\,\bar{\omega}_{1,\alpha,k}\,d\mathrm{vol}_{g_k}\right)^{\frac{1}{2}}.
    \end{align}
    
    \textbf{Step 2: Estimate on the Neck Region.} We have for all $x\in \Omega_k(\alpha)$
    \begin{align*}
        \D \bar{\omega}_{1,\alpha,k}=-2m\frac{x}{|x|^2}\bar{\omega_{2,\alpha,k}}+2\gamma\frac{x}{|x|^{2m+2}}\left(\left(\frac{|x|}{\alpha}\right)^{2\gamma}-\left(\frac{\alpha^{-1}\rho_k}{|x|}\right)^{2\gamma}\right)
    \end{align*}
    which implies the elementary inequality
    \begin{align*}
        |\D\bar{\omega}_{1,\alpha,k}|\leq \frac{2(m+2\gamma)}{|x|}\bar{\omega}_{1,\alpha,k}.
    \end{align*}
    We first estimate by the Cauchy--Schwarz inequality
    \begin{align}\label{ipp_poids_3}
        \left|\int_{\Omega_k(\alpha)}u\D u\cdot \frac{x}{|x|^{2m+2}}\frac{1}{\log^2\left(\frac{\alpha^2}{\rho_k}\right)}d\mathrm{vol}_{g_k}\right|&\leq \frac{1}{\log^2\left(\frac{\alpha^2}{\rho_k}\right)}\left(\int_{\Omega_k(\alpha)}\frac{u^2}{|x|^{2m+2}}dx\right)^{\frac{1}{2}}\left(\int_{\Omega_k(\alpha)}\frac{|\D u|^2}{|x|^{2m}}dx\right)^{\frac{1}{2}}\nonumber\\
        &\leq \left(\int_{\Omega_k(\alpha)}u^2\omega_{1,\alpha,k}\,d\mathrm{vol}_{g_k}\right)^{\frac{1}{2}}\left(\int_{\Omega_k(\alpha)}|du|_{g_k}^2\bar{\omega}_{1,\alpha,k}\,d\mathrm{vol}_{g_k}\right)^{\frac{1}{2}}.
    \end{align}
    Using Cauchy--Schwarz inequality once more, we deduce that for all $0<\epsilon<1$
    \begin{align}\label{ipp_poids_4}
        &\int_{\Omega_k(\alpha)}\frac{|\D u|}{|x|^m}\frac{|u|}{|x|^{m+1}}\left(\left(\frac{|x|}{\alpha}\right)^{2\gamma}+\left(\frac{\alpha^{-1}\rho_k}{|x|}\right)^{2\gamma}\right)dx\leq 2\left(\int_{\Omega_k(\alpha)}\frac{|\D u|^2}{|x|^{2m}}\left(\left(\frac{|x|}{\alpha}\right)^{4\epsilon\gamma}+\left(\frac{\alpha^{-1}\rho_k}{|x|}\right)^{4\epsilon\gamma}\right)dx\right)^{\frac{1}{2}}\nonumber\\
        &\times \left(\int_{\Omega_k(\alpha)}\frac{u^2}{|x|^{2m+2}}\left(\left(\frac{|x|}{\alpha}\right)^{4(1-\epsilon)\gamma}+\left(\frac{\alpha^{-1}\rho_k}{|x|}\right)^{4(1-\epsilon)\gamma}\right)dx\right)^{\frac{1}{2}}\nonumber\\
        &\leq 2\left(\int_{\Omega_k(\alpha)}\frac{|\D u|^2}{|x|^{2m}}\left(\left(\frac{|x|}{\alpha}\right)^{4\epsilon\gamma}+\left(\frac{\alpha^{-1}\rho_k}{|x|}\right)^{4\epsilon\gamma}\right)dx\right)^{\frac{1}{2}}\left(\int_{\Omega_k(\alpha)}u^2\omega_{1,\alpha,k}|x|^{2m-2}dx\right)^{\frac{1}{2}}
    \end{align}
    provided that $(1-\epsilon)\gamma\geq \beta_1$.
    Now, using Corollary \ref{gagliardo_nirenberg_frak_m}, we deduce that for all $2\epsilon\gamma\leq \beta<\min\ens{\dfrac{m-1}{4},1}$
\begin{align}\label{ipp_poids_5}
    \int_{\Omega_k(\alpha)}\frac{|\D u|^2}{|x|^{2m}}\left(\left(\frac{|x|}{b}\right)^{4\epsilon\gamma}+\left(\frac{a}{|x|}\right)^{4\epsilon\gamma}\right)dx&\leq C_{m,\beta,2\epsilon\gamma}\left(\int_{\Omega_k(\alpha)}\frac{u^2}{|x|^{2m+2}}\left(\left(\frac{|x|}{\alpha}\right)^{4\beta}+\left(\frac{\alpha^{-1}\rho_k}{|x|}\right)^{4\beta}\right)dx\right.\nonumber\\
    &\left.+\int_{\Omega_k(\alpha)}\frac{|\D^2u|^2}{|x|^{2m-2}}\left(\left(\frac{|x|}{\alpha}\right)^{4\epsilon\gamma}+\left(\frac{\alpha^{-1}\rho_k}{|x|}\right)^{4\epsilon\gamma}\right)dx\right)
\end{align}
At this point, the constraints on $\beta$ and $\gamma$ are given as follows:
\begin{align*}
    \left\{\begin{alignedat}{1}
        &\beta\geq \beta_1\\
        &(1-\epsilon)\gamma\geq \beta_1\\
        &\beta\geq 2\epsilon\gamma.
    \end{alignedat}\right.
\end{align*}
Therefore, we can take $\beta=\beta_1$, $0<\beta_1<\gamma<1$ and
\begin{align*}
    0<\epsilon<\min\ens{1-\frac{\beta_1}{\gamma},\frac{\beta_1}{2\gamma}}.
\end{align*}
Now, recall Bochner's identity
\begin{align*}
    \frac{1}{2}\Delta_{g_k}|du|_{g_k}^2&=|\mathrm{Hess}(u)|_{g_k}^2+\s{du}{d\left(\Delta_{g_k}u\right)}_{g_k}+\mathrm{Ric}_{g_k}(du,du)\\
    &=|\mathrm{Hess}(u)|_{g_k}^2+\s{du}{d\left(\Delta_{g_k}u\right)}_{g_k}+2K_{g_k}|du|_{g_k}^2.
\end{align*}
Integrating by parts, we deduce as $\Sigma$ is closed that
\begin{align*}
    0&=\frac{1}{2}\int_{\Sigma}\Delta_{g_k}|du|_{g_k}^2\,d\mathrm{vol}_{g_k}=\int_{\Sigma}\left(|\mathrm{Hess}(u)|^2_{g_k}+\s{du}{d\left(\Delta_{g_k}u\right)}_{g_k}+2K_{g_k}|du|_{g_k}^2\right)d\mathrm{vol}_{g_k}\\
    &=\int_{\Sigma}\left(|\mathrm{Hess}(u)|^2_{g_k}-(\Delta_{g_k}u)^2+2K_{g_k}|du|_{g_k}^2\right)d\mathrm{vol}_{g_k},
\end{align*}
which implies that
\begin{align}\label{bochner_gk}
    \int_{\Sigma}\left|\mathrm{Hess}(u)\right|^2_{g_k}\,d\mathrm{vol}_{g_k}=\int_{\Sigma}\left(\left(\Delta_{g_k}u\right)^2-2K_{g_k}|du|_{g_k}^2\right)d\mathrm{vol}_{g_k}.
\end{align}
On the other hand, by \eqref{bochner_metric}, we have
\begin{align*}
    e^{-2\lambda_k}|\D^2u|^2&=\left(|\mathrm{Hess}(u)|_{g_k}^2-2\s{du}{d\lambda_k}_{g_k}\Delta_{g_k}u+2\,e^{2\lambda_k}\s{d\lambda_k}{d|\D u|^2}_{g_k}-2|d\lambda_k|_{g_k}^2|du|_{g_k}^2\right)e^{2\lambda_k}\\
    &=\left(|\mathrm{Hess}(u)|_{g_k}^2-2\s{du}{d\lambda_k}_{g_k}\Delta_{g_k}u+2\,e^{-2\lambda_k}\s{d\lambda_k}{d(e^{2\lambda_k}|du|_{g_k}^2)}_{g_k}-2|d\lambda_k|_{g_k}^2|du|_{g_k}^2\right)e^{2\lambda_k}\\
    &=\left(|\mathrm{Hess}(u)|_{g_k}^2-2\s{du}{d\lambda_k}_{g_k}\Delta_{g_k}u+2\,\s{d\lambda_k}{d|du|_{g_k}^2}_{g_k}+2|d\lambda_k|_{g_k}^2|du|_{g_k}^2\right)e^{2\lambda_k}.
\end{align*}
Using the charts of Laurain--Rivière (\cite{lauriv1}) where the conformal parameter is controlled in $L^{2,\infty}$, we can assume without loss of generality that 
\begin{align}\label{laurain_rivière_estimate}
    \np{\D\lambda_k}{2,\infty}{B(0,1)}\leq C_{\mathrm{LR}}\,W(\phi_k),
\end{align}
where $C_{\mathrm{LR}}<\infty$ is a universal constant. Thanks to the main result of \cite{pointwise}, for all $0<\alpha<\alpha_0$ there exists a decomposition $\lambda_k=\mu_k+\nu_k$ in $\Omega_k(\alpha)$, where  $\mu_k\in W^{2,1}(B(0,\alpha)$ is such that
\begin{align*}
    \np{\D^2\mu_k}{1}{B(0,\alpha)}+\np{\D\mu_k}{2,1}{B(0,\alpha)}+\np{\mu_k}{\infty}{B(0,\alpha)}\leq C_{\mathrm{MR}}\,\int_{\Omega_k(\alpha)}|\D\n_k|^2dx,
\end{align*}
and $\nu_k$ is a harmonic function such that
\begin{align*}
    \np{\D(\nu_k-(m-1)\log|x|)}{2,1}{\Omega_k(\alpha)}\leq C_{\mathrm{MR}}\left(\sqrt{\alpha}\np{\D\lambda_k}{2,\infty}{B(0,1)}+\int_{\Omega_k(\alpha)}|\D\n_k|^2dx\right).
\end{align*}
Furthermore, refinements of those results show that there exists a sequence of bounded holomorphic functions $\psi_k:B(0,\alpha_0)\rightarrow \C$ such that
\begin{align*}
    \psi_k(z)=e^{c_k}z^{m-1}(1+O(|z|)),
\end{align*}
where $c_k\conv{k\rightarrow \infty}c$ and 
\begin{align}\label{neck_domain_expansion}
    e^{\lambda_k}=e^{\mu_k}|\psi_k(z)|=e^{\mu_k}e^{c_k}|z|^{m-1}(1+O(|z|))
\end{align}
In particular, there exists a universal constant $0<\Gamma<\infty$ such that
\begin{align}\label{harnack}
    \Gamma^{-1}|z|^{2m-2}\leq e^{2\lambda_k(z)}\leq \Gamma\,|z|^{2m-2}\quad \text{for all}\;\, z\in \Omega_k(\alpha_0).
\end{align}
Therefore, we have the estimate 
\begin{align*}
    &\int_{\Omega_k(\alpha)}\frac{|\D^2u|^2}{|x|^{2m-2}}\left(\left(\frac{|x|}{\alpha}\right)^{4\epsilon\gamma}+\left(\frac{\alpha^{-1}\rho_k}{|x|}\right)^{4\epsilon\gamma}\right)dx\leq \Gamma\int_{\Omega_k(\alpha)}|\D^2u|^2\left(\left(\frac{|x|}{\alpha}\right)^{4\epsilon\gamma}+\left(\frac{\alpha^{-1}\rho_k}{|x|}\right)^{4\epsilon\gamma}\right)d\mathrm{vol}_{g_k}\\
    &=\Gamma\left(\int_{\Omega_k(\alpha)}|\mathrm{Hess}(u)|_{g_k}^2\left(\left(\frac{|x|}{\alpha}\right)^{4\epsilon\gamma}+\left(\frac{\alpha^{-1}\rho_k}{|x|}\right)^{4\epsilon\gamma}\right)d\mathrm{vol}_{g_k}\right.\\
    &-2\int_{\Omega_k(\alpha)}\s{du}{d\lambda_k}_{g_k}\Delta_{g_k}u\,\left(\left(\frac{|x|}{\alpha}\right)^{4\epsilon\gamma}+\left(\frac{\alpha^{-1}\rho_k}{|x|}\right)^{4\epsilon\gamma}\right)d\mathrm{vol}_{g_k}\\
    &+2\int_{\Omega_k(\alpha)}\s{d\lambda_k}{d|du|_{g_k}^2}_{g_k}\left(\left(\frac{|x|}{\alpha}\right)^{4\epsilon\gamma}+\left(\frac{\alpha^{-1}\rho_k}{|x|}\right)^{4\epsilon\gamma}\right)d\mathrm{vol}_{g_k}\\
    &\left.+2\int_{\Omega_k(\alpha)}|d\lambda_k|_{g_k}^2|du|_{g_k}^2\left(\left(\frac{|x|}{\alpha}\right)^{4\epsilon\gamma}+\left(\frac{\alpha^{-1}\rho_k}{|x|}\right)^{4\epsilon\gamma}\right)d\mathrm{vol}_{g_k}\right).
\end{align*}
%Now, define 
%\begin{align*}
 %   \eta_{\alpha,k}=\np{\frac{A_k^2}{\omega_{2,\alpha,k}}}{\infty}{\Sigma},
%\end{align*}
%we deduce as above that there exists $\eta_0<\infty$
       % \begin{align}
      %  \left\{\begin{alignedat}{1}
      %      &\lim_{\alpha\rightarrow 0}\limsup_{k\rightarrow %\infty}\eta_{\alpha,k}=0,\\
       %     &\text{For all}\;\, k\in\N,\;\, \text{for all}\;\,  %0<\alpha<\alpha_0, \;\, \text{we have}\;\,  0<\eta_{\alpha,k}<\eta_0.
      %      \end{alignedat}\right.
      %  \end{align}
Now, using the pointwise bound $2|K_{g_k}|\leq |A_k|^2\leq \kappa_{\alpha,k}^2\,\bar{\omega}_{1,\alpha,k}$, we deduce that
\begin{align}\label{bochner_gk_2}
    \int_{\Sigma}-2K_{g_k}|du|_{g_k}^2\,d\mathrm{vol}_{g_k}\leq \int_{\Sigma}|A_k|^2|du|_{g_k}^2\,d\mathrm{vol}_{g_k}\leq \kappa_{\alpha,k}^2\int_{\Sigma}|du|_{g_k}^2\bar{\omega}_{1,\alpha,k}\,d\mathrm{vol}_{g_k}.
\end{align}
Using the Bochner identity \eqref{bochner_gk} and \eqref{bochner_gk_2}, we deduce that
\begin{align}\label{bochner_gk_3}
    \int_{\Omega_k(\alpha)}|\mathrm{Hess}(u)|_{g_k}^2\left(\left(\frac{|x|}{\alpha}\right)^{4\epsilon\gamma}+\left(\frac{\alpha^{-1}\rho_k}{|x|}\right)^{4\epsilon\gamma}\right)d\mathrm{vol}_{g_k}&\leq 2\int_{\Sigma}|\mathrm{Hess}(u)|_{g_k}^2\,d\mathrm{vol}_{g_k}\nonumber\\
    &=2\int_{\Sigma}\left((\Delta_{g_k}u)^2-2K_{g_k}|du|_{g_k}^2\right)d\mathrm{vol}_{g_k}\nonumber\\
    &\leq 2\int_{\Sigma}(\Delta_{g_k}u)^2\,d\mathrm{vol}_{g_k}+\kappa_{\alpha,k}^2\int_{\Sigma}|du|_{g_k}^2\bar{\omega}_{1,\alpha,k}dx.
\end{align}
The next term is easy to treat using Cauchy--Schwarz inequality
\begin{align}\label{bochner_gk_4}
    &-2\int_{\Omega_k(\alpha)}\s{du}{d\lambda_k}_{g_k}\Delta_{g_k}u\,\left(\left(\frac{|x|}{\alpha}\right)^{4\epsilon\gamma}+\left(\frac{\alpha^{-1}\rho_k}{|x|}\right)^{4\epsilon\gamma}\right)d\mathrm{vol}_{g_k}\nonumber\\
    &\leq 2\sqrt{2}\left(\int_{\Omega_k(\alpha)}|d\lambda_k|_{g_k}^2|du|_{g_k}^2\left(\left(\frac{|x|}{\alpha}\right)^{8\epsilon\gamma}+\left(\frac{\alpha^{-1}\rho_k}{|x|}\right)^{8\epsilon\gamma}\right)d\mathrm{vol}_{g_k} \right)^{\frac{1}{2}}\left(\int_{\Omega_k(\alpha)}(\Delta_{g_k}u)^2\,d\mathrm{vol}_{g_k}\right)^{\frac{1}{2}}.
\end{align}
For the next term, we could integrate by parts (using the Liouville equation) once we replace the integral by the one on the whole domain, but this quantity has no clear sign, so we instead reason as in \cite{willmore_scs}. Using that 
\begin{align}
    |\D \lambda_k\cdot \D |\D u|^2|\leq 2|\D \lambda_k||\D u||\D^2u|\leq \frac{1}{4}|\D^2u|^2+4|\D\lambda_k|^2|\D u|^2,
\end{align}
we obtain that
\begin{align}
    e^{-4\lambda_k}|\D^2u|^2\leq 2|\mathrm{Hess}(u)|_{g_k}^2-4\s{d\lambda_k}{du}_{g_k}\Delta_{g_k}u+16|d\lambda_k|_{g_k}^2|du|_{g_k}^2, 
\end{align}
so we need only estimate the last term. Using the independence on the domain of the Wente constant in the $L^{\infty}$ and $W^{1,2}$ inequality, we could perform another expansion in $\Omega_k(\alpha)$ (with a term that solves the Liouville equation with Dirichlet boundary data), but in this case, one would not be able to derive an estimate that holds up to the boundary for the harmonic component, due to the following lemma (see \cite{pointwise}). 
\begin{lemme}\label{pointwise_l2infty_harmonic}
    Let $0<a<b<\infty$ and $\Omega=B_b\setminus\bar{B}_a(0)$. For all harmonic function $u:\Omega\rightarrow \R$, 
    we have for all $2a<|x|<b/2$
    \begin{align*}
        |\D u(x)|\leq \frac{8}{\log(2)}\sqrt{\frac{3}{\pi}}\frac{1}{|x|}\np{\D u}{2,\infty}{\Omega}.
    \end{align*}
\end{lemme}
Indeed, if one restricts to $x\in \Omega_{\alpha}=B_{\alpha}\setminus\bar{B}_{\alpha^{-1}\rho_k}(0)$, the constant blows up like $(1-\alpha)^{-1}$. Therefore, one could replace this harmonic component by another one defined on a larger domain, but to recover the standard metric on the rest of $\Sigma$, one would need to use a cutoff function and this would yield a new term in the Bochner formula, which seems equally intractable. Instead, we will derive a global pointwise estimate of $\D\lambda_k$ in $\Omega_k(\alpha)$ for all $0<\alpha<\alpha_0/2$.

Let $x\in \Omega_k(\alpha_0/2)$, $r>0$ be such that $B(x,2r)\subset \Omega_k(\alpha_0/2)$. Let $\mu_k:B(x,r)\rightarrow \R$ be the solution to the system
\begin{align*}
    \left\{\begin{alignedat}{2}
    -\Delta \mu_k&=\D^{\perp}\vec{e}_{k,1}\cdot \D\vec{e}_{k,2}\qquad &&\text{in}\;\, B(x,2r)\\
    \mu_k&=0\qquad &&\text{on}\;\,\partial B(x,2r).
    \end{alignedat}\right.
\end{align*}
where $(\vec{e}_{k,1},\vec{e}_{k,2})$ is a controlled moving frame (\cite{helein}) such that
\begin{align*}
    \int_{B(x,2r)}\left(|\D\vec{e}_{k,1}|^2+|\D\vec{e}_{k,2}|^2\right)dx\leq 2\int_{B(x,2r)}|\D\n_k|^2dx.
\end{align*}
Therefore, Wente's inequality shows that
\begin{align*}
    \np{\D\mu_k}{2}{B(x,2r)}&\leq \sqrt{\frac{3}{16\pi}}\np{\D\e_{k,1}}{2}{B(x,2r)}\np{\D\e_{k,2}}{2}{B(x,2r)}\\
    &\leq \frac{1}{2}\sqrt{\frac{3}{16\pi}}\int_{B(x,2r)}\left(|\D\vec{e}_{k,1}|^2+|\D\vec{e}_{k,2}|^2\right)dx\leq \sqrt{\frac{3}{16\pi}}\int_{B(x,2r)}|\D\n_k|^2dx.
\end{align*}
Likewise, we have
\begin{align*}
    \np{\mu_k}{\infty}{B(x,2r)}\leq \frac{1}{2\pi}\int_{B(x,2r)}|\D\n_k|^2dx.
\end{align*}
Therefore, using the bound \eqref{laurain_rivière_estimate}, if $\lambda_k=\mu_k+\nu_k$, we deduce that
\begin{align*}
    \np{\D\nu_k}{2,\infty}{B(x,2r)}&\leq \np{\D\lambda_k}{2,\infty}{B(x,2r)}+\np{\D\mu_k}{2,\infty}{B(x,2r)}\\
    &\leq \np{\D\lambda_k}{2,\infty}{B(x,2r)}+C_{\mathrm{L}}\np{\D\mu_k}{2}{B(x,2r)}\\
    &\leq C_{\mathrm{LR}}W(\phi_k)+C_{\mathrm{L}}\sqrt{\frac{3}{16\pi}}\np{\D\n_k}{2}{B(x,2r)}\leq C(1+W(\phi_k)).
\end{align*}
where $C_{\mathrm{L}}<\infty$ is a constant independent of the domain (here, L stands for Lorentz), and $C<\infty$ is a universal constant. Since $\lambda_k$ also solves the Liouville equation, we deduce that $\nu_k=\lambda_k-\mu_k$ is a harmonic function in $B(x,2r)$. Therefore, Lemma \ref{pointwise_l2infty_harmonic} implies that for all $y\in B(x,r)$, we have
\begin{align*}
    |\D\nu_k(y)|&\leq \frac{8}{\log(2)}\sqrt{\frac{3}{\pi}}\frac{1}{|y|}\np{\D\nu_k}{2,\infty}{B(x,2r)}\leq \frac{8}{\log(2)}\sqrt{\frac{3}{\pi}}C(1+W(\phi_k))\frac{1}{|y|}\\
    &\leq \frac{\Gamma_0}{|y|}=\frac{1}{|y|}\frac{8}{\log(2)}\sqrt{\frac{3}{\pi}}C\left(1+\sup_{k\in \N}W(\phi_k)\right).
\end{align*}
First, we trivially estimate by the decomposition $\lambda_k=\mu_k+\nu_k$
\begin{align*}
    \int_{B(x,r)}|\D\lambda_k|^2dy\leq 2\int_{B(x,r)}|\D\mu_k|^2dy+2\int_{B(x,r)}|\D\nu_k|^2dy.
\end{align*}
Now, recall that by the $\epsilon$-regularity (\cite{riviere1}; we used more precise estimates above and in \cite{willmore_scs}, but they  will not be necessary here) for all $z\in \Omega_k(\alpha_0/2)$, we have
\begin{align*}
    |\D\n_k(z)|\leq \frac{\Gamma_1}{|z|}\np{\D\n_k}{2}{B_{2|z|}\setminus\bar{B}_{|z|/2}(0)}.
\end{align*}
Integrating by parts, we deduce that
\begin{align*}
    \int_{B(x,r)}|\D\mu_k|^2d&y\leq \int_{B(x,2r)}|\D\mu_k|^2dy=-\int_{B(x,2r)}\mu_k\,\Delta\mu_k\,dy=\int_{B(x,2r)}\mu_ke^{2\lambda_k}K_{g_k}dx\\
    &\leq \frac{1}{2}\np{\mu_k}{\infty}{B(x,2r)}\int_{B(x,2r)}|\D\n_k|^2dy\leq \frac{\Gamma^2}{2}\np{\mu_k}{\infty}{B(x,2r)}\frac{4\pi r^2}{(|x|-2r)^2}\\
    &\leq \frac{\Gamma^2}{4\pi}\int_{B(x,2r)}|\D\n_k|^2dy\frac{4\pi r^2}{(|x|-2r)^2}.
\end{align*}
The next term can be directly estimated thanks to the pointwise estimate:
\begin{align*}
    \int_{B(x,r)}|\D\nu_k|^2dy\leq \int_{B(x,r)}\frac{\Gamma_0^2}{|y|^2}dy\leq \frac{\pi\Gamma_0^2r^2}{(|x|-r)^2}.
\end{align*}
Therefore, we deduce that
\begin{align*}
    \frac{1}{\pi r^2}\int_{B(x,r)}|\D\lambda_k|^2dx\leq \frac{\Gamma^2}{\pi(|x|-2r)^2}\int_{B(x,2r)}|\D\n_k|^2dy+\frac{2\Gamma_0^2}{(|x|-r)^2}.
\end{align*}
Therefore, since (here, we could be using Lebesgue's differentiation theorem, but this is not needed) $\lambda_k$ is continuous (recall that all tensors are smooth since $\phi_k$ is a smooth \textbf{immersion} and $\lambda_k=\log\left(|\D\phi_k|/\sqrt{2}\right)$), taking $r\rightarrow 0$, we obtain the estimate
\begin{align}
    |\D\lambda_k(x)|^2\leq \frac{2\Gamma_0^2}{|x|^2}.
\end{align}
Now, coming back to the Bochner formula, we deduce that
\begin{align*}
    \int_{\Omega_k(\alpha)}|x|^{2-2m}|\D u|^2|\D\lambda_k|^2\omega_{4,\alpha,k}dx&\leq 2\Gamma_0^2\int_{\Omega_k(\alpha)}|x|^{2-2m}|\D u|^2\frac{1}{|x|^2}\omega_{4,\alpha,k}^{+}dx\\
    &\leq 2\Gamma_0^2\Gamma\int_{\Omega_k(\alpha)}|du|_{g_k}^2\bar{\omega}_{1,\alpha,k}\,d\mathrm{vol}_{g_k}. 
\end{align*}
Therefore, we deduce that
\begin{align}\label{inequality_finale1}
    \Gamma^{-1}\int_{\Omega_k(\alpha)}|x|^{2-2m}|\D^2u|^2\omega_{4,\alpha,k}^{+}dx&\leq \int_{\Omega_k(\alpha)}\left(2|\mathrm{Hess}(u)|_{g_k}^2+(\Delta_{g_k}u)^2+20|d\lambda_k|^2_{g_k}|du|_{g_k}^2\right)d\mathrm{vol}_{g_k}\nonumber\\
    &\leq 5\int_{\Sigma}(\Delta_{g_k}u)^2\,d\mathrm{vol}_{g_k}+\Gamma'\int_{\Sigma}|du|_{g_k}^2\bar{\omega}_{1,\alpha,k}\,d\mathrm{vol}_{g_k}.
\end{align}
Therefore, by \eqref{ipp_poids_1}, \eqref{ipp_poids_2}, \eqref{bubbling_conformal5}, \eqref{ipp_poids_3}, \eqref{ipp_poids_4}, and \eqref{inequality_finale1}, we deduce that for some universal constant $C<\infty$ (depending only on $0<\beta_1<\min\ens{\dfrac{m-1}{4},1}$, $0<\epsilon<\min\ens{1-\dfrac{\beta_1}{\gamma},\dfrac{\beta_1}{2\gamma}}$, $0<\beta_1<\gamma<1$ and $m\geq 2$)
\begin{align}\label{inequality_finale2}
    &\int_{\Sigma}|du|_{g_k}^2\bar{\omega}_{1,\alpha,k}\,d\mathrm{vol}_{g_k}\leq C\int_{\Sigma}(\Delta_{g_k}u)^2\,d\mathrm{vol}_{g_k}+C\int_{\Sigma}u^2\omega_{1,\alpha,k}\,d\mathrm{vol}_{g_k}\nonumber\\
    &+C\left(\int_{\Sigma}u^2\omega_{1,\alpha,k}\,d\mathrm{vol}_{g_k}\right)^{\frac{1}{2}}\left(\int_{\Sigma}(\Delta_{g_k}u)^2\,d\mathrm{vol}_{g_k}+\int_{\Sigma}u^2\omega_{1,\alpha,k}\,d\mathrm{vol}_{g_k}+\int_{\Sigma}|du|_{g_k}^2\bar{\omega}_{1,\alpha,k}\,d\mathrm{vol}_{g_k}\right)^{\frac{1}{2}}\nonumber\\
    &\leq \left(C+\frac{1}{2}C^2\right)\left(\int_{\Sigma}(\Delta_{g_k}u)^2\,d\mathrm{vol}_{g_k}+\int_{\Sigma}u^2\omega_{1,\alpha,k}\,d\mathrm{vol}_{g_k}\right)+\frac{1}{2}\int_{\Sigma}|du|_{g_k}^2\bar{\omega}_{1,\alpha,k}\,d\mathrm{vol}_{g_k},
\end{align}
which yields
\begin{align}\label{inequality_finale3}
    \int_{\Sigma}|du|_{g_k}^2\bar{\omega}_{1,\alpha,k}\,d\mathrm{vol}_{g_k}\leq \left(2\,C+C^2\right)\left(\int_{\Sigma}(\Delta_{g_k}u)^2\,d\mathrm{vol}_{g_k}+\int_{\Sigma}u^2\omega_{1,\alpha,k}\,d\mathrm{vol}_{g_k}\right).
\end{align}

\textbf{Part 4: Estimate on the Fourth Component of the Second Derivative.}
Thanks to the uniformisation theorem, for all $k\in \N$, there exists a metric $g_{0,k}$ on $\Sigma$ of constant Gauss curvature and volume $1$ and a smooth function $\alpha_k:\Sigma\rightarrow \R$ such that
\begin{align*}
    g_{k}=e^{2\alpha_k}g_{0,k}.
\end{align*}
First assume that the underlying sequence of Riemann surfaces $(\Sigma,g_{0,k})_{k\in \N}$ does not degenerate in the moduli space. Then, we have
\begin{align}\label{fourth_term1}
    &\left|\int_{\Sigma}\s{\partial u\otimes \partial H_k}{h_{0,k}}_{\mathrm{WP}}\,u\,d\mathrm{vol}_{g_k}\right|\leq \nu_{\alpha,k}\int_{\Sigma}|u||du|_{g_{0,k}}\omega_{3,\alpha,k}\,d\mathrm{vol}_{g_{0,k}}.
\end{align}
This estimate follows once one takes the metric $g_{0,k}$ (that smoothly converges everywhere) in the definition of $\omega_{3,\alpha,k}$. Namely, the expression $|\D A_k|$ stands for $|dA_k|_{g_{0,k}}$ in the definition \eqref{third_weight}.

First, we have 
\begin{align*}
    &\int_{\Sigma\setminus\bar{B}(0,\alpha)}|u||du|_{g_{0,k}}\omega_{3,\alpha,k}\,d\mathrm{vol}_{g_{0,k}}\leq \left(\int_{\Sigma\setminus\bar{B}(0,\alpha)}u^2\frac{1}{\alpha^{2m+2}}\left(1+\left(\frac{\alpha^{-1}\delta_k}{\alpha}\right)^{2\beta_3}+\frac{1}{\log^2\left(\frac{\alpha^2}{\rho_k}\right)}\right)d\mathrm{vol}_{g_{0,k}}\right)^{\frac{1}{2}}\\
    &\times \left(\int_{\Sigma\setminus\bar{B}(0,\alpha)}|du|_{g_{0,k}}^2\frac{1}{\alpha^{2m}}\left(1+\left(\frac{\alpha^{-1}\delta_k}{\alpha}\right)^{2\beta_3}+\frac{1}{\log^2\left(\frac{\alpha^2}{\rho_k}\right)}\right)d\mathrm{vol}_{g_{0,k}}\right)^{\frac{1}{2}}.
\end{align*}
Recall that thanks to \eqref{harnack} there exists $0<\alpha_0<1$ such that for all $k\in \N$ large enough, we have (for some $0<\Gamma<\infty$)
\begin{align*}
    \Gamma^{-1}|x|^{2m-2}\leq e^{2\lambda_k(x)}\leq \Gamma|x|^{2m-2}\qquad\text{for all}\;\,x\in \Omega_k(\alpha_0).
\end{align*}
Therefore, we have for all $0<\alpha<_0$
\begin{align*}
    &\int_{B_{\alpha_0}\setminus\bar{B}_{\alpha}(0)}u^2\frac{1}{\alpha^{2m+2}}d\mathrm{vol}_{g_{0,k}}\leq\int_{B_{\alpha_0}\setminus\bar{B}_{\alpha}(0)}u^2\frac{1}{\alpha^{4m}}|x|^{2m-2}\,d\mathrm{vol}_{g_{0,k}}
    &\leq \Gamma'\int_{B_{\alpha_0}\setminus\bar{B}_{\alpha}(0)}u^2\omega_{1,\alpha,k}\,d\mathrm{vol}_{g_k}.
\end{align*}
On the other hand, by absence of branch points of the limiting branched immersion $\phi_{\infty}:\Sigma\rightarrow \R^3$ outside $0$, and by smooth convergence of $\{\phi_k\}_{k\in \N}$ in $C^{\infty}_{\mathrm{loc}}(\Sigma\setminus\ens{0})$, we deduce that for $k\in \N$ large enough, we have
\begin{align*}
    \inf_{\Sigma\setminus\bar{B}(0,\alpha_0)}e^{2\alpha_k}\geq \alpha^{2m-2},
\end{align*}
which shows that 
\begin{align*}
    \int_{\Sigma\setminus\bar{B}(0,\alpha_0)}u^2\frac{1}{\alpha^{2m+2}}d\mathrm{vol}_{g_{0,k}}\leq \int_{\Sigma\setminus\bar{B}(0,\alpha_0)}u^2\frac{1}{\alpha^{4m}}e^{2\alpha_k}\,d\mathrm{vol}_{g_{0,k}}\leq \int_{\Sigma\setminus\bar{B}(0,\alpha_0)}u^2\omega_{1,\alpha,k}\,d\mathrm{vol}_{g_k}.
\end{align*}
On the other hand, for $k$ large enough, we have 
\begin{align*}
    &\int_{\Sigma\setminus\bar{B}(0,\alpha)}|du|_{g_{0,k}}^2\frac{1}{\alpha^{2m}}\left(1+\left(\frac{\alpha^{-1}\delta_k}{\alpha}\right)^{2\beta_3}+\frac{1}{\log^2\left(\frac{\alpha^2}{\rho_k}\right)}\right)d\mathrm{vol}_{g_{0,k}}\\
    &=\int_{\Sigma\setminus\bar{B}(0,\alpha)}|du|_{g_{k}}^2\frac{1}{\alpha^{2m}}\left(1+\left(\frac{\alpha^{-1}\delta_k}{\alpha}\right)^{2\beta_3}+\frac{1}{\log^2\left(\frac{\alpha^2}{\rho_k}\right)}\right)d\mathrm{vol}_{g_{k}}
    \leq 2\int_{\Sigma\setminus\bar{B}(0,\alpha)}|du|_{g_k}^2\bar{\omega}_{1,\alpha,k}\,d\mathrm{vol}_{g_k}.
\end{align*}
independently of the choice of $0<\beta_3<1$. We now move to the estimate in the neck region.
We have
\begin{align*}
    &\int_{\Omega_k(\alpha)}|u||\D u|\omega_{3,\alpha,k}dx=\int_{\Omega_k(\alpha)}|u||\D u|\frac{1}{|x|^{2m+1}}\left(\left(\frac{|x|}{\alpha}\right)^{2\beta_3}+\left(\frac{\alpha^{-1}\rho_k}{|x|}\right)^{2\beta_3}+\frac{1}{\log^2\left(\frac{\alpha^2}{\rho_k}\right)}\right)dx\\
    &\leq 3\left(\int_{\Omega_k(\alpha)}\frac{u^2}{|x|^{2m+2}}\left(\left(\frac{|x|}{\alpha}\right)^{4(1-\epsilon)\beta_3}+\left(\frac{\alpha^{-1}\rho_k}{|x|}\right)^{4(1-\epsilon)\beta_3}+\frac{1}{\log^2\left(\frac{\alpha^2}{\rho_k}\right)}\right)dx\right)^{\frac{1}{2}}\\
    &\times \left(\int_{\Omega_k(\alpha)}\frac{|\D u|^2}{|x|^{2m}}\left(\left(\frac{|x|}{\alpha}\right)^{4\epsilon\beta_3}+\left(\frac{\alpha^{-1}\rho_k}{|x|}\right)^{4\epsilon\beta_3}+\frac{1}{\log^2\left(\frac{\alpha^2}{\rho_k}\right)}\right)dx\right)^{\frac{1}{2}}\\
    &\leq 3\sqrt{\Gamma}\left(\int_{\Omega_k(\alpha)}u^2\,\omega_{1,\alpha,k}\,d\mathrm{vol}_{g_k}\right)^{\frac{1}{2}} \left(\int_{\Omega_k(\alpha)}\frac{|\D u|^2}{|x|^{2m}}\left(\left(\frac{|x|}{\alpha}\right)^{4\epsilon\beta_3}+\left(\frac{\alpha^{-1}\rho_k}{|x|}\right)^{4\epsilon\beta_3}+\frac{1}{\log^2\left(\frac{\alpha^2}{\rho_k}\right)}\right)dx\right)^{\frac{1}{2}}.
\end{align*}
provided that we picked $(1-\epsilon)\beta_3\geq \beta_1$. The second component is estimated as in \textbf{Step 3}, so we need only estimate the term on the bubble region.

We have 
\begin{align*}
    &\int_{B(0,\alpha^{-1}\rho_k)}|u||\D u|\omega_{3,\alpha,k}dx=\int_{B(0,\alpha^{-1}\rho_k)}|u||\D u|\frac{1}{(\alpha^{-1}\rho_k)^{2m+1}}\frac{1}{\alpha^{2m+3}}\frac{(1+\alpha^2)^{\frac{2m+3}{2}}}{\left(1+\left(\frac{|x|}{\rho_k}\right)^2\right)^{\frac{2m+3}{2}}}\left(1+o_k(1)\right)dx\\
    &\leq 2\left(\int_{B(0,\alpha^{-1}\rho_k)}u^2\frac{1}{\alpha^4\rho_k^{2m+2}}\frac{(1+\alpha^2)^{m+3}}{\left(1+\left(\frac{|x|}{\rho_k}\right)^{2}\right)^{m+3}}dx\right)^{\frac{1}{2}}\left(\int_{B(0,\alpha^{-1}\rho_k)}|\D u|^2\frac{1}{\rho_k^{2m}}\frac{(1+\alpha^2)^{m}}{\left(1+\left(\frac{|x|}{\rho_k}\right)^2\right)^m}dx\right)^{\frac{1}{2}}.
\end{align*}
On the one hand, we have by \eqref{bubbling_conformal}
\begin{align*}
    &\int_{B(0,\alpha^{-1}\rho_k)}u^2\frac{1}{\alpha^4\rho_k^{2m+2}}\frac{(1+\alpha^2)^{m+3}}{\left(1+\left(\frac{|x|}{\rho_k}\right)^{2}\right)^{m+3}}dx\\
    &=\int_{B(0,\alpha^{-1}\rho_k)}u^2\frac{1}{\alpha^4\rho_k^{4m}}\frac{(1+\alpha^2)^{m+3}}{\left(1+\left(\frac{|x|}{\rho_k}\right)^{2}\right)^{2m+2}}\,\rho_k^{2m-2}\left(1+\left(\frac{|x|}{\rho_k}\right)^{2}\right)^{m-1}dx\\
    &\leq e^{2\,\Gamma}\int_{B(0,\alpha^{-1}\rho_k)}u^2\,\omega_{1,\alpha,k}\,d\mathrm{vol}_{g_k}.
\end{align*}
On the other hand, we have (for some $C_{\alpha}<\infty$)
\begin{align*}
    \frac{1}{\rho_k^{2m}}\frac{(1+\alpha^2)^m}{\left(1+\left(\frac{|x|}{\rho_k}\right)^{2}\right)^m}\leq \frac{C_{\alpha}}{\alpha^2\rho_k^{2m}}\frac{(1+\alpha^2)^{m+1}}{\left(1+\left(\frac{|x|}{\rho_k}\right)^{2}\right)^{m+1}}
\end{align*}
if and only if
\begin{align*}
    \alpha^2\left(1+\left(\frac{|x|}{\rho_k}\right)^2\right)\leq C_{\alpha}(1+\alpha^2)\qquad\text{for all}\;\, x\in B(0,\alpha^{-1}\rho_k),
\end{align*}
which is satisfied for $C_{\alpha}=1$ since 
\begin{align*}
    \sup_{|x|\leq \alpha^{-1}\rho_k}\alpha^2\left(1+\left(\frac{|x|}{\rho_k}\right)^2\right)=\alpha^2(1+\alpha^{-2})=(1+\alpha^2).
\end{align*}
Therefore, we deduce that 
\begin{align*}
    &\int_{B(0,\alpha^{-1}\rho_k)}|\D u|^2\frac{1}{\rho_k^{2m}}\frac{(1+\alpha^2)^{m}}{\left(1+\left(\frac{|x|}{\rho_k}\right)^2\right)^m}dx\leq \int_{B(0,\alpha^{-1}\rho_k)}|\D u|^2\frac{1}{\alpha^2\rho_k^{2m}}\frac{(1+\alpha^2)^{m+1}}{\left(1+\left(\frac{|x|}{\rho_k}\right)^2\right)^{m+1}}dx\\
    &\leq \int_{B(0,\alpha^{-1}\rho_k)}|\D u|^2\,\bar{\omega}_{1,\alpha,k}dx=\int_{B(0,\alpha^{-1}\rho_k)}|du|_{g_k}^2\,\bar{\omega}_{1,\alpha,k}\,d\mathrm{vol}_{g_k}
\end{align*}
and the theorem follows in the case of non-degenerating conformal class. If the conformal class degenerates, we follow the analysis of Laurain--Rivière \cite{lauriv1}. Let $\ens{\gamma_k^1,\cdots,\gamma_k^N}_{k\in \N}$ be the set of shrinking geodesics of $(\Sigma,g_{0,k})$. According to the compactness theorem of Deligne--Mumford (\cite{deligne_mumford}), the sequence of Riemann surfaces $(\Sigma,g_{0,k})_{k\in \N}$ converges to a nodal surface $(\widetilde{\Sigma},\widetilde{g_0})$ and there exists a sequence of diffeomorphism $\psi_k:\widetilde{\Sigma}\setminus \ens{p_1,\cdots,p_{2N}}\rightarrow \Sigma\setminus\ens{\gamma_k^1,\cdots,\gamma_k^N}$ (where $\ens{p_1,\cdots,p_{2N}}$ are the cusps of $\widetilde{\Sigma}$) that extends to a continuous map $\psi_k:\widetilde{\Sigma}\rightarrow \Sigma$ and such that $\widetilde{g}_{0,k}=\psi_k^{\ast}g_{0,k}$ converges to $\widetilde{g}_0$ in $C^{\infty}_{\mathrm{loc}}$ away from cusps. For all connected component $\Sigma_i$ of $\widetilde{\Sigma}$, consider the map $\phi^i_k=\phi_k\circ \psi_k|\Sigma_{i}:\Sigma_i\rightarrow \R^3$. Thanks to the renormalisation of Laurain--Rivière (\cite[p. 906]{lauriv1}), there exists a sequence of conformal transformations $M_k:\R^3\cup\ens{\infty}\rightarrow \R^3\cup\ens{\infty}$ such that $\ens{M_k\circ \phi_k^i}_{k\in \N}:\Sigma_i\rightarrow \R^3$ converges smoothly away from finitely points to a limiting branched immersion $\phi_{\infty}^i:\Sigma_i\rightarrow \R^3$. Therefore, we can now apply the previous bubbling analysis to $\ens{M_k\circ \phi_k^i}_{k\in \N}$ and this concludes the proof of the theorem.  
\end{proof}

\textbf{We finally move to the proof of the main theorem (Theorem \ref{main_theorem}) of the article.}
\begin{proof}
    We follow the lines of the proof of the main theorem of \cite{willmore_scs}. They are unchanged up to one important point that renormalising the variations to 
    \begin{align*}
        \int_{\Sigma}u_k^2\,\omega_{1,\alpha,k}\,d\mathrm{vol}_{g_k}=1.
    \end{align*}
    Using the previous Theorem \ref{theorem_eigenvalue_bounded}, we deduce in particular that for $0<\alpha<\alpha_0$ small enough and $k\in \N$ large enough, we have
    \begin{align*}
        \int_{\Sigma}(\Delta_{g_k}u_k)^2\,d\mathrm{vol}_{g_k}\leq 4\mu_0.
    \end{align*}
    Therefore, integrating by parts, since $\Sigma$ is closed, we deduce that
    \begin{align*}
        \int_{\Sigma}|du_k|_{g_k}^2\,d\mathrm{vol}_{g_k}=-\int_{\Sigma}u_k\,\Delta_{g_k}u_k\,d\mathrm{vol}_{g_k}&\leq \left(\int_{\Sigma}(\Delta_{g_k}u)^2\,d\mathrm{vol}_{g_k}\right)^{\frac{1}{2}}\left(\int_{\Sigma}u_k^2\,d\mathrm{vol}_{g_k}\right)^{\frac{1}{2}}\\
        &\leq C_{\alpha}\left(\int_{\Sigma}(\Delta_{g_k}u)^2\,d\mathrm{vol}_{g_k}\right)^{\frac{1}{2}}\left(\int_{\Sigma}u_k^2\,\omega_{1,\alpha,k}\,d\mathrm{vol}_{g_k}\right)^{\frac{1}{2}}\\
        &\leq 2C_{\alpha}\sqrt{\mu_0}<\infty. 
    \end{align*}
    In $\Sigma\setminus\bar{B}(0,\alpha)$, the sequence of immersions $\{\phi_k\}_{k\in \N}$ converges locally smoothly. %, while\\ $\vec{\Psi}_k=e^{-\lambda_k(x_k)}\phi_k(\rho_k(\,\cdot\,))$ converges also smooth locally on $B(0,\alpha^{-1})$. 
    Furthermore, we have for a uniformly bounded sequence of functions $\ens{\mu_k}_{k\in\N}$ the expansion
    \begin{align*}
        e^{2\lambda_k}=e^{2\mu_k}|x|^{2m-2}\qquad \text{for all}\;\, x\in \Omega_k(\alpha).
    \end{align*}
    In particular, for any fixed smooth metric $g_0$ on $\Sigma$, we have $\ens{u_k}_{k\in \N}$ bounded in $W^{2,2}_{\mathrm{loc}}(\Sigma\setminus\ens{0},g_0)$. Therefore, up to a subsequence, it converges weakly in $W^{2,2}(\Sigma)$ to $u_{\infty}\in W^{2,2}(\Sigma)$. However, the convergence in the bubble domain is much more delicate due to the singularity of the metric. To define the bubble, we follow the approach of Laurain--Rivière (\cite[p. 2089, 2113]{quantamoduli}). Let $\ens{\mu_k}_{k\in \N}\subset \R$ to be determined later and define \begin{align*}
        \vec{\Psi}_k=e^{-\mu_k}\left(\phi_k(\rho_k(\,\cdot\,))-\phi_k(0)\right):B(0,2\alpha^{-1})\rightarrow\R^3.
    \end{align*}
    Notice that thanks to the Harnack inequality in the bubble domain, for all $0<\alpha<\infty$, there exists $C_{\alpha}<\infty$ such that 
    \begin{align}\label{harnack_bubble_domain}
        0<\sup_{x\in B(0,2\alpha^{-1}\rho_k)}e^{\lambda_k(x)}\leq C_{\alpha}\inf_{x\in B(0,2\alpha^{-1}\rho_k)}e^{\lambda_k(x)}
    \end{align}
    Therefore, thanks to the expansion \eqref{neck_domain_expansion}, we deduce that 
    \begin{align*}
        \Gamma_0^{-1}C_{\alpha}^{-1}(\alpha^{-1}\rho_k)^{m-1}\leq e^{\lambda_k(x)}\leq \Gamma_0C_{\alpha}(\alpha^{-1}\rho_k)^{m-1}\qquad \text{for all}\;\, x\in B(0,2\alpha^{-1}\rho_k),
    \end{align*}
    where $\Gamma_0<\infty$ is a universal constant. In other words, there exists a bounded sequence of functions $\ens{\zeta_{\alpha,k}}_{k\in \N}$ such that
    \begin{align*}
        e^{\lambda_k(x)}=e^{\zeta_{\alpha,k}(x)}(\alpha^{-1}\rho_k)^{m-1}\qquad \text{for all}\;\, x\in B(0,2\alpha^{-1}\rho_k).
    \end{align*}
    Now, we have
    \begin{align*}
        \p{z}\vec{\Psi}_k(z)=\rho_k\,e^{-\mu_k}\p{z}\phi_k(\rho_kz),
    \end{align*}
    which shows that the conformal parameter of $\vec{\Psi}_k$ is given by
    \begin{align*}
        \lambda_{\vec{\Psi}_k}(z)=\lambda_k(\rho_kz)-\mu_k+\log(\rho_k).
    \end{align*}
    The previous estimate shows that
    \begin{align*}
        m\,\log(\rho_k)-\mu_k-\log(\Gamma_0C_{\alpha}\alpha^{1-m})\leq \lambda_{\vec{\Psi}_k}(z)\leq m\,\log(\rho_k)-\mu_k+\log(\Gamma_0C_{\alpha}\alpha^{1-m}).
    \end{align*}
    In order to extract a subsequence, we want to bound this quantity independently of $k$, which prompts us to pick 
    \begin{align*}
        \mu_k=m\,\log(\rho_k).
    \end{align*}
    Therefore, we deduce that for all $0<\alpha<0$, the conformal factor of $\{\vec{\Psi}_k\}_{k\in \N}$ is uniformly bounded in $B(0,2\alpha^{-1})$, so using the analysis of Bernard--Rivière (\cite{quanta}), we deduce that there exists a branched Willmore sphere $\vec{\Psi}_{\infty}:S^2\rightarrow\R^3$ such that
    \begin{align*}
        \vec{\Psi}_k=\rho_k^{-m}\left(\phi_k\left(\rho_k(\,\cdot\,)\right)-\phi_k(0)\right)\conv{k\rightarrow \infty}\vec{\Psi}_{\infty}\circ \pi^{-1}\qquad \text{in}\;\, C^l_{\mathrm{loc}}(\C)\;\,\text{for all}\;\, l\in \N,
    \end{align*}
    where $\pi^{-1}:\C\rightarrow S^2$ is the inverse stereographic projection. Notice that following the afore-mentioned article \cite[p. 2089]{quantamoduli}, we would have chosen
    \begin{align*}
        e^{\mu_k}&=\mathscr{H}^1(\phi_k(2\alpha^{-1}\rho_k\,S^1))=\int_{0}^{2\pi}\left|\frac{d}{d\theta}\phi_k(\rho_ke^{i\theta})\right|d\theta\\
        &=\rho_k\int_{0}^{2\pi}e^{\mu_k(2\alpha^{-1}\rho_ke^{i\theta})}e^{c_k}(\alpha^{-1}\rho_k)^{m-1}(1+O(\alpha^{-1}\rho_k))=e^{\Gamma_k}\alpha^{1-m}e^{c_k}\rho_k^{m}(1+O(\alpha^{-1}\rho_k)),
    \end{align*}
    where $\ens{\Gamma_k}_{k\in \N}\subset \R$ is a bounded sequence. Therefore, the choice leads to a scaled immersion, and by scaling invariance of the Willmore immersion, to the same branched surface. Now, using \eqref{harnack_bubble_domain}, if we define $w_k(z)=u_k(\rho_k z)$ we deduce that
    \begin{align*}
        \int_{B(0,2\alpha^{-1}\rho_k)}(\Delta_{g_k}u_k)^2\,d\mathrm{vol}_{g_k}&=\int_{B(0,2\alpha^{-1}\rho_k)}e^{-2\lambda_k}(\Delta u_k(z))|dz|^2\\
        &=\int_{B(0,2\alpha^{-1}\rho_k)}e^{-2\,\zeta_{\alpha,k}(z)}(\alpha^{-1}\rho_k)^{2-2m}\left(\rho_k^{-2}\Delta w_k(\rho_k^{-1}z)\right)^2|dz|^2\\
        &=\alpha^{2m-2}\rho_k^{-2m}\int_{B(0,2\alpha^{-1})}e^{-2\,\zeta_{\alpha,k}(\rho_kz)}(\Delta w_k(z))^2|dz|^2.
    \end{align*}
    By boundedness of $\ens{\zeta_{\alpha,k}}_{k\in \N}$, we deduce that the correct renormalisation is to define
    \begin{align*}
        w_k(z)=\rho_k^{m}v_k(z), 
    \end{align*}
    which yields
    \begin{align}
        v_k(z)=\rho_k^{-m}u_k(\rho_kz). 
    \end{align}

    Let $\ens{\varphi_{k,j}}_{1\leq j\leq N_k}$ be an orthonormal base of 
        \begin{align*}
            \mathscr{E}_{\alpha,k}^0=\bigoplus_{\lambda\leq 0}\mathscr{E}_{\alpha,k}(\lambda).
        \end{align*}
        For all $1\leq j\leq N_k$, we have
        \begin{align*}
            \mathcal{L}_{g_k}\varphi_{k,j}=\lambda_{k,j}\,\omega_{1,\alpha,k}\,\varphi_{k,j}.
        \end{align*}
        Since $u_k\in \mathscr{S}_{\alpha,k}\simeq S^{N_k-1}$, we have a decomposition
        \begin{align*}
            u_k=\sum_{j=1}^{N_k}a_{k,j}\varphi_{k,j},
        \end{align*}
        where $a=\ens{a_{k,j}}_{1\leq j\leq N_k}\in S^{N_k-1}$. Then, we get as $-\mu_0\leq \lambda_{k,j}\leq 0$ the inequality
        \begin{align}\label{bound_w42}
            \int_{\Sigma}\frac{1}{\omega_{1,\alpha,k}}|\mathcal{L}_{g_k}u_k|^2\,d\mathrm{vol}_{g_k}&=\int_{\Sigma}\sum_{i,j=1}^{N_k}a_{k,i}a_{k,j}\lambda_{k,i}\lambda_{k,j}\varphi_{k,i}\varphi_{k,j}\,\omega_{1,\alpha,k}\,d\mathrm{vol}_{g_k}\nonumber\\
            &=\sum_{j=1}^{N_k}|a_{k,j}|^2|\lambda_{k,j}|^2\leq \mu_0^2.
        \end{align}
        Therefore, there exists a uniformly bounded sequence of functions $\ens{f_k}_{k\in\N}\subset L^2(\Sigma)$ such that
        \begin{align*}
            \mathcal{L}_{g_k}u_k=\sqrt{\omega_{\alpha,k}}f_k.
        \end{align*}
    Since $\{\phi_k\}_{k\in \N}$ converges strongly in $C^l_{\mathrm{loc}}(\Sigma\setminus\ens{0})$ towards $\phi_{\infty}$, the restriction of family of metrics $\ens{g_k}_{k\in \N}$ to $\Sigma\setminus\bar{B}(0,\alpha)$ is uniformly elliptic which implies that there exists a constant $C_{\alpha}<\infty$ such that
    \begin{align*}
        \wp{u_k}{4,2}{\Sigma\setminus\bar{B}(0,\alpha)}\leq C_{\alpha}.
    \end{align*}
    Likewise, thanks to the strong convergence in the bubble region, we also have
    \begin{align*}
        \wp{v_k}{4,2}{B(0,2\alpha^{-1})}\leq C_{\alpha}.
    \end{align*}
    Let us justify this step. Thanks to \eqref{bound_w42}, we have
    \begin{align*}
            &\int_{B(0,2\alpha^{-1}\rho_k)}\frac{1}{\omega_{1,\alpha,k}}|e^{-4\lambda_k}\Delta^2u_k|^2\,d\mathrm{vol}_{g_k}\\
            &=\int_{B(0,2\alpha^{-1}\rho_k)}(\alpha^{-1}\rho_k)^{4m}e^{-6\zeta_{\alpha,k}}(\alpha^{-1}\rho_k)^{-6m+6}\left|\rho_k^{m-4}\Delta^2v_k(\rho_k^{-1}z)\right|^2|dz|^2\\
            &=\alpha^{4m-4}\int_{B(0,2\alpha^{-1}\rho_k)}\rho_k^{-2}|\Delta^2v_k(\rho_k^{-1}z)|^2|dz|^2=\alpha^{4m-4}\int_{B(0,2\alpha^{-1})}|\Delta^2v_k(w)|^2|dw|^2. 
    \end{align*}
    Therefore, $\ens{v_k}_{k\in \N}$ is bounded in $W^{4,2}_{\mathrm{loc}}(B(0,2\alpha^{-1}))$. We deduce that up to a subsequence, we have
    \begin{align}\label{w42_convergence}
         u_k\underset{k\rightarrow \infty}{\rightharpoonup}u_{\infty}\quad \text{in}\;\, W^{4,2}_{\mathrm{loc}}(\Sigma\setminus\ens{0})\quad \text{and}\quad v_k\underset{k\rightarrow \infty}{\rightharpoonup}v_{\infty}\quad \text{in}\;\, W^{4,2}_{\mathrm{loc}}(\C). 
    \end{align}
    Assume by contradiction that $u_{\infty}=0$ and $v_{\infty}=0$. Let $\chi_{\alpha,k}\in C^{\infty}(\Sigma)$ be a cutoff function such that $\chi_{\alpha,k}=1$ in $\Sigma\setminus\bar{\Omega_k(\alpha)}$ and $\chi_{\alpha,k}=0$ in $\Omega_k(\alpha/2)$ and $\widetilde{u}_k=\chi_{\alpha,k}u_k$. We can also assume that for some universal constant $C<\infty$ independent of $\alpha$ and $k$, we have
    \begin{align*}
        |\D^l\chi_{\alpha,k}|\leq \frac{C}{|x|^l}\quad \forall \;\,x\in \left(B_{\alpha}\setminus\bar{B}_{\alpha/2}(0)\right)\cup\;\,\left(B_{2\alpha^{-1}\rho_k}\setminus\bar{B}_{\alpha^{-1}\rho_k}(0)\right), \forall\;\,l=0,1,2.
    \end{align*}
    Recall that by \eqref{formula_second_derivative}, we have
    \begin{align}\label{formula_second_derivative2}
    Q_{\phi_k}(u)&=\frac{1}{2}\int_{\Sigma}\left(\Delta_{g_k}u+|A_k|^2u\right)^2\,d\mathrm{vol}_{g_k}+\int_{\Sigma}\left(|du|_h^2+4|h_{0,k}|_{\mathrm{WP}}u^2\right)H_k^2\,d\mathrm{vol}_{g_k}\nonumber\\
    &-8\int_{\Sigma}\s{\partial u\otimes \partial u}{h_{0,k}}_{\mathrm{WP}}H_k\,d\mathrm{vol}_{g_k}
    -16\int_{\Sigma}\s{\partial u\otimes \partial H_k}{h_{0,k}}_{\mathrm{WP}}u\,d\mathrm{vol}_{g_k}.
\end{align}
    For simplicity of notation, define for all $A\subset \Sigma$
    \begin{align*}
        Q_{\phi_k}(u|A)&=\frac{1}{2}\int_{A}\left(\Delta_{g_k}u+|A_k|^2u\right)^2\,d\mathrm{vol}_{g_k}+\int_{A}\left(|du|_h^2+4|h_{0,k}|_{\mathrm{WP}}u^2\right)H_k^2\,d\mathrm{vol}_{g_k}\nonumber\\
    &-8\int_{A}\s{\partial u\otimes \partial u}{h_{0,k}}_{\mathrm{WP}}H_k\,d\mathrm{vol}_{g_k}
    -16\int_{A}\s{\partial u\otimes \partial H_k}{h_{0,k}}_{\mathrm{WP}}u\,d\mathrm{vol}_{g_k}.
    \end{align*}
    Thanks to \eqref{w42_convergence}, we deduce that for all fixed $0<\beta<1$, up to a subsequence, $\ens{u_k}_{k\in \N}$ converges strongly in $C^{2,\beta}_{\mathrm{loc}}(\Sigma\setminus\ens{0})$ to $0$ and likewise, $\ens{v_k}_{k\in \N}$ converges strongly in $C^{2,\beta}_{\mathrm{loc}}(\C)$ to $0$. 
    In particular, we have
    \begin{align*}
        Q_{\phi_k}(u_k|\Sigma\setminus\bar{B}(0,\alpha))\conv{k\rightarrow \infty}0\qquad Q_{\phi_k}(u_k|B(0,\alpha^{-1}\rho_k))=Q_{\vec{\Psi}_k}(v_k|B(0,\alpha^{-1}))\conv{k\rightarrow \infty}0.
    \end{align*}
    Then, we have
    \begin{align*}
        Q_{\phi_k}(u_k)-Q_{\phi_k}(\widetilde{u}_k)=Q_{\phi_k}(u_k|\Sigma\setminus\bar{B}(0,\alpha))+Q_{\phi_k}(u_k|B(0,\alpha^{-1}\rho_k))+Q_{\phi_k}(u_k|\Omega_k(\alpha))-Q_{\phi_k}(\widetilde{u}_k|\Omega_k(\alpha)).
    \end{align*}
    We have
    \begin{align*}
        &\frac{1}{2}\int_{\Omega_k(\alpha)}\left(\Delta_{g_k}u_k+|A_k|^2u_k\right)^2\,d\mathrm{vol}_{g_k}-\frac{1}{2}\int_{\Omega_k(\alpha)}\left(\Delta_{g_k}\widetilde{u}_k+|A_k|^2\widetilde{u}_k\right)^2\,d\mathrm{vol}_{g_k}\\
        &=\frac{1}{2}\int_{B_{\alpha}\setminus\bar{B}_{\alpha/2}(0)}\left(1-\chi_{\alpha,k}\right)^2\left(\Delta_{g_k}u_k+|A_k|^2u_k\right)^2\,d\mathrm{vol}_{g_k}\\
        &+\frac{1}{2}\int_{B_{2\alpha^{-1}\rho_k}\setminus\bar{B}_{\alpha^{-1}\rho_k}(0)}\left(1-\chi_{\alpha,k}\right)^2\left(\Delta_{g_k}u_k+|A_k|^2u_k\right)^2\,d\mathrm{vol}_{g_k}\\
        &-\frac{1}{2}\int_{B_{\alpha}\setminus\bar{B}_{\alpha/2}(0)}u_k^2\left(\Delta_{g_k}\chi_{\alpha,k}\right)^2\,d\mathrm{vol}_{g_k}-\frac{1}{2}\int_{B_{2\alpha^{-1}\rho_k}\setminus\bar{B}_{\alpha^{-1}\rho_k}(0)}u_k^2\left(\Delta_{g_k}\chi_{\alpha,k}\right)^2\,d\mathrm{vol}_{g_k}\\
        &-2\int_{B_{\alpha}\setminus\bar{B}_{\alpha/2}(0)}\s{du_k}{d\chi_{\alpha,k}}_{g_k}^2\,d\mathrm{vol}_{g_k}-2\int_{B_{2\alpha^{-1}\rho_k}\setminus\bar{B}_{\alpha^{-1}\rho_k}(0)}\s{du_k}{d\chi_{\alpha,k}}_{g_k}^2\,d\mathrm{vol}_{g_k}\\
        &-\int_{B_{\alpha}\setminus\bar{B}_{\alpha/2}(0)}\Big(\chi_{\alpha,k}(\Delta_{g_k}\chi_{\alpha,k})u_k+2\s{du_k}{d\chi_{\alpha,k}}_{g_k}\chi_{\alpha,k}\Big)\left(\Delta_{g_k}u_k+|A_k|^2u_k\right)d\mathrm{vol}_{g_k}\\
        &-\int_{B_{2\alpha^{-1}\rho_k}\setminus\bar{B}_{\alpha^{-1}\rho_k}(0)}\Big(\chi_{\alpha,k}(\Delta_{g_k}\chi_{\alpha,k})u_k+2\s{du_k}{d\chi_{\alpha,k}}_{g_k}\chi_{\alpha,k}\Big)\left(\Delta_{g_k}u_k+|A_k|^2u_k\right)d\mathrm{vol}_{g_k}.
    \end{align*}
    Thanks to the strong local convergence of $\ens{u_k}_{k\in \N}$ and $\ens{v_k}_{k\in \N}$ to $0$, and the strong local convergence of $\{\phi_k\}_{k\in \N}$ to $\phi_{\infty}$ in $C^l(\Sigma\setminus B(0,\alpha/2))$ and $\{\vec{\Psi}_k\}_{k\in \N}$ to $\vec{\Psi}_{\infty}$ in $C^l(B(0,2\alpha^{-1}))$, combined with the pointwise bound on $\ens{\chi_{\alpha,k}}_{k\in \N}$, we deduce that all integrals converge to $0$. Therefore, we deduce that
    \begin{align}\label{penultimate_convergence}
        \lim_{k\rightarrow \infty }\left|Q_{\phi_k}(u_k)-Q_{\phi_k}(\widetilde{u}_k)\right|=0.
    \end{align}    
    Furthermore, the strong convergence shows that 
    \begin{align}\label{last_convergence}
        \lim_{k\rightarrow \infty}\left|\int_{\Sigma}u_k^2\omega_{1,\alpha,k}\,d\mathrm{vol}_{g_k}-\int_{\Sigma}\widetilde{u}_k^2\omega_{1,\alpha,k}\,d\mathrm{vol}_{g_k}\right|=0. 
    \end{align}
    However, as $\widetilde{u}_k\in W^{2,2}_0(\Omega_k(\alpha))$, we have by Theorem \ref{lower_bound_main_theorem} the bound
    \begin{align*}
        Q_{\phi_k}(\widetilde{u}_k)&\geq \lambda_0\int_{\Omega_k(\alpha)}\frac{\widetilde{u}_k^2}{|x|^{2m+2}}\left(\left(\frac{|x|}{\alpha}\right)^{4\beta_1}+\left(\frac{\alpha^{-1}\rho_k}{|x|}\right)^{4\beta_1}+\frac{1}{\log^2\left(\frac{\alpha^2}{\rho_k}\right)}\right)dx\\
        &\geq \lambda_0\Gamma^{-1}\int_{\Omega_k(\alpha)}\widetilde{u}_k^2\,\omega_{1,\alpha,k}\,d\mathrm{vol}_{g_k}
        \geq \frac{1}{2}\lambda_0\Gamma^{-1}>0
    \end{align*}
    for $k$ large enough, where we used the Harnack inequality \eqref{harnack}, the convergence \eqref{last_convergence}, and the fact that $u_k\in \mathscr{S}_{\alpha,k}$. This is a contradiction since $Q_{\phi_k}(u_k)\leq 0$ and using the convergence \eqref{penultimate_convergence}. Therefore, we have $(u_{\infty},v_{\infty})\neq (0,0)$ and we can finish the argument as in \cite{riviere_morse_scs} and \cite{willmore_scs}.     
    Notice that thanks to Fatou's lemma, $u_k\wconv{k\rightarrow \infty}u_{\infty}$ in $W^{2,2}(\Sigma)$ and that
    \begin{align}\label{bounded}
        \int_{\Sigma}u_{\infty}^2\,\omega_{1,\alpha,\infty}\,d\mathrm{vol}_{g_{\infty}}\leq 1.
    \end{align}
    In particular, $u_{\infty}$ is an admissible variation (see \cite{index4}) since it vanishes sufficiently quickly at branch points, which shows that it is a genuine negative variation of the limiting immersion of the bubble. Indeed, notice that by standard elliptic regularity, $u_{\infty}$ is smooth. Since
    \begin{align*}
        \omega_{1,\alpha,\infty}d\mathrm{vol}_{g_{\infty}}=\frac{e^{2\mu_{\infty}}}{\alpha^{4\beta_1}}\frac{1}{|x|^{2m+2-4\beta_1}}dx
    \end{align*}    
    (where $\mu_{\infty}$ is a smooth function) and $1/2<\beta_1<1$, the bound \eqref{bounded} shows that $u(z)=O(|z|^{m-1})$. Then, the bound of $\Delta_{g_{\infty}}u_{\infty}$ in $L^2(\Sigma,g_{\infty})$ shows that 
    \begin{align*}
        u_{\infty}(z)=\Re(\gamma z^m)+O(|z|^{m+1})
    \end{align*}
    which implies that $u_{\infty}$ is indeed admissible. Furthermore, for such variations, making the change of variation $u_{\infty}=(\eta|x|^{m-1}+(1-\eta))v_{\infty}$ for some well-chosen cutoff function on $\Sigma$ (this is easy to see for smooth variations away from branched points), the index operator has a discrete spectrum for the weighted operator (see Theorem \ref{discrete_spectrum}).
    The final argument by contradiction can therefore be reproduced \emph{mutadis mutandis} which concludes the proof of the theorem. 
\end{proof}

\section{Appendix}

\subsection{Discreetness of the Spectrum for the Weighted Index Operator}

In this appendix, we prove that in the class of variations of finite weighted $W^{2,2}$ norm, the index operator has a discrete spectrum. 

\begin{theorem}\label{discrete_spectrum}
    Let $\Sigma$ be a closed Riemann surface and $\phi:\Sigma\rightarrow \R^3$ be a branched Willmore surface whose branched points $p_1,\cdots,p_n$ are of respective multiplicities $m_1,\cdots,m_n\geq 1$. Let $\omega\in C^{\infty}(\Sigma\setminus\ens{p_1,\cdots,p_n})$ be such that there exists $1/2<\beta_i<1$ such that 
    \begin{align*}
        0<C_0^{-1}\leq \omega(x)\leq C_0\sum_{i=1}^n\frac{1}{\mathrm{dist}(x,p_i)^{4(m_i-\beta_i)}},
    \end{align*}
    where $\ens{C_l}_{l\in\N}\subset (0,\infty)$. Denote
    \begin{align}\label{weighted_L2}
        L^2_{\omega}(\Sigma)=L^2(\Sigma)\cap\ens{u:\int_{\Sigma}u^2\,\omega\,d\mathrm{vol}_{g}<\infty},
    \end{align}
    where $g=\phi^{\ast}g_{\R^3}$ is any smooth metric on $\Sigma$. Consider the operator
    \begin{align*}
        \mathcal{L}_{g,\omega}&=\omega^{-1}\mathcal{L}_g=\omega^{-1}\bigg\{\Delta_g^2+|A|^2\Delta_g+2\s{d|A|^2}{d(\,\cdot\,)}_g+2\,d^{\ast_g}\left(H^2\,d(\,\cdot\,)\right)\\
        &+\left(|A|^4+\Delta_g|A|^2+24\,H^2|h_0|_{\mathrm{WP}}^2\right)+16\,\ast_g\,d\,\Re\left(g^{-1}\otimes h_0\otimes\partial(H\,\cdot\,)\right)\nonumber\\
            &-16\s{\partial(\,\cdot\,)\otimes\partial H}{h_0}_{\mathrm{WP}}\bigg\}
    \end{align*}
    acting on 
    \begin{align*}
        W^{2,2}_{\omega}(\Sigma)=W^{2,2}(\Sigma)\cap\ens{u:u\in L^2_{\omega}(\Sigma), \int_{\Sigma}(\Delta_gu)^2d\vg<\infty},
    \end{align*}
    that we equip with the norm
    \begin{align*}
        \Vert u\Vert_{\mathrm{W}^{2,2}_{\omega}(\Sigma)}=\left(\int_{\Sigma}u^2\,\omega\,d\vg\right)^{\frac{1}{2}}+\left(\int_{\Sigma}(\Delta_gu)^2d\vg\right)^{\frac{1}{2}}.
    \end{align*}
    Then, there exists a Hilbertian base of $L^2_{\omega}(\Sigma)$ made of eigenvectors of $\mathcal{L}_{g,\omega}$ whose eigenvalues satisfy
    \begin{align*}
        \lambda_1\leq \lambda_2\leq \cdots\leq \lambda_k\conv{k\rightarrow \infty}\infty. 
    \end{align*}
\end{theorem}
\begin{proof}
    Recall that
    \begin{align}\label{index_form}
        Q_{\phi}(u)&=\frac{1}{2}\int_{\Sigma}(\Delta_gu+|A|^2u)^2d\vg+\int_{\Sigma}\left(|du|_g^2+4|h_0|_{\mathrm{WP}}^2u^2\right)H^2d\vg\nonumber\\
        &-8\int_{\Sigma}\s{\partial u\otimes \partial u}{h_0}_{\mathrm{WP}}H\,d\vg
        -16\int_{\Sigma}\s{\partial u\otimes\partial H}{h_0}_{\mathrm{WP}}u\,d\vg.
    \end{align}
    It suffices to treat the case $n=1$, in which case we write for simplicity $m=m_1$. Fixed a conformal chart centred in $p_1$ in $B(0,1)\subset \R^2$. Thanks to the analysis of \cite{beriviere}, we deduce that there exists a constant $\Gamma_1<\infty$ such that for all $x\in B(0,1)$, we have
    \begin{align*}
        |H(x)|\leq \frac{\Gamma_1}{|x|^{m-1}}\qquad |h_0(x)|_{WP}\leq \frac{\Gamma_1}{|x|^{m-1}}.
    \end{align*}
    Furthermore, up to scaling if $\phi^{\ast}g_{\R^3}=e^{2\lambda}dx$ in our conformal chart, we have
    \begin{align*}
        e^{2\lambda(x)}=|x|^{2m-2}\left(1+O(|z|)\right). 
    \end{align*}
    and we can therefore assume (up to shrinking the domain) that 
    \begin{align*}
        \frac{1}{2}|x|^{2m-2}\leq e^{2\lambda(x)}\leq 2|x|^{2m-2}\qquad \text{for all}\;\, x\in B(0,1).
    \end{align*}
    The hypothesis shows that for all $u\in W^{2,2}_{\omega}(\Sigma)$, we have 
    \begin{align*}
        \int_{B(0,1)}u^2(x)\frac{dx}{|x|^{2m+2-4\beta_1}}<\infty. 
    \end{align*}
    In particular, we have
    \begin{align*}
        \int_{B(0,1)}u^2|A|^4d\mathrm{vol}_{g}\leq 32\,\Gamma_1^4\int_{B(0,1)}u^2(x)\frac{dx}{|x|^{2m-2}}\leq 32\,\Gamma_1^2\int_{B(0,1)}u^2(x)\frac{dx}{|x|^{2m+2-4\beta_1}}<\infty,
    \end{align*}
    where we used that $\beta_1<1$ (notice that the proof would also work for $\beta_1=1$) and $|A|^2=2|H|^2+2|h_0|_{WP}^2$. Then, we have
    \begin{align}\label{final_gagliardo1}
        \int_{\Sigma}|du|_{g}^2|A|^2\,d\mathrm{vol}_{g}=-\int_{\Sigma}u\,\Delta_gu\,|A|^2d\vg-\int_{\Sigma}u\s{du}{d|A|^2}_{g}d\vg.
    \end{align}
    We estimate by Cauchy--Schwarz inequality
    \begin{align}\label{final_gagliardo2}
        -\int_{\Sigma}u\,\Delta_gu\,|A|^2d\vg\leq \left(\int_{\Sigma}u^2|A|^4d\vg\right)^{\frac{1}{2}}\left(\int_{\Sigma}(\Delta_gu)^2d\vg\right)^{\frac{1}{2}}<\infty.
    \end{align}
    The second integral is bounded away from $p$ by smoothness of $\phi$ outside of the branch point. We therefore deduce that for some $C=C(\phi)<\infty$
    \begin{align}\label{final_gagliardo2bis}
        &\left|\int_{\Sigma\setminus B(0,1/2)}u\s{du}{d|A|^2}_gd\vg\right|\leq C\int_{\Sigma\setminus B(0,1/2)}|u||du|_gd\vg\leq C\int_{\Sigma}u^2d\vg+C\int_{\Sigma}|du|_g^2d\vg\nonumber\\
        &\leq C\left(\int_{\Sigma}u^2d\vg\right)^{\frac{1}{2}}\left(\left(\int_{\Sigma}u^2d\vg\right)^{\frac{1}{2}}+\left(\int_{\Sigma}(\Delta_gu)^2d\vg\right)^{\frac{1}{2}}\right)
    \end{align}
    Furthermore, by either the $\epsilon$-regularity of Rivière (\cite{riviere1}) or the higher order Taylor expansion of Bernard--Rivière (\cite{beriviere}; see also \cite{classification}), we have
    \begin{align*}
        |\D A(x)|\leq \frac{\Gamma_1}{|x|^m}\qquad \text{for all}\;\, x\in B(0,1). 
    \end{align*}
    which yields
    \begin{align}\label{final_gagliardo3}
        \int_{B(0,1/2)}|u||\D u||\D |A|^2|dx&\leq 2\,\Gamma_1\Gamma_2\int_{B(0,1/2)}|u||\D u|\frac{dx}{|x|^{2m-1}}\nonumber\\
        &\leq 2\,\Gamma_1\Gamma_2\left(\int_{B(0,1/2)}u^2\frac{dx}{|x|^{2m+2-4\beta_1}}\right)^{\frac{1}{2}}\left(\int_{B(0,1/2)}\frac{|\D u|^2}{|x|^{2m}}|x|^{4(1-\beta_1)}dx\right)^{\frac{1}{2}}.
    \end{align}
    Therefore, we can use Corollary \ref{gagliardo_nirenberg_frak_m} to get for all $1-\beta_1\leq \beta<\min\ens{\dfrac{m-1}{4},1}$
    \begin{align}\label{final_gagliardo4}
        \int_{B(0,1/2)}\frac{|\D u|^2}{|x|^{2m}}|x|^{4(1-\beta_1)}dx\leq C_{m,\beta,2(1-\beta_1)}\left(\int_{B(0,1/2)}\frac{u^2}{|x|^{2m+2}}|x|^{4\beta}dx+\int_{B(0,1/2)}\frac{|\D^2u|^2}{|x|^{2m-2}}|x|^{4(1-\beta_1)}dx\right).
    \end{align}
    Notice that in the inequality \eqref{gagliardo_nirenberg_frak_m_ineq}, we require that $0<\gamma<1$, which is $0<2(1-\beta_1)<1$, or $\beta_1>1/2$, which is not restrictive due to the above choice. Furthermore, we can apply Theorem \ref{gagliardo_nirenberg_frak_m} for $a=0$ for functions in $W^{2,2}_{\omega}(\Sigma)$. Indeed, for all $u\in L^2_{\omega}(\Sigma)$, we have $u(p)=0$ (since the weight $|x|^{-(2m+2-4\beta_1)}$ does not belong to $L^1(B(0,1))$ for all $m\geq 2$ and $0<\beta_1<1$), and using the density of $C^{\infty}_c(\Sigma\setminus\ens{0})$ in
    \begin{align*}
        W^{2,2}(\Sigma)\cap\ens{u:u(p)=0},
    \end{align*}
    the inequality follows for all $u\in W^{2,2}_{\omega}(\Sigma)$. 
    Taking $\beta=\beta_1$ in the inequality, we need only estimate the second term, and this will require a non-trivial variation of the previous argument involving the Bochner identity. First, thanks to the Taylor expansion of \cite{beriviere} (see also \cite{classification}), there exists $\mu\in W^{2,p}(B(0,1))$ for all $p<\infty$ such that
    \begin{align*}
        \lambda(x)=(m-1)\log|x|+\mu.
    \end{align*}
    Then, assume without loss of generality that $u\in C^{\infty}(\Sigma\setminus\ens{p})$. Thanks to Lemma \ref{bochner_metric}, we have
    \begin{align}\label{bochner_appendix}
        \frac{|\D^2u|^2}{|x|^{2m-2}}=|\mathrm{Hess}(u)|_{g}^2\frac{e^{4\lambda}}{|x|^{2m-2}}-\frac{2}{|x|^{2m-2}}\left(\D u\cdot \D\lambda\right)\Delta u+\frac{2}{|x|^{2m-2}}\D\lambda\cdot \D|\D u|^2-\frac{2}{|x|^{2m-2}}|\D\lambda|^2|\D u|^2.
    \end{align}
    Thanks to the expansion of $\lambda$, we have
    \begin{align*}
        \frac{2}{|x|^{2m-2}}\D\lambda\cdot \D|\D u|^2=2(m-1)\frac{x}{|x|^{2m}}\cdot \D|\D u|^2+\frac{2}{|x|^{2m-2}}\D\mu\cdot \D|\D u|^2.
    \end{align*}
    Then, let $\eta\in C^{\infty}_c(B(0,1))$ such that $0\leq \eta\leq 1$ and $\eta=1$ in $B(0,1/2)$
    \begin{align}\label{final_gagliardo5}
        &\int_{B(0,1)}\eta \frac{2}{|x|^{2m-2}}\D\lambda\cdot \D|\D u|^2|x|^{4(1-\beta_1)}dx=2(m-1)\int_{B(0,1)}\eta\, \frac{1}{|x|^{2m-6+4\beta_1}}\dive\left(\frac{x}{|x|^2}\cdot \D u\right)dx\nonumber\\
        &+2\int_{B(0,1)}\eta\frac{1}{|x|^{2m-6+4\beta_1}}\D\mu\cdot \D|\D u|^2dx\nonumber\\
        &=-4(m-1)(m-3+2\beta_1)\int_{B(0,1)}\eta\frac{|\D u|^2}{|x|^{2m}}|x|^{4(1-\beta_1)}dx-2(m-1)\int_{B(0,1)}\left(\frac{x}{|x|^{2m-4+4\beta_1}}\cdot \D\eta\right)|\D u|^2dx\nonumber\\
        &-4(m-3+2\beta_1)\int_{B(0,1)}\eta \frac{x}{|x|^{2m-6+4\beta_1}}\cdot \D\mu |\D u|^2\nonumber\\
        &-2\int_{B(0,1)}\eta \frac{1}{|x|^{2m-6+4\beta_1}}\Delta\mu |\D u|^2dx-2\int_{B(0,1)}\frac{1}{|x|^{2m-6+4\beta_1}}(\D\eta\cdot \D\mu)|\D u|^2dx.
    \end{align}
    We trivially have
    \begin{align}\label{final_gagliardo6}
        &\left|2(m-1)\int_{B(0,1)}\left(\frac{x}{|x|^{2m-4+4\beta_1}}\cdot \D\eta\right)|\D u|^2dx\right|\nonumber\\
        &\leq 2^{2m}(m-1)\np{\D\eta}{\infty}{B(0,1)}\int_{B(0,1)\setminus\bar{B}(0,1/2)}|\D u|^2dx\nonumber\\
        &\leq 2^{2m}(m-1)\np{\D\eta}{\infty}{B(0,1)}\int_{\Sigma}|du|_{g}^2d\vg\nonumber\\
        &\leq 2^{2m}(m-1)\np{\D\eta}{\infty}{B(0,1)}\left(\int_{\Sigma}u^2d\vg\right)^{\frac{1}{2}}\left(\int_{\Sigma}(\Delta_gu)^2d\vg\right)^{\frac{1}{2}},
    \end{align}
    where we integrated by parts and used Cauchy--Schwarz inequality. Likewise, we have
    \begin{align}\label{final_gagliardo7}
        &\left|2\int_{B(0,1)}\frac{1}{|x|^{2m-6+4\beta_1}}(\D\eta\cdot \D\mu)|\D u|^2dx\right|\nonumber\\
        &\leq 2^{2m}\np{\D\eta}{\infty}{B(0,1)}\np{\D\mu}{\infty}{B(0,1)}\left(\int_{\Sigma}u^2d\vg\right)^{\frac{1}{2}}\left(\int_{\Sigma}(\Delta_gu)^2d\vg\right)^{\frac{1}{2}}.
    \end{align}
    Thanks to the Liouville equation, we deduce that
    \begin{align}\label{final_gagliardo8}
        &\left|\int_{B(0,1)}\eta \frac{1}{|x|^{2m-6+4\beta_1}}\Delta\mu |\D u|^2dx\right|=\left|\int_{B(0,1)}\eta\, e^{-2\mu}K_g|\D u|^2|x|^{4(1-\beta_1)}d\vg\right|\nonumber\\
        &\leq 2\np{e^{2\lambda}K_g}{\infty}{B(0,1)}\exp\left(4\np{\mu}{\infty}{B(0,1)}\right)\epsilon^{4(1-\beta_1)}\int_{B(0,\epsilon)}\eta\frac{|\D u|^2}{|x|^{2m}}dx\nonumber\\
        &+\left|\int_{B(0,1)\setminus\bar{B}(0,\epsilon)}\eta\,e^{-2\mu}K_g|\D u|^2|x|^{4(1-\beta_1)}d\mathrm{vol}_g\right|
    \end{align}
    where we used (\cite{beriviere}) that $e^{2\lambda}K_g\in L^{\infty}(\Sigma)$. Furthermore, we have 
    \begin{align}\label{final_gagliardo9}
        \left|\int_{B(0,1)\setminus\bar{B}(0,\epsilon)}\eta\,e^{-2\mu}K_g|\D u|^2|x|^{4(1-\beta_1)}d\mathrm{vol}_g\right|&\leq 2\np{e^{2\lambda}K_g}{\infty}{B(0,1)}\exp\left(4\np{\mu}{\infty}{B(0,1)}\right)\frac{1}{\epsilon^{2m-4+4\beta_1}}\nonumber\\
        &\times \left(\int_{\Sigma}u^2d\vg\right)^{\frac{1}{2}}\left(\int_{\Sigma}(\Delta_gu)^2d\vg\right)^{\frac{1}{2}}.
    \end{align}
    Therefore, we deduce by \eqref{bochner_appendix}, \eqref{final_gagliardo5}, \eqref{final_gagliardo6}, \eqref{final_gagliardo7}, \eqref{final_gagliardo8}, and \eqref{final_gagliardo9} that for some constant $C<\infty$ depending only on $\phi$ that
    \begin{align}\label{final_gagliardo10}
        &\int_{B(0,1/2)}\frac{|\D^2u|^2}{|x|^{2m-2}}|x|^{4(1-\beta_1)}dx\leq \int_{B(0,1)}\eta\frac{|\D^2 u|^2}{|x|^{2m-2}}|x|^{4(1-\beta_1)}dx=\int_{B(0,1)}\eta|\mathrm{Hess}(u)|_{g}^2\frac{e^{4\lambda}}{|x|^{2m-2}}|x|^{4(1-\beta_1)}dx\nonumber\\
        &-2\int_{B(0,1)}\frac{\eta}{|x|^{2m-2}}(\D u\cdot \D\lambda)\Delta u\,|x|^{4(1-\beta_1)}dx+\int_{B(0,1)}\eta\frac{2}{|x|^{2m-2}}\D\lambda\cdot \D|\D u|^2|x|^{4(1-\beta_1)}|x|^{4(1-\beta_1)}dx\nonumber\\
        &-2\int_{B(0,1)}\frac{\eta}{|x|^{2m-2}}|\D\lambda|^2|\D u|^2|x|^{4(1-\beta_1)}dx\nonumber\\
        &\leq C\int_{\Sigma}|\mathrm{Hess}(u)|_{g}^2d\vg\nonumber\\
        &-2\int_{B(0,1)}\frac{\eta}{|x|^{2m-2}}(\D u\cdot \D\lambda)\Delta u\,|x|^{4(1-\beta_1)}dx-2\int_{B(0,1)}\frac{\eta}{|x|^{2m-2}}|\D\lambda|^2|\D u|^2|x|^{4(1-\beta_1)}dx\nonumber\\
        &-\left(4(m-1)(m-3+2\beta_1)-C\,\epsilon^{4(1-\beta_1)}\right)\int_{B(0,1)}\eta\frac{|\D u|^2}{|x|^{2m-2}}|x|^{4(1-\beta_1)}dx\nonumber\\
        &+C\left(1+\frac{1}{\epsilon^{2m-4+4\beta_1}}\right)\left(\int_{\Sigma}u^2d\vg\right)^{\frac{1}{2}}\left(\int_{\Sigma}(\Delta_gu)^2d\vg\right)^{\frac{1}{2}}\nonumber\\
        &\leq C\int_{\Sigma}|\mathrm{Hess}(u)|_{g}^2d\vg+\frac{1}{2}\int_{B(0,1)}\eta |x|^{2-2m}(\Delta u)^2|x|^{4(1-\beta_1)}dx\nonumber\\
        &-\left(4(m-1)(m-3+2\beta_1)-C\,\epsilon^{4(1-\beta_1)}\right)\int_{B(0,1)}\eta\frac{|\D u|^2}{|x|^{2m-2}}|x|^{4(1-\beta_1)}dx\nonumber\\
        &+C\left(1+\frac{1}{\epsilon^{2m-4+4\beta_1}}\right)\left(\int_{\Sigma}u^2d\vg\right)^{\frac{1}{2}}\left(\int_{\Sigma}(\Delta_gu)^2d\vg\right)^{\frac{1}{2}},
    \end{align}
    where we used Cauchy's inequality to estimate
    \begin{align}\label{final_gagliardo11}
        &-2\int_{B(0,1)}\frac{\eta}{|x|^{2m-2}}(\D u\cdot \D\lambda)\Delta u\,dx|x|^{4(1-\beta_1)}\nonumber\\
        &\leq \frac{1}{2}\int_{B(0,1)}\eta|x|^{2-2m}(\Delta u)^2dx+2\int_{B(0,1)}\frac{\eta}{|x|^{2m-2}}(\D u\cdot \D\lambda)^2dx\nonumber\\
        &\leq \frac{1}{2}\int_{B(0,1)}\eta|x|^{2-2m}(\Delta u)^2|x|^{4(1-\beta_1)}dx+2\int_{B(0,1)}\frac{\eta}{|x|^{2m-2}}|\D u|^2| \D\lambda|^2|x|^{4(1-\beta_1)}dx
    \end{align}
    Choosing 
    \begin{align*}
        \epsilon=\left(\frac{2(m-1)(m-3+2\beta_1)}{C}\right)^{\frac{1}{4(1-\beta_1)}},
    \end{align*}
    we deduce that
    \begin{align}\label{final_gagliardo12}
        &2(m-1)(m-3+2\beta_1)\int_{B(0,1)}\eta \frac{|\D u|^2}{|x|^{2m-2}}|x|^{4(1-\beta_1)}dx+\int_{B(0,1)}\eta\frac{|\D^2u|^2}{|x|^{2m-2}}|x|^{4(1-\beta_1)}dx\nonumber\\
        &\leq C\int_{\Sigma}|\mathrm{Hess}(u)|_{g}^2d\vg+C\int_{\Sigma}(\Delta_gu)^2d\vg\nonumber\\
        &+C\left(1+\frac{1}{\big(2(m-1)(m-3+2\beta_1)\big)^{\frac{1}{4(1-\beta_1)}}}\right)\left(\int_{\Sigma}u^2d\vg\right)^{\frac{1}{2}}\left(\int_{\Sigma}(\Delta_gu)^2d\vg\right)^{\frac{1}{2}}.
    \end{align}
    Notice that the estimate is non-trivial even for $m=2$ as $2\beta_1>1$. Now, using Bochner's formula, we get by closeness of $\Sigma$ the estimate
    \begin{align}\label{final_gagliardo13}
        \int_{\Sigma}|\mathrm{Hess}(u)|_g^2d\vg&=\int_{\Sigma}\left(\Delta_gu\right)^2d\vg-2\int_{\Sigma}K_g|du|_g^2d\vg\nonumber\\
        &\leq \int_{\Sigma}(\Delta_gu)^2d\vg+\int_{\Sigma}|du|_g^2|A|^2d\vg.
    \end{align}
    Therefore, we deduce that 
    \begin{align}\label{final_gagliardo14}
        &2(m-1)(m-3+2\beta_1)\int_{B(0,1)}\eta \frac{|\D u|^2}{|x|^{2m-2}}|x|^{4(1-\beta_1)}dx+\int_{B(0,1)}\eta\frac{|\D^2u|^2}{|x|^{2m-2}}|x|^{4(1-\beta_1)}dx\\
        &\leq C\int_{\Sigma}(\Delta_gu)^2d\vg+C\int_{\Sigma}|du|_g^2|A|^2d\vg\nonumber\\
        &+C\left(1+\frac{1}{\big(2(m-1)(m-3+2\beta_1)\big)^{\frac{1}{4(1-\beta_1)}}}\right)\left(\int_{\Sigma}u^2d\vg\right)^{\frac{1}{2}}\left(\int_{\Sigma}(\Delta_gu)^2d\vg\right)^{\frac{1}{2}}
    \end{align}
    Finally, \eqref{final_gagliardo1}, \eqref{final_gagliardo2}, \eqref{final_gagliardo2bis}, \eqref{final_gagliardo3}, \eqref{final_gagliardo4}, \eqref{final_gagliardo5}, and \eqref{final_gagliardo14}
    \begin{align}\label{final_gagliardo15}
        &\int_{\Sigma}|du|_g^2|A|^2d\vg\leq \left(\int_{\Sigma}u^2|A|^4d\vg\right)^{\frac{1}{2}}\left(\int_{\Sigma}(\Delta_gu)^2\right)^{\frac{1}{2}}\nonumber\\
        &+C\left(\int_{\Sigma}u^2d\vg\right)^{\frac{1}{2}}\left(\left(\int_{\Sigma}u^2d\vg\right)^{\frac{1}{2}}+\left(\int_{\Sigma}(\Delta_gu)^2d\vg\right)^{\frac{1}{2}}\right)\nonumber\\
        &+C\left(\int_{B(0,1/2)}\frac{u^2}{|x|^{2m+2}}|x|^{4\beta_1}\right)^{\frac{1}{2}}\left(\int_{B(0,1/2)}\frac{u^2}{|x|^{2m+2}}|x|^{4\beta_1}+\int_{\Sigma}(\Delta_gu)^2d\vg+\int_{\Sigma}|du|_g^2|A|^2d\vg
        \right.\nonumber\\
        &\left.+\left(1+\frac{1}{\big(2(m-1)(m-3+2\beta_1)\big)^{\frac{1}{4(1-\beta_1)}}}\right)\left(\int_{\Sigma}u^2d\vg\right)^{\frac{1}{2}}\left(\int_{\Sigma}(\Delta_gu)^2d\vg\right)^{\frac{1}{2}}\right)^{\frac{1}{2}}\nonumber\\
        &\leq C\left(\int_{\Sigma}u^2\,\omega\,d\vg\right)^{\frac{1}{2}}\left(\int_{\Sigma}u^2\,\omega\,d\vg+\int_{\Sigma}(\Delta_gu)^2d\vg\right)^{\frac{1}{2}}+\frac{1}{2}\int_{\Sigma}|du|_g^2|A|^2d\vg,
    \end{align}
    which furnishes the needed inequality
    \begin{align}\label{final_gagliardo16}
        \int_{\Sigma}|du|_g^2|A|^2d\vg\leq C\left(\int_{\Sigma}u^2\,\omega\,d\vg\right)^{\frac{1}{2}}\left(\int_{\Sigma}u^2\,\omega\,d\vg+\int_{\Sigma}(\Delta_gu)^2d\vg\right)^{\frac{1}{2}}.
    \end{align}
    Therefore, we deduce that for all $u\in W^{2,2}_{\omega}(\Sigma)$, each component of the index form \eqref{index_form} is bounded. Notice that the last term is bounded thanks to \eqref{final_gagliardo3}. Therefore, the estimates \eqref{final_gagliardo3} and \eqref{final_gagliardo16} show that 
    \begin{align}\label{final_gagliardo17}
        Q_{\phi}(u)&\geq \frac{1}{4}\int_{\Sigma}(\Delta_gu)^2d\vg-C\int_{\Sigma}u^2\,\omega\,d\vg-C\left(\int_{\Sigma}u^2\,\omega\,d\vg\right)^{\frac{1}{2}}\left(\int_{\Sigma}u^2\,\omega\,d\vg+\int_{\Sigma}(\Delta_gu)^2d\vg\right)^{\frac{1}{2}}\nonumber\\
        &\geq \frac{1}{8}\int_{\Sigma}(\Delta_gu)^2d\vg-C\int_{\Sigma}u^2\,\omega\,d\vg.
    \end{align}
    We can therefore follow the exact same steps as in \cite[Lemma 1.4.5 p. 96]{willmore_scs} (see also \cite[Lemma B.1]{riviere_morse_scs}) to complete the proof of the theorem. 
\end{proof}

\subsection{The Second Fourth-Order Elliptic Operator Associated to Ends of Multiplicity Two}

In this section, we investigate the behaviour of the operator 
\begin{align*}
    \mathscr{D}_m&=\Delta^2-\frac{6(m^2-1)}{|x|^2}+4(m^2-1)\left(\frac{x}{|x|^2}\right)^t\cdot \D^2(\,\cdot\,)\cdot \left(\frac{x}{|x|^2}\right)+\frac{8(m-1)^2}{|x|^2}\frac{x}{|x|^2}\cdot \D u\\
    &+\frac{(m+1)(m-1)^2(m-3)}{|x|^4}
\end{align*}
for $m=2$, in which case
\begin{align*}
    \mathscr{D}_2=\Delta^2-\frac{18}{|x|^2}\Delta+4\left(\frac{x}{|x|^2}\right)^t\cdot \D^2(\,\cdot\,)\cdot \left(\frac{x}{|x|^2}\right)+\frac{8}{|x|^2}\frac{x}{|x|^2}\cdot \D u-\frac{3}{|x|^4}.
\end{align*}
Recall that the linearised ordinary differential equation associated to the $n$-th Fourier mode is 
\begin{align*}
    &X^4-2((m-1)^2+n^2+1)X^2+2(m-1)^2+2n^2+1\\
    &+\left(n^4+n^2(6(m-1)^2-4)+(m+1)(m-1)^2(m-3)\right)\\
    &=X^4-2(n^2+2)X^2+n^4+4n^2=(X^2-(n^2+2))^2-4=(X^2-n^2)(X^2-(n^2+4)).
\end{align*}
This implies that for $n\neq 0$, a basis of solutions of $\Pi_{n^2}(\mathscr{D}_2)u=0$ is given by
\begin{align*}
    r^{1+n},r^{1-n},r^{1+\sqrt{n^2+4}},r^{1-\sqrt{n^2+4}},
\end{align*}
while for $n=0$, the solutions are given by 
\begin{align*}
    r,r\log(r),r^3,r^{-1}.
\end{align*}
This explains why the Gagliardo--Nirenberg estimate should fail for $m=2$, due to the need of controlling a logarithm as in Theorem \ref{lm_last_lemma}. We will not pursue this direction further. 

\subsection{Example of Bubbling with Branch Point of Multiplicity Two in Codimension Two}

Identify $\R^4$ with $\C^2$, let $p_1,p_2,p_3\in \C$ be two distinct points and let $F:S^2\setminus\ens{p_1,p_2}\rightarrow \C^2$ be a complex minimal sphere of finite total curvature with three embedded planar ends. By the $3$-transitivity of the conformal group, we can fix $p_3=\infty$. Then, there exists $(a_1,b_2),(a_2,b_2),(a_3,b_3)\in \C^2\setminus\ens{0}$ such that
\begin{align*}
    F(z)=\left(\frac{a_1}{z-p_1}+\frac{a_2}{z-p_2}+a_3\,z,\frac{b_1}{z-p_1}+\frac{b_2}{z-p_2}+b_3\,z\right).
\end{align*}
This is a complete minimal surface with total curvature $-8\pi$ (\cite{shiohama,hoffman_gauss}). 
The map $F$ also needs to be an immersion, which imposes open conditions on the parameters. If we impose embeddedness, an elementary analysis (\cite{hoffman_gauss,lee_minimal_surfaces}) shows that the following conditions must hold:
\begin{align*}
	\det\begin{pmatrix}
		a_1 & b_1\\
		a_2 & b_2
	\end{pmatrix}\qquad \mathrm{rank}\begin{pmatrix}
	a_1 & a_2 & a_3\\
	b_1 & b_2 & b_3
	\end{pmatrix}=2.
\end{align*}
To find the example of bubbling of multiplicity two, we proceed as in \cite{marque_minimal_bubbling}. Let $\mu>0$, choose $p_1=\mu$, $p_2=i\mu$ and define
\begin{align}
	F_{\mu}(z)&=\left(\frac{a_1(\mu)}{z-\mu}+\frac{a_2(\mu)}{z-i\mu},z\right)\nonumber\\
	&=\left(\frac{(a_1(\mu)+a_2(\mu))z-(a_2(\mu)+i\,a_1(\mu))\mu}{z^2-(1+i)\mu z+i\mu^2},z\right).
\end{align}
Therefore, we choose $a_2(\mu)=-a_1(\mu)$ and 
\begin{align*}
	a_1(\mu)=-\frac{1}{(-1+i)\mu}=\frac{1+i}{2\mu},
\end{align*}
which yields
\begin{align*}
	F_{\mu}(z)=\left(\frac{1}{z^2-(1+i)\mu\,z+i\mu^2},z\right).
\end{align*}
As $\mu\rightarrow 0$, for all $z\in \C\setminus\ens{0}$, we have
\begin{align*}
	F_{\mu}(z)\conv{\mu\rightarrow 0}F_0(z)=\left(\frac{1}{z^2},z\right),
\end{align*}
which is a minimal surface with two ends, one of multiplicity $2$ and one of multiplicity $1$. Therefore, its total curvature is equal to 
\begin{align*}
    -2\pi(0-2+(2+1)+(1+1))=-6\pi,
\end{align*}
which shows that if  
\begin{align*}
	\phi_{\mu}=\frac{F_{\mu}}{|F_{\mu}|^2}:S^2\rightarrow \R^4,
\end{align*}
then $\{\phi_{\mu}\}_{\mu>0}$ is a family of Willmore spheres of energy $12\pi$ (notice that this is consistent with the Li--Yau inequality) which converges outside $0$ to a branched Willmore sphere $\phi_0:S^2\rightarrow\R^4$. Since $\phi_0$ has a unique branch point of multiplicity $2$, we deduce that
\begin{align*}
    \int_{S^2}K_{\phi_0}d\mathrm{vol}_{g_{\phi_0}}=4\pi+2\pi(2-1)=6\pi. 
\end{align*}
This implies that $W(\phi_0)=6\pi+6\pi=12\pi$, and we deduce that a minimal bubble is formed at $0$, and that it has a single end of multiplicity $2$ and total curvature $-2\pi$. By a classical theorem of Hoffman--Osserman (\cite[Theorem 6.2]{hoffman_gauss}), the bubble is the graph of the function $z\rightarrow z^2$, or in other words, up to conformal transformations, the bubble $\vec{\Psi}:\C\rightarrow \C^2$ is given by $\vec{\Psi}(z)=(z,z^2)$ for all $z\in \C$.

    \nocite{}
	 \bibliographystyle{plain}
	 \bibliography{biblio_full}

\end{document}